\documentclass[9pt]{article}
\usepackage{amssymb}
\usepackage{tikz}
\usetikzlibrary{arrows, calc, decorations.pathmorphing, backgrounds, positioning, fit, petri, automata}
\definecolor{grey1}{rgb}{0.5,0.5,0.5}
\usepackage[top=2.54cm, bottom=2.54cm, left=2.2cm, right=2.2cm]{geometry}
\usepackage{algorithmicx}
\usepackage{algcompatible}
\usepackage{algorithm}

\usepackage{changebar}

\usepackage{hyperref}
\hypersetup{ 
  colorlinks,
  citecolor=blue,
  linkcolor=red,
  urlcolor=magenta}

\usepackage{mathtools}
\usepackage{graphicx}
\usepackage{subfigure}
\usepackage{amsmath,amsthm,bm,bbm}
\usepackage{rotating}
\usepackage{mathrsfs}
\usepackage{chngcntr}
\usepackage{apptools}
\usepackage[titletoc,title]{appendix}
\usepackage{comment}

\usepackage{xcolor}         
\definecolor{grau}{rgb}{0.8,0.8,0.8}

\newcommand{\chen}[1]{\color{orange}}
\usepackage{setspace,lscape,longtable}
\usepackage{amsmath,amsthm,bm}
\usepackage{fullpage}  
\usepackage[shortlabels]{enumitem}
\usepackage{titlesec}
\usepackage[english]{babel}
\numberwithin{equation}{section}
\numberwithin{equation}{section}

\newtheorem{theorem}{Theorem}[section]

\newtheorem{lemma}[theorem]{Lemma}
\newtheorem{proposition}{Proposition}[section]
\theoremstyle{remark}
\newtheorem{remark}{Remark}

\newtheorem{assumption}{Assumption}
\newtheorem*{example}{Example}

\newtheorem{result}{Result}[section]

\newcommand{\prob}{{\mathbb{P}}}
\newcommand{\var}{{\mathrm{var}}}
\newcommand{\expect}{\mathbb{E}}

\newcommand{\bea}{\begin{equation}\begin{aligned}}
\newcommand{\eae}{\end{aligned}\end{equation}}

\newcommand{\transpose}{^{\mathrm{T}}}

\newcommand{\bLambda}{{\bm{\Lambda}}}

\newcommand{\calA}{{\mathcal{A}}}

\newcommand{\calE}{{\mathcal{E}}}

\newcommand{\calJ}{{\mathcal{J}}}

\newcommand{\calP}{{\mathcal{P}}}

\newcommand{\calR}{{\mathcal{R}}}
\newcommand{\calS}{{\mathcal{S}}}

\newcommand{\bT}{{\mathbf{T}}}
\newcommand{\bU}{{\mathbf{U}}}
\newcommand{\bP}{{\mathbf{P}}}

\newcommand{\bY}{{\mathbf{Y}}}
\newcommand{\bv}{{\mathbf{v}}}
\newcommand{\bx}{{\mathbf{x}}}
\newcommand{\bX}{{\mathbf{X}}}
\newcommand{\bA}{{\mathbf{A}}}
\newcommand{\bB}{{\mathbf{B}}}

\newcommand{\bD}{{\mathbf{D}}}
\newcommand{\bE}{{\mathbf{E}}}
\newcommand{\bF}{{\mathbf{F}}}
\newcommand{\bG}{{\mathbf{G}}}
\newcommand{\bH}{{\mathbf{H}}}
\newcommand{\bM}{{\mathbf{M}}}

\newcommand{\bR}{{\mathbf{R}}}

\newcommand{\bS}{{\mathbf{S}}}
\newcommand{\bW}{{\mathbf{W}}}

\newcommand{\bt}{{\mathbf{t}}}
\newcommand{\bu}{{\mathbf{u}}}
\newcommand{\bq}{{\mathbf{q}}}

\newcommand{\by}{{\mathbf{y}}}

\newcommand{\be}{{\mathbf{e}}}
\newcommand{\bb}{{\mathbf{b}}}

\newcommand{\bV}{{\mathbf{V}}}

\newcommand{\bSigma}{{\bm{\Sigma}}}

\newcommand{\eye}{{\mathbf{I}}}
\newcommand{\one}{{\mathbf{1}}}

\newcommand{\bPi}{{\bm{\Pi}}}

\newcommand{\bpsi}{{\bm{\psi}}}

\newcommand{\zero}{{\bm{0}}}
\newcommand{\eps}{\epsilon}

\newcommand{\keywords}[1]{\par\addvspace\baselineskip\noindent\enspace\ignorespaces \textbf{Keywords: }#1}

\author{Fangzheng Xie\thanks{Department of Statistics, Indiana University}
\footnotemark[1] \thanks{Correspondence should be addressed to Fangzheng Xie (fxie@iu.edu)}
\and
Yichi Zhang\thanks{Department of Biostatistics and Bioinformatics \& The Fuqua School of Business, Duke University}
}
\title{Higher-Order Entrywise Eigenvectors Analysis of Low-Rank Random Matrices: Bias Correction, Edgeworth Expansion, and Bootstrap}
\begin{document}
\allowdisplaybreaks

\maketitle

\begin{abstract}
Understanding the distributions of spectral estimators in low-rank random matrix models, also known as signal-plus-noise matrix models, is fundamentally important in various statistical learning problems, including network analysis, matrix denoising, and matrix completion. This paper investigates the entrywise eigenvector distributions in a broad range of low-rank signal-plus-noise matrix models by establishing their higher-order accurate stochastic expansions. At a high level, the stochastic expansion states that the eigenvector perturbation approximately decomposes into the sum of a first-order term and a second-order term, where the first-order term in the expansion is a linear function of the noise matrix, and the second-order term is a linear function of the squared noise matrix. Our theoretical finding is used to derive the bias correction procedure for the eigenvectors. We further establish the Edgeworth expansion formula for the studentized entrywise eigenvector statistics. In particular, under mild conditions, we show that Cram\'er's condition on the smoothness of noise distribution is not required, thanks to the self-smoothing effect of the second-order term in the eigenvector stochastic expansion. The Edgeworth expansion result is then applied to justify the higher-order correctness of the residual bootstrap procedure for approximating the distributions of the studentized entrywise eigenvector statistics. 
\end{abstract}


\tableofcontents

\section{Introduction}
\label{sec:introduction}

\subsection{Background}
\label{sub:background}
In the contemporary world of data science, low-rank structures are pervasive not only in statistics but also in an enormous variety of disciplines, including network analysis in social science \cite{HOLLAND1983109,nickel2008random,young2007random}, connectomics in neuroscience \cite{eichler2017complete,8570772}, gene expression analysis in biology \cite{10.1093/bioinformatics/btt091,10.1093/bioinformatics/17.suppl_1.S279,10.1214/17-AOAS1123}, compressed sensing in signal processing \cite{1614066,eldar2012compressed}, distance matrix analysis in molecular chemistry and remote sensing \cite{https://doi.org/10.1002/jcc.540140115, 1143830}, face recognition in computer vision \cite{1407873}, collaborative filtering \cite{bennett2007netflix}, and recommendation systems \cite{goldberg1992using}.

Motivated by the urgent need for learning high-dimensional matrix data emerging in various application fields, random matrix models with low expected ranks, also known as signal-plus-noise matrix models \cite{cape2019signal} or the generalized spiked Wigner model \cite{doi:10.1080/01621459.2020.1840990,yau2012universality}, and the accompanying theory and methods, have been developed. A broad range of statistical models can also be represented by the low-rank random matrix models, such as low-rank matrix denoising \cite{chatterjee2015,10.1214/14-AOS1257,SHABALIN201367}, matrix completion \cite{tight_oracle_inequalities,candes2009exact,5466511}, network models \cite{abbe2017community,airoldi2008mixed,HOLLAND1983109,PhysRevE.83.016107,young2007random}, and topic models \cite{blei2003latent,doi:10.1080/01621459.2022.2089574,10.1145/312624.312649}. Theoretical endeavors for analyzing low-rank random matrix models are largely based on both the matrix perturbation theory \cite{10.1214/17-AOS1541,cape2017two,doi:10.1137/0707001,wedin1972perturbation,doi:10.1093/biomet/asv008} and the random matrix theory (RMT), which was originally investigated by the seminal work \cite{PhysRev.98.145}. 
\subsection{Overview}
\label{sub:overview}

Consider the low-rank signal-plus-noise random matrix model $\bA = \bP + \bE$, where $\bP$ is an $n\times n$ symmetric deterministic low-rank signal matrix, $\bE$ is a symmetric random noise matrix whose upper triangular entries are independent mean-zero random variables, and $\bA$ is the noisy version of $\bP$, which is the only accessible data to the practitioners whereas neither $\bP$ nor $\bE$ is observed. This paper investigates the entrywise behavior of the sample eigenvector $\widehat{\bu}_k$ of $\bA$ (associated with the $k$th largest eigenvalue of $\bA$) as an estimator for the unknown population eigenvector $\bu_k$ of $\bP$ (associated with the $k$th largest eigenvalue of $\bP$) when the noise level of $\bE$ is negligible compared to the signal strength encoded in $\bP$. Roughly, one of our main results reads
\[
\widehat{\bu}_k - \bu_k \approx \underbrace{\mbox{linear function of }\bE}_{\text{first-order term}} + \underbrace{\mbox{linear function of }\bE^2}_{\text{second-order term}}.
\]
Although the dependence of $\widehat{\bu}_k$ on $\bA$ is complicated and highly nonlinear, our analysis shows a fine-grained approximation of $\widehat{\bu}_k$ using the sum of a linear function of $\bE$ that plays the role of the first-order term and a linear function of $\bE^2$ that serves as the second-order term. The first-order term 
has been explored recently \cite{10.1214/19-AOS1854,cape2019signal,xie2021entrywise}, while the second-order term, to our best knowledge, has never been studied before. In particular, the second-order term has a rather simple structure because it only contains a linear function of $\bE^2$ and does not involve other  quadratic forms of $\bE$. Under mild conditions, we prove that the approximation error is uniformly small, thereby establishing a higher-order entrywise analysis of $\widehat{\bu}_k - \bu_k$. 

Our higher-order entrywise eigenvector stochastic expansion also leads to sharp results of the entrywise sample eigenvector distribution beyond asymptotic normality. Suppose we are interested in the inference for the $i$th entry of $\widehat{\bu}_k$, namely $\widehat{u}_{ik}$, with $u_{ik}$ being the $i$th entry of $\bu_k$. Let $\widehat{s}_{ik}^2$ be an estimator of $\var(\widehat{u}_{ik})$, the formal definition of which is deferred to Section \ref{sub:entrywise_eigenvector_edgeworth_expansion}. We focus on the sampling distribution of the studentized entrywise eigenvector statistic $T_{ik} = (\widehat{u}_{ik} - u_{ik} - b_{ik})/\widehat{s}_{ik}$, where $b_{ik}$ is the additional second-order bias term. Our work establishes the first provable higher-order accurate distributional approximation to the sampling distribution of $T_{ik}$ via an Edgeworth expansion formula. From the technical perspective, we show that, under mild conditions, the second-order term has a self-smoothing effect on the cumulative distribution function of $T_{ik}$, which helps to establish the Edgeworth expansion without Cram\'er's condition on the smoothness of the noise distribution. The self-smoothing phenomenon is a blessing of the high-dimensional nature of the signal-plus-noise matrix model, which is typically unavailable in classical low-dimensional models. Such a phenomenon is similar to the observation in \cite{10.1214/21-AOS2125} for the sampling distributions of network moment statistics. Figure \ref{fig:Tik_Edgeworth_expansion_ER_n80} below visualizes the comparison among the true cumulative distribution function (CDF) of $T_{11}$, the standard normal CDF, and our Edgeworth expansion approximation in a toy synthetic example with $n = 80$, where the distribution of $\bE$ is purely discrete. It demonstrates the advantage of our result over the coarse normal approximation with a moderately small matrix size.
\begin{figure}[h]
\centerline{
\includegraphics[width = 11cm]{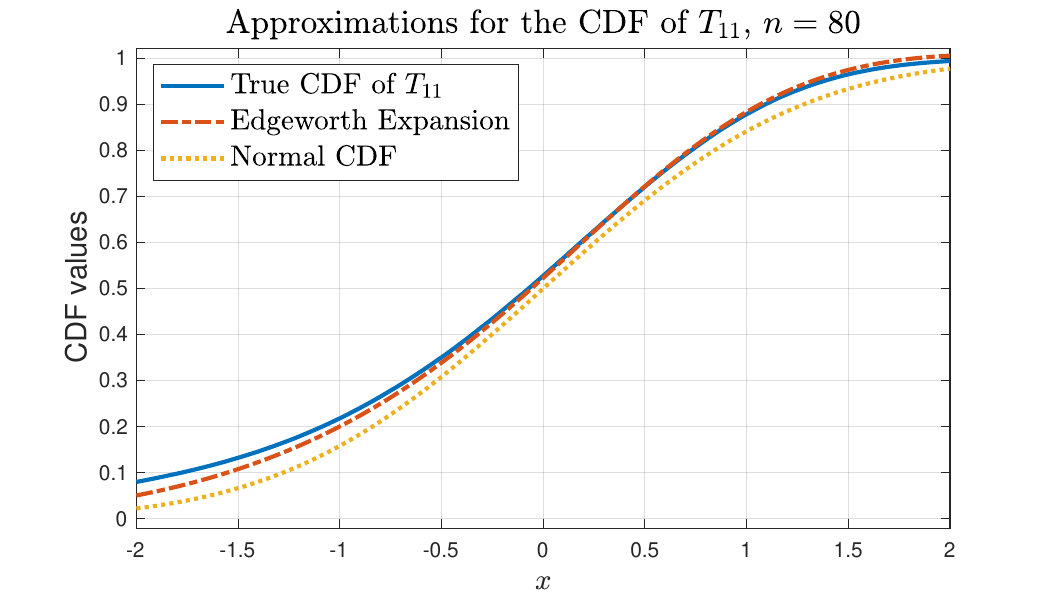}
}
\caption{Comparison among the true CDF of $T_{11}$, the standard normal CDF, and our Edgeworth expansion approximation in a toy synthetic example. Here, $n = 80$, the upper triangular entries of $\bE$ are independently generated as $X$ with $\prob(X = 4) = (1/5)n^{-1/3}$, $\prob(X = -1) = (4/5)n^{-1/3}$, $\prob(X = 0) = 1 - n^{-1/3}$, and $\bP$ is a matrix whose entries are all $n^{-5/12}$. }
\label{fig:Tik_Edgeworth_expansion_ER_n80}
\end{figure}
Our Edgeworth expansion result is further applied to justify the higher-order correctness of the residual bootstrap scheme for approximating the distribution of $T_{ik}$. 

\subsection{Related work}
\label{sub:related_work}
Spectral methods, or more precisely, methods based on spectral decomposition, are ubiquitous for analyzing high-dimensional low-rank matrix and network data. They are conceptually simple and methodologically easy to implement and interpret. Depending on the context, these spectral estimators can be either directly applied to subsequent inferential tasks such as community detection \cite{rohe2011,868688,sussman2012consistent} or goodness-of-fit \cite{10.1214/15-AOS1370} in network models, or ``warm-start'' certain iterative algorithms to obtain refined estimators \cite{MAL-079,doi:10.1073/pnas.1910053116,8811622,10.5555/3122009.3153016,8007083,ImplicitRegularizationMa2020,wu2022statistical,xie2021entrywise,xie2019efficient}. As mentioned in Section \ref{sub:background}, the spectral methods are fundamentally supported by the purely algebraic and deterministic matrix perturbation theory \cite{10.1214/17-AOS1541,cape2017two,doi:10.1137/0707001,wedin1972perturbation,doi:10.1093/biomet/asv008} as well as probabilistic tools such as random perturbation analysis \cite{cape2019signal,OROURKE201826} and random matrix theory \cite{arnold1967asymptotic,arnold1971wigner,bai2010spectral,AIHPB_2008__44_3_447_0,10.1214/009117905000000233,10.1214/20-AOS1960,10.1214/21-AOS2143,bourgade2017eigenvector,https://doi.org/10.1002/cpa.21895,dekel2007eigenvectors,erdHos2013delocalization,ERDOS20121435,erdos2013,doi:10.1080/01621459.2020.1840990,furedi1981eigenvalues,10.1214/009117906000000917,https://doi.org/10.1002/cpa.21450,10.1214/13-AOP855,knowles2017anisotropic,marchenko1967distribution,AIHPB_2013__49_1_64_0,doi:10.1142/S2010326312500153,rudelson2016no,tao2012topics,yau2012universality}. 

Recently, the entrywise eigenvector/singular vector analysis has been attracting research attention in the network analysis community and related fields \cite{10.1214/19-AOS1854,agterberg2021entrywise,athreya2016limit,cape2019signal,cape2017two,pmlr-v83-eldridge18a,lyzinski2014,doi:10.1080/01621459.2020.1751645,tang2018,xie2021entrywise,xie2019efficient}. It provides fine-grained uniform control rather than the on-average guarantee of some popular spectral-method-based matrix learning algorithms. One canonical example is the exact recovery of the community memberships in the stochastic block model and its variants \cite{10.1214/19-AOS1854,sussman2012consistent}. These entrywise eigenvector/singular vector analysis results mainly focus on either the zeroth-order uniform deviation result $\|\widehat{\bu}_k - \bu_k\|_\infty$ \cite{cape2019signal,cape2017two,doi:10.1080/01621459.2020.1751645} or the first-order fluctuation results \cite{10.1214/19-AOS1854}, which typically imply the distributional approximation to $\widehat{u}_{ik} - u_{ik}$ \cite{agterberg2021entrywise,athreya2016limit,cape2019signal,tang2018,xie2021entrywise,xie2019efficient}. Informally, these results establish (or imply) that $(\widehat{u}_{ik} - u_{ik})/s_{ik}$ converges to $\mathrm{N}(0, 1)$ weakly under various contexts, where $s_{ik}^2$ is the asymptotic variance of $\widehat{u}_{ik}$. Nevertheless, none of the above literature addresses the higher-order entrywise eigenvector analysis beyond asymptotic normality. Our work fills this blank in the literature by establishing the first ever higher-order accurate distributional approximation to $\widehat{u}_{ik} - u_{ik}$ through a second-order stochastic expansion and an Edgeworth expansion formula. The Edgeworth expansion result also sharpens the existing Berry-Esseen bounds in \cite{agterberg2021entrywise,xie2021entrywise}. 

In classical low-dimensional models with independent and identically distributed (i.i.d.) random samples, Edgeworth expansion \cite{Edgeworth1905,10.1007/BF02392223,10.2307/2237255} provides a higher-order accurate approximation to the CDF of the studentized sample mean beyond the central limit theorem and the Berry-Esseen bound. Roughly speaking, if $X_1,\ldots,X_n$ are i.i.d. random variables, then the Edgeworth expansion of the studentized sample mean $\bar{X}_n = (1/n)\sum_{i = 1}^nX_i$ takes the form of
\[
\prob\left\{\frac{\sqrt{n}(\bar{X}_n - \expect X_i)}{S_n}\leq x\right\} = \Phi(x) + \frac{(2x^2 + 1)}{6\sqrt{n}}\frac{\expect X_i^3}{\sigma^3} + O(n^{-1}),
\]
where $S_n^2$ is the sample variance and $\sigma^2 = \var(X_i)$. Further expansion with even higher orders is obtainable \cite{hall2013bootstrap} but is less practical because the population cumulants are typically unknown. The Edgeworth expansion is also the fundamental supporting component for justifying the higher-order correctness of the bootstrap approximation to the distribution of the studentized sample mean. For an incomplete list of related references, see \cite{10.2307/25050481,10.1214/aos/1176344134,10.1214/aos/1176342609,davison_hinkley_1997,10.1214/aos/1176345638,10.1214/aop/1176992073,hall2013bootstrap,10.1214/aos/1176345636}. 

Edgeworth expansions for nonlinear statistics are typically more involved. In \cite{10.2307/4615839,10.2307/4615841}, the author studied the Edgeworth expansions of maximum likelihood estimators. More efforts have been devoted to the Edgeworth expansions of $U$-statistics and symmetric statistics \cite{10.1214/aos/1031833676,10.1214/aos/1176350170,Bloznelis2022,10.1214/aos/1176344955,10.1214/aos/1176347994,10.2307/24304972}. 
In the context of high-dimensional and low-rank matrix models, little is known regarding the higher-order distributional approximation of the spectral estimators. 
We also remark that all the cited references above require Cram\'er's condition on the smoothness of the data distribution: $\limsup_{|t|\to\infty}|\expect e^{\mathbbm{i}tX_i}| < 1$. In contrast, our analysis unveils that  Cram\'er's condition is not required under certain third-moment conditions on the noise distribution. As mentioned earlier in Section \ref{sub:overview}, such a self-smoothing effect is unavailable in classical i.i.d. models. In high-dimensional models, to our best knowledge, the self-smoothing effect has only been found in \cite{10.1214/21-AOS2125} for the Edgeworth expansion of network moments statistics. However, the network moments considered in \cite{10.1214/21-AOS2125} are $U$-statistics over networks generated from infinitely exchangeable graphons, and they enjoy convenient symmetric properties. In contrast, the entrywise eigenvector statistics considered in the current work are asymmetric, which disables the direct application of techniques for studying the Edgeworth expansions of $U$-statistics and symmetric statistics and is thus technically more challenging.

\subsection{Organization}
\label{sub:organization}

We set the stage of the signal-plus-noise matrix model and the spectral decomposition framework in Section \ref{sec:model_setup}. Section \ref{sec:eigenvector_expansion} studies the higher-order entrywise eigenvector stochastic expansion.
In Section \ref{sec:applications}, we elaborate on the bias correction procedure for the sample eigenvectors, establish the Edgeworth expansion formula for the studentized entrywise eigenvector statistics, and use it to justify the higher-order correctness of entrywise eigenvector bootstrap procedure. Numerical experiments are presented in Section \ref{sec:numerical_results}, and we conclude the paper with discussion in Section \ref{sec:discussion}.

\subsection{Notations}
\label{sub:notations}
The notation $:=$ assigns definitions of quantities. Let $\mathbbm{i} = \sqrt{-1}$ be the notation for the imaginary unit. Given a positive integer $n$, let $[n]$ denote the set of consecutive integers $\{1,\ldots,n\}$. For any $a,b\in\mathbb{R}$, we use $a\vee b$ to denote $\max(a, b)$. Absolute constants that may vary from line to line but are fixed throughout are denoted by $C, C_1, C_2, c, c_1, c_2$, etc. For any two non-negative real sequences $(a_n)_{n = 1}^\infty$ and $(b_n)_{n = 1}^\infty$, we write $a_n = O(b_n)$ or $a_n\lesssim b_n$ if $a_n\leq Cb_n$, and write $b_n = \Omega(a_n)$ or $b_n\gtrsim a_n$ if $a_n\geq Cb_n$. We also write $a_n = o(b_n)$ or $b_n = \omega(a_n)$ if $\lim_{n\to\infty}a_n/b_n = 0$, and $a_n = \Theta(b_n)$ or $a_n\asymp b_n$ if $a_n = O(b_n)$ and $b_n = O(a_n)$. If $(\bA_n)_{n = 1}^\infty$ is a sequence of real matrices and $(b_n)_{n = 1}^\infty$ is a non-negative real sequence, then we write $\bA_n = O(b_n)$ if $\|\bA_n\|_2 = O(b_n)$ and $\bA_n = o(b_n)$ if $\|\bA_n\|_2 = o(b_n)$, where $\|\cdot\|_2$ denotes the matrix spectral norm (the largest singular value of a matrix). Let the total variation distance between two cumulative distribution functions $F(\cdot)$ and $G(\cdot)$ be denoted by $\|F(x) - G(x)\|_\infty := \sup_{x\in\mathbb{R}}|F(x) - G(x)|$. 

For any positive integers $d$, $n$, $\eye_d$ denotes the $d\times d$ identity matrix, $\zero_{n\times d}$ denotes the $n\times d$ zero matrix, and $\zero_d$ denotes the zero vector in $\mathbb{R}^d$. For positive integers $n, d$ with $n\geq d$, let $\mathbb{O}(n, d) = \{\bU\in\mathbb{R}^{n\times d}:\bU\transpose\bU = \eye_d\}$ denote the collection of all orthonormal $d$-frames in $\mathbb{R}^n$, and we write $\mathbb{O}(d)$ for $\mathbb{O}(d, d)$. 
For an Euclidean vector $\bx = [x_1,\ldots,x_n]\transpose\in\mathbb{R}^n$, the Euclidean norm of $\bx$ is denoted by $\|\bx\|_2 = (\sum_{i = 1}^nx_i^2)^{1/2}$, and the infinity norm of $\bx$ is denoted by $\|\bx\|_\infty = \max_{i\in[n]}|x_i|$. 
Given a general matrix $\bM\in\mathbb{R}^{p_1\times p_2}$, let $\sigma_k(\bM)$ denote the $k$th largest singular value of $\bM$, i.e., $\sigma_1(\bM)\geq\ldots\geq\sigma_{\min(p_1, p_2)}(\bM)$. The spectral norm of $\bM$ is defined to be its largest singular value, denoted by $\|\bM\|_2 = \sigma_1(\bM)$, and the two-to-infinity norm of $\bM$ is defined as $\|\bM\|_{2\to\infty} = \sup\{\|\bM\bx\|_\infty:\bx\in\mathbb{R}^{p_2},\|\bx\|_2 = 1\}$. In particular, if the $(i, j)$th entry of $\bM$ is $M_{ij}$, $(i, j)\in[p_1]\times [p_2]$, then $\|\bM\|_{2\to\infty}$ can be explicitly computed as $\|\bM\|_{2\to\infty} = \max_{i\in[p_1]}(\sum_{j = 1}^{p_2}M_{ij})^{1/2}$. If $\bM\in\mathbb{R}^{n\times n}$ is a square symmetric matrix, then we let $\lambda_k(\bM)$ denote its $k$th largest eigenvalue, i.e., $\lambda_1(\bM)\geq\ldots\geq\lambda_n(\bM)$. 

With a slight abuse of notation, we say that a sequence of random matrices $(\bX_n)_{n = 1}^\infty$ and a sequence of non-negative random variables $(Y_n)_{n = 1}^\infty$ satisfy that $\|\bX_n\|_2$ is upper bounded by  $Y_n$ with high probability (w.h.p.), denoted by $\bX_n = \widetilde{O}_{\prob}(Y_n)$, if for any $c > 0$, there exist constants $K_c > 0$ and $N_c\in\mathbb{N}_+$ that may depend on $c$, such that $P(\|\bX_n\|_2 \leq K_cY_n)\geq 1 - n^{-c}$ for all $n\geq N_c$. It is often the case that $(Y_n)_{n = 1}^\infty$ is a sequence of deterministic positive real numbers. Similarly, a sequence of events $(\calE_n)_{n = 1}^\infty$ is said to occur w.h.p., if for all $c > 0$, there exists a constant $N_c\in\mathbb{N}_+$ that may depend on $c$, such that $\prob(E_n)\geq 1 - n^{-c}$ for all $n\geq N_c$. 
Note that our $\widetilde{O}_{\prob}(\cdot)$ notation is stronger than the standard stochastic $O_{\prob}(\cdot)$ notation. 

\section{Model setup}
\label{sec:model_setup}

Let $\bP$ be an $n\times n$ symmetric matrix that plays the role of the \emph{signal} matrix in the setting below with $\mathrm{rank}(\bP) = d\ll n$. Throughout this work, we assume that $d$ is a fixed constant. Consider the additive perturbation of the form
\begin{align}
\label{eqn:signal_plus_noise_matrix_model}
\bA = \bP + \bE,
\end{align}
where $\bE = [E_{ij}]_{n\times n}$ is referred to as the \emph{noise} matrix whose upper triangular entries $(E_{ij}:1\leq i\leq j\leq n)$ are independent random variables with $\expect E_{ij} = 0$. In practice, neither the signal matrix $\bP$ nor the noise matrix $\bE$ is accessible to the users, as only the noisy data matrix $\bA$ is observed. Model \eqref{eqn:signal_plus_noise_matrix_model} is often referred to as the generalized spiked Wigner model \cite{doi:10.1080/01621459.2020.1840990,yau2012universality} or the signal-plus-noise matrix model \cite{cape2019signal,xie2021entrywise} in the literature. 

Let $\bP$ yield  spectral decomposition $\bP = \bU_\bP\bS_\bP\bU_\bP\transpose$, where $\bU_\bP\in\mathbb{O}(n, d)$ is the $n\times d$ matrix of orthonormal eigenvectors of $\bP$ associated with the nonzero eigenvalues, and $\bS_\bP$ is the diagonal matrix of these eigenvalues sorted in nonincreasing order. Let $\bu_k$ denote the $k$th column of $\bU_\bP$, $k\in [d]$. Let $p, q$ be the numbers of positive and negative eigenvalues of $\bP$, respectively, $\bS_{\bP_+} = \mathrm{diag}\{\lambda_1(\bP),\ldots,\lambda_p(\bP)\}$, $\bS_{\bP_-} = \mathrm{diag}\{\lambda_{n - q + 1}(\bP),\ldots,\lambda_n(\bP)\}$, $\bU_{\bP_+}$ be the eigenvector matrix of $\bP$ associated with $\bS_{\bP_+}$, and $\bU_{\bP_-}$ be the eigenvector matrix of $\bP$ associated with $\bS_{\bP_-}$. Let $\lambda_1\geq\ldots\geq\lambda_d$ be the nonzero eigenvalues of $\bP$ sorted in nonincreasing order, namely, $\bS_\bP = \mathrm{diag}(\lambda_1,\ldots,\lambda_d)$. 

This work establishes a fine-grained higher-order entrywise eigenvector stochastic expansion for the signal-plus-noise matrix model \eqref{eqn:signal_plus_noise_matrix_model}. Because neither $\bP$ nor $\bE$ is observed, it is natural to use the eigenvectors of $\bA$ as the sample versions of their unknown population counterparts. Let $\bA$ yield spectral decomposition
\[
\bA = \sum_{k = 1}^n{\lambda}_k(\bA)\widehat{\bu}_k\widehat{\bu}_k\transpose.
\]
Here, ${\lambda}_1(\bA)\geq\ldots\geq{\lambda}_n(\bA)$ are the eigenvalues and $(\widehat{\bu}_k)_{k = 1}^n$ are associated orthonormal eigenvectors. Let $\bU_{\bA_+} = [\widehat{\bu}_1,\ldots,\widehat{\bu}_p]$, $\bU_{\bA_-} = [\widehat{\bu}_{n - q + 1},\ldots,\widehat{\bu}_n]$, and $\bU_\bA = [\bU_{\bA_+},\bU_{\bA_-}]$. Likewise, let $\bS_{\bA_+} = \mathrm{diag}\{\lambda_1(\bA),\ldots,$ $\lambda_p(\bA)\}$, $\bS_{\bA_-} = \mathrm{diag}\{\lambda_{n - q + 1}(\bA),\ldots,\lambda_n(\bA)\}$, and $\bS_{\bA} = \mathrm{diag}(\bS_{\bA_+}, \bS_{\bA_-})$.  We use $\mathrm{sgn}(\widehat{\bu}_k\transpose\bu_k)$ to denote the sign of $\widehat{\bu}_k\transpose\bu_k$ such that $\widehat{\bu}_k\mathrm{sgn}(\widehat{\bu}_k\transpose\bu_k)\approx\bu_k$ intuitively. Let $\hat{\lambda}_m$ be the $m$th diagonal entry of $\bS_{\bA}$ for $m\in[d]$.
\begin{remark}[Rectangular matrix and symmetric dilation]
\label{remark:symmetric_dilation}
Although the current model \eqref{eqn:signal_plus_noise_matrix_model} is designed for symmetric square signal-plus-noise matrices, it is straightforward to adopt it to asymmetric or rectangular low-rank random matrices by applying the ``symmetric dilation'' trick \cite{10.1214/19-AOS1854,paulsen2002completely}. Specifically, let $\bar{\bP}$ be an $n_1\times n_2$ rank-$d$ matrix with compact singular value decomposition $\bar{\bP} = \bU_{\bar{\bP}}\bS_{\bar{\bP}}\bV_{\bar{\bP}}\transpose$, where $\bU_{\bar{\bP}}\in\mathbb{O}(n_1, d)$, $\bV_{\bar{\bP}}\in\mathbb{O}(n_2, d)$, and $\bar{\bA} = \bar{\bP} + \bar{\bE}$, where the entries of $\bar{\bE} = [\bar{E}_{ij}]_{n_1\times n_2}$ are independent mean-zero random variables. The symmetric dilation trick states that it is sufficient to study the following symmetrized version:
\[
\underbrace{\begin{bmatrix}
\zero_{n_1\times n_1} & \bar{\bA}\\
\bar{\bA}\transpose & \zero_{n_2\times n_2}
\end{bmatrix}}_{\bA} = \underbrace{\begin{bmatrix}
\zero_{n_1\times n_1} & \bar{\bP}\\
\bar{\bP}\transpose & \zero_{n_2\times n_2}
\end{bmatrix}}_{\bP} + \underbrace{\begin{bmatrix}
\zero_{n_1\times n_1} & \bar{\bE}\\
\bar{\bE}\transpose & \zero_{n_2\times n_2}
\end{bmatrix}}_{\bE}
\]
because the spectral decomposition of the signal matrix above has the following form
\begin{align*}
\bP
& = \underbrace{\frac{1}{\sqrt{2}}\begin{bmatrix}
\bU_{\bar{\bP}} & \bU_{\bar{\bP}}\\
\bV_{\bar{\bP}} & -\bV_{\bar{\bP}}
\end{bmatrix}}_{\bU_\bP}
\underbrace{\begin{bmatrix}
\bS_{\bar{\bP}} & \\
 & -\bS_{\bar{\bP}}
\end{bmatrix}}_{\bS_{\bP}}
\underbrace{\frac{1}{\sqrt{2}}\begin{bmatrix}
\bU_{\bar{\bP}} & \bU_{\bar{\bP}}\\
\bV_{\bar{\bP}} & -\bV_{\bar{\bP}}
\end{bmatrix}\transpose}_{\bU_\bP}.
\end{align*}
Then, our current setup and notations for symmetric matrices apply directly.
\end{remark}

Next, we present several assumptions regarding the eigenvalues and eigenvectors of $\bP$. 
\begin{assumption}\label{assumption:Signal_strength}
There exist constants $\beta_\Delta, C_1, C_2, c > 0$ and an $n$-dependent quantity $\rho_n\in (0, 1]$, such that $n\rho_n = \omega(n^c)$ and 
\[
C_1n^{\beta_\Delta}\sqrt{n\rho_n}\leq \min_{k\in[d]}|\lambda_k|\leq \max_{k\in[d]}|\lambda_k|\leq C_2n^{\beta_\Delta}\sqrt{n\rho_n}.
\]
\end{assumption}

\begin{assumption}\label{assumption:Eigenvalue_separation}
There exists a constant $\delta_0 > 0$ such that 
\[
\Delta_n := \min\Big(\min_{k\in [d]}|\lambda_k|, \min_{k\in [d - 1]}|\lambda_k - \lambda_{k + 1}|\Big)\geq \delta_0\min_{k\in [d]}|\lambda_k|.
\]
\end{assumption}

The parameter $\rho_n$ in Assumption \ref{assumption:Signal_strength} serves as the scaling factor for the noise variance, thereby controlling the overall noise level in model \eqref{eqn:signal_plus_noise_matrix_model}. The characteristics of this noise level are detailed in the forthcoming Assumption \ref{assumption:Noise_matrix_distribution}. In particular, it can be shown that $\|\bE\|_2 = \widetilde{O}_{\prob}(\sqrt{n\rho_n})$. The magnitudes of the eigenvalues control the signal level, and Assumption \ref{assumption:Signal_strength} requires that the signal-to-noise ratio (SNR), namely, $\Delta_n/\sqrt{n\rho_n}$, scales at least at the order of a polynomial of $n$. Readers who are familiar with the random graph model may recognize that the signal strength can be connected to the noise level in the following sense: In a random graph model where $\bA$ is the graph adjacency matrix with expected value $\bP$, the signal strength is equivalent to the average expected degree of the graph through $n\rho_n$ and is also connected to the noise level $\|\bE\|_2 = \widetilde{O}_{\prob}(\sqrt{n\rho_n})$. Nevertheless, by setting SNR flexibly through $\beta_{\Delta}$, our Assumption \ref{assumption:Signal_strength} above allows the signal strength to be decoupled from the noise level to accommodate a broader class of low-rank random matrix models. For example, in the low-rank matrix denoising or spiked covariance matrix model, the noise variance scale $\rho_n$ can be independent of the magnitudes of the eigenvalues of $\bP$. Assumption \ref{assumption:Eigenvalue_separation} requires that the nonzero eigenvalues of $\bP$ are simple and well separated. 
It turns out that this eigenvalue separation assumption is critical for us to obtain fine-grained higher-order entrywise eigenvector expansion. 
\par
The next assumption describes the variance and tail behavior of the distributions of $E_{ij}$'s via a collection of moment bounds.
\begin{assumption}\label{assumption:Noise_matrix_distribution}
Let $\eps \in (0, 1/2]$ be a constant and $\rho_n$ be the $n$-dependent quantity given in Assumption \ref{assumption:Signal_strength}. There exists another $n$-dependent quantity $q_n\in [n^\eps, \sqrt{n}]$ and a constant $C > 0$, such that the distribution of $\bE$ satisfies the moment bounds
\begin{align}\label{eqn:noise__moment_condition}
\expect E_{ij}^2\leq C\rho_n,\quad \expect |E_{ij}|^p\leq \frac{C^p(n\rho_n)^{p/2}}{nq_n^{p - 2}}
\end{align}
for $1\leq i,j\leq n$ and $3\leq p\leq (2\log n)^{A_0\log\log n}$, where $A_0 \geq 11$. 
\end{assumption}
Assumption \ref{assumption:Noise_matrix_distribution} is an adaptation of Definition 2.1 in \cite{erdos2013}. As observed there, the moment bound condition in \eqref{eqn:noise__moment_condition} is more flexible than the commonly assumed sub-Gaussian or sub-exponential assumptions for the noise matrix $\bE$. This is because $\var(E_{ij}/\sqrt{n\rho_n}) = O(1/n)$ but $\expect |E_{ij}/\sqrt{n\rho_n}|^p = O(1/q_n^p)$ and note that $q_n\leq \sqrt{n}$. In contrast, in the case of the classical generalized Wigner matrix, $E_{ij}/\sqrt{n\rho_n}$ has a natural scaling of $1/\sqrt{n}$. Below, we provide two examples where Assumption \ref{assumption:Noise_matrix_distribution} is satisfied. 
\begin{example}[Sub-exponential decay $\bE$]
\label{example:sub_exponential_noise}
Suppose $\bE = [E_{ij}]_{n\times n}$ has entrywise sub-exponential decay with
$
\prob(\rho_n^{-1/2}|E_{ij}| > x)\leq C\exp(-x^{1/\theta})
$
for some constants $C, \theta > 0$. This includes the classical sub-Gaussian and sub-exponential random variables (see, e.g., \cite{vershynin2018high}) with $\theta = 1/2$ and $\theta = 1$, respectively. Then Assumption \ref{assumption:Noise_matrix_distribution} is satisfied with $q_n := \sqrt{n}\{\max(1,\theta)^3(2\log n)^{3A_0\log\log n}\}^{-\theta}$; Also see Remark 2.5 in \cite{erdos2013}.
\end{example}

\begin{example}[Random graph]
\label{example:random_graph}
If $A_{ij}\sim\mathrm{Bernoulli}(p_{ij})$ independently for all $i,j\in[n]$, $i\leq j$, where $p_{ij}$ is the $(i, j)$th entry of $\bP$, and $\max_{i, j\in[n]}p_{ij} = O(\rho_n)$, then $\expect E_{ij} = 0$, and there exists a constant $C > 0$, such that 
\[
\expect E_{ij}^2 = p_{ij}(1 - p_{ij})\leq C\rho_n,\quad \expect|E_{ij}|^p\leq \expect E_{ij}^2\leq \frac{C(n\rho_n)^{p/2}}{nq_n^{p - 2}}
\]
with $q_n:=\sqrt{n\rho_n}$ (note that $|E_{ij}|\leq 1$ with probability one).
This model corresponds to the generalized random dot product graph model \cite{rubin2017statistical} where $\bA = [A_{ij}]_{n\times n}$ is the graph adjacency matrix and $\bP$ is the underlying edge probability matrix. 
\end{example}

\section{Higher-order entrywise eigenvector stochastic expansion} 
\label{sec:eigenvector_expansion}

\subsection{Main result}
\label{sub:main_result}
To motivate the development of our main result, we first review the existing zeroth-order entrywise eigenvector deviation and first-order entrywise eigenvector fluctuation results (see, e.g., \cite{10.1214/19-AOS1854,cape2019signal,cape2017two,pmlr-v83-eldridge18a,doi:10.1080/01621459.2020.1840990,xie2021entrywise}). Informally speaking, under Assumptions \ref{assumption:Signal_strength}--\ref{assumption:Noise_matrix_distribution}, for each $i \in [n]$ and $k\in [d]$, the zeroth-order entrywise eigenvector deviation establishes
\begin{align}\label{eqn:zeroth_order_eigenvector_deviation}
\widehat{u}_{ik}
 - u_{ik} = \widetilde{O}_{\prob}\left\{\|\bU_\bP\|_{2\to\infty}\frac{(n\rho_n)^{1/2}(\log n)^{\xi}}{\Delta_n}\right\},
\end{align}
and the first-order entrywise eigenvector fluctuation states that 
\begin{equation}
\label{eqn:first_order_eigenvector_fluctuation}
\begin{aligned}
\widehat{u}_{ik}
 - u_{ik}& = \be_i\transpose\bigg(\eye_n - \bu_k\bu_k\transpose + \sum_{m\in[d]\backslash\{k\}}\frac{\lambda_m\bu_m\bu_m\transpose}{\lambda_k - \lambda_m}\bigg)\frac{\bE\bu_k}{\lambda_k}
\\&\quad+ \widetilde{O}_{\prob}\left[\|\bU_\bP\|_{2\to\infty}\left\{\frac{(n\rho_n)(\log n)^{2\xi}}{\Delta_n^2} + \frac{(\log n)^\xi}{\Delta_n}\right\}\right],
\end{aligned}
\end{equation}
where $\xi$ is any constant with $\xi > 1$, $\widehat{u}_{ik}$ and $u_{ik}$ are the $i$th entries of $\widehat{\bu}_k$ and $\bu_k$, respectively (see Lemma B.4 in the Supplement for a rigorous statement of these results). 
Observe that 
\[
\be_i\transpose\bigg(\eye_n - \bu_k\bu_k\transpose + \sum_{m\in[d]\backslash\{k\}}\frac{\lambda_m\bu_m\bu_m\transpose}{\lambda_k - \lambda_m}\bigg)\frac{\bE\bu_k}{\lambda_k} = O_{\prob}\bigg(\frac{\rho_n^{1/2}}{\Delta_n}\bigg)
\]
by Chebyshev's inequality. Therefore, under the condition that
\[
\rho_n^{-1/2}\max\left\{(\log n)^\xi, \frac{n\rho_n(\log n)^{2\xi}}{\Delta_n}\right\}\|\bU_\bP\|_{2\to\infty} \to 0\quad\mbox{w.h.p.},
\]
the remainder in \eqref{eqn:first_order_eigenvector_fluctuation} is a higher order term. Note that the condition in the preceding display can be satisfied when the eigenvectors are delocalized, \emph{i.e.}, $\|\bU_\bP\|_{2\to\infty} = O(n^{-1/2})$ (see Assumption \ref{assumption:Eigenvector_delocalization} in Section \ref{sub:bias_correction} below; Also see \cite{cape2019signal,zhang2022perturbation}). Indeed, $u_{ik}$ and $\widehat{u}_{ik}$ are nonlinear functionals of $\bP$ and $\bA$, respectively, and the leading term in \eqref{eqn:first_order_eigenvector_fluctuation}, viewed as a linear function of $\bE = \bA - \bP$, is exactly the Fr\'echet derivative of $u_{ik}$ with respect to $\bP$. 

Given the above observation, it is natural to pursue entrywise eigenvector analysis that provides further information beyond the first-order fluctuation \eqref{eqn:first_order_eigenvector_fluctuation}. Theorem \ref{thm:Eigenvector_Expansion} below, our first main result, tackles this task by rigorously establishing a higher-order entrywise eigenvector stochastic expansion. 

\begin{theorem}\label{thm:Eigenvector_Expansion}
Suppose Assumptions \ref{assumption:Signal_strength}--\ref{assumption:Noise_matrix_distribution} hold. Let $\alpha_n:=\|\bU_\bP\transpose\bE\bU_\bP\|_2$.
Assume that there exists a constant $\bar{\xi} > 1$, such that
\begin{equation}\label{thm3.1:con}
\begin{aligned}
\frac{n\rho_n^{1/2}(\log n)^{3\bar{\xi}}\|\bU_\bP\|_{2\to\infty}}{\Delta_n}\to 0.
\end{aligned}
\end{equation}
Then for any fixed $k\in [p]$ and any $\xi \in (1, \bar{\xi}]$, we have
\begin{equation}
\label{eqn:eigenvector_expansion}
\begin{aligned}
\widehat{\bu}_k\mathrm{sgn}(\bu_k\transpose\widehat{\bu}_k) - \bu_k & = 
{\bigg(}\eye_n - \bu_k\bu_k\transpose + \sum_{m\in [d]/\{k\}}\frac{\lambda_m\bu_m\bu_m\transpose}{\lambda_k - \lambda_m}{\bigg)}\frac{\bE\bu_k}{\lambda_k}\\
&\quad + 
{\bigg(}\eye_n - \frac{3}{2}\bu_k\bu_k\transpose + \sum_{m\in [d]/\{k\}}\frac{\lambda_m\bu_m\bu_m\transpose}{\lambda_k - \lambda_m}{\bigg)}\frac{\bE^2\bu_k}{\lambda_k^2}
 + \bt_k,
\end{aligned}
\end{equation}
where 
\begin{align*}
\|\bt_k\|_{\infty} & = \widetilde{O}_{\prob}\left[\|\bU_\bP\|_{2\to\infty}\left\{\frac{(n\rho_n)^{1/2}(\log n)^\xi}{\Delta_n^2}\left(\alpha_n + \frac{n\rho_n}{\Delta_n}\right)^2 + \frac{(n\rho_n)^{3/2}(\log n)^{3\xi}}{\Delta_n^3}\right\}\right]\\
&\quad + \widetilde{O}_{\prob}\left\{\|\bU_\bP\|_{2\to\infty}\frac{n^2\rho_n^{3/2}(\log n)^{3\xi}\|\bU_\bP\|_{2\to\infty}^2}{\Delta_n^3}\right\}.
\end{align*}
\end{theorem}
\begin{remark}[Intuition of Theorem \ref{thm:Eigenvector_Expansion}]
\label{rmk:intuition_stochastic_expansion}
To gain insight into expansion \eqref{eqn:eigenvector_expansion}, let us consider the simplified situation where $\Delta_n = \Theta(n\rho_n)$ and $\|\bU_\bP\|_{2\to\infty} = O(n^{-1/2})$, which is common in low-rank random graphs such as the stochastic block model \cite{HOLLAND1983109} and its offspring \cite{airoldi2008mixed,PhysRevE.83.016107,PhysRevX.4.011047,https://doi.org/10.1111/rssb.12245}; See also \cite{cape2019signal,xie2021entrywise}. It can be easily checked that condition \eqref{thm3.1:con} is satisfied for any $\xi > 1$. It can also be shown that 
\begin{align}\label{alphan:bound}
\alpha_n = \|\bU_\bP\transpose\bE\bU_\bP\|_2 = \widetilde{O}_{\prob}\{(\log n)^\xi\}
\end{align} for any $\xi > 1$
(See, for example, Equation (3.19) of Lemma 3.8 in \cite{erdos2013}). 
Thus, we have the following sharp entrywise characterization of the terms appearing in \eqref{eqn:eigenvector_expansion} for any $i\in [n]$:
\begin{align}\nonumber
\be_i\transpose{\bigg(}\eye_n - \bu_k\bu_k\transpose + \sum_{m\in [d]/\{k\}}\frac{\lambda_m\bu_m\bu_m\transpose}{\lambda_k - \lambda_m}{\bigg)}\frac{\bE\bu_k}{\lambda_k} &= O_{\prob}\left(\frac{1}{n\rho_n^{1/2}}\right),
\\\label{thm3.1:rate:2}
\be_i\transpose{\bigg(}\eye_n - \frac{3}{2}\bu_k\bu_k\transpose + \sum_{m\in [d]/\{k\}}\frac{\lambda_m\bu_m\bu_m\transpose}{\lambda_k - \lambda_m}{\bigg)}\frac{\bE^2\bu_k}{\lambda_k^2} &= O_{\prob}\left(\frac{1}{n^{3/2}\rho_n}\right),
\end{align}
with $\|\bt_k\|_\infty  = \widetilde{O}_{\prob}\{(\log n)^{3\xi}n^{-2}\rho_n^{-3/2}\}$, which is implied to be a high-order negligible remainder compared to the first two terms on the right-hand side of \eqref{eqn:eigenvector_expansion}. 
\end{remark}

\subsection{Proof architecture}
\label{sub:proof_sketch}
We now briefly discuss the proof architecture of Theorem \ref{thm:Eigenvector_Expansion}. We use $\mathrm{sgn}(\widehat{\bu}_k\transpose\bu_k)$ to adjust $\widehat{\bu}_k$ for the identifiability of the eigenvector $\bu_k$ up to a sign. Let $\bW^* = \mathrm{diag}\{\mathrm{sgn}(\widehat{\bu}_1\transpose\bu_1),\ldots,\mathrm{sgn}(\widehat{\bu}_d\transpose\bu_d)\}$ and $\bW^*_+ = \mathrm{diag}\{\mathrm{sgn}(\widehat{\bu}_1\transpose\bu_1),\ldots,\mathrm{sgn}(\widehat{\bu}_p\transpose\bu_p)\}$. Without loss of generality, we only need to consider $\bU_{\bA_+}$ and $\bU_{\bP_+}$, because the result for $\bU_{\bA_-}$ and $\bU_{\bP_-}$ can be obtained immediately by observing  $-\bA = -\bP + (-\bE)$. 
We first present a warm-up decomposition of $\bU_{\bA_+} - \bU_{\bP_+}\bW$ due to \cite[$\mathsection$D.1]{xie2021entrywise} motivated by \cite{10.1214/19-AOS1854,cape2019signal}:
\begin{align}\label{eqn:first_order_eigenvector_expansion}
&\bU_{\bA_+} - \bU_{\bP_+}\bW_+^* = \bE\bU_{\bP_+}\bS_{\bP_+}^{-1}\bW_+^* + \bR_{\bU_+} + \bU_{\bP_-}\bS_{\bP_-}\bU_{\bP_-}\transpose\bU_{\bA_+}\bS_{\bA_+}^{-1},
\end{align}
where
\begin{equation}\label{def:RU+}
\begin{aligned}
\bR_{\bU_+} &= \bR_1 + \bR_2 + \bR_3 + \bR_4,\\
\bR_1 &= \bE(\bU_{\bA_+} - \bU_{\bP_+}\bW_+^*)\bS_{\bA_+}^{-1},\\
\bR_2 &= \bU_{\bP_+}(\bS_{\bP_+}\bU_{\bP_+}\transpose\bU_{\bA_+} - \bU_{\bP_+}\transpose\bU_{\bA_+}\bS_{\bA_+})\bS_{\bA_+}^{-1},\\
\bR_3 & = \bU_{\bP_+}(\bU_{\bP_+}\transpose\bU_{\bA_+} - \bW_+^*),\\
\bR_4 &= \bE\bU_{\bP_+}(\bW_+^*\bS_{\bA_+}^{-1} - \bS_{\bP_+}^{-1}\bW_+^*). 
\end{aligned}
\end{equation}
Decomposition \eqref{eqn:first_order_eigenvector_expansion} is purely algebraic and holds for any matrix equation of the form $\bA = \bU_\bP\bS_\bP\bU_\bP\transpose + \bE$ even without the randomness of $\bE$. See Section 2.2 and Section 3.2 of \cite{xie2021entrywise} for the detailed derivation and motivation. In particular, the first term appearing on the right-hand side of \eqref{eqn:first_order_eigenvector_expansion} unveils that the leading term in $\widehat{\bu}_k\mathrm{sgn}(\widehat{\bu}_k\transpose\bu_k) - \bu_k$ should contain $\bE\bu_k/\lambda_k$ and is inspiring for pursuing the higher-order expansion. It has also been observed in \cite{10.1214/19-AOS1854,xie2021entrywise} that decomposition \eqref{eqn:first_order_eigenvector_expansion} paves the way for the sharp first-order entrywise eigenvector analysis for model \eqref{eqn:signal_plus_noise_matrix_model} under challenging low SNR regimes. Nevertheless, the analysis of the terms $\bR_{\bU_+}$ and $\bU_{\bP_-}\bS_{\bP_-}\bU_{\bP_-}\transpose\bU_{\bA_+}\bS_{\bA_+}^{-1}$ in \cite{10.1214/19-AOS1854} and \cite{10.1214/19-AOS1854,xie2021entrywise} are not sufficient to explore the higher-order entrywise expansion of $\widehat{\bu}_k\mathrm{sgn}(\widehat{\bu}_k\transpose\bu_k) - \bu_k$.

The proof architecture of Theorem \ref{thm:Eigenvector_Expansion} is based on a collection of delicate analyses of the terms in $\bR_{\bU_+}$ and the term $\bU_{\bP_-}\bS_{\bP_-}\bU_{\bP_-}\transpose\bU_{\bA_+}\bS_{\bA_+}^{-1}$ in \eqref{eqn:first_order_eigenvector_expansion}.  
For the first term $\bR_1$ in $\bR_{\bU_+}$, decomposition \eqref{eqn:first_order_eigenvector_expansion} suggests that $\bU_{\bA_+} - \bU_{\bP_+}\bW_+^*\approx \bE\bU_{\bP_+}\bS_{\bP_+}^{-1}\bW_+^*$. It is thus natural to conceive that $\bR_1\approx \bE^2\bU_\bP\bS_\bP^{-2}\bW_+^*$ because $\bS_\bA\approx\bS_\bP$ by the Weyl's inequality for eigenvalue concentration and $\bE^2\bU_\bP\bS_\bP^{-2}$ is exactly a quadratic function of $\bE$. The analysis of the second term $\bR_2$ in $\bR_{\bU_+}$ is based on the definition of eigenvalues and eigenvectors, namely, $\bU_{\bP_+}\bS_{\bP_+} = \bP\bU_{\bP_+}$ and $\bU_{\bA_+}\bS_{\bA_+} = \bA\bU_{\bA_+}$. This entails that
\begin{align*}
\bR_2&\approx \bU_{\bP_+}(\bS_{\bP_+}\bU_{\bP_+}\transpose\bU_{\bA_+} - \bU_{\bP_+}\transpose\bU_{\bA_+}\bS_{\bA_+})\bS_{\bP_+}^{-1}\\
& = \bU_{\bP_+}(\bU_{\bP_+}\transpose\bP\bU_{\bA_+} - \bU_{\bP_+}\transpose\bA\bU_{\bA_+})\bS_{\bP_+}^{-1}\\
& \approx \bU_{\bP_+}(\bU_{\bP_+}\transpose\bP\bU_{\bP_+} - \bU_{\bP_+}\transpose\bA\bU_{\bP_+})\bS_{\bP_+}^{-1}\bW_+^*\\
& = -\bU_{\bP_+}\bU_{\bP_+}\transpose\bE\bU_{\bP_+}\bS_{\bP_+}^{-1}\bW_+^*
\end{align*}
because $\bU_{\bA_+} \approx \bU_{\bP_+}\bW_+^*$ heuristically, so that $\bR_2$ can be approximated by a linear function of $\bE$. The fourth term $\bR_4$ can be shown to be negligible compared to the first-order and second-order terms in the expansion \eqref{eqn:eigenvector_expansion}. It remains to work with the  $\bR_3$ in $\bR_{\bU_+}$ and the last term $\bU_{\bP_-}\bS_{\bP_-}\bU_{\bP_-}\transpose\bU_{\bA_+}\bS_{\bA_+}^{-1}$ in \eqref{eqn:first_order_eigenvector_expansion}. The key factor in these two terms, which is also the most challenging one to analyze, is the angles between the sample eigenvectors in $\bU_\bA$ and their population counterparts in $\bU_\bP$. More specifically, one needs to deal with the asymptotic expansion of $\bU_{\bP_+}\transpose 
 \bU_{\bA_+}$ in $\bR_3$ and $\bU_{\bP_-}\transpose 
 \bU_{\bA_+}$ in $\bU_{\bP_-}\bS_{\bP_-}\bU_{\bP_-}\transpose\bU_{\bA_+}\bS_{\bA_+}^{-1}$. Proposition \ref{prop:eigenvector_angle_expansion} below addresses the asymptotic expansion of $\bu_m\transpose\widehat{\bu}_k$ for any $m\in[d]$ and $k\in[p]$, thereby filling the last piece in the proof architecture of Theorem \ref{thm:Eigenvector_Expansion}.
\begin{proposition}\label{prop:eigenvector_angle_expansion}
Suppose the conditions of Theorem \ref{thm:Eigenvector_Expansion} hold. For each fixed $k\in [p]$, $m\in [d]$, $k\neq m$, we have
\begin{align*}
\bu_k\transpose\widehat{\bu}_k - \mathrm{sgn}(\bu_k\transpose\widehat{\bu}_k)
&= -\mathrm{sgn}(\bu_k\transpose\widehat{\bu}_k)\frac{\bu_k\transpose\bE^2\bu_k}{2\lambda_k^{2}} + t_{kk},\\
\mathrm{sgn}(\bu_k\transpose\widehat{\bu}_k)(\bu_m\transpose\widehat{\bu}_k)
&= \frac{\bu_m\transpose\bE\bu_k}{\lambda_k - \lambda_m} + \frac{\bu_m\transpose \bE^2\bu_k}{\lambda_k(\lambda_k - \lambda_m)} + t_{mk},
\end{align*}
where for any $\xi\in(1,\bar{\xi}]$, 
\[
\max_{k\in[p],m\in [d]}|t_{mk}| = \widetilde{O}_{\prob}\left\{\frac{\alpha_n^2}{\Delta_n^2} + \frac{n^2\rho_n^{3/2}(\log n)^{3\xi}\|\bU_\bP\|_{2\to\infty}^2}{\Delta_n^3} + \frac{(n\rho_n)^{3/2}}{\Delta_n^3}\right\}.
\]
\end{proposition}
\begin{remark}
We briefly compare Proposition \ref{prop:eigenvector_angle_expansion} with the similar result in \cite{doi:10.1080/01621459.2020.1840990}. Theorem 2 in \cite{doi:10.1080/01621459.2020.1840990} characterizes the asymptotic normality of general linear functionals of eigenvectors based on stochastic expansions similar to those in Proposition \ref{prop:eigenvector_angle_expansion}. However, Proposition \ref{prop:eigenvector_angle_expansion} differs from Theorem 2 in \cite{doi:10.1080/01621459.2020.1840990} in the following aspects. First of all, in \cite{doi:10.1080/01621459.2020.1840990}, the entries of the noise matrix $\bE$ are required to be bounded with probability one or satisfy the moment bound $\expect |E_{ij}|^l \leq C^{l - 2}\rho_n$ for $l\geq 2$ for some constant $C > 0$. In contrast, Assumption \ref{assumption:Noise_matrix_distribution} allows for a much broader class of noise matrix distributions because the moment bound in \eqref{eqn:noise__moment_condition} is equivalent to $\expect |E_{ij}|^l \leq (C\sqrt{n\rho_n}/q_n)^{l - 2}\rho_n$ for some constant $C > 0$ and $\sqrt{n\rho_n}/q_n$ is allowed to diverge to $+\infty$ as $n\to\infty$. Secondly, the results in \cite{doi:10.1080/01621459.2020.1840990} are obtained with remainders characterized in terms of the classical $O_{\prob}(\cdot)$ notation, whereas the remainders in Proposition \ref{prop:eigenvector_angle_expansion} are obtained under the stronger ``with high probability'' notation $\widetilde{O}_{\prob}(\cdot)$. The high-probability bound for the remainder term is crucial for us to develop the entrywise Edgeworth expansion in Section \ref{sub:entrywise_eigenvector_edgeworth_expansion} while, in contrast, the classical $O_{\prob}(\cdot)$ characterization would be insufficient. The proof of Proposition \ref{prop:eigenvector_angle_expansion} follows the same roadmap as in \cite{doi:10.1080/01621459.2020.1840990}, but we borrow the sharp concentration inequalities in \cite{erdos2013} for terms in the form of $\bx\transpose (\bE^l - \expect\bE^l)\by$ instead of Chebyshev's inequality (which is the treatment in \cite{doi:10.1080/01621459.2020.1840990}) in order to provide stronger $\widetilde{O}_{\prob}(\cdot)$ characterization of the remainders. 
\end{remark}

\section{Applications}
\label{sec:applications}

\subsection{Eigenvector bias correction}
\label{sub:bias_correction}

In this section, we introduce the bias correction procedure for the eigenvectors as an immediate application of Theorem \ref{thm:Eigenvector_Expansion}. 
We focus on the following eigenvector delocalization scenario. 
\begin{assumption}\label{assumption:Eigenvector_delocalization}
There exists a constant $C_\mu\geq 1$ such that $\|\bU_\bP\|_{2\to\infty}\leq C_\mu\sqrt{d/n}$.
\end{assumption}
Assumption \ref{assumption:Eigenvector_delocalization} is also referred to as the incoherent condition in the literature of random matrix theory, compressed sensing, and network analysis \cite{10.1214/19-AOS1854,agterberg2021entrywise,candes2009exact,candes2010power,cape2019signal,cape2017two,10.1215/00127094-3129809,xie2021entrywise}. It requires that the magnitudes of the entries of $\bU_\bP$ cannot be too distinct from each other and they are uniformly spread out. In our current entrywise eigenvector analysis, such eigenvector delocalization implies that the first-order term in the expansion \eqref{eqn:eigenvector_expansion} is a sum of holospoudic independent random variables and converges weakly to the standard normal distribution by the Lindeberg-Feller central limit theorem \cite{chung2001course}. 

Bias correction is a method to construct a ``new'' estimator that removes the second-order bias of an ``old'' estimator. There are several approaches to construct bias-corrected estimators, including the jackknife method \cite{10.1214/aoms/1177729989,10.1093/biomet/43.3-4.353} and the analytic method \cite{https://doi.org/10.1111/j.1468-0262.2004.00482.x}. The analytic method is preferred when a closed-form formula for the second-order bias is available, which is exactly our case in light of Theorem \ref{thm:Eigenvector_Expansion}. We now describe the bias correction procedure in our context. 
The key observation for the eigenvector bias correction is that the second-order term
\[
{\bigg(}\eye_n - \frac{3}{2}\bu_k\bu_k\transpose + \sum_{m\in [d]/\{k\}}\frac{\lambda_m\bu_m\bu_m\transpose}{\lambda_k - \lambda_m}{\bigg)}\frac{\bE^2\bu_k}{\lambda_k^2}
\]
on the right-hand side of the expansion \eqref{eqn:eigenvector_expansion} has a non-zero expected value. Formally, let $\widetilde{\bu}_k$ be any estimator of $\bu_k$ admitting an expansion of the following form 
\begin{equation}\label{tu:expansion}
\widetilde{\bu}_k\mathrm{sgn}(\bu_k\transpose\widetilde{\bu}_k) - \bu_k = \bpsi_k^{\widetilde{\bu}_k}(\bE) + \bq_k^{\widetilde{\bu}_k}(\bE) + \bt_k^{\widetilde{\bu}_k};
\end{equation}
Here, $\bpsi_k^{\widetilde{\bu}_k}(\cdot):L_2(\mathbb{R}^{n\times n})\to L_2(\mathbb{R}^n)$ is a linear function of $\bE$, $\bq_k^{\widetilde{\bu}_k}(\cdot):L_2(\mathbb{R}^{n\times n})\to L_2(\mathbb{R}^n)$ is a quadratic function of $\bE$, $\bt_k^{\widetilde{\bu}_k}\in\mathbb{R}^n$ is the higher-order remainder term, and they satisfy
\begin{align*}
\be_i\transpose\bpsi_k^{\widetilde{\bu}_k}(\bE)& = O_{\prob}\mathrel{\Big(}\frac{\rho_n^{1/2}}{\Delta_n}\mathrel{\Big)},\quad \be_i\transpose\bq_k^{\widetilde{\bu}_k}(\bE) = O_{\prob}\Big(\frac{\sqrt{n}\rho_n}{\Delta_n^2}\Big)\quad
\be_i\transpose\bt_k^{\widetilde{\bu}_k} = o_{\prob}\Big(\frac{\sqrt{n}\rho_n}{\Delta_n^2}\Big)
\end{align*}
for each $i\in[n]$, where we use $L_2(\mathbb{R}^{n\times n})$ and $L_2(\mathbb{R}^n)$ to denote the collection of all $n\times n$ random matrices and $n$-dimensional random vectors with finite second moments. 
The second-order bias of $\widetilde{\bu}_k$ is defined as
\begin{equation}\label{def:bias}
\begin{aligned}
\bb(\widetilde{\bu}_k) = \expect \bq_k^{\widetilde{\bu}_k}(\bE).
\end{aligned}
\end{equation}
When $\widetilde{\bu}_k = \widehat{\bu}_k$ and Assumptions~\ref{assumption:Signal_strength}--\ref{assumption:Eigenvector_delocalization} hold, by Theorem~\ref{thm:Eigenvector_Expansion}, it can be shown that
\begin{align}
\label{eqn:eigenvector_first_order_term}
\be_i\transpose\bpsi_k^{\widehat{\bu}_k}(\bE) & = \be_i\transpose{\bigg(}\eye_n - \bu_k\bu_k\transpose + \sum_{m\in [d]/\{k\}}\frac{\lambda_m\bu_m\bu_m\transpose}{\lambda_k - \lambda_m}{\bigg)}\frac{\bE\bu_k}{\lambda_k} = O_{\prob}\mathrel{\Big(}\frac{\rho_n^{1/2}}{\Delta_n}\mathrel{\Big)},\\
\label{eqn:eigenvector_second_order_term}
\be_i\transpose\bq_k^{\widehat{\bu}_k}(\bE) & = \be_i\transpose{\bigg(}\eye_n - \frac{3}{2}\bu_k\bu_k\transpose + \sum_{m\in [d]/\{k\}}\frac{\lambda_m\bu_m\bu_m\transpose}{\lambda_k - \lambda_m}{\bigg)}\frac{\bE^2\bu_k}{\lambda_k^2} = O_{\prob}\Big(\frac{\sqrt{n}\rho_n}{\Delta_n^2}\Big),
\end{align}
and $\be_i\transpose\bt_k^{\widehat{\bu}_k} = \widetilde{o}_{\prob}\big({\sqrt{n}\rho_n}/{\Delta_n^2}\big)$, provided that $(n\rho_n)^{3/2}/\Delta_n^2\to 0$. 
The key idea of the bias correction for the eigenvectors is to estimate the second-order bias $\bb_k = \bb(\widehat{\bu}_k) = \expect \bq_k^{\widehat{\bu}_k}(\bE)$ using the analytic formula
\[
\bb_k(\widehat{\bu}_k) = {\bigg(}\eye_n - \frac{3}{2}\bu_k\bu_k\transpose + \sum_{m\in [d]/\{k\}}\frac{\lambda_m\bu_m\bu_m\transpose}{\lambda_k - \lambda_m}{\bigg)}\frac{\expect\bE^2\bu_k}{\lambda_k^2}.
\]
To this end, we define
\[
\bD := \expect\bE^2 = \mathrm{diag}\mathrel{\bigg(}\sum_{j = 1}^n\sigma_{1j}^2,\ldots,\sum_{j = 1}^n\sigma_{nj}^2\mathrel{\bigg)},
\]
where $\sigma_{ij}^2:= \var(E_{ij})$.
Let $\widehat{\bP} = \bU_\bA\bS_\bA\bU_\bA\transpose$, $\widehat{p}_{ij} = \be_i\transpose\widehat{\bP}\be_j$, where $\be_i$ is the $i$th standard basis vector in $\mathbb{R}^n$ whose coordinates are zeros except the $i$th one being one. Define 
\[
\widehat{\bD} = \mathrm{diag}\mathrel{\bigg\{}\sum_{j = 1}^n(A_{1j} - \widehat{p}_{1j})^2,\ldots,\sum_{j = 1}^n(A_{nj} - \widehat{p}_{nj})^2\mathrel{\bigg\}}
\]
as the plug-in estimator of $\bD$ and 
\begin{align}\label{eqn:estimated_bias}
\widehat{\bb}_k = {\bigg(}\eye_n - \frac{3}{2}\widehat{\bu}_k\widehat{\bu}_k\transpose + \sum_{m\in [d]/\{k\}}\frac{\widehat{\lambda}_m\widehat{\bu}_m\widehat{\bu}_m\transpose}{\widehat{\lambda}_k - \widehat{\lambda}_m}{\bigg)}\frac{\widehat{\bD}\widehat{\bu}_k}{\widehat{\lambda}_k^2},
\end{align}
as the plug-in estimator of $\bb(\widehat{\bu}_k)$. Theorem \ref{thm:bias_corrected_ASE} shows that the bias-corrected eigenvector, 
\begin{equation}\label{def:biascorr}
\begin{aligned}
\widehat{\bu}_k^{(c)} := \widehat{\bu}_k - \widehat{\bb}_k,
\end{aligned}
\end{equation}
has zero second-order bias. Note that $\bq_k^{\widehat{\bu}_k^{(c)}}(\bE) = \bq_k^{\widehat{\bu}_k}(\bE) - \expect \bq_k^{\widehat{\bu}_k}(\bE)$ is also a quadratic function of $\bE$ because $\expect \bq_k^{\widehat{\bu}_k}(\cdot):L_2(\mathbb{R}^{n\times n})\to L_2(\mathbb{R}^n)$ is quadratic in $\bE$. 
\begin{theorem}\label{thm:bias_corrected_ASE}
Suppose Assumptions \ref{assumption:Signal_strength}--\ref{assumption:Eigenvector_delocalization} hold. Then for any $\xi > 1$, we have
\[
\|\widehat{\bb}_k\mathrm{sgn}(\bu_k\transpose\widehat{\bu}_k) - \bb(\widehat{\bu}_k)\|_{\infty} = \widetilde{O}_{\prob}\left\{\frac{\sqrt{n}\rho_n}{\Delta_n^2}\max\left(\frac{1}{q_n},  \frac{\sqrt{n\rho_n}}{\Delta_n}\right)(\log n)^{\xi}\right\}.
\]
Furthermore, if $(n\rho_n)^{3/2}/\Delta_n^2\to 0$, then the bias-corrected eigenvector $\widehat{\bu}_k^{(c)} = \widehat{\bu}_k - \widehat{\bb}_k$ satisfies $\bb(\widehat{\bu}_k^{(c)}) = \zero_n$. 
\end{theorem}

\begin{example}[Random graph inference]
\label{example:GRPDG_bias_correction}
As an illustrative example, we consider the bias correction for the eigenvectors of a random graph model with a low-rank edge probability matrix. Suppose $\bP$ satisfies Assumption \ref{assumption:Signal_strength}, Assumption \ref{assumption:Eigenvalue_separation}, and Assumption \ref{assumption:Eigenvector_delocalization} with $\Delta_n = \Theta(n\rho_n)$. 
Furthermore, suppose $A_{ij}\sim\mathrm{Bernoulli}(p_{ij})$ independently for all $i,j\in [n]$, $i\leq j$, where $p_{ij}$ is the $(i, j)$th entry of $\bP$, and $A_{ij} = A_{ji}$ if $i > j$. The random matrix $\bA = [A_{ij}]_{n\times n}$ is called the adjacency matrix of a generalized random dot product graph. By Theorem \ref{thm:Eigenvector_Expansion}, Theorem \ref{thm:bias_corrected_ASE}, and the concentration bound \eqref{alphan:bound}, we have the following stochastic expansions for $\widehat{\bu}_k$ and $\widehat{\bu}_k^{(c)}$:
\begin{align*}
&\left\|\widehat{\bu}_k\mathrm{sgn}(\widehat{\bu}_k\transpose\bu_k) - \bu_k - \bpsi_k^{\widehat{\bu}_k}(\bE) - \bq_k^{\widehat{\bu}_k}(\bE)\right\|_2 = \widetilde{O}_{\prob}\left\{\frac{(\log n)^{3\xi}}{(n\rho_n)^{3/2}}\right\},\\
&\left\|\widehat{\bu}_k^{(c)}\mathrm{sgn}(\widehat{\bu}_k\transpose\bu_k) - \bu_k - \bpsi_k^{\widehat{\bu}_k}(\bE) - \{\bq_k^{\widehat{\bu}_k}(\bE) - \expect \bq_k^{\widehat{\bu}_k}(\bE)\}\right\|_2 = \widetilde{O}_{\prob}\left\{\frac{(\log n)^{3\xi}}{(n\rho_n)^{3/2}}\right\},
\end{align*}
where $\bpsi_k^{\widehat{\bu}_k}(\bE)$ and $\bq_k^{\widehat{\bu}_k}(\bE)$ are defined in \eqref{eqn:eigenvector_first_order_term} and \eqref{eqn:eigenvector_second_order_term}. 
Because $\max_{i,j\in [n]}|E_{ij}| \leq 1$ almost surely, it follows from Assumption \ref{assumption:Eigenvector_delocalization} that  $\|\bpsi_k^{\widehat{\bu}_k}(\bE)\|_2\leq Cn^{c_1}$ and $\|\bq_k^{\widehat{\bu}_k}(\bE)\|_2\leq Cn^{c_2}$ almost surely for some fixed constants $C, c_1, c_2 > 0$. This entails that
\begin{align*}
&\expect\|\widehat{\bu}_k\mathrm{sgn}(\widehat{\bu}_k\transpose\bu_k) - \bu_k\|_2^2 - \expect\|\widehat{\bu}_k^{(c)}\mathrm{sgn}(\widehat{\bu}_k\transpose\bu_k) - \bu_k\|_2^2
 = \|\expect \bq_k^{\widehat{\bu}_k}(\bE)\|_2^2 + O\left\{\frac{(\log n)^{6\xi}}{(n\rho_n)^{3}}\right\}.
\end{align*}
Under the condition that $\min_{i,j\in [n]}p_{ij}\geq c\rho_n$ for some constant $c > 0$, it is clear that  $\|\expect \bq_k^{\widehat{\bu}_k}(\bE)\|_2 \geq c'/(n\rho_n)$ for some fixed constant $c' \geq 0$. Hence, we conclude with the following comparison between the mean-squared error of eigenvectors for generalized random dot product graphs and their bias-corrected counterpart:
\[
\liminf_{n\to\infty}(n\rho_n)^2\{\expect\|\widehat{\bu}_k\mathrm{sgn}(\widehat{\bu}_k\transpose\bu_k) - \bu_k\|_2^2 - \expect\|\widehat{\bu}_k^{(c)}\mathrm{sgn}(\widehat{\bu}_k\transpose\bu_k) - \bu_k\|_2^2\}\geq 0.
\]

Below, we provide a toy simulation to illustrate the above comparison. Let $d = 1$, $\lambda_1 = n(p^2 + q^2)/2$, $\bu_1 = \lambda_1^{-1/2}[p,\ldots,p,q,\ldots,q]\transpose$, and $\bP = \lambda_1\bu_1\bu_1\transpose$. Suppose $A_{ij}\sim\mathrm{Bernoulli}(p_{ij})$ independently for all $i\leq j$, $i,j\in [n]$, $A_{ij} = A_{ji}$ if $i > j$, where $p_{ij}$ is the $(i, j)$th entry of $\bP$. Here, we let $n$ range over $\{320, 640, 1280, 2560, 5120\}$ and $\rho_n \in \{n^{-1/2}, n^{-1/3}, n^{-1/4}\}$. Given an observed adjacency matrix $\bA$, we compute the leading sample eigenvector $\widehat{\bu}_1$, its bias-corrected version $\widehat{\bu}^{(c)}_1$, and their corresponding estimation errors $\|\widehat{\bu}_1\mathrm{sgn}(\widehat{\bu}_1\transpose\bu_1) - \bu_1\|_2^2$ and $\|\widehat{\bu}_1^{(c)}\mathrm{sgn}(\widehat{\bu}_1\transpose\bu_1) - \bu_1\|_2^2$. The same experiment is repeated for $500$ independent Monte Carlo replicates, and the boxplots of
\[
(n\rho_n)^2\{\|\widehat{\bu}_1\mathrm{sgn}(\widehat{\bu}_1\transpose\bu_1) - \bu_1\|_2^2 - \|\widehat{\bu}_1^{(c)}\mathrm{sgn}(\widehat{\bu}_1\transpose\bu_1) - \bu_1\|_2^2\}
\] 
are visualized in Figure \ref{fig:GRDPG_bias_correction}. It is clear that the bias-corrected eigenvector $\widehat{\bu}_k^{(c)}$ has less higher-order estimation error than the sample eigenvector $\widehat{\bu}_k$. 
\begin{figure}[t]
\centerline{\includegraphics[width = 12cm]{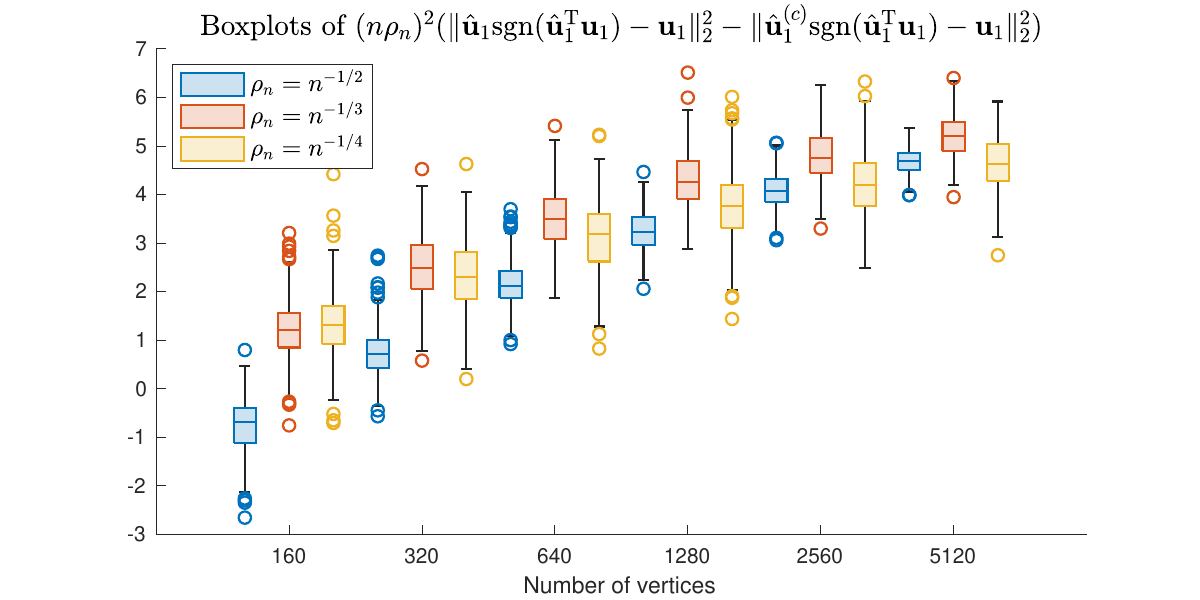}}
\caption{Boxplots of $(n\rho_n)^2\{\|\widehat{\bu}_1\mathrm{sgn}(\widehat{\bu}_1\transpose\bu_1) - \bu_1\|_2^2 - \|\widehat{\bu}_1^{(c)}\mathrm{sgn}(\widehat{\bu}_1\transpose\bu_1) - \bu_1\|_2^2\}$ across $500$ independent Monte Carlo replicates for $n \in \{320,640,1280,2560,5120\}$ and $\rho_n\in\{n^{-1/2},n^{-1/3},n^{-1/4}\}$ for Example \ref{example:GRPDG_bias_correction}. }
\label{fig:GRDPG_bias_correction}
\end{figure}
\end{example}


\subsection{Entrywise Edgeworth expansion for eigenvectors}
\label{sub:entrywise_eigenvector_edgeworth_expansion}

In addition to the bias correction for the eigenvectors discussed in Section \ref{sub:bias_correction}, another interesting application of Theorem \ref{thm:Eigenvector_Expansion} is a fine-grained higher-order approximation to the distributions of the entries of eigenvectors. This goal is achieved through an Edgeworth expansion formula. The Edgeworth expansion can be applied to justify the higher-order correctness of the eigenvector bootstrap for the signal-plus-noise matrix model \eqref{eqn:signal_plus_noise_matrix_model}, as will be seen in Section \ref{sub:entrywise_eigenvector_bootstrap}. This section addresses the Edgeworth expansion for the distributions of the entries of eigenvectors. 

The starting point of the Edgeworth expansion analysis is the second-order eigenvector stochastic expansion formula \eqref{eqn:eigenvector_expansion} in Theorem \ref{thm:Eigenvector_Expansion}. It can be alternatively written in the following entrywise form:
\begin{align*}
\widehat{u}_{ik}\mathrm{sgn}(\bu_k\transpose\widehat{\bu}_k) - u_{ik} -  b_{ik}(\widehat{\bu}_k) & = \frac{\be_i\transpose\bE\bu_k}{\lambda_k} + \frac{\bv_{ik}\transpose\bE\bu_k}{\lambda_k} + \frac{\be_i\transpose(\bE^2 - \expect\bE^2)\bu_k}{\lambda_k^2} + \be_i\transpose\bt_k',
\end{align*}
where $\widehat{u}_{ik}$, $u_{ik}$, $b_{ik}(\widehat{\bu}_k)$ denote the $i$th entries of $\widehat{\bu}_k$, $\bu_k$, $\bb_k(\widehat{\bu}_k)$, respectively, and
\begin{align}\label{def:v}
\bv_{ik} & = -u_{ik}\bu_k + \sum_{m\in [d]/\{k\}}\frac{\lambda_mu_{im}\bu_m}{\lambda_k - \lambda_m},\\\nonumber
\bt_{k}' & = - \frac{3\bu_k\bu_k\transpose(\bE^2 - \expect\bE^2)\bu_k}{2\lambda_k^{2}} + \sum_{m\in [d]/\{k\}}\frac{\lambda_m\bu_m\bu_m\transpose(\bE^2 - \expect\bE^2)\bu_k}{\lambda_k^{2}(\lambda_k - \lambda_m)} + \bt_k.
\end{align}
The bias-corrected eigenvector $\widehat{\bu}_{k}^{(c)}$ (defined in~\eqref{def:biascorr}) also yields the same expansion:
\begin{align}\label{eqn:bias_corrected_eigenvector_expansion}
\widehat{u}_{ik}^{(c)}\mathrm{sgn}(\bu_k\transpose\widehat{\bu}_k) - u_{ik} = \frac{\be_i\transpose\bE\bu_k}{\lambda_k} + \frac{\bv_{ik}\transpose\bE\bu_k}{\lambda_k} + \frac{\be_i\transpose(\bE^2 - \expect\bE^2)\bu_k}{\lambda_k^{2}} + \be_i\transpose\bt_{k}^{(c)},
\end{align}
where $\widehat{u}_{ik}^{(c)}$ denotes the $i$th entry of $\widehat{\bu}_k^{(c)}$. Theorem \ref{thm:Eigenvector_Expansion} and Theorem \ref{thm:bias_corrected_ASE} guarantee that
\begin{align*}
\|\bt_k'\|_\infty\vee\|\bt_{k}^{(c)}\|_\infty & = \widetilde{O}_{\prob}\left[\frac{\sqrt{n}\rho_n}{\Delta_n^2}\max\left\{\frac{1}{q_n},\frac{(n\rho_n)^{1/2}}{\Delta_n},\frac{(n\rho_n)^{3/2}}{\Delta_n^2}\right\}(\log n)^{3\xi}\right],
\end{align*}
after invoking \eqref{alphan:bound}. Therefore, it is sufficient to consider $\widehat{u}_{ik}\mathrm{sgn}(\bu_k\transpose\widehat{\bu}_k) - u_{ik} - b_{ik}(\widehat{\bu}_k)$ without loss of generality. 

In the literature of Edgeworth expansion, researchers are mainly interested in the distribution of the studentized estimator because they can provide insight into the higher-order correctness of the so-called percentile-$t$ bootstrap \cite{hall2013bootstrap}. In our current entrywise eigenvector estimation context, this requires us to find an estimator for the asymptotic variance of $\widehat{u}_{ik}\mathrm{sgn}(\widehat{\bu}_k\transpose\bu_k) - u_{ik} - b_{ik}(\widehat{\bu}_k)$. Specifically, let
\begin{align}
\label{eqn:plug_in_estimate_variance}
s_{ik}^2 = \var\left(\frac{\be_i\transpose\bE\bu_k}{\lambda_k}\right) = \frac{1}{\lambda_k^{2}}\sum_{j = 1}^n\sigma_{ij}^2u_{jk}^2,\quad
\widehat{s}_{ik}^2 = \frac{1}{\widehat{\lambda}_k^2}\sum_{j = 1}^n(A_{ij} - \widehat{p}_{ij})^2\widehat{u}_{jk}^2,
\end{align}
where $\widehat{p}_{ij}$ is the $(i, j)$th entry of $\widehat{\bP} := \bU_\bA\bS_\bA\bU_\bA\transpose$. Here, $s_{ik}^2$ is exactly the asymptotic variance of $\widehat{u}_{ik}^{(c)}\mathrm{sgn}(\widehat{\bu}_k\transpose\bu_k) - u_{ik}$, and $\widehat{s}_{ik}^2$ is a natural plug-in estimator for $s_{ik}^2$.  
We thus consider the studentized entrywise eigenvector statistic  
\[
T_{ik} = \frac{\widehat{u}_{ik}\mathrm{sgn}(\widehat{\bu}_k\transpose\bu_k) - u_{ik} - b_{ik}(\widehat{\bu}_k)}{\widehat{s}_{ik}}
\]
for each $i\in [n]$ and $k\in [p]$. Deriving the Edgeworth expansion of $T_{ik}$ requires not only the aforementioned entrywise expansion formula \eqref{eqn:bias_corrected_eigenvector_expansion} but also an expansion formula for $\widehat{s}_{ik}^2$, the latter of which is formally stated in Lemma \ref{lemma:variance_expansion} below. 
\begin{lemma}\label{lemma:variance_expansion}
Suppose Assumptions \ref{assumption:Signal_strength}--\ref{assumption:Eigenvector_delocalization} hold and $\min_{i\in[n]}s_{ik}^2 = \Theta(\rho_n/\Delta_n^2)$. Then, uniformly over $i\in[n]$,
\begin{align*}
\frac{\widehat{s}_{ik}^2 - s_{ik}^2}{s_{ik}^2}
& = \sum_{j = 1}^n\frac{(E_{ij}^2 - \sigma_{ij}^2)u_{jk}^2}{s_{ik}^2\lambda_k^{2}} + \sum_{a,b = 1}^n\frac{2E_{ia}^2E_{ab}u_{ak}u_{bk}}{\lambda_k^3s_{ik}^2}
\\
&\quad + 
\widetilde{O}_{\prob}\left\{\frac{(\log n)^{2\xi}}{n\rho_n^{1/2}} + \frac{(\log n)^{2\xi}}{\Delta_n} + \frac{n\rho_n(\log n)^{2\xi}}{\Delta_n^2}\right\}.
\end{align*}
\end{lemma}
Now we combine expansion \eqref{eqn:bias_corrected_eigenvector_expansion} with Lemma \ref{lemma:variance_expansion} to sketch the derivation for the Edgeworth expansion of $T_{ik}$. Loosely speaking, the entrywise eigenvector expansion \eqref{eqn:bias_corrected_eigenvector_expansion} and the variance expansion (Lemma \ref{lemma:variance_expansion}) suggest that
\begin{align}
T_{ik} & = \frac{\widehat{u}_{ik}\mathrm{sgn}(\widehat{\bu}_k\transpose\bu_k) - u_{ik} - b_{ik}(\widehat{\bu}_k)}{{s}_{ik}}\left(1 + \frac{\widehat{s}_{ik}^2 - s_{ik}^2}{s_{ik}^2}\right)^{-1/2}\nonumber\\
&\approx \left\{\frac{\be_i\transpose\bE\bu_k}{s_{ik}\lambda_k} + \frac{\bv_{ik}\transpose\bE\bu_k}{s_{ik}\lambda_k} + \frac{\be_i\transpose(\bE^2 - \expect\bE^2)\bu_k}{s_{ik}\lambda_k^{2}}\right\}\left(1 - \frac{\widehat{s}_{ik}^2 - s_{ik}^2}{2s_{ik}^2}\right)\nonumber\\
\label{eqn:Tik_sharp_and_deltaik}
&\approx \frac{\be_i\transpose\bE\bu_k}{s_{ik}\lambda_k} - \frac{1}{2}\frac{\be_i\transpose\bE\bu_k}{s_{ik}\lambda_k}\sum_{j = 1}^n\frac{(E_{ij}^2 - \sigma_{ij}^2)u_{jk}^2}{s_{ik}^2\lambda_k^{2}}\\
\label{eqn:Delta_ik_approximation}
&\quad + \frac{\bv_{ik}\transpose\bE\bu_k}{s_{ik}\lambda_k} + \frac{\be_i\transpose(\bE^2 -\expect\bE^2)\bu_k}{s_{ik}\lambda_k^{2}} - \frac{\be_i\transpose\bE\bu_k}{s_{ik}\lambda_k}\sum_{a, b = 1}^n\frac{E_{ia}^2E_{ab}u_{ak}u_{bk}}{\lambda_k^3s_{ik}^2}.
\\\nonumber
&\approx T_{ik}^\sharp - \frac{1}{2}T_{ik}^\sharp \delta_{ik} + \Delta_{ik},
\end{align}
where we define
\begin{align*}
T_{ik}^\sharp &= \frac{\be_i\transpose\bE\bu_k}{s_{ik}\lambda_k},\quad
\delta_{ik} = \sum_{j = 1}^n\frac{(E_{ij}^2 - \sigma_{ij}^2)u_{jk}^2}{s_{ik}^2\lambda_k^{2}},\\
\Delta_{ik} &= \sum_{a, b\in [n]\backslash\{i\}}\left\{\frac{(E_{ia}/\lambda_k + v_a)u_{bk}}{s_{ik}\lambda_k} - \frac{\be_i\transpose\bE\bu_k E_{ia}^2u_{ak}u_{bk}}{\lambda_k^4s_{ik}^3}\right\}E_{ab}.
\end{align*}
It can be shown that line \eqref{eqn:Delta_ik_approximation} can be well approximated by $\Delta_{ik}$. 
Therefore, it is sufficient to establish the Edgeworth expansion for $T_{ik}^\sharp - (1/2)T_{ik}^\sharp \delta_{ik} + \Delta_{ik}$. As seen in the proof, the sum of the first two terms $T_{ik}^\sharp - (1/2)T_{ik}^\sharp\delta_{ik}$ jointly determines the Edgeworth expansion formula. 

The key to the success of the Edgeworth expansion for $T_{ik}$ is the following critical observation: The conditional distribution of $\Delta_{ik}$ given $\be_i\transpose\bE = [E_{i1},\ldots,E_{in}]$ is approximately Gaussian. This is because $\Delta_{ik}$ can be written as a sum of mean-zero independent random variables $h_{ab}^{(ik)}(\be_i\transpose\bE)E_{ab}$, $a,b\neq i$, given $\be_i\transpose\bE$, where $\{h_{ab}^{(ik)}(\be_i\transpose\bE):a,b\in[n],a,b\neq i\}$ are functions of $\be_i\transpose\bE$ and can be shown to be quite homogeneous. In particular, under the condition $\max_{i,j\in[n]}\expect |E_{ij}|^3 = o(n\rho_n^2/(\Delta_n\sqrt{\log n}))$, the conditional asymptotic normality of $\Delta_{ik}$ given $\be_i\transpose\bE$ has a self-smoothing effect on the CDF of $T_{ik}$, allowing us to circumvent the Cram\'er's condition regarding the smoothness of the distributions of $[E_{ij}]_{n\times n}$. The idea of self-smoothing also appears in the Edgeworth expansion for network moments \cite{10.1214/21-AOS2125}. The difference is that our analysis relies on a collection of concentration inequalities for random variables satisfying Assumption \ref{assumption:Noise_matrix_distribution}, whereas the authors of \cite{10.1214/21-AOS2125} leverage the properties of $U$-statistics, and they are not applicable in our current setup due to the asymmetry nature of the studentized entrywise eigenvector statistic.

We are now in a position to present the formal statement of the entrywise eigenvector Edgeworth expansion in Theorem \ref{thm:edgeworth_expansion} below, which is the main contribution of this section.
\begin{theorem}\label{thm:edgeworth_expansion}
Suppose Assumptions \ref{assumption:Signal_strength}--\ref{assumption:Eigenvector_delocalization} hold. Assume that $\min_{i\in [n]}s_{ik}^2 = \Theta(\rho_n/\Delta_n^2)$, $\beta_\Delta\in[1/2 - \eps, 1/2]$, and $\eps \geq 1/3$. Let $\beta = \max_{i,j\in [n]}\expect|E_{ij}|^3$. Further, assume that either one of the following conditions holds:
\begin{enumerate}[(i)]
  \item $\beta = o(n\rho_n^2/(\Delta_n\sqrt{\log n}))$;
  \item $\beta = \Omega(n\rho_n^2/(\Delta_n\sqrt{\log n}))$ but the Cram\'er's condition holds: There exists some constant $\eta > 0$ such that $\limsup_{|t|\to\infty}|\expect e^{\mathbbm{i}t\rho_nE_{ij}/\beta}| \leq 1 - \eta$. 
\end{enumerate}
Let
\begin{align}\label{eqn:edgeworth_expansion_formula}
G_n^{(ik)}(x) = \Phi(x) + \frac{(2x^2 + 1)}{6}\phi(x)\sum_{j = 1}^n\frac{\expect E_{ij}^3u_{jk}^3}{s_{ik}^3\lambda_k^3},
\end{align}
where $\Phi(x)$ and $\phi(x)$ are the cumulative distribution function and probability density function of $\mathrm{N}(0, 1)$, respectively.
Then
\begin{align*}
&\sup_{x\in \mathbb{R}}\left|\prob\left(T_{ik}\leq x\right) - G_n^{(ik)}(x)\right|\\
&\quad = O\left\{
\frac{(\log n)^{3\xi}}{n\rho_n^{1/2}} + \frac{(\log n)^{3\xi}}{\Delta_n} + \frac{(n\rho_n)(\log n)^{3\xi}}{\Delta_n^2} + \frac{(\log n)^{3\xi\vee4}}{q_n^2} + \frac{(n\rho_n)^2(\log n)^{3\xi}}{\Delta_n^3}
\right\}.
\end{align*}
\end{theorem}

Theorem \ref{thm:edgeworth_expansion} can be applied to obtain a higher-order accurate distributional approximation to the studentized entrywise eigenvector statistics in a broad collection of concrete models. Below, we exemplify Theorem \ref{thm:edgeworth_expansion} in the context of low-rank matrix denoising.
\begin{example}[Low-rank matrix denoising]
\label{example:matrix_denoising}
The low-rank matrix denoising model includes a broad class of high-dimensional models and has been extensively studied during the past decade 
\cite{BENAYCHGEORGES2011494,BENAYCHGEORGES2012120,BURA20082275,6545395,Capitaine2018,10.1214/08-AOP394,10.3150/19-BEJ1129,10.1214/14-AOS1257,6846297,SHABALIN201367}. It has also been broadly applied in signal processing and image analysis \cite{elad2010sparse,212753}. Formally, suppose $\bM$ is a $p_1\times p_2$ matrix with rank $r$ and compact singular value decomposition (SVD) $\bM = \sum_{k = 1}^r\sigma_k\bu_k\bv_k\transpose$, where $(\bu_k)_{k = 1}^r$ and $(\bv_k)_{k = 1}^r$ are orthonormal singular vectors, and $\bLambda = \mathrm{diag}(\sigma_1,\ldots,\sigma_r)$ with $\sigma_1 > \ldots > \sigma_r > 0$. Let $\bY = [Y_{ij}]_{p_1\times p_2}$ be a $p_1\times p_2$ noise matrix whose entries are independent mean-zero random variables, and one observes the noisy version $\bX = [X_{ij}]_{p_1\times p_2}$ of $\bM$ through the additive model
\begin{align}
\label{eqn:matrix_denoising}
\bX = \bM + \bY.
\end{align}
Neither the signal matrix $\bM$ nor the matrix $\bY$ is accessible to practitioners. Now let $\widehat{\sigma}_1\geq\ldots\geq\widehat{\sigma}_r$ be the $r$-largest singular values of $\bX$, and $(\widehat{\bu}_k)_{k = 1}^r$, $(\widehat{\bv}_k)_{k = 1}^r$ are the associated left and right singular vectors. We
use the sample left singular vectors $(\widehat{\bu}_k)_{k = 1}^r$ to estimate the population left singular vectors $(\bu_k)_{k = 1}^r$.
Denote by $n = p_1 + p_2$ and $\widehat{\bM} = [\widehat{M}_{ij}]_{p_1\times p_2} = \sum_{k = 1}^r\widehat{\sigma}_k\widehat{\bu}_k\widehat{\bv}_k\transpose$. 

\begin{theorem}\label{thm:matrix_denoising}
Under the above setup, assume the following conditions hold:
\begin{enumerate}[(a)]
    \item There exist constants $C, C_1, C_2 > 0$, such that
    $\sigma_1/\sigma_r\leq C$ and $C_1\leq p_1/p_2\leq C_2$.
    \item There exists a constant $\delta_0 > 0$ such that $\min_{k\in[r - 1]}(\sigma_k - \sigma_{k + 1})\geq \delta_0\sigma_r$. 
    \item $(Y_{ij}:i\in[p_1],j\in[p_2])$ are independent mean-zero random variables, and there exist a constant $\beta_\Delta \in (0, 1/2]$ and an $n$-dependent quantity $\rho_n\in(0, 1]$, such that 
    \[
    \var(y_{ij}) =  \rho_n,\;
    \sup_{p\geq 1}p^{-1/2}(\expect|Y_{ij}|^p)^{1/p} = \Theta(\rho_n^{1/2}),\;
    \frac{\sigma_r\sqrt{\log n}}{n\rho_n^{1/2}}\to 0,\; \frac{\sigma_r}{\sqrt{n\rho_n}} = \Theta(n^{\beta_\Delta}).
    \]
    \item $\max_{k\in[r]}(\|\bu_k\|_\infty\vee\|\bv_k\|_\infty) = O(1/\sqrt{n})$. 
\end{enumerate}
For any fixed $k\in\{1,\ldots,r\}$, $i\in \{1,\ldots,p_1\}$, we  define the following quantities, 
\begin{align*}
&
\widehat{\tau}_{ik}^2 = \sum_{j = 1}^{p_2}\frac{(X_{ij} - \widehat{M}_{ij})^2\widehat{v}_{jk}^2}{\widehat{\sigma}_k^2},\quad
b_{ik} =-\frac{\rho_n}{2\sigma_k^2}p_1 u_{ik},\quad
T_{ik} = \frac{\widehat{u}_{ik} - u_{ik} - b_{ik}}{\widehat{\tau}_{ik}},\\
&G_n^{(ik)}(x) = \Phi(x) + \frac{(2x^2 + 1)}{6}\phi(x)\sum_{j = 1}^{p_2}\frac{\expect Y_{ij}^3v_{jk}^3}{\rho_n^{3/2}}.
\end{align*}
Here, $\Phi(x)$ and $\phi(x)$ are the cumulative distribution function and the probability density function of $\mathrm{N}(0, 1)$, respectively. 
Then for any $\xi > 1$,
\begin{equation}
\label{eqn:Edgeworth_expansion_matrix_denoising}
\begin{aligned}
&\sup_{x\in\mathbb{R}}|\prob(T_{ik} \leq x) - G_n^{(ik)}(x)|\\
&\quad = O\left\{
\frac{(\log n)^{3\xi}}{\sigma_r} + \frac{(\log n)^{3\xi}}{n^{2\beta_\Delta}} + \frac{(2\log n)^{3\xi\vee4 + 33\log\log n}}{n} + \frac{(n\rho_n)^{1/2}(\log n)^{3\xi}}{n^{3\beta_\Delta}}
\right\}
\end{aligned}
\end{equation}
\end{theorem} 
In the context above, the entrywise eigenvector Berry-Esseen bound in \cite{agterberg2021entrywise} entails that
\begin{align}\label{eqn:Berry_Esseen_Agterberg}
\left\|\prob\left(\frac{\widehat{u}_{ik} - u_{ik}}{s_{ik}} \leq x\right) - \Phi(x)\right\|_\infty = O\left(\frac{\log n}{n^{\beta_\Delta}} + \sqrt{\frac{\log n}{n}}\right);
\end{align}
See their Corollary 1 there for the details. In contrast, under the condition
\[
\frac{(\log n)^{3\xi}}{\sqrt{n\rho_n}} + \frac{(n\rho_n)^{1/2}(\log n)^{3\xi}}{n^{2\beta_\Delta}} = O(1),
\]
Theorem \ref{thm:matrix_denoising} provides a more refined distribution approximation to the studentized $\widehat{u}_{ik} $ (where the asymptotic variance $s_{ik}^2$ is replaced by a plug-in estimator) through an Edgeworth expansion formula. Note that the higher-order term in the Edgeworth expansion satisfies
\[
\sup_{x\in\mathbb{R}}\left|\frac{(2x^2 + 1)}{6}\phi(x)\sum_{j = 1}^{p_2}\frac{\expect y_{ij}^3v_{jk}^3}{\rho_n^{3/2}}\right| = O\left(\frac{1}{\sqrt{n}}\right).
\]
Theorem \ref{thm:matrix_denoising} then implies the following improved Berry-Esseen bound compared to \eqref{eqn:Berry_Esseen_Agterberg}:
\begin{align*}
\|\prob(T_{ik}\leq x) - \Phi(x)\|_\infty = O\left(\frac{1}{n^{\beta_\Delta}}\right).
\end{align*}
Note that Theorem \ref{thm:matrix_denoising} does not require the commonly imposed Cram\'er's condition on the smoothness or non-lattice property of the distributions of $Y_{ij}$'s, and only the condition $\sigma_r\sqrt{\log n} = o(n\rho_n^{1/2})$ is required.
\end{example}

Theorem \ref{thm:edgeworth_expansion} requires either $\max_{i,j\in [n]}\expect|E_{ij}|^3 = o(n\rho_n^2/(\Delta_n\sqrt{\log n}))$ or the Cram\'er's condition. Unfortunately, random graph models do not satisfy either one. Specifically, for the random graph model considered in Example \ref{example:random_graph} and \ref{example:GRPDG_bias_correction} with $\min_{i,j\in[n]}p_{ij} = \Theta(\rho_n)$, it is clear that Assumption \ref{assumption:Noise_matrix_distribution} is satisfied with $q_n = \sqrt{n\rho_n}$ and $\beta = \Theta(\rho_n) = \Theta(n\rho_n^2/\Delta_n)$ so that neither condition (i) nor condition (ii) in Theorem \ref{thm:edgeworth_expansion} is satisfied; Note that $\bE$ is purely discrete and thus Cram\'er's condition no longer holds. To cover this important class of signal-plus-noise matrix models, we follow the idea of \cite{LAHIRI1993247,10.1214/aos/1176345636} and modify Theorem \ref{thm:Eigenvector_Expansion} by injecting an artificial Gaussian smoother. This modification is formally stated in Theorem \ref{thm:edgeworth_expansion_smoothed} below. 

\begin{theorem}\label{thm:edgeworth_expansion_smoothed}
Suppose Assumptions \ref{assumption:Signal_strength}--\ref{assumption:Eigenvector_delocalization} hold and $\min_{i\in [n]}s_{ik}^2 = \Theta(\rho_n/\Delta_n^2)$. 
Assume that $\beta_\Delta\in[1/2 - \eps, 1/2]$ and $\eps \geq 1/3$.
Let $\beta = \max_{i,j\in [n]}\expect|E_{ij}|^3$ and assume that $\beta = \Omega(n\rho_n^2/(\Delta_n\sqrt{\log n}))$.
Let $z\sim\mathrm{N}(0, 1)$ be independent of $\bA$. 
Then there exists a sufficiently large constant $\tau > 0$, such that for any $\xi > 1$,
\begin{align*}
&\left\|\prob\left(T_{ik} + \tau\sqrt{\frac{\beta^2\log n}{n\rho_n^3}}z\leq x\right) - G_n^{(ik)}(x)\right\|_{\infty}\\
&\quad = O\left\{\frac{(\log n)^{3\xi}}{n\rho_n^{1/2}} + \frac{(\log n)^{3\xi}}{\Delta_n} + \frac{(n\rho_n)(\log n)^{3\xi}}{\Delta_n^2} + \frac{(\log n)^{3\xi}}{q_n^2} + \frac{(n\rho_n)^2(\log n)^{3\xi}}{\Delta_n^3}\right\}.
\end{align*}
\end{theorem}

\begin{remark}
In Theorem \ref{thm:edgeworth_expansion_smoothed}, $z\sim\mathrm{N}(0, 1)$ is artificially injected into the studentized $i$th entry of the sample eigenvector $\widehat{u}_{ik}$ (with the appropriate scaling) to provide a sufficient smoothing effect on the cumulative distribution function of $T_{ik}$. By doing so, the Edgeworth expansion formula provides the correct approximation order with regard to the total variation distance. 
When $\rho_n = \Omega(1)$ and $\beta = \Theta(1)$, the scaling rate $\sqrt{(\log n)/n}$ coincides with the optimal rate in \cite{LAHIRI1993247} in the sense that it does not affect the asymptotic variance of $\widehat{u}_{ik}$ and provides a sufficient smoothing effect to guarantee the validity of the Edgeworth expansion formula. 
The artificial smoother is necessary when neither $\beta$ decays sufficiently fast nor does the Cram\'er's condition hold. This is the case of random graph models as discussed in Examples \ref{example:random_graph} and  \ref{example:GRPDG_bias_correction}. It also motivates the smoothed bootstrap for eigenvectors in low-rank random graphs, which will be introduced in Section \ref{sub:entrywise_eigenvector_bootstrap}. 
\end{remark}

\subsection{Entrywise eigenvector bootstrap}
\label{sub:entrywise_eigenvector_bootstrap}
We now discuss an important consequence of the entrywise eigenvector Edgeworth expansion established in Theorem \ref{thm:edgeworth_expansion}: The higher-order correctness of the entrywise eigenvector bootstrap for signal-plus-noise matrices. As illustrated in Section~\ref{sub:entrywise_eigenvector_edgeworth_expansion} (c.f., Example~\ref{example:matrix_denoising}), $G_n^{(ik)}(n)$ can be a more accurate distribution approximation for $T_{ik}$ than $\Phi(x)$. However, $G_n^{(ik)}(n)$ is not computable in practice because the third-order cumulants $(\expect E^3_{ij})_{i,j\in[n]}$ are unknown. Below, we show that the entrywise eigenvector bootstrap is a fully data-driven remedy and, meanwhile, as accurate as $G_n^{(ik)}(n)$. In general, bootstrap is a method to reconstruct a distribution that closely resembles the original data distribution using only the observed data \cite{efron1994introduction}. In the context of signal-plus-noise matrices \eqref{eqn:signal_plus_noise_matrix_model}, we focus on the case where the noise $E_{ij}$'s are independent and identically distributed sub-exponential random variables and design a residual bootstrap scheme by following the idea of \cite{10.1214/aos/1176345638}. 

Let $(A_{ij} - \widehat{p}_{ij}:i,j\in[n],i\leq j)$ be the residuals after fitting $\bP$ using $\widehat{\bP} = [\widehat{p}_{ij}]_{n\times n} =\bU_\bA\bS_\bA\bU_\bA\transpose$. We centralize the residuals and set $\widehat{E}_{ij} = A_{ij} - \widehat{p}_{ij} - \widehat{\mu}_\bE$, $i,j\in[n]$, $i\leq j$, where
\[
\widehat{\mu}_\bE = \frac{2}{n(n + 1)}\sum_{1\leq i\leq j\leq n}(A_{ij} - \widehat{p}_{ij}).
\]
Let $\widehat{F}_n$ denote the empirical distribution of the centered residuals: 
\[
\widehat{F}_n(\mathrm{d}x) = \frac{2}{n(n + 1)}\sum_{1\leq i\leq j\leq n}\delta_{\widehat{E}_{ij}}(\mathrm{d}x).
\]
The bootstrap distribution of the signal-plus-noise matrix is constructed by resampling the centered residuals from $\widehat{F}_n$. Specifically, let $\bE^* = [E_{ij}^*]_{n\times n}$ be a symmetric $n\times n$ matrix whose upper triangular entries $(E_{ij}^*)_{1\leq i\leq j\leq n}$ are independent and identically distributed from $\widehat{F}_n$,
and we set $\bA^*  = [A_{ij}^*]_{n\times n} = \widehat{\bP} + \bE^*$. 
Let $\sum_{k = 1}^n{\lambda}_k(\bA^*)(\widehat{\bu}_k^*)(\widehat{\bu}_k^*)\transpose$ be the spectral decomposition of $\bA^*$, where $\lambda_1(\bA^*)\geq\ldots\geq\lambda_n(\bA^*)$ are the eigenvalues of $\bA^*$, $(\widehat{\bu}_k^*)_{k = 1}^n$ are the associated orthornomal eigenvectors, and let $\bU_\bA^* = [\widehat{\bu}_1^*,\ldots,\widehat{\bu}_p^*,\widehat{\bu}_{n - q + 1}^*,\ldots,\widehat{\bu}_n^*]$, $\bS_\bA^* = \mathrm{diag}\{\lambda_1(\bA^*),\ldots,\lambda_p(\bA^*), \lambda_{n - q + 1}(\bA^*),\ldots,\lambda_n(\bA^*)\}$. 
Let $\widehat{\bP}^* = \bU_\bA^*\bS_\bA^*(\bU_\bA^*)\transpose$, $\widehat{p}_{ij}^*$ be the $(i, j)$th entry of $\widehat{\bP}^*$, and $\widehat{\lambda}_k^*$ be the $k$th diagonal entry of $\bS_\bA^*$.
Define the bootstrap version of $T_{ik}$ by
\[
T_{ik}^* = \frac{\widehat{u}_{ik}^*\mathrm{sgn}(\widehat{\bu}_k\transpose\widehat{\bu}_k^*) - \widehat{u}_{ik} - \widehat{b}_{ik}}{\widehat{s}_{ik}^*},\quad i\in[n],\;k\in[p],
\]
where $\widehat{u}_{ik}^*$ and $\widehat{b}_{ik}$ are the $i$th entries of $\widehat{\bu}_k^*$ and the estimated bias $\widehat{\bb}_k$ defined in \eqref{eqn:estimated_bias}, respectively, and $(\widehat{s}_{ik}^*)^{2}$ is the bootstrap version of the plug-in variance estimator given by
\[
(\widehat{s}_{ik}^{*})^2 = \sum_{j = 1}^n\frac{(A_{ij}^* - \widehat{p}_{ij}^*)^2\widehat{u}_{jk}^{*2}}{\widehat{\lambda}_k^{*2}}.
\]
Let $\prob$ denote the probability corresponding to $\bA$, $\bE$, and $\prob^*$ denote the conditional probability corresponding to $\bA^*$, $\bE^*$ given $\bA$. 

Theorem \ref{thm:bootstrap_subgaussian} below formally establishes the approximation error bound of the distribution of the studentized entrywise eigenvector bootstrap statistic $T_{ik}^*$. 
\begin{theorem}\label{thm:bootstrap_subgaussian}
Suppose Assumptions \ref{assumption:Signal_strength}--\ref{assumption:Eigenvector_delocalization} hold, and $(E_{ij})_{1\leq i\leq j\leq n}$ are independent and identically distributed sub-exponential random variables such that $\|E_{ij}\|_{\psi_1} = C\rho_n^{1/2}$ for some $C > 0$, where $\|E_{ij}\|_{\psi_1}:=\sup_{p\geq 1}\{p^{-1}(\expect|E_{ij}^p|)^{1/p}\}$. Further assume that $\beta_\Delta\in (0, 1/2]$, $\Delta_n = o(n\rho_n^{1/2}/\sqrt{\log n})$, and $\widehat{\bu}_k\transpose\bu_k > 0$ for all $k\in [p]$. Then for any fixed $i\in[n]$, $k\in[p]$, and any $\xi > 1$, 
\begin{align*}
&\left\|\prob^*(T_{ik}^*\leq x) - \prob(T_{ik}\leq x)\right\|_\infty\\
&\quad = O\left\{
\frac{(\log n)^{3\xi}}{\Delta_n} + \frac{(\log n)^{3\xi}}{n^{2\beta_\Delta}} + \frac{(2\log n)^{3\xi\vee4 + 66\log\log n}}{n} + \frac{(n\rho_n)^{1/2}(\log n)^{3\xi}}{n^{3\beta_\Delta}}
\right\}
\end{align*}
with probability at least $1 - O(n^{-1})$. 
\end{theorem}
To illustrate Theorem \ref{thm:bootstrap_subgaussian}, let us assume that $\Delta_n = \Theta(n\rho_n)$ for simplicity. Then in the context of Theorem \ref{thm:bootstrap_subgaussian}, we obtain
\begin{align}
\label{eqn:sub_Gaussian_bootstrap_error}
\left\|\prob^*(T_{ik}^*\leq x) - \prob(T_{ik}\leq x)\right\|_\infty = O\left\{\frac{(\log n)^{3\xi}}{n\rho_n} + \frac{(2\log n)^{4 + 33\log\log n}}{n}\right\}.
\end{align}
On the other hand, Theorem \ref{thm:edgeworth_expansion} also immediately entails the following rate-exact Berry-Esseen bound:
\begin{align}
\label{eqn:sub_Gaussian_Berry_Esseen}
\left\|\Phi(x) - \prob(T_{ik}\leq x)\right\|_\infty = O\left\{\frac{1}{\sqrt{n}} + \frac{(\log n)^{3\xi}}{n\rho_n}\right\}.
\end{align}
When $\sqrt{n}\rho_n = \omega\{(\log n)^{3\xi}\}$ for some $\xi \in(1, 4/3]$, the leading term on the right-hand side of \eqref{eqn:sub_Gaussian_Berry_Esseen} is $\Theta(n^{-1/2})$, which also dominates the right-hand side of \eqref{eqn:sub_Gaussian_bootstrap_error}. We thereby establish the higher-order correctness of the bootstrap scheme beyond the normal approximation.

The residual bootstrap scheme does not apply to random graphs, where the noises are independent but not identically distributed. For low-rank random graphs, we follow the idea of the parametric bootstrap scheme developed \cite{levin2019bootstrapping}, which we review here. We follow the same set of notations and definitions above but further assume that $\bA= [A_{ij}]_{n\times n}$ is a random adjacency matrix with edge probability matrix $\bP$, namely, $A_{ij}\sim\mathrm{Bernoulli}(p_{ij})$ independently for all $i\leq j$, and $A_{ij} = A_{ji}$ for all $i > j$, where $p_{ij}$ is the $(i, j)$th entry of $\bP$. Also see Example \ref{example:random_graph} and Example \ref{example:GRPDG_bias_correction}. 
The major difference is that rather than bootstrapping $\bA$ by resampling the residuals, we generate the bootstrap samples of $\bA$ directly from Bernoulli distributions. Specifically,
given $\bA$, let $A_{ij}^*\sim\mathrm{Bernoulli}(\widehat{p}_{ij})$ independently for all $i \leq j$, where $\widehat{p}_{ij}$ is the $(i, j)$th entry of $\widehat{\bP} = \bU_\bA\bS_\bA\bU_\bA\transpose$, and set $A_{ij}^* = A_{ji}^*$ for all $i > j$. We then let $\bA^* = [A_{ij}^*]_{n\times n}$ and define $T_{ik}^*$ the same as before. 

However, rather than considering the bootstrap distribution of $T_{ik}$ directly, we consider the so-called smoothed bootstrap by injecting independent artificial Gaussian smoothers with the appropriate scaling into $T_{ik}$ and $T_{ik}^*$, respectively. The smoothed bootstrap is motivated by Theorem \ref{thm:edgeworth_expansion_smoothed} and is rooted in \cite{LAHIRI1993247}. Theorem \ref{thm:bootstrap_GRDPG} below provides the formal approximation error bound for the smoothed entrywise eigenvector bootstrap for low-rank random graphs.
\begin{theorem}\label{thm:bootstrap_GRDPG}
Let $\bP = \bU_\bP\bS_\bP\bU_\bP\transpose$ be a $n\times n$ rank-$d$ matrix with $\bU_\bP\in\mathbb{O}(n, d)$, $\bS_\bP = \mathrm{diag}(\lambda_1,\ldots,\lambda_d)$,
\[
\min\Big(\min_{k\in [d]}|\lambda_k|, \min_{k\in [d - 1]}|\lambda_k - \lambda_{k + 1}|\Big) = \Theta(n\rho_n),
\]
and $\sqrt{n\rho_n} = \Theta(n^{\beta_\Delta})$ for some $\beta_\Delta > 0$. Assume that $\bU_\bP$ satisfies Assumption \ref{assumption:Eigenvector_delocalization}, and further assume that $\delta_1\rho_n\leq p_{ij}\leq \delta_2\rho_n$ for all $i,j\in[n]$ for some constants $\delta_1,\delta_2 \in (0,1]$. Suppose
$A_{ij}\sim\mathrm{Bernoulli}(p_{ij})$ independently for all $i\leq j$, where $p_{ij}$ is the $(i, j)$th entry of $\bP$, and $A_{ij} = A_{ji}$ for all $i > j$. Let $\widehat{\bu}_k\transpose\bu_k > 0$ and $(z,z^*)\sim\mathrm{N}\{(0,0), \mathrm{diag}(1,1)\}$ be independent of $\bA$ and $\bA^*$. Then there exists a sufficiently large $\tau > 0$, such that for any $\xi > 1$,
\[
\left\|\prob^*\left(T_{ik}^* + \tau\sqrt{\frac{\log n}{n\rho_n}}z^*\leq x\right) - \prob\left(T_{ik} + \tau \sqrt{\frac{\log n}{n\rho_n}}z\leq x\right)\right\|_\infty = \widetilde{O}_{\prob}\left\{\frac{(\log n)^{3\xi}}{n\rho_n}\right\}.
\]
\end{theorem}
\section{Numerical results}
\label{sec:numerical_results}

This section presents the results of numerical experiments to illustrate the Edgeworth expansion and bootstrap theory established in Section \ref{sec:applications}. Specifically, Section \ref{sub:higher_order_accuracy_of_edgeworth_expansion} demonstrates the higher-order accuracy of Edgeworth expansion formula, and Section \ref{sub:higher_order_accuracy_of_empirical_edgeworth_expansion_and_bootstrap} compares the approximation quality of the entrywise eigenvector bootstrap and the so-called empirical Edgeworth expansion where the true value of the higher-order term appearing in the Edgeworth expansion formula \eqref{eqn:edgeworth_expansion_formula} is replaced by a plug-in estimate. 

\subsection{Higher-order accuracy of Edgeworth expansion}
\label{sub:higher_order_accuracy_of_edgeworth_expansion}

We consider a weighted two-block stochastic block model (SBM) graph with the block profile matrix and community assignment function given by
\[
\bB = \frac{\Delta_n}{n}\begin{bmatrix}
a & & b\\ b&  & a
\end{bmatrix},\quad
\tau(i) = \left\{\begin{aligned}
&1,&\quad&\mbox{if }i = 1,\ldots,n/2,\\
&2,&\quad&\mbox{if }i = n/2 + 1,\ldots,n,
\end{aligned}\right.
\]
where $\Delta_n$ governs the signal strength of the model, $a, b > 0$ with $a\neq b$, and $n$ is an even integer. The weighted graph adjacency matrix $\bA = [A_{ij}]_{n\times n}$ is generated via the signal-plus-noise structure $A_{ij} = \expect A_{ij} + E_{ij}$, where $(E_{ij}:1\leq i\leq j\leq n)$ are independent and identically distributed (i.i.d.) discrete random variables with 
\begin{align*}
\prob(E_{ij} = 4) = \frac{\rho_n}{5},\quad
\prob(E_{ij} = -1) = \frac{4\rho_n}{5},\quad
\prob(E_{ij} = 0) = 1 - \rho_n.
\end{align*}
We set $\expect A_{ij} = \Delta_na/n$ if $\tau(i) = \tau(j)$ and $\expect A_{ij} = \Delta_nb/n$ if $\tau(i)\neq \tau(j)$, $i\leq j$, $i,j\in [n]$, and $A_{ij} = A_{ji}$ if $i > j$. 
The mean matrix $\bP = \expect\bA$ can be alternatively written as $\bP = \bU_\bP\bS_\bP\bU_\bP\transpose$, where $\bU_\bP = [\bu_1,\bu_2]$, $\bu_1 = [1,\ldots,1]\transpose/\sqrt{n}$, $\bu_2 = [1,\ldots,1,-1,\ldots,-1]\transpose/\sqrt{n}$, $\bS_\bP = \mathrm{diag}\{n\rho_n(a + b)/2, n\rho_n(a - b)/2\}$. 
The number of vertices $n$ ranges in $\{80, 160, 320\}$, and we let $\rho_n = n^{-1/3}$ and the signal-to-noise ratio $\Delta_n/\sqrt{n\rho_n}$ vary across $\{n^{1/6}, n^{1/5}, n^{1/4}\}$. 

One challenge in numerical experiments is obtaining the exact distributions of the studentized entrywise eigenvector statistic $T_{ik} = \{\widehat{u}_{ik}\mathrm{sgn}(\widehat{\bu}_k\transpose\bu_k) - u_{ik} - b_{ik}(\widehat{\bu}_k)\}/\widehat{s}_{ik}$, $i\in [n]$, $k\in [p]$. In general, the closed-form distributions of $T_{ik}$'s are not available, and we resort to Monte Carlo approximations to $F_{T_{ik}}(x) = \prob(T_{ik}\leq x)$. Specifically, we generate $n_{\mathrm{mc}}$ replicates of $\bA$ from the above weighted SBM, compute $T_{ik}^{(t)}$ for each $t\in \{1,2,\ldots,n_{\mathrm{mc}}\}$ (here $t$ is the index of the Monte Carlo replicate), and take the empirical CDF of $(T_{ik}^{(t)})_{t = 1}^{n_{\mathrm{mc}}}$ 
\[
F_{ik}(x) = \frac{1}{n_{\mathrm{mc}}}\sum_{t = 1}^{n_{\mathrm{mc}}}\mathbbm{1}(T_{ik}^{(t)}\leq x)
\]
as a proxy for the true CDF $\prob(T_{ik}\leq x)$. We follow the suggestion in \cite{10.1214/21-AOS2125} and take $n_{\mathrm{mc}} = 10^6 \gg n$. This choice of $n_{\mathrm{mc}}$ ensures that $\|F(x) - \prob(T_{ik}\leq x)\|_\infty = O_{\prob}(n_{\mathrm{mc}}^{-1/2})$ (by Glivenko-Cantelli Theorem, see, for example, Theorem 12.4 in \cite{devroye2013probabilistic}) is negligible compared to the Edgeworth expansion approximation error in Theorem \ref{thm:edgeworth_expansion} but allows for a manageable computation cost of $F_{ik}(x)$. 

The numerical results are visualized in Figure \ref{fig:Binary_WSBM_CDF}. 
Note that in the current context, it is easy to see that, for the entries of the second eigenvector $\widehat{\bu}_2$ and $\bu_2$, the corresponding higher-order term in the Edgeworth expansion formula is zero because $\sum_{j = 1}^n\expect E_{ij}^3u_{j2}^3 = 0$.
Thus, we focus on the entrywise analysis of $\bu_1$. The three columns of Figure \ref{fig:Binary_WSBM_CDF} correspond to different signal-to-noise ratios with $\beta_\Delta\in\{1/6, 1/5, 1/4\}$, respectively. 
The first, third, and fifth rows of Figure \ref{fig:Binary_WSBM_CDF} visualize the comparison among $F_{11}(x)$ (denoted by true CDFs of $T_{ik}$ but computed using the aforementioned Monte Carlo approximations), the normal CDF $\Phi(x)$, and the Edgeworth expansion formula $G_n^{(11)}(x)$'s for $T_{11}$ with $n \in \{80, 160, 320\}$. The second, fourth, and sixth rows of Figure \ref{fig:Binary_WSBM_CDF} show the total variation distances $\|F_{T_{i1}}(x) - \Phi(x)\|_\infty$ and $\|F_{T_{i1}}(x) - G_n^{(i1)}(x)\|_\infty$ with $n \in \{80, 160, 320\}$, respectively, where the horizontal axis labels $i\in[n]$. 
It can be seen that the approximation quality of the Edgeworth expansion formula $G_n^{(i1)}(x)$ is superior to that of the normal CDF $\Phi(x)$ across $n\in\{80,160,320\}$ and $\beta_\Delta\in\{1/6,1/5,1/4\}$. 
\begin{figure}[htbp]
\centerline{
\includegraphics[width = 16cm]{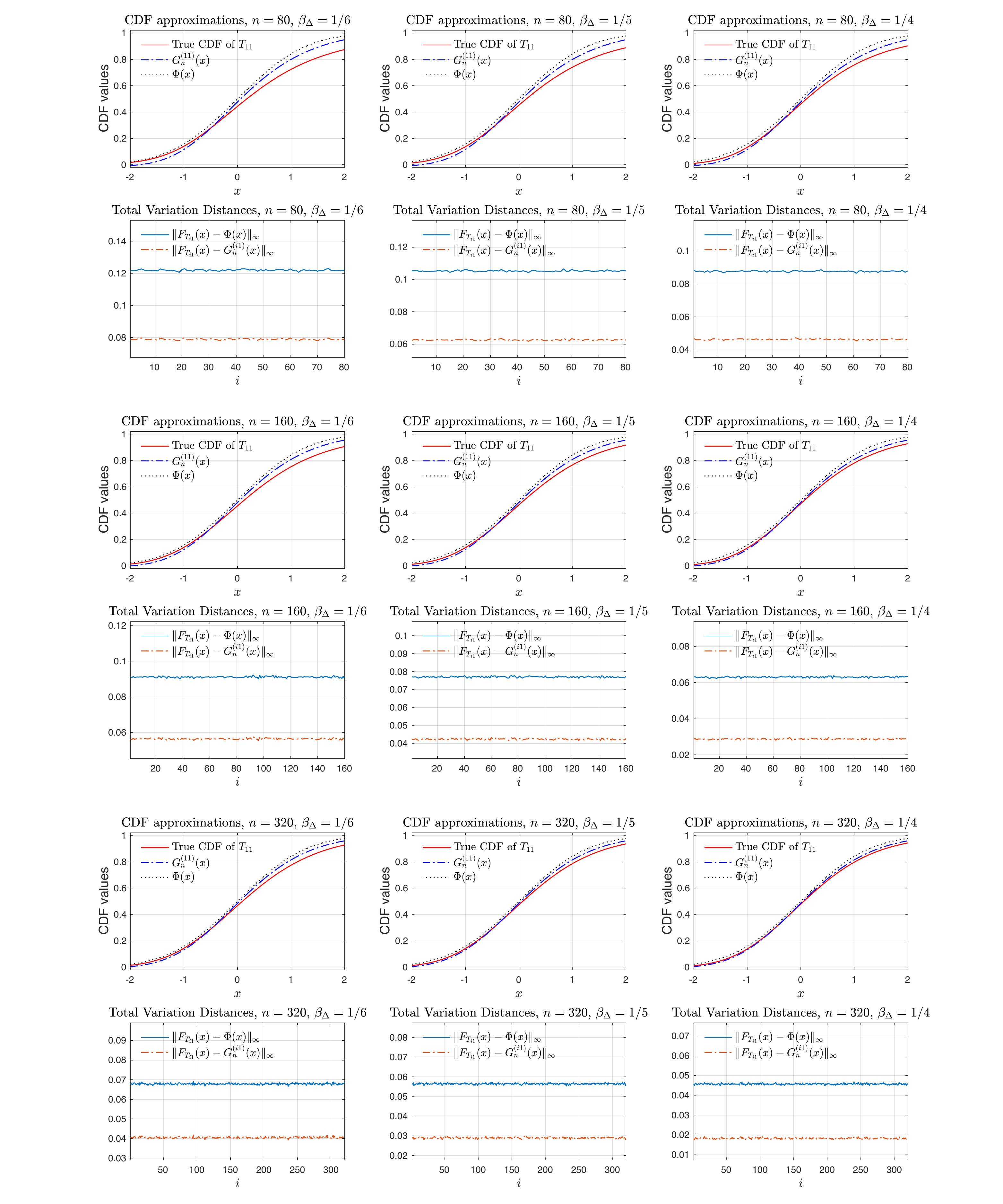}
}
\caption{Simulation results for Section \ref{sub:higher_order_accuracy_of_edgeworth_expansion}: Comparison of the Edgeworth expansion formula $G_n^{(i1)}(x)$ versus the normal CDF $\Phi(x)$ for approximating the true CDF of $T_{i1}$, $i\in [n]$ for $n\in\{80,160,320\}$ and $\Delta_n = n^{\beta_\Delta}(n\rho_n)^{1/2}$ with $\beta_\Delta\in\{1/6,1/5,1/4\}$. }
\label{fig:Binary_WSBM_CDF}
\end{figure}

\subsection{Higher-order accuracy of empirical Edgeworth expansion and bootstrap}
\label{sub:higher_order_accuracy_of_empirical_edgeworth_expansion_and_bootstrap}

Section \ref{sub:higher_order_accuracy_of_edgeworth_expansion} showcases the power of the Edgeworth expansion approximation to the CDF of $T_{ik}$. Nevertheless, the higher-order accuracy of this phenomenon is not directly obtainable in practice because the third-order cumulants $\expect E_{ij}^3u_{jk}^3/(s_{ik}^3\lambda_k^3)$ in the Edgeworth expansion formula \eqref{eqn:edgeworth_expansion_formula} is generally unavailable in real-world problems. This subsection elaborates on a more realistic case when the Edgeworth expansion formula is not directly available to practitioners. Two approximation methods for the CDFs of $T_{ik}$'s are investigated: The entrywise eigenvector bootstrap and the empirical version of the Edgeworth expansion formula. 

We first describe the simulation setup. We set the upper triangular entries $(E_{ij}:i,j\in[n],i\leq j)$ to be independent and identically distributed random variables with probability density function 
\[
f(x) = \frac{1}{\sqrt{\rho_n}}e^{-\rho_n^{-1/2}(x + \rho_n^{1/2})},\quad x\geq -\rho_n^{-1/2}.
\] Namely, $E_{ij}$'s are centered independent $\mathrm{Exp}(\rho_n^{-1/2})$ random variables. We still focus on $T_{i1}$ because the higher-order term in the Edgeworth expansion formula for $T_{i2}$ is zero. The signal matrix $\bP$ is the same as in Section \ref{sub:higher_order_accuracy_of_edgeworth_expansion}. Given a realization of $\bA = \bP + \bE$, where $\bE = [E_{ij}]_{n\times n}$, we consider the following two methods for approximating $\prob(T_{i1}\leq x)$, $i\in [n]$:
\begin{enumerate}[(a)]
    \item Entrywise eigenvector bootstrap using the residual bootstrap scheme for $T_{i1}$ described in Section \ref{sub:entrywise_eigenvector_bootstrap} with $n_{\mathrm{boot}} = 10000$ bootstrap samples. Given bootstrap samples $(T_{i1}^{*(b)})_{b = 1}^{n_{\mathrm{boot}}}$, we use the empirical CDF
    \[
    F_{i1}^*(x) = \frac{1}{n_{\mathrm{boot}}}\sum_{b = 1}^{n_{\mathrm{boot}}}\mathbbm{1}(T_{i1}^{*(b)}\leq x)
    \]
    as a proxy for $F_{T_{i1}^*}(x) = \prob^*(T_{i1}^*\leq x)$. 
    \item Empirical Edgeworth expansion using the formula
    \[
    \widehat{G}_n^{(i1)} = \Phi(x) + \frac{(2x^2 + 1)}{6}\sum_{j = 1}^n\frac{(A_{ij} - \widehat{p}_{ij})^3\widehat{u}_{j1}^3}{\widehat{s}_{i1}^3\widehat{\lambda}_1^3}.
    \]
    In the above formula, $\widehat{p}_{ij}$ is the $(i, j)$th entry of $\widehat{\bP} = \bU_\bA\bS_\bA\bU_\bA\transpose$, $\bU_\bA\bS_\bA\bU_\bA\transpose$ is the truncated spectral decomposition of $\bA$ introduced in Section \ref{sec:model_setup} with $d = 2$, $\widehat{\lambda}_1$ is the second largest of $\bA$ in magnitude, $\widehat{\bu}_1 = [\widehat{u}_{11},\ldots,\widehat{u}_{n1}]\transpose$, and $\widehat{s}_{ik}^2$ is the plug-in estimate of the variance $s_{i1}^2$ defined in \eqref{eqn:plug_in_estimate_variance}. 
\end{enumerate}
We let $n\in\{80, 160, 320\}$, $\Delta_n = (n\rho_n)^{1/2}n^{\beta_\Delta}$ with $\beta_\Delta\in\{1/6, 1/5, 1/4\}$, and set $\rho_n = n^{-1/4}$, so the setup here aligns with that in Section \ref{sub:higher_order_accuracy_of_edgeworth_expansion}.  
The same experiment is repeated $500$ independent times for each pair of $(n, \beta_\Delta)$. The first, third, and fifth rows of Figure \ref{fig:matrix_denoising} present the deviations $|F_{T_{11}}(x) - \Phi(x)|$ (for normal approximation), $|F_{T_{11}}(x) - G_n^{(11)}(x)|$ (for the theoretical Edgeworth expansion), $|F_{T_{11}}(x) - F_{T_{11}^*}(x)|$ (for the bootstrap approximation), and $|F_{T_{11}}(x) - \widehat{G}_n^{(11)}(x)|$ (for the Empirical Edgeworth expansion) as functions of $x\in[-2, 2]$, for $n\in\{80, 160, 320\}$, respectively, using a randomly selected realization of $\bA$. 
The second, fourth, and sixth rows of Figure \ref{fig:matrix_denoising} show
the total variation distances for approximating $F_{T_{i1}}(x)$ (as functions of the entry index $i\in [n]$ as the horizontal axis in these plots) using the above four approaches. The three columns of Figure \ref{fig:matrix_denoising} correspond to different signal-to-noise ratios with $\beta_\Delta\in\{1/6,1/5,1/4\}$, respectively. 
The total variation distances based on the latter two methods are computed for each experiment, and we plot the averages of total variation distances and error bars with one standard deviation (in red shaded regions for the bootstrap approximation and in green shaded regions for the empirical Edgeworth expansion) across $500$ repeated experiments. 
\begin{figure}[htbp]
\centerline{
\includegraphics[width = 17cm]{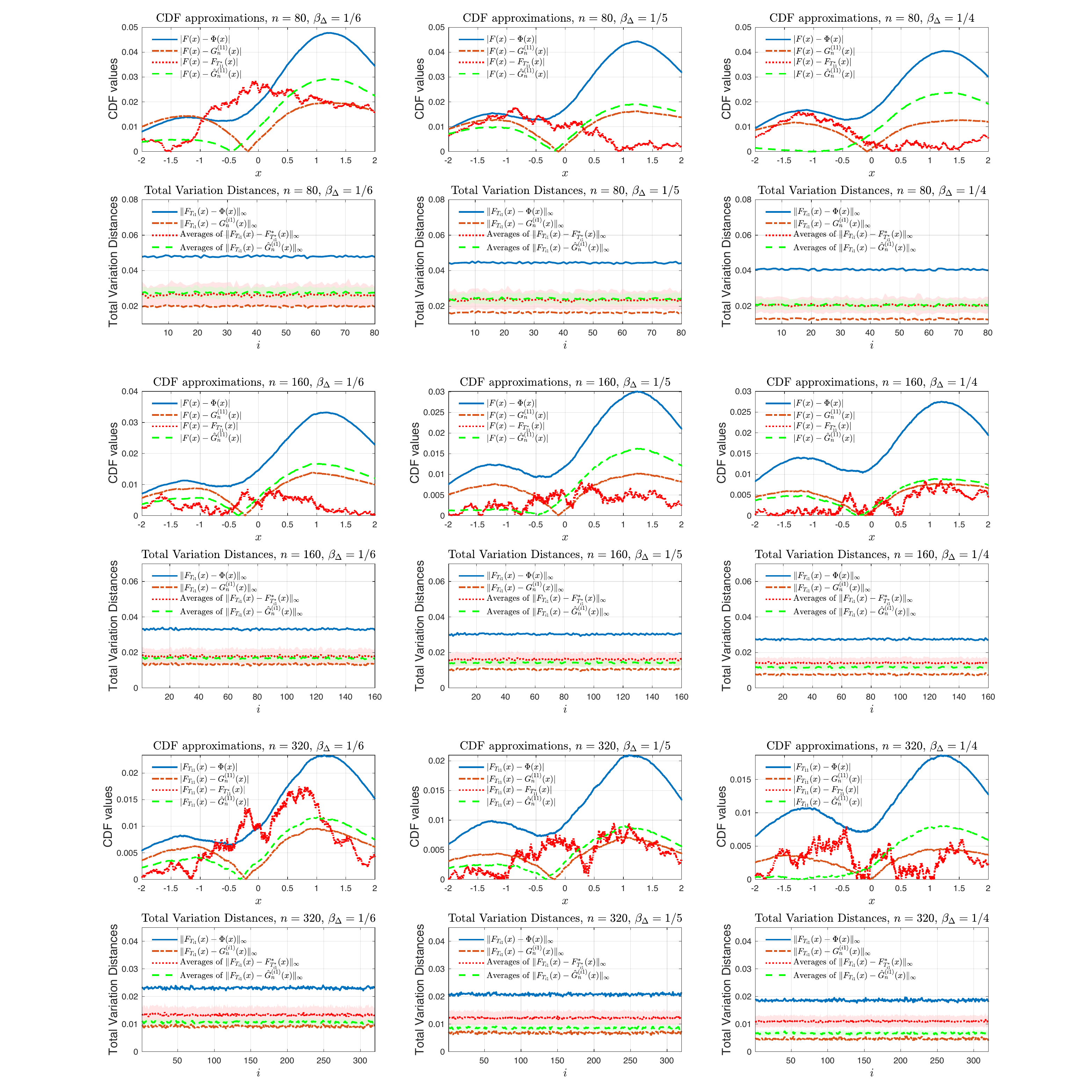}}
\caption{Simulation results for Section \ref{sub:higher_order_accuracy_of_empirical_edgeworth_expansion_and_bootstrap}: CDF approximation errors using the normal approximation (solid lines), the theoretical Edgeworth expansion (dash-dotted lines), the bootstrap approximation (dotted lines), and the empirical Edgeworth expansion Comparison (dashed lines) for the true CDF of $T_{i1}$, $i\in [n]$, where $n$ varies over $\{80,160,320\}$ and $\Delta_n = n^{\beta_\Delta}(n\rho_n)^{1/2}$ with $\beta_\Delta\in\{1/6,1/5,1/4\}$. }
\label{fig:matrix_denoising}
\end{figure}
We observe that the bootstrap approximation has similar accuracy with the empirical Edgeworth expansion when $n$ is small ($n \in \{80, 160\}$), but the empirical Edgeworth expansion slightly outperforms the bootstrap approximation when $n$ is larger with a stronger signal-to-noise ratio ($n = 320$ and $\beta_\Delta = 1/4$). Both methods also significantly outperform the baseline normal approximation with less total variation distances, which is expected by Theorem \ref{thm:edgeworth_expansion} and Theorem \ref{thm:bootstrap_subgaussian}. 

\section{Discussion}
\label{sec:discussion}

In this work, we study the distributions of the entries of the sample eigenvectors in low-rank random matrix models. Specifically, we investigate the higher-order entrywise eigenvector stochastic expansion with uniform error control. The stochastic expansion directly leads to the bias correction for the eigenvectors with reduced estimation mean-squared error. Furthermore, we establish the first-ever provable Edgeworth expansions of the studentized entrywise eigenvector statistics and apply them to justify the higher-order correctness of the bootstrap distributions for the entrywise eigenvector statistics. 

The signal-plus-noise matrix models considered in this work can be applied to deal with rectangular matrices directly through the ``symmetric dilation'' trick introduced in Remark \ref{remark:symmetric_dilation}. Nevertheless, our theory only works when the two dimensions are balanced, namely, $p_1/p_2$ stays bounded away from $0$ and $\infty$. A potential future research direction is to explore the higher-order entrywise singular vector analysis with unbalanced dimensions when $p_1$ and $p_2$ grow at different rates \cite{10.1214/20-AOS1986,10.1214/17-AOS1541}. We conjecture that this may require significant extra work because the entrywise recovery of the underlying signal matrix $\bP$ could be impossible for unbalanced matrices. 

The theory developed in this work crucially relies on the eigenvalue separation condition (Assumption \ref{assumption:Eigenvalue_separation}), under which the nontrivial eigenvectors are all identifiable up to sign flips. This condition can be too strong in some contexts where the estimation of subspaces is of interest rather than a single eigenvector, such as the principal component analysis and multiple graph inference with a common subspace \cite{arroyo2021inference}. An interesting future direction is to extend the current theory to the row-wise subspace analysis for low-rank random matrix models with higher-order stochastic expansion and/or Edgeworth expansions. 

The objective we consider in this paper is the entrywise functionals of sample eigenvectors. Another potential extension is to study sample eigenvalues and general linear functionals of sample eigenvectors. This may require different techniques than those developed in the current work. Specifically, in \cite{doi:10.1080/01621459.2020.1840990}, the authors show that if $\bx\transpose\widehat{\bu}_k$ is a general linear functional of $\widehat{\bu}_k$ and $|\bx\transpose\bu_k|$ is bounded away from $0$, then the first-order term in the stochastic expansion of $\bx\transpose\widehat{\bu}_k\mathrm{sgn}(\widehat{\bu}_k\transpose\bu_k) - \bx\transpose\bu_k$ involves a quadratic form $\bx\transpose\bE^2\by$ for some vectors $\bx,\by\in\mathbb{R}^n$. Also see Proposition \ref{prop:eigenvector_angle_expansion} when $\bx = \bu_k$. Although higher-order stochastic expansion is possible by pushing forward the roadmap developed there, the Edgeworth expansion of the CDF of the studentized version of $\bx\transpose\widehat{\bu}_k$ is enormously different. This is because $\bx\transpose\bE^2\by$ is a martingale rather than a sum of independent random variables, and Edgeworth expansion of martingales is treated differently under a more involved ``test-function'' topology rather than the total variation topology \cite{mykland1989edgeworth,10.1214/aos/1176348649,10.2307/2244676,10.1214/aos/1176324617}. We defer this direction to future research.

\section*{Acknowledgement}
This research was supported in part by Lilly Endowment, Inc., through its support for the Indiana University Pervasive Technology Institute.

\bigskip
\begin{center}
  \begin{Large}
    \textbf{Supplement to ``Higher Order Entrywise Eigenvectors Analysis of Low-Rank Random Matrices: Bias Correction, Edgeworth Expansion, and Bootstrap''}
  \end{Large}
\end{center}
\appendix

\allowdisplaybreaks

\begin{abstract}
This supplementary material contains the proofs of theoretical results in Section \ref{sec:eigenvector_expansion} and Section \ref{sec:applications}.
\end{abstract}

\appendix

\section{Overview of the Supplement}
\label{sec:overview}

Before proceeding to the proofs of the main results in the manuscript, we first provide a brief overview of this supplement. In Section \ref{sec:auxiliary_results}, we establish and collect several auxiliary results that are useful throughout the entire document. These results are mainly based on the concentration properties of quadratic forms of powers of generalized Wigner matrix (\emph{i.e.}, $\bx\transpose\bE^l\by$) established in \cite{erdos2013} and some recent progress in the subspace perturbation analysis of random matrices with low expected ranks studied in \cite{cape2019signal,xie2021entrywise}. Section \ref{sec:technical_tools_from_RMT} is devoted to the proof of Proposition \ref{prop:eigenvector_angle_expansion}, \emph{i.e.}, the expansion of the angles between the sample eigenvectors and the population eigenvectors. This section borrows the random matrix theory tools established in \cite{doi:10.1080/01621459.2020.1840990} via the residual theorem in complex analysis, but our results are obtained with refined remainder characterization and weaker assumptions compared to \cite{doi:10.1080/01621459.2020.1840990}. Section \ref{sec:proof_of_eigenvector_expansion} completes the proof of Theorem \ref{thm:Eigenvector_Expansion}. In Section \ref{sec:proof_of_edgeworth_expansion}, we provide the complete recipe for the proof of Theorem \ref{thm:edgeworth_expansion}. The main technical tools used here are Esseen's smoothing lemma \cite{10.1007/BF02392223} and a collection of highly nontrivial and delicate analyses of the characteristic function of $T_{ik}$. Sections \ref{sec:proof_of_matrix_denoising}, \ref{sec:proof_of_bootstrap_subgaussian}, and \ref{sec:proof_of_smoothed_edgeworth_expansion} collect the proofs of the remaining results. 

\section{Auxiliary Results}
\label{sec:auxiliary_results}

\begin{result}\label{result:Noise_matrix_moment_bound}
Under Assumption \ref{assumption:Noise_matrix_distribution} of the manuscript, for each fixed $l\in \mathbb{N}_+$, any constant $\xi > 1$, and any unit vectors $\bx,\by\in\mathbb{R}^n$, 
\begin{align*}
\bx\transpose\bE^l\by - \expect(\bx\transpose\bE^l\by) &= \widetilde{O}_{\prob}\left\{\sqrt{n}(n\rho_n)^{l/2}(\log n)^{l\xi}\|\bx\|_\infty\|\by\|_\infty\right\},\\
\expect\bx\transpose\bE^l\by &= O\{\sqrt{n}(n\rho_n)^{l/2}\|\by\|_\infty\},\\
\expect\bx\transpose\bE^2\by &= O(n\rho_n).
\end{align*} 
\end{result}
\begin{proof}
The proof of the first assertion is almost the same as that of Lemma 6.5 in \cite{erdos2013}. To invoke the proof there, we set $h_{ij} = E_{ij}/\sqrt{n\rho_n}$ and $\bH = [h_{ij}]_{n\times n}$. The differences are the following: rather than considering $\bx = \by = \one_n/\sqrt{n}$, where $\one_n$ is the $n$-dimensional vector of all ones, we consider general unit vectors $\bx, \by$ so that the upper bound comes with $\|\bx\|_\infty\|\by\|_\infty$, and we do not restrict $\expect h_{ij}^2 = 1/n$ but only need $\expect h_{ij}^2\leq 1/n$. In Definition 2.1 and equation (2.5) in \cite{erdos2013}, we can take $C = 1$ and $q = q_n$. Then, by the proof of Lemma 6.5 there, we see immediately that for any $\xi > 1$, there exists a constant $\nu > 0$, such that with probability at least $1 - \exp\{-\nu(\log n)^{\xi}\}$, 
\[
\frac{1}{n}\left|\bx\transpose\bH^l\by - \expect(\bx\transpose\bH^l\by)\right| = O\left\{\frac{(\log n)^{l\xi}}{\sqrt{n}}\|\bx\|_\infty\|\by\|_\infty\right\},
\]
and hence,
\[
\left|\bx\transpose\bE^l\by - \expect(\bx\transpose\bE^l\by)\right| = O\left\{\sqrt{n}(n\rho_n)^{l/2}(\log n)^{l\xi}\|\bx\|_\infty\|\by\|_\infty\right\}.
\]
For the second assertion, by the proof of Lemma 7.10 in \cite{erdos2013}, we know that 
\[
\max_{i\in[n]}|\expect \be_i\transpose\bE^l\by| = O\{(n\rho_n)^{l/2}\|\by\|_\infty\}.
\]
It follows from Cauchy-Schwarz inequality that
\begin{align*}
\left|\expect\bx\transpose\bE^l\by\right|& = \left|\sum_{i = 1}^nx_i\expect\be_i\transpose\bE^l\by\right|\leq \left(\sum_{i = 1}^nx_i^2\right)^{1/2}\left\{\sum_{i = 1}^n(\expect \be_i\transpose\bE^l\by)^2\right\}^{1/2}\\
& = O\{\sqrt{n}(n\rho_n)^{l/2}\|\by\|_\infty\},
\end{align*}
where $x_i$ is the $i$th entry of $\bx$. The third assertion follows directly from the fact that
\[
\|\expect \bE^2\|_2 = \left\|\mathrm{diag}\left\{\sum_{j = 1}^n\var(E_{1j}),\ldots,\sum_{j = 1}^n\var(E_{nj})\right\}\right\|_2 = O(n\rho_n).
\]
\end{proof}
\begin{result}\label{result:Noise_matrix_concentration}
Under Assumption \ref{assumption:Noise_matrix_distribution} of the manuscript, for unit vectors $\bu,\bv\in\mathbb{R}^n$ and any $\xi > 1$, 
\[
\bu\transpose\bE\bv = \widetilde{O}_{\prob}\{n\rho_n^{1/2}(\log n)^{\xi}\|\bu\|_\infty\|\bv\|_\infty\}
\quad\mbox{and}\quad
\|\bE\|_2 = \widetilde{O}_{\prob}(\sqrt{n\rho_n}).
\] 
\end{result}
\begin{proof}
The first assertion follows directly from Result \ref{result:Noise_matrix_moment_bound} and the second assertion is a consequence of Lemma 4.3 in \cite{erdos2013}.
\end{proof}
\begin{result}\label{result:Noise_matrix_rowwise_concentration}
Under Assumption \ref{assumption:Noise_matrix_distribution} of the manuscript, for unit vector $\bu\in\mathbb{R}^n$, 
\[
|\be_i\transpose\bE\bu| = \widetilde{O}_{\prob}\left\{\|\bu\|_{\infty}(n\rho_n)^{1/2}(\log n)^{\xi}\right\}
\]
\end{result}
\begin{proof}
This follows directly from Equation (3.19) of Lemma 3.8 in \cite{erdos2013}. 
\end{proof}
\begin{result}\label{result:SU_exchange_concentration}
Under Assumption \ref{assumption:Noise_matrix_distribution} of the manuscript, 
\[
\|\bS_\bP\bU_\bP\transpose\bU_\bA - \bU_\bP\transpose\bU_\bA\bS_\bA\|_2 = \widetilde{O}_{\prob}
\left(\alpha_n + \frac{n\rho_n}{\Delta_n}\right)
\]
\end{result}
\begin{proof}
To see this, first note that by Davis-Kahan theorem,
\[
\|\bU_\bA - \bU_\bP\bW^*\|_2\leq \sqrt{d}\max_{k\in [d]}\|\widehat{\bu}_k\mathrm{sgn}(\widehat{\bu}_k\transpose\bu_k) - \bu_k\|_2\leq \frac{\sqrt{d}\|\bE\|_2}{2\Delta_n} = \widetilde{O}_{\prob}\left(\frac{\sqrt{n\rho_n}}{\Delta_n}\right).
\]
Then it follows directly from the decomposition
\begin{align*}
\bS_\bP\bU_\bP\transpose\bU_\bA - \bU_\bP\transpose\bU_\bA\bS_\bA &= \bU_\bP\transpose\bP\bU_\bA - \bU_\bP\transpose\bA\bU_\bA\\
& = -\bU_\bP\transpose\bE\bU_\bP\bW^* - \bU_\bP\transpose\bE(\bU_\bA - \bU_\bP\bW^*),
\end{align*}
that
\[
\|\bS_\bP\bU_\bP\transpose\bU_\bA - \bU_\bP\transpose\bU_\bA\bS_\bA\|_2 =
\widetilde{O}_{\prob}\left(\underbrace{\|\bU_\bP\transpose\bE\bU_\bP\|_2}_{\alpha_n} + \frac{n\rho_n}{\Delta_n}\right).
\]
\end{proof}
\begin{result}
\label{result:concentration}
Suppose Assumption~\ref{assumption:Noise_matrix_distribution} holds. Then
\begin{align}\label{result:concentration:res1}
\max_{(i,j)\in[n]^2}|E_{ij}|  = \widetilde{O}_{\prob}\Bigg(\frac{\sqrt{n\rho_n}}{q_n}\Bigg).
\end{align}
Furthermore, let $l\in \mathbb{N}_+$ be any fixed integer and $\mathcal{E}_m$ be any subset of $\{E_{ij}\mid 1\leq i\leq j\leq n\}$ such that  $\mathcal{E}_m$ contains $m$ identical elements: $\widetilde{E}_1,\dots ,\widetilde{E}_m$, and suppose $m \geq   \eta n$ for a fixed constant $\eta \in(0,1]$.  Let $\{\gamma_i\}_{i = 1}^m$ be a collection of deterministic coefficients. Then we have
\begin{equation}\label{result:concentration:res2}
\begin{aligned}
&\sum_{i = 1}^m \gamma_i\Big\{\widetilde{E}_i^l - \expect\big(\widetilde{E}_i^l\big)\Big\}\\
&\quad = \widetilde{O}_{\prob}\Bigg\{(\log m)^{\xi}\Bigg(\frac{\max_{i\in[m]}|\gamma_i|}{q_n} + \sqrt{\frac{1}{m}\sum_{i = 1}^m|\gamma_i|^2}\Bigg) \frac{m^{1/2}n^{(l-1)/2}\rho_n^{l/2}}{q_n^{l - 1}}\Bigg\}
\\
&\quad = \widetilde{O}_{\prob}\Bigg\{(\log n)^{\xi}{\max_{i\in[m]}|\gamma_i|}\frac{m^{1/2}n^{(l-1)/2}\rho_n^{l/2}}{q_n^{l - 1}}\Bigg\}.
\end{aligned}
\end{equation}
Also, if $\{\gamma_{i,j}\}_{i\neq j,(i,j)\in[m]^2}$ is a collection of deterministic coefficients, then we have
\begin{equation}\label{result:concentration:res3}
\begin{aligned}
&\sum_{i \neq j\atop (i,j)\in[m]^2} \gamma_{ij}\big\{\widetilde{E}_i^l - \expect\big(\widetilde{E}_i^l\big)\big\}\big\{\widetilde{E}_j^l - \expect\big(\widetilde{E}_j^l\big)\big\}
\\
&\quad = \widetilde{O}_{\prob}\Bigg\{(\log m)^{2\xi}\Bigg(\frac{\max_{i\neq j}|\gamma_{ij}|}{q_n} + \sqrt{\frac{1}{m^2}\sum_{i \neq j}|\gamma_{ij}|^2}\Bigg)\frac{mn^{l-1}\rho_n^{l}}{q_n^{2l - 2}}\Bigg\}
\\
&\quad = \widetilde{O}_{\prob}\Bigg\{(\log n)^{2\xi}\max_{i\neq j}|\gamma_{ij}|\frac{mn^{l-1}\rho_n^{l}}{q_n^{2l - 2}}\Bigg\}.
\end{aligned}
\end{equation}
\end{result}
\begin{proof}
The first assertion follows from Lemma~3.7 in \cite{erdos2013} directly. In particular, for any $(i,j)\in[n]^2$, by high-order Markov inequality, we have that with $p_n = a(\log n)^{\xi}$ for any fixed $\xi > 1$ and $a > 0$, and a sufficiently large $n$,
\begin{align}\nonumber
\prob \Bigg(|E_{ij}| \geq 2C\frac{\sqrt{n\rho_n}}{q_n} \Bigg)&\leq \frac{q_n^{p_n}\expect |E_{ij}|^{p_n}}{(2C)^{p_n}(n\rho_n)^{p_n/2}} \leq \frac{q_n^{2}}{2^{p_n}n},
\end{align}  
where the second inequality is by Assumption~\ref{assumption:Noise_matrix_distribution}. Since the above probability vanishes superpolynomially, an uniform argument over all $(i,j)\in[n]^2$ yields \eqref{result:concentration:res1} immediately. 
\par
For the second assertion, we use Lemma~3.8 in \cite{erdos2013}. 
We set 
$$
a_i = \frac{n^{1/2}q_n^{l - 1}}{m^{1/2}(n\rho_n)^{l/2}}\{\widetilde{E}^{l}_{i} - \expect(\widetilde{E}^{l}_{i})\}.
$$
Then we have for any $2 \leq p \leq (\log n)^{A_0\log\log m}$ and $n \geq \exp\big(l^{1/(A_0\log 2)}\big)$, 
\begin{align*}
\expect|a_i|^{p} &\leq \frac{n^{p/2}q_n^{lp - p}}{m^{p/2}(n\rho_n)^{lp/2}}2^p \expect\big(|\widetilde{E}_{i}|^{lp}\big) 
\\
&\leq  \frac{(2C^l\sqrt{\eta})^{p}}{mq_n^{p - 2}}\Big(\frac{m}{\eta n}\Big)^{1 - p/2 } \leq \frac{(2C^l\sqrt{\eta})^{p}}{mq_n^{p - 2}},
\end{align*}
which verifies the Condition (3.18) in \cite{erdos2013}; Here the first inequality is by Jensen's inequality, the second inequality is by Assumption~\ref{assumption:Noise_matrix_distribution} and $0< \eta \leq 1$, and the third inequality is due to the fact that $p\geq 2$ and $m/(\eta n) \geq 1$. Thus Lemma~3.8, Equation~(3.19) in \cite{erdos2013} implies that
\begin{align*}
\sum_{i = 1}^m\gamma_i a_i = \widetilde{O}_{\prob}\Bigg\{(\log m)^{\xi}\Bigg(\frac{\max_{i\in[m]}|\gamma_i|}{q_n} + \sqrt{\frac{1}{m}\sum_{i = 1}^m|\gamma_i|^2}\Bigg)\Bigg\},
\end{align*}
and Lemma~3.8, Equation~(3.21) in \cite{erdos2013} implies that
\begin{align*}
\sum_{i \neq j\atop (i,j)\in[m]^2}\gamma_{i,j} a_ia_j = \widetilde{O}_{\prob}\Bigg\{(\log m)^{2\xi}\Bigg(\frac{\max_{i\neq j}|\gamma_{ij}|}{q_n} + \sqrt{\frac{1}{m^2}\sum_{i \neq j}|\gamma_{ij}|^2}\Bigg)\Bigg\},
\end{align*}
which directly yield the desired results.
\end{proof}
\begin{lemma}\label{lemma:Rowwise_higher_order_concentration}
Under Assumption \ref{assumption:Noise_matrix_distribution} of the manuscript, there exists a constant $C > 0$ such that for each fixed $i\in [n]$, any deterministic vector $\bv\in\mathbb{R}^n$, and any $\xi > 1$,
\[
\prob\left[\bigcup_{k = 1}^{\log n}\left\{|\be_i\transpose\bE^k\bv|  >  C^k(n\rho_n)^{k/2}(\log n)^{\xi k}\|\bv\|_\infty\right\}\right]\lesssim \exp\{-\nu(\log n)^\xi\},
\]
where $\nu > 0$ is a constant. 
\end{lemma}
\begin{proof}
The proof is almost the same as that of Lemma 7.10 in \cite{erdos2013}. Also see Lemma B.1 in \cite{xie2019efficient}, Lemma 5.4 in \cite{doi:10.1080/01621459.2020.1751645}, and Lemma B.3 in \cite{tang2017}. The only difference is that rather than using the vector $\one_n/\sqrt{n}$, where, $\one_n$ is the $n$-dimensional vector of all ones, and considering $\be_i\transpose\bE\one_n/\sqrt{n}$, we consider a more general vector $\bv$. The remaining part of the proof follows exactly the same as that of Lemma 7.10 in \cite{erdos2013}, together with a union bound over $k = 1,\ldots,\log n$. 
\end{proof}

\begin{lemma}\label{lemma:Eigenvector_angle_analysis}
Suppose Assumptions \ref{assumption:Signal_strength}, \ref{assumption:Eigenvalue_separation}, and \ref{assumption:Noise_matrix_distribution} in the manuscript hold. Then
\[
\|\bU_\bP\transpose\bU_\bA - \bW^*\|_2 = 
\widetilde{O}_{\prob}\left(\frac{\alpha_n}{\Delta_n}  + \frac{n\rho_n}{\Delta_n^2}\right)
,
\]
where $\bW^* = \mathrm{diag}\{\mathrm{sgn}(\widehat{\bu}_1\transpose\bu_1),\ldots,\mathrm{sgn}(\widehat{\bu}_d\transpose\bu_d)\}$.  
\end{lemma}
\begin{proof}
It suffices to show that 
\begin{align*}
|\bu_k\transpose\widehat{\bu}_k - \mathrm{sgn}(\widehat{\bu}_k\transpose\bu_k)| &= \widetilde{O}_{\prob}\left(\frac{n\rho_n}{\Delta_n^2}\right),\quad
|\bu_r\transpose\widehat{\bu}_k| = \widetilde{O}_{\prob}\left(\alpha_n + \frac{n\rho_n}{\Delta_n^2}\right)
\end{align*}
for all $k, r\in [d]$, $r\neq k$. By the definition of $\Delta_n$, $\lambda_k - \lambda_{k + 1}\geq \Delta_n$ and $\min_{k\in [d]}|\lambda_k|\geq \Delta_n$. Then by Davis-Kahan theorem,
\begin{align*}
|\mathrm{sgn}(\bu_k\transpose\widehat{\bu}_k) - \bu_k\transpose\widehat{\bu}_k|
& = 1 - \mathrm{sgn}(\bu_k\transpose\widehat{\bu}_k)\bu_k\transpose\widehat{\bu}_k
  = \frac{1}{2}\|\bu_k - \mathrm{sgn}(\bu_k\transpose\widehat{\bu}_k)\widehat{\bu}_k\|_2^2 = \widetilde{O}_{\prob}\left(\frac{n\rho_n}{\Delta_n^2}\right)
  .  
\end{align*}
For $k\neq r$, by the definitions of $\bu_r$ and $\widehat{\bu}_k$, we have $\bu_r\transpose\widehat{\bu}_k\widehat{\lambda}_k = \bu_r\transpose\bA\widehat{\bu}_k$ and $\widehat{\bu}_k\transpose\bu_r\lambda_r = \widehat{\bu}_k\transpose\bP\bu_r$, where $\widehat{\lambda}_k$ is the eigenvalue associated with the eigenvector $\widehat{\bu}_k$. Therefore,
\[
\bu_r\transpose\widehat{\bu}_k(\widehat{\lambda}_k - \lambda_r) = \bu_r\transpose\bE\widehat{\bu}_k = \bu_r\transpose\bE\bu_k\mathrm{sgn}(\bu_k\transpose\widehat{\bu}_k) + \bu_r\transpose\bE\{\widehat{\bu}_k - \bu_k\mathrm{sgn}(\bu_k\transpose\widehat{\bu}_k)\}.
\]
By Davis-Kahan theorem and Result \ref{result:Noise_matrix_concentration}, 
\[
|\bu_r\transpose\bE\{\widehat{\bu}_k - \mathrm{sgn}(\bu_k\transpose\widehat{\bu}_k)\}|\leq \|\bE\|_2\|\widehat{\bu}_k - \bu_k\mathrm{sgn}(\bu_k\transpose\widehat{\bu}_k)\|_2 = \widetilde{O}_{\prob}\left(\frac{n\rho_n}{\Delta_n}\right). 
\]
By Weyls' inequality, Result \ref{result:Noise_matrix_concentration} and Assumption \ref{assumption:Signal_strength}, we have that $\sqrt{n\rho_n}/\Delta_n\to0$, 
\[
|\widehat{\lambda}_k - \lambda_r|\geq |\lambda_k - \lambda_r| - |\widehat{\lambda}_k - \lambda_k|\geq \Delta_n - \|\bE\|_2\geq\frac{\Delta_n}{2}\quad\mbox{w.h.p.}
\]
by Assumption \ref{assumption:Signal_strength}.
Hence, we conclude that 
\begin{align*}
|\bu_r\transpose\widehat{\bu}_k|&\leq \frac{|\bu_r\transpose\bE\bu_k| + |\bu_r\transpose\bE\{\widehat{\bu}_k - \mathrm{sgn}(\bu_k\transpose\widehat{\bu}_k)\}|}{|\widehat{\lambda}_k - \lambda_r|}\\
& = \widetilde{O}_{\prob}\left(\frac{\|\bU_\bP\transpose\bE\bU_\bP\|_2}{\Delta_n} + \frac{n\rho_n}{\Delta_n^2}\right).
\end{align*}
The proof is thus completed.
\end{proof}

\begin{lemma}\label{lemma:Utranspose_E_cubic_U_concentration}
Suppose Assumptions \ref{assumption:Signal_strength}--\ref{assumption:Noise_matrix_distribution} in the manuscript hold. Then for any $\xi > 1$,
\begin{align*}
\bU_\bP\transpose\expect\big(\bE^3\big)\bU_\bP &= O\left\{\frac{(n\rho_n)^{3/2}}{q_n}\right\}\\
\bU_\bP\transpose\bE^3\bU_\bP &= \widetilde{O}_{\prob}\left[(n\rho_n)^{3/2}\max\left\{\frac{1}{q_n}, \sqrt{n}(\log n)^{3\xi}\|\bU_\bP\|_{2\to\infty}^2\right\}\right].
\end{align*}
\end{lemma}
\begin{proof}
It is sufficient to show that $\expect(\bu\transpose\bE^3\bv) = O\{(n\rho_n)^{3/2}/q_n\}$ for any unit vectors $\bu, \bv$. The second assertion then follows directly from Result \ref{result:Noise_matrix_moment_bound}. For the first assertion, we have
\[
\expect\big(\bu\transpose\bE^3\bv\big) = \sum_{a,b,c,d\in [n]}u_av_d\expect (E_{ab}E_{bc}E_{cd}). 
\]
For any $(a,b,c,d)\in [n]^4$, $\expect (E_{ab}E_{bc}E_{cd})\neq 0$ if and only if $|\{\{a,b\},\{b,c\},\{c,d\}\}| = 1$, where $|\cdot|$ denotes the cardinality of a set. This happens only if $a = c$, in which case we have
\[
\expect \big(\bu\transpose\bE^3\bv\big) = \sum_{a,b,c,d\in [n]}u_av_d\expect \big(E_{ab}E_{bc}E_{cd}\big) = \sum_{a,b,d\in[n]}u_av_d\expect \big(E_{ab}^2E_{ad}\big) = \sum_{a,b\in[n]}u_av_d\expect E_{ab}^3.
\]
By Cauchy-Schwarz inequality and Assumption \ref{assumption:Noise_matrix_distribution}, we then have 
\begin{align*}
|\expect(\bu\transpose\bE^3\bv)| &\leq \sum_{a,b\in [n]}|u_a||v_b|\expect |E_{ab}|^3\leq C\sum_{a,b\in[n]}|u_a||v_b|\frac{(n\rho_n)^{3/2}}{nq_n}\\
&\leq C\frac{(n\rho_n)^{3/2}}{nq_n}\left(\sum_{a,b\in[n]}u_a^2\right)^{1/2}\left(\sum_{a,b\in[n]}v_b^2\right)^{1/2} = O\left\{\frac{(n\rho_n)^{3/2}}{q_n}\right\}.
\end{align*}
The proof is thus completed.
\end{proof}

\begin{lemma}\label{lemma:Eigenvector_zero_order_deviation}
Suppose Assumptions \ref{assumption:Signal_strength}--\ref{assumption:Noise_matrix_distribution} in the manuscript hold. Then, for any $\xi > 1$, we have the following decomposition,
\begin{align*}
\bU_\bA - \bU_\bP\bW^* & = \bE\bU_\bP\bS_\bP^{-1}\bW^* + \bR_\bU,
\end{align*}
where $\bW^* = \mathrm{diag}\{\mathrm{sgn}(\bu_1\transpose\widehat{\bu}_1),\ldots,\mathrm{sgn}(\bu_d\transpose\widehat{\bu}_d)\}$ and $\bR_\bU$ satisfies
\begin{equation}\label{eq:decompose}
\begin{aligned}
\|\bR_\bU\|_{2\to\infty} = \widetilde{O}_{\prob}\left\{\|\bU_\bP\|_{2\to\infty}\frac{n\rho_n(\log n)^{2\xi}}{\Delta_n^2} + \|\bU_\bP\|_{2\to\infty}\frac{\alpha_n}{\Delta_n} \right\}.
\end{aligned}
\end{equation}
Furthermore, for any $k\in[p]$,
\begin{equation}
\label{eqn:Frechet_derivative}
\begin{aligned}
&\Bigg\|\widehat{\bu}_k\mathrm{sgn}(\bu_k\transpose\widehat{\bu}_k) - \bu_k
 - \bigg(\eye_n - \bu_k\bu_k\transpose + \sum_{m\in[d]\backslash\{k\}}\frac{\lambda_m\bu_m\bu_m\transpose}{\lambda_k - \lambda_m} \bigg)\frac{\bE\bu_k}{\lambda_k}\Bigg\|_\infty\\
&\quad = \widetilde{O}_{\prob}\left\{\|\bU_\bP\|_{2\to\infty}\frac{n\rho_n(\log n)^{2\xi}}{\Delta_n^2} + \|\bU_\bP\|_{2\to\infty}\frac{\alpha_n}{\Delta_n} \right\}.
 \end{aligned}
 \end{equation}
In addition, we have the following bounds for $\bU_\bA - \bU_\bP\bW^*$,
\begin{equation}\label{equaupw}
\begin{aligned}
\|\bU_\bA - \bU_\bP\bW^*\|_2 & = \widetilde{O}_{\prob}\left(\frac{\sqrt{n\rho_n}}{\Delta_n}\right),\\
\|\bU_\bA - \bU_\bP\bW^*\|_{2\to\infty} & = \widetilde{O}_{\prob}\left\{\|\bU_\bP\|_{2\to\infty}\frac{(n\rho_n)^{1/2}(\log n)^{\xi}}{\Delta_n}\right\}.
\end{aligned}
\end{equation}
\end{lemma}
\begin{proof}
We first prove \eqref{equaupw}. For the first assertion in \eqref{equaupw}, it follows directly from Davis-Kahan theorem that 
\[
\|\widehat{\bu}_k - \bu_k\mathrm{sgn}(\bu_k\transpose\widehat{\bu}_k)\|_2\leq \frac{2\|\bE\|_2}{\Delta_n}\]
and $\|\bE\|_2 = \widetilde{O}_{\prob}(\sqrt{n\rho_n})$ from Result \ref{result:Noise_matrix_concentration}.
For the second assertion in \eqref{equaupw}, the proof is based on a straightforward modification of that of Theorem 2 in \cite{cape2019signal}. Note that the key difference here is that rather than having $\|\bU_\bP\transpose\bU_\bA - \bW^*\|_2 = \widetilde{O}_{\prob}\{1/(n\rho_n)\}$, we only have 
\[
\|\bU_\bP\transpose\bU_\bA - \bW^*\|_2 = 
\widetilde{O}\left(\frac{\alpha_n}{\Delta_n}  + \frac{n\rho_n}{\Delta_n^2}\right)
\]
by Lemma \ref{lemma:Eigenvector_angle_analysis} here, and the signal strength $\Delta_n$ is not necessarily the squared noise level $n\rho_n$. 

\par
Since $\bU_\bA$ satisfies the matrix Sylvester equation $\bU_\bA\bS_\bA - \bE\bU_\bA = \bP\bU_\bA$ and the spectra of $\bE$ and $\bS_\bA$ are disjoint w.h.p., by matrix series expansion of $\mathbf{U}_{\mathbf{A}}$ (See e.g., \cite[Section 7.2]{bhatia2013matrix}), we have
\begin{equation}\label{eqn:Eigenvector_series_expansion}
\begin{aligned}
\bU_\bA - \bU_\bP\bW^*& = \sum_{m = 1}^\infty \bE^m\bU_\bP\bS_\bP\bU_\bP\transpose\bU_\bA\bS_\bA^{-(m + 1)}
\\
&\quad + \bU_\bP(\bU_\bP\transpose\bU_\bA - \bW^*)\\
&\quad + \bU_\bP(\bS_\bP\bU_\bP\transpose\bU_\bA - \bU_\bP\transpose\bU_\bA\bS_\bA)\bS_\bA^{-1}.
\end{aligned}
\end{equation}
The two-to-infinity norm of the third term is 
\[
\widetilde{O}\left(\|\bU_\bP\|_{2\to\infty}\frac{\|\bU_\bP\transpose\bE\bU_\bP\|_2}{\Delta_n}  + \|\bU_\bP\|_{2\to\infty}\frac{n\rho_n}{\Delta_n^2}\right)
\] 
by Result \ref{result:SU_exchange_concentration}. The two-to-infinity norm of the second term is 
\[
\widetilde{O}\left(\|\bU_\bP\|_{2\to\infty}\frac{\|\bU_\bP\transpose\bE\bU_\bP\|_2}{\Delta_n} + \|\bU_\bP\|_{2\to\infty}\frac{n\rho_n}{\Delta_n^2}\right)
\]
 by Lemma \ref{lemma:Eigenvector_angle_analysis}. For the first term, by Lemma \ref{lemma:Rowwise_higher_order_concentration}, we can let $M\in\mathbb{N}_+$ satisfy
\[
\|\bE^m\bU_\bP\|_{2\to\infty} = \widetilde{O}_{\prob}\left\{\|\bU_\bP\|_{2\to\infty}(\log n)^{m\xi}(n\rho_n)^{m/2}\right\}
\]
for all $m \geq 1$ for any  $\xi > 1$. In addition, we have $(\sqrt{n\rho_n}/\Delta_n)^{M}\leq n^{-1/2}\leq \|\bU_\bP\|_{2\to\infty}$ for some sufficiently large $M$; Such a constant $M$ exists because $\Delta_n/\sqrt{n\rho_n} \gtrsim n^{\beta_\Delta}$ for some $\beta_\Delta > 0$ by Assumption \ref{assumption:Signal_strength}. Then, for the first term, we have
\begin{equation}\label{eusuus:series}
\begin{aligned}
&\left\|\sum_{m = 1}^\infty \bE^m\bU_\bP\bS_\bP\bU_\bP\transpose\bU_\bA\bS_\bA^{-(m + 1)}\right\|_{2\to\infty}\\
&\quad\leq \sum_{m = 1}^M\|\bE^m\bU_\bP\|_{2\to\infty}\|\bS_\bP\|_2\|\bS_\bA^{-(m + 1)}\|_2
 + \sum_{m = M + 1}^\infty\|\bE\|_2^m\|\bS_\bP\|_2\|\bS_\bA^{-(m + 1)}\|_2\\
&\quad\leq \|\bU_\bP\|_{2\to\infty}
\sum_{m = 1}^M\left[\widetilde{O}_{\prob}\left\{\frac{c(\log n)^{2\xi}(n\rho_n)}{\Delta_n^2}\right\}\right]^{m/2} + 
\sum_{m = M + 1}^\infty\left[\widetilde{O}_{\prob}\left\{\frac{(n\rho_n)}{\Delta^2_n}\right\}\right]^{m/2}\\
&\quad = \widetilde{O}_{\prob}\left\{\|\bU_\bP\|_{2\to\infty}\frac{(\log n)^{\xi}(n\rho_n)^{1/2}}{\Delta_n}\right\}.
\end{aligned}
\end{equation}
This completes the proof of the second assertion of \eqref{equaupw}. Here, we have used the fact that 
$\|\bE^m\bU_\bP\|_2\leq\|\bE\|_2^m = \{\widetilde{O}_{\prob}(\sqrt{n\rho_n})\}^m$.
\par
For \eqref{eq:decompose}, we follow the proof of Theorem 2 in \cite{cape2019signal} and write
\begin{align*}
\sum_{m = 1}^\infty \bE^m\bU_\bP\bS_\bP\bU_\bP\transpose\bU_\bA\bS_\bA^{-(m + 1)}
& = \bE\bU_\bP\bS_\bP^{-1}\bW^* + \bE\bU_\bP\bS_\bP^{-1}(\bU_\bP\transpose\bU_\bA - \bW^*)\\
&\quad + \bR_2^{(1)} + \bR_2^{(\infty)},
\end{align*}
where
\begin{align*}
\bR_2^{(1)} = \bE\bU_\bP\bS_\bP(\bU_\bP\transpose\bU_\bA\bS_\bA^{-2} - \bS_\bP^{-2}\bU_\bP\transpose\bU_\bA),\quad
\bR_2^{(\infty)} = \sum_{m = 2}^\infty\bE^m\bU_\bP\bS_\bP\bU_\bP\transpose\bU_\bA\bS_\bA^{-(m + 1)}.
\end{align*}
For $\bR_2^{(\infty)}$, similar to \eqref{eusuus:series}, we have
\begin{align*}
\|\bR_2^{(\infty)}\|_{2\to\infty}  = \widetilde{O}_{\prob}\left\{\|\bU_\bP\|_{2\to\infty}\frac{n\rho_n(\log n)^{2\xi}}{\Delta_n^2}\right\}.
\end{align*}
For $\bR_2^{(1)}$, write $\bR_2^{(1)} = \bE\bU_\bP\bS_\bP\bR_2^{(3)}$ where $\bR_2^{(3)} = \bU_\bP\transpose\bU_\bA\bS_\bA^{-2} - \bS_\bP^{-2}\bU_\bP\transpose\bU_\bA$. The $(k, l)$th entry of $\bR_2^{(3)}$ satisfies
\begin{align*}
[\bR_2^{(3)}]_{kl} &= \bu_k\transpose\widehat{\bu}_l(\widehat{\lambda}_l^{-2} - \lambda_k^{-2})
 = \bu_k\transpose\widehat{\bu}_l(\lambda_k^2 - \widehat{\lambda}_l^2)\lambda_k^{-2}\widehat{\lambda}_l^{-2}.
\end{align*}
With $\bH_2$ defined as a $d\times d$ matrix whose $(k, l)$th entry is $\lambda_k^{-2}\widehat{\lambda}_l^{-2}$, we obtain
\[
\bR_2^{(3)} = -\bH_2\circ(\bU_\bP\transpose\bU_\bA\bS_\bA^{2} - \bS_\bP^{2}\bU_\bP\transpose\bU_\bA),
\]
where $\circ$ denotes the entrywise (Hadamard) matrix product. For $\bU_\bP\transpose\bU_\bA\bS_\bA^{2} - \bS_\bP^{2}\bU_\bP\transpose\bU_\bA$, we have
\begin{align*}
\|\bU_\bP\transpose\bU_\bA\bS_\bA^{2} - \bS_\bP^{2}\bU_\bP\transpose\bU_\bA\|_2
& = \|\bU_\bP\transpose(\bA^2 - \bP^2)\bU_\bA\|_2 = \|\bU_\bP\transpose(\bP\bE + \bE\bP + \bE^2)\bU_\bA\|_2\\
& = \widetilde{O}_{\prob}(\sqrt{n\rho_n}\Delta_n).
\end{align*}
It follows that
\begin{align*}
\|\bR_2^{(1)}\|_{2\to\infty}
&\leq d\|\bE\bU_\bP\|_{2\to\infty}\|\bS_\bP\|_2\|\bH_2\|_{\max}\|\bU_\bP\transpose\bU_\bA\bS_\bA^{2} - \bS_\bP^{2}\bU_\bP\transpose\bU_\bA\|_2\\
& = \widetilde{O}_{\prob}\left\{\|\bU_\bP\|_{2\to\infty}\frac{n\rho_n(\log n)^\xi}{\Delta_n^2}\right\},
\end{align*}
where $\|\cdot\|_{\max}$ is the matrix entrywise maximum norm.
The term $\bE\bU_\bP\bS_\bP^{-1}(\bU_\bP\transpose\bU_\bA - \bW^*)$ can be bounded as
\begin{align*}
&\|\bE\bU_\bP\bS_\bP^{-1}(\bU_\bP\transpose\bU_\bA - \bW^*)\|_{2\to\infty}\\
&\quad\leq \|\bE\bU_\bP\|_{2\to\infty}\|\bS_\bP^{-1}\|_2\|\bU_\bP\transpose\bU_\bA - \bW^*\|_2\\
&\quad = \widetilde{O}_{\prob}\left\{\|\bU_\bP\|_{2\to\infty}\frac{(\log n)^{2\xi}(n\rho_n)^{1/2}}{\Delta_n}\left(\frac{n\rho_n}{\Delta_n^2} + \frac{\|\bU_\bP\transpose\bE\bU_\bP\|_2}{\Delta_n} \right)\right\}\\
&\quad = \widetilde{O}_{\prob}\left(\|\bU_\bP\|_{2\to\infty}\frac{n\rho_n}{\Delta_n^2}\right)
\end{align*}
by Lemma \ref{lemma:Eigenvector_angle_analysis}, where we have used the fact that $\|\bU_\bP\transpose\bE\bU_\bP\|_2/\Delta_n = \widetilde{O}_{\prob}(\sqrt{n\rho_n}/\Delta_n)$. This completes the proof of \eqref{eq:decompose}. The proof of \eqref{eqn:Frechet_derivative} follows directly from \eqref{eq:decompose} and the definition of $\alpha_n$. The proof of all results is therefore completed.
\end{proof}

\begin{lemma}\label{lemma:Remainder_analysis_I}
Suppose Assumptions \ref{assumption:Signal_strength}, \ref{assumption:Eigenvalue_separation}, and \ref{assumption:Noise_matrix_distribution} in the manuscript hold. Then for each fixed $i\in [n]$,
\begin{align*}
\|\bW^*\bS_\bA - \bS_\bP\bW^*\|_2 &= 
\widetilde{O}_{\prob}\left(\alpha_n + \frac{n\rho_n}{\Delta_n}\right)
,
\\
\|\bW^*\bS_\bA^{-1} - \bS_\bP^{-1}\bW^*\|_2 &= 
\widetilde{O}_{\prob}\left(\frac{\alpha_n}{\Delta_n^2} + \frac{n\rho_n}{\Delta_n^3}\right)
,
\end{align*}
where $\bW^* = \mathrm{diag}\{\mathrm{sgn}(\bu_1\transpose\widehat{\bu}_1),\ldots,\mathrm{sgn}(\bu_d\transpose\widehat{\bu}_d)\}$.
\end{lemma}

\begin{proof}
The proof is similar to that of Lemma S2.3 and Lemma S2.4 in \cite{xie2021entrywise}. For the first assertion, write
\begin{align*}
\bW^*\bS_\bA - \bS_\bP\bW^*& = (\bW^* - \bU_\bP\transpose\bU_\bA)\bS_\bA + \bS_\bP(\bU_\bP\transpose\bU_\bA - \bW^*)\\
&\quad + (\bU_\bP\transpose\bU_\bA\bS_\bA - \bS_\bP\bU_\bP\transpose\bU_\bA).
\end{align*}
The third term 
is \[
\widetilde{O}_{\prob}\left(\underbrace{\|\bU_\bP\transpose\bE\bU_\bP\|_2}_{\alpha_n} + \frac{n\rho_n}{\Delta_n}\right)
\] 
by Result \ref{result:SU_exchange_concentration}.
The first two terms are 
$
\widetilde{O}_{\prob}\left(\alpha_n + \frac{n\rho_n}{\Delta_n}\right)
$
by Lemma \ref{lemma:Eigenvector_angle_analysis}. This completes the proof of the first assertion. The second assertion follows directly from the first assertion and the fact that the absolute values of the diagonals of $\bS_\bA$ and $\bS_\bP$ are bounded away from $\Delta_n/2$ w.h.p.. 
The proof is thus completed.
\end{proof}

\section{Technical Tools From Random Matrix Theory}
\label{sec:technical_tools_from_RMT}
In this section, we introduce some technical tools from random matrix theory to investigate the asymptotic expansion of the angles between eigenvectors. Specifically, this section establishes the expansion of $\widehat{\bu}_k\transpose\bu_k - \mathrm{sgn}(\widehat{\bu}_k\transpose\bu_k)$ and $\widehat{\bu}_k\transpose\bu_r$ for any $k\in [p], r\in [d]$, $k\neq r$. These tools primarily rely on the Cauchy's residual theorem in complex analysis. The analysis in this section are closely aligned with that in Appendix A of \cite{doi:10.1080/01621459.2020.1840990}. However, because our setup is more general than that considered in \cite{doi:10.1080/01621459.2020.1840990} whereas our probability bounds of the remainders are of the more refined form $\widetilde{O}_{\prob}(\cdot)$ instead of the classical $O_{\prob}(\cdot)$ notation, we present the analysis here for completeness.

The ultimate purpose of this section is to prove the following Proposition \ref{prop:bilinear_form_expansion} that characterizes the asymptotic expansion of the bilinear forms $\bu_k\transpose\widehat{\bu}_k\widehat{\bu}_k\transpose\bu_k$ and $\bu_m\transpose\widehat{\bu}_k\widehat{\bu}_k\transpose\bu_k$ ($k\neq m$) of the projection matrix $\widehat{\bu}_k\widehat{\bu}_k\transpose$. Proposition \ref{prop:eigenvector_angle_expansion} follows from Proposition \ref{prop:bilinear_form_expansion} as an immediate corollary.
\begin{proposition}\label{prop:bilinear_form_expansion}
Suppose the conditions of Theorem \ref{thm:Eigenvector_Expansion} hold. For each fixed $k\in [p]$, $m\in [d]$, $k\neq m$, for any $\xi > 1$, we have
\begin{align*}
\bu_k\transpose\widehat{\bu}_k\widehat{\bu}_k\transpose\bu_k
& = 1 - \frac{\bu_k\transpose\bE^2\bu_k}{\lambda_k^2} + \widetilde{O}_{\prob}\left\{\frac{\alpha_n^2}{\Delta_n^2} + \frac{n^2\rho_n^{3/2}(\log n)^{3\xi}\|\bU_\bP\|_{2\to\infty}^2}{\Delta_n^3} + \frac{(n\rho_n)^{3/2}}{\Delta_n^3}\right\},\\
\bu_m\transpose\widehat{\bu}_k\widehat{\bu}_k\transpose\bu_k
& = \frac{\bu_m\transpose\bE\bu_k}{\lambda_k - \lambda_m} + \frac{\bu_m\transpose\bE^2\bu_k}{\lambda_k(\lambda_k - \lambda_m)}\\
&\quad + \widetilde{O}_{\prob}\left\{\frac{\alpha_n^2}{\Delta_n^2} + \frac{n^2\rho_n^{3/2}(\log n)^{3\xi}\|\bU_\bP\|_{2\to\infty}^2}{\Delta_n^3} + \frac{(n\rho_n)^{3/2}}{\Delta_n^3}\right\}.
\end{align*}
\end{proposition}

\subsection{Notation and preliminary results}
\label{sub:preliminary_results}

\begin{table}[tp]
\caption{Table of notation}
\label{def:sum}
\centering
\begin{tabular}{l|l} 
\hline
 Notation & Definition 
 \\ 
 \hline
$\bG(z)$ &$(\bE - z\eye_n)^{-1}$\\[1.5ex]
$\widetilde{\bG}(z)$ & $(\bA - z\eye_n)^{-1}$ \\[1.5ex]
$\bV_k$ & Submatrix of $\bU_\bP$ formed by removing $\bu_k$\\[1.5ex] 
$\bS_k$&Submatrix of $\bS_\bP$ obtained by removing its $k$th row and $k$th column\\[1.5ex]
$\alpha_n$ & $\big\|\bU_\bP\transpose\bE\bU_\bP\big\|_2$\\[1.5ex]
$L$ & $\inf\{l\in\mathbb{N}_+:L\geq 3,(\sqrt{n\rho_n})^{l + 1}/\Delta_n^{l - 2}\lesssim \|\bU_\bP\|_{2\to\infty}\}$ guaranteed by Assumption \ref{assumption:Signal_strength}\\[1.5ex]
 $M$ & $\inf\{m\in\mathbb{N}_+:m\geq 3, (\sqrt{n\rho_n}/\Delta_n)^{m - 2}\lesssim \|\bU_\bP\|_{2\to\infty}\}$ guaranteed by Assumption \ref{assumption:Signal_strength}\\[1.5ex]
 $c_0$ & A positive constant such that $\min\{|\lambda_k|/|\lambda_l|:1\leq k < l \leq d + 1,\lambda_k\neq-\lambda_l\} \geq 1 + c_0$\\[1.5ex]
$\calR(\bM_1, \bM_2, t)$ &$ -\sum_{l = 0,l\neq 1}^L t^{-(l + 1)}\bM_1\transpose\expect\big(\bE^l\big)\bM_2$\\[1.5ex]  
$\calP(\bM_1, \bM_2, t)$ &$ t\calR(\bM_1, \bM_2, t)$\\[1.5ex]
$\bb(\bv, k, t)$ &$ \bv - \bV_k\big\{\bS_k^{-1} + \calR(\bV_k, \bV_k, t)\big\}^{-1}\calR\big(\bV_k, \bv, t\big)$\\[1.5ex]
$\calA(\bv, k, t)$ &$ \calP(\bv, \bu_k, t) - \calP(\bv, \bV_k, t)\big\{t\bS_k^{-1} + \calP(\bV_k, \bV_k, t)\big\}^{-1}\calP(\bV_k, \bu_k, t)$\\[1.5ex]
$\widetilde{\calP}(k, t)$ &$ \left[t^2\frac{\partial}{\partial t}\left\{\frac{\calA(\bu_k, k, t)}{t}\right\}\right]^{-1}$\\[1.5ex]
$f_k(t)$ &$ 1 + \lambda_k\calR(\bu_k, \bu_k, t) - \lambda_k\calR(\bu_k, \bV_k, t)\{\bS_k^{-1} + \calR(\bV_k, \bV_k, t)\}^{-1}\calR(\bV_k, \bu_k, t)$\\[1.5ex]
$\bF_k(z)$ & $\bG(z)\bV_k\{\bS_k^{-1} + \bV_k\transpose\bG(z)\bV_k\}^{-1}\bV_k\transpose\bG(z)$\\[1.5ex]
\hline
\end{tabular}
\end{table}

Let $\bM_1,\bM_2$ be any matrices of appropriate dimensions, $t\in\mathbb{R}$, $\bv\in\mathbb{R}^n$, and $k\in [p]$. Then we define several  functions that will be used throughout the paper in Table~\ref{def:sum}. More specifically, we let $\bG(\cdot):\mathbb{C}\to\mathbb{C}^{n\times n}$ be the Green function associated with $\bE$ defined by $\bG(z) = (\bE - z\eye_n)^{-1}$, and $\widetilde{\bG}(z)$ denote the matrix-valued function $\widetilde{\bG}(z) = (\bA - z\eye_n)^{-1}$ for any $z\in\mathbb{C}$. 
 With the convention that $\lambda_d/0 = +\infty$, the existence of $c_0 > 0$ is guaranteed by Assumption \ref{assumption:Eigenvalue_separation}. For all $k\in [d]$, define
\bea\label{akbkdef}
a_k = \left\{\begin{array}{ll}
\lambda_k/(1 + c_0/2),&\quad\lambda_k > 0\\
\lambda_k(1 + c_0/2),&\quad\lambda_k < 0
\end{array}\right.,\quad
b_k = \left\{\begin{array}{ll}
\lambda_k(1 + c_0/2),&\quad\lambda_k > 0\\
\lambda_k/(1 + c_0/2),&\quad\lambda_k < 0
\end{array}\right..
\eae 
\par
In the following, we let $\bx,\by\in\{\bu_1,\ldots,\bu_d\}$. 
We focus on the asymptotic expansion of the bilinear form $\bx\transpose\widehat{\bu}_k\widehat{\bu}_k\transpose\by$ of the projection matrix $\widehat{\bu}_k\widehat{\bu}_k\transpose$. By Cauchy's residual theorem (See e.g., (A.26) in \cite{stein2010complex}), with high probability, we have
\begin{equation}
\label{eqn:Cauchy_residual_formula}
\begin{aligned}
\bx\transpose\widehat{\bu}_k\widehat{\bu}_k\transpose\by
& = -\frac{1}{2\pi\mathbbm{i}}\oint_{\Omega_k}\bx\transpose\widetilde{\bG}(z)\by\mathrm{d}z\\
& = -\frac{1}{2\pi\mathbbm{i}}\oint_{\Omega_k}\bx\transpose\left(\bE + \sum_{r = 1}^d\lambda_r\bu_r\bu_r\transpose - z\eye_n\right)^{-1}\by\mathrm{d}z\\
& = \frac{1}{2\pi\mathbbm{i}}\oint_{\Omega_k}\frac{\lambda_k\bx\transpose\{\bG(z) - \bF_k(z)\}\bu_k\bu_k\transpose\{\bG(z) - \bF_k(z)\}\by}{1 + \lambda_k\bu_k\transpose\{\bG(z) - \bF_k(z)\}\bu_k}\mathrm{d}z,
\end{aligned}
\end{equation}
where $\Omega_k\subset\mathbb{C}$ is the complex contour centered at $(a_k + b_k)/2$ with radius $|b_k - a_k|/2$, recalling that,
\bea\label{def:Fk}
\bF_k(z) = \bG(z)\bV_k\{\bS_k^{-1} + \bV_k\transpose\bG(z)\bV_k\}^{-1}\bV_k\transpose\bG(z).
\eae
Also see Equations (A.35)--(A.38) in Appendix A of \cite{doi:10.1080/01621459.2020.1840990}. 

In the following, we enumerate a collection of equations and high-probability upper bounds that will greatly facilitate our subsequent theoretical analysis. In this section, we assume that $z\in[a_k, b_k]$. First, observe that by Result \ref{result:Noise_matrix_moment_bound} and the fact that $n\rho_n^{1/2}\|\bU_\bP\|_{2\to\infty}/\Delta_n \to 0$ and $|z| = \Theta(\Delta_n)$, for $\bu\in\{\bx,\by\}$, we have
\begin{align}
\label{eqn:R(Vk,Vk,z)}
\calR(\bV_k, \bV_k, z) &= -\frac{\eye_{d - 1}}{z} - \sum_{l = 2}^L\frac{\bV_k\transpose\expect\bE^l\bV_k}{z^{l + 1}}= O\left(\frac{1}{|z|}\right)
,\\
\label{eqn:R'(Vk,Vk,z)}
\calR'(\bV_k, \bV_k, z) &= \frac{\eye_{d - 1}}{z^2} + \sum_{l = 2}^L\frac{(l + 1)\bV_k\transpose\expect\bE^l\bV_k}{z^{l + 2}} = O\left(\frac{1}{|z|^2}\right)
,\\
\label{eqn:R(Vk,uk,z)}
\calR(\bV_k, \bu_k, z) &= - \sum_{l = 2}^L\frac{\bV_k\transpose\expect\bE^l\bu_k}{z^{l + 1}}= O\left(\frac{n\rho_n}{|z|^3}\right),\\
\label{eqn:R'(Vk,uk,z)}
\calR'(\bV_k, \bu_k, z) &= \sum_{l = 2}^L\frac{(l + 1)\bV_k\transpose\expect\bE^l\bu_k}{z^{l + 2}}= O\left(\frac{n\rho_n}{|z|^4}\right),\\
\label{eqn:R(Vk,u,z)}
\calR(\bV_k, \bu, z) & = -\frac{\bV_k\transpose\bu}{z} - \sum_{l = 2}^L\frac{\bV_k\transpose\expect\bE^l\bu}{z^{l + 1}} = O\left(\frac{1}{|z|}\right)
,\\
\label{eqn:R'(Vk,u,z)}
\calR'(\bV_k, \bu, z) & = \frac{\bV_k\transpose\bu}{z^2} + \sum_{l = 2}^L\frac{(l + 1)\bV_k\transpose\expect(\bE^l)\bu}{z^{l + 2}} = O\left(\frac{1}{|z|^2}\right)
,
\end{align}
where $\calR'(\bM_1,\bM_2,z) = \mathrm{d}\calR(\bM_1,\bM_2,z)/\mathrm{d}z$ for any matrices $\bM_1,\bM_2$ of appropriate dimensions.
By Results \ref{result:Noise_matrix_moment_bound} and \ref{result:Noise_matrix_concentration}, the following equations hold w.h.p.:
\begin{equation}
\label{eqn:x_G_y}
\begin{aligned}
\bx\transpose\bG(z)\by
& = - \frac{\bx\transpose\bE\by}{z^2} - \frac{\bx\transpose(\bE^2 - \expect\bE^2)\by}{z^3} + \calR(\bx, \by, z)\\
&\quad + \widetilde{O}_{\prob}\left\{\frac{n^2\rho_n^{3/2}(\log n)^{3\xi}\|\bU_\bP\|_{2\to\infty}^2}{|z|^4}\right\},
\end{aligned}
\end{equation}
\begin{equation}
\label{eqn:x_G_Vk}
\begin{aligned}
\bx\transpose\bG(z)\bV_k& = -\frac{\bx\transpose\bE\bV_k}{z^2}  -\frac{\bx\transpose(\bE^2 - \expect\bE^2)\bV_k}{z^3} + \calR(\bx, \bV_k, z)\\
&\quad + \widetilde{O}_{\prob}\left\{\frac{n^2\rho_n^{3/2}(\log n)^{3\xi}\|\bU_\bP\|_{2\to\infty}^2}{|z|^4}\right\},
\end{aligned}
\end{equation}
\begin{equation}
\label{eqn:Vk_G_Vk}
\begin{aligned}
\bV_k\transpose\bG(z)\bV_k& = -\frac{\bV_k\transpose\bE\bV_k}{z^2}  -\frac{\bV_k\transpose(\bE^2 - \expect\bE^2)\bV_k}{z^3} + \calR(\bV_k, \bV_k, z)\\
&\quad + \widetilde{O}_{\prob}\left\{\frac{n^2\rho_n^{3/2}(\log n)^{3\xi}\|\bU_\bP\|_{2\to\infty}^2}{|z|^4}\right\}.
\end{aligned}
\end{equation}
{\color{black}To see above bounds more clearly, we will take \eqref{eqn:x_G_y} as an example. By Result~\ref{result:Noise_matrix_moment_bound}, we have
\begin{align*}
&\Bigg|\bx\transpose\bG(z)\by -\Bigg[ -\frac{\bx\transpose\bE\by}{z^2}  -\frac{\bx\transpose(\bE^2 - \expect\bE^2)\by}{z^3} + \calR(\bx,\by, z)\Bigg]\Bigg|
\\
& = \Bigg|\sum_{\ell = 3}^{L}- \frac{\bx\transpose(\bE^\ell - \expect\bE^\ell)\by}{z^{\ell + 1}} + \sum_{\ell = L + 1}^{\infty}- \frac{\bx\transpose\bE^\ell\by}{z^{\ell + 1}} \Bigg|
\\
& \leq \sum_{\ell = 3}^{L} \frac{\big| \bx\transpose(\bE^\ell - \expect\bE^\ell)\by\big|}{|z| ^{\ell + 1}}  + \sum_{\ell = L + 1}^\infty \frac{\|\bE\|_2^{\ell}}{|z|^{\ell+1}}
\\
& = \widetilde{O}_{\prob}\left\{\frac{n^2\rho_n^{3/2}(\log n)^{3\xi}\|\bU_\bP\|_{2\to\infty}^2}{|z|^4}\right\}  +  \widetilde{O}_{\prob}\left\{ \frac{(\sqrt{n\rho_n})^{L + 1}}{\Delta_n^{L + 2}}\sum_{\ell = 0}^{\infty}\Bigg(\frac{\sqrt{n\rho_n}}{\Delta_n}\Bigg)^\ell\right\}
\\
& = \widetilde{O}_{\prob}\left\{\frac{n^2\rho_n^{3/2}(\log n)^{3\xi}\|\bU_\bP\|_{2\to\infty}^2}{|z|^4}\right\},
\end{align*}
where the last equality is because
\begin{align*}
 \frac{(\sqrt{n\rho_n})^{L + 1}}{\Delta_n^{L + 2}}\sum_{\ell = 0}^{\infty}\Bigg(\frac{\sqrt{n\rho_n}}{\Delta_n}\Bigg)^\ell &\lesssim   \frac{(\sqrt{n\rho_n})^{L + 1}}{\Delta_n^{L + 2}} \lesssim \frac{1}{\Delta_n^{4 }}\|\bU_{\bP}\|_{2\rightarrow \infty}
\lesssim
\frac{\sqrt{n}}{\Delta_n^{4 }}\|\bU_{\bP}\|^2_{2\rightarrow \infty}
\\
 &=o\left\{  \frac{\sqrt{n}(n\rho_n)^{3/2}(\log n)^{3\xi}}{|z|^4}\|\bU_{\bP}\|^2_{2\rightarrow \infty}\right\};
\end{align*}
Note here we use the following facts by definitions: (i) $(\sqrt{n\rho_n})^{L + 1} \lesssim \|\bU_{\bP}\|_{2\rightarrow\infty}\Delta_n^{L - 2}$, (ii) $|z| = \Theta(\Delta_n)$, (iii) $n\rho_n\rightarrow \infty$, (iv) $\sqrt{n}\|\bU_{\bP}\|_{2\rightarrow\infty}\gtrsim 1$.}
\par
\par
In particular, for $\bu\in\{\bx,\by\}$, by \eqref{eqn:R(Vk,uk,z)}, \eqref{eqn:R(Vk,u,z)},  \eqref{eqn:x_G_Vk}, and Result~\ref{result:Noise_matrix_moment_bound},
\begin{equation}
\label{eqn:u_G_Vk_and_uk_G_Vk}
\begin{aligned}
\bu\transpose\bG(z)\bV_k & = 
 \widetilde{O}_{\prob}\left(\frac{1}{|z|}\right)
,
\\
\bu_k\transpose\bG(z)\bV_k 
&= 
\widetilde{O}_{\prob}\left(\frac{\alpha_n}{|z|^2} + \frac{\|\bU_\bP\transpose(\bE^2 - \expect\bE^2)\bU_\bP\|_2 + n\rho_n}{|z|^3}\right)
\\
&= 
\widetilde{O}_{\prob}\left(\frac{\alpha_n}{|z|^2} + \frac{\|\bE\|^2 + \|\expect\bE^2\|_2 + n\rho_n}{|z|^3}\right)
\\
&= \widetilde{O}_{\prob}\left(\frac{\alpha_n}{|z|^2} + \frac{n\rho_n}{|z|^3}\right).
\end{aligned}
\end{equation}
For $\{\bS_k^{-1} + \bV_k\transpose\bG(z)\bV_k\}^{-1}$, we have by \eqref{eqn:Vk_G_Vk},
\begin{align*}
\bS_k^{-1} + \bV_k\transpose\bG(z)\bV_k& = \bS_k^{-1} + \calR(\bV_k, \bV_k, z)
  + \widetilde{O}_{\prob}\left(\frac{\alpha_n}{|z|^2} 
+ \frac{n\rho_n}{|z|^3}\right)
.
\end{align*}
Since the eigenvalues of 
\begin{align}\label{boundSRVV}\bS_k^{-1} + \calR(\bV_k, \bV_k, z)=\Theta(1/|z|)
\end{align} by Assumptions \ref{assumption:Signal_strength} and \ref{assumption:Eigenvalue_separation},  \eqref{eqn:R(Vk,Vk,z)}, and recalling that $z\in[a_k,b_k]$ and \eqref{akbkdef}. Then we have,
\begin{equation}
\label{eqn:Sk_plus_VkGVk_inverse}
\begin{aligned}
\{\bS_k^{-1} + \bV_k\transpose\bG(z)\bV_k\}^{-1}
&
 = \{\bS_k^{-1} + \calR(\bV_k, \bV_k, z)\}^{-1}
 + \widetilde{O}_{\prob}\left(\alpha_n + \frac{n\rho_n}{|z|}\right)
 = \widetilde{O}_{\prob}(|z|),
\end{aligned}
\end{equation}
where we have used the fact $\alpha_n = \|\bU_\bP\transpose\bE\bU_\bP\|_2\leq\|\bE\|_2 = \widetilde{O}_{\prob}(\sqrt{n\rho_n})$ and $z\in[a_k,b_k ]$. 
Furthermore, recalling that $\bG(z) = -\sum_{m = 0}^\infty\bE^m/(z^{m + 1})$ w.h.p., by the same reasoning, we also obtain,
\begin{equation}
\label{eqn:x_G'_y}
\begin{aligned}
\bx\transpose\bG'(z)\by
& = \frac{2\bx\transpose\bE\by}{z^3} + \frac{3\bx\transpose(\bE^2 - \expect\bE^2)\by}{z^4} + \calR'(\bx, \by, z)\\
&\quad + \widetilde{O}_{\prob}\left\{\frac{n^2\rho_n^{3/2}(\log n)^{3\xi}\|\bU_\bP\|_{2\to\infty}^2}{|z|^5}\right\},
\end{aligned}
\end{equation}
\begin{equation}
\label{eqn:uk_G'_uk}
\begin{aligned}
\bu_k\transpose\bG'(z)\bu_k
& = \frac{2\bu_k\transpose\bE\bu_k}{z^3} + \frac{3\bu_k\transpose(\bE^2 - \expect\bE^2)\bu_k}{z^4} + \calR'(\bu_k, \bu_k, z)\\
&\quad + \widetilde{O}_{\prob}\left\{\frac{n^2\rho_n^{3/2}(\log n)^{3\xi}\|\bU_\bP\|_{2\to\infty}^2}{|z|^5}\right\}\\
& = \widetilde{O}_{\prob}\left(\frac{1}{|z|^2}\right),
\end{aligned}
\end{equation}
\begin{equation}
\label{eqn:Vk_G'_Vk}
\begin{aligned}
\bV_k\transpose\bG'(z)\bV_k
& = \frac{2\bV_k\transpose\bE\bV_k}{z^3} + \frac{3\bV_k\transpose(\bE^2 - \expect\bE^2)\bV_k}{z^4} + \calR'(\bV_k, \bV_k, z)\\
&\quad + \widetilde{O}_{\prob}\left\{\frac{n^2\rho_n^{3/2}(\log n)^{3\xi}\|\bU_\bP\|_{2\to\infty}^2}{|z|^5}\right\}\\
& = \widetilde{O}_{\prob}\left(\frac{1}{|z|^2}\right),
\end{aligned}
\end{equation}
\begin{equation}
\label{eqn:uk_G'_Vk}
\begin{aligned}
\bu_k\transpose\bG'(z)\bV_k
& = \frac{2\bu_k\transpose\bE\bV_k}{z^3} + \frac{3\bu_k\transpose(\bE^2 - \expect\bE^2)\bV_k}{z^4} + \calR'(\bu_k, \bV_k, z)\\
&\quad + \widetilde{O}_{\prob}\left\{\frac{n^2\rho_n^{3/2}(\log n)^{3\xi}\|\bU_\bP\|_{2\to\infty}^2}{|z|^5}\right\}\\
& = \widetilde{O}_{\prob}\left(\frac{\alpha_n}{|z|^3} + \frac{n\rho_n}{|z|^4}\right);
\end{aligned}
\end{equation}
Note here $f'(z) = df(z)/dz$ represents the complex derivative of $f(z)$ with respect to $z$.

\subsection{Technical preparations}
\label{sub:technical_preparation}
Now let $\widehat{t}_k$ be the solution to the equation
\[
1 + \lambda_k\bu_k\transpose\{\bG(z) - \bF_k(z)\}\bu_k = 0
\]
for $z\in [a_k, b_k]$; Recall the definition of $\bF_k(z)$ in \eqref{def:Fk}. Then Cauchy's residual theorem entails that
\begin{align}\label{eqn:bilinear_form_residual_formula}
\bx\transpose\widehat{\bu}_k\widehat{\bu}_k\transpose\by = \frac{\widehat{t}_k^2\bx\transpose\{\bG(\widehat{t}_k) - \bF_k(\widehat{t}_k)\}\bu_k\bu_k\transpose\{\bG(\widehat{t}_k) - \bF_k(\widehat{t}_k)\}\by}{\widehat{t}_k^2\bu_k\transpose\{\bG'(\widehat{t}_k) - \bF_k'(\widehat{t}_k)\}\bu_k},
\end{align}
where
\[
\bG'(\widehat{t}_k) = \frac{\mathrm{d}\bG(z)}{\mathrm{d}z}\mathrel{\bigg|}_{z = \widehat{t}_k},\quad
\bF_k'(\widehat{t}_k) = \frac{\mathrm{d}\bF_k(z)}{\mathrm{d}z}\mathrel{\bigg|}_{z = \widehat{t}_k}.
\]
Also see (A.64) in Appendix A of \cite{doi:10.1080/01621459.2020.1840990}. Below, our analysis of the asymptotic expansion of $\bx\transpose\widehat{\bu}_k\widehat{\bu}_k\transpose\by$ will be centered around the formula \eqref{eqn:bilinear_form_residual_formula}. 

\begin{lemma}\label{lemma:uk_G_minus_Fk_uk}
Suppose the conditions of Theorem \ref{thm:Eigenvector_Expansion} hold. For each fixed $k\in [p]$, $z\in [a_k, b_k]$, any $\xi > 1$, and $\bu\in\mathbb{R}^n$ with $\|\bu\|_\infty = O(1/\sqrt{n})$, recalling the definition in Table~\ref{def:sum}, we have
\begin{align*}
&z\bu\transpose\{\bG(z) - \bF_k(z)\}\bu_k\\
&\quad = \calA(\bu, k, z) - \frac{\bb(\bu, k, z)\transpose\bE\bu_k}{z} - \frac{\bb(\bu, k, z)\transpose(\bE^2 - \expect\bE^2)\bu_k}{z^2}\\
&\quad\quad  + \widetilde{O}_{\prob}\left\{\frac{\alpha_n^2}{|z|^2} + \frac{n^2\rho_n^{3/2}(\log n)^{3\xi}\|\bU_\bP\|_{2\to\infty}^2}{|z|^3} + \frac{(n\rho_n)^2}{|z|^4}\right\}.
\end{align*}
\end{lemma}
\begin{proof}
By \eqref{eqn:Sk_plus_VkGVk_inverse}, we have
\begin{align*}
\{\bS_k^{-1} + \bV_k\bG(z)\bV_k\transpose\}^{-1}& = \{\bS_k^{-1} + \calR(\bV_k, \bV_k, z)\}^{-1}
+ \widetilde{O}_{\prob}\left(\alpha_n + \frac{n\rho_n}{|z|}\right).
\end{align*}
Recall that $\alpha_n = \|\bU_\bP\transpose\bE\bU_\bP\|_2\leq\|\bE\|_2 = \widetilde{O}(\sqrt{n\rho_n})$. 
Then by \eqref{eqn:x_G_Vk} and \eqref{eqn:u_G_Vk_and_uk_G_Vk}, we obtain
\begin{align*}
\bu\transpose\bF_k(z)\bu_k
& = \bu\transpose\bG(z)\bV_k\{\bS_k^{-1} + \calR(\bV_k, \bV_k, z)\}^{-1}\bV_k\transpose\bG(z)\bu_k\\
&\quad +  \bu\transpose\bG(z)\bV_k\widetilde{O}_{\prob}\left(\alpha_n + \frac{n\rho_n}{|z|}\right)\bV_k\transpose\bG(z)\bu_k\\
& = \bu\transpose\bG(z)\bV_k\{\bS_k^{-1} + \calR(\bV_k, \bV_k, z)\}^{-1}\bV_k\transpose\bG(z)\bu_k
 + \widetilde{O}_{\prob}\left\{\frac{\alpha_n^2}{|z|^3} + \frac{(n\rho_n)^2}{|z|^5}\right\}\\
& = \left\{-\frac{\bu\transpose\bE\bV_k}{z^2} - \frac{\bu\transpose(\bE^2 - \expect\bE^2)\bV_k}{z^3} + \calR(\bu, \bV_k, z)\right\}\\
&\quad\times\{\bS_k^{-1} + \calR(\bV_k, \bV_k, z)\}^{-1}\\
&\quad\times \left\{-\frac{\bV_k\transpose\bE\bu_k}{z^2} - \frac{\bV_k\transpose(\bE^2 - \expect\bE^2)\bu_k}{z^3} + \calR(\bV_k, \bu_k, z)\right\}
\\&\quad
 + \widetilde{O}_{\prob}\left\{\frac{\alpha_n^2}{|z|^3} + \frac{n^2\rho_n^{3/2}(\log n)^{3\xi}\|\bU_\bP\|_{2\to\infty}^2}{|z|^4} + \frac{(n\rho_n)^2}{|z|^5}\right\}\\
& = \calR(\bu, \bV_k, z)\{\bS_k^{-1} + \calR(\bV_k, \bV_k, z)\}^{-1}\\
&\quad\times\left\{-\frac{\bV_k\transpose\bE\bu_k}{z^2} - \frac{\bV_k\transpose(\bE^2 - \expect\bE^2)\bu_k}{z^3} + \calR(\bV_k, \bu_k, z)\right\}
\\&\quad
 + \widetilde{O}_{\prob}\left\{\frac{\alpha_n^2}{|z|^3} + \frac{n^2\rho_n^{3/2}(\log n)^{3\xi}\|\bU_\bP\|_{2\to\infty}^2}{|z|^4} + \frac{(n\rho_n)^2}{|z|^5}\right\}\\
& = -\frac{1}{z^2}\calR(\bu, \bV_k, z)\{\bS_k^{-1} + \calR(\bV_k, \bV_k, z)\}^{-1}\bV_k\transpose\bE\bu_k\\
&\quad - \frac{1}{z^3}\calR(\bu, \bV_k, z)\{\bS_k^{-1} + \calR(\bV_k, \bV_k, z)\}^{-1}\bV_k\transpose(\bE^2 - \expect\bE^2)\bu_k\\
&\quad + \calR(\bu, \bV_k, z)\{\bS_k^{-1} + \calR(\bV_k, \bV_k, z)\}^{-1}\calR(\bV_k, \bu_k, z)\\ 
&\quad+ \widetilde{O}_{\prob}\left\{\frac{\alpha_n^2}{|z|^3} + \frac{n^2\rho_n^{3/2}(\log n)^{3\xi}\|\bU_\bP\|_{2\to\infty}^2}{|z|^4} + \frac{(n\rho_n)^2}{|z|^5}\right\}.
\end{align*}
By \eqref{eqn:x_G_y}, we have
\begin{align*}
\bu\transpose\bG(z)\bu_k
& = - \frac{\bu\transpose\bE\bu_k}{z^2} - \frac{\bu\transpose(\bE^2 - \expect\bE^2)\bu_k}{z^3} + \calR(\bu, \bu_k, z)\\
&\quad + \widetilde{O}_{\prob}\left\{\frac{n^2\rho_n^{3/2}(\log n)^{3\xi}\|\bU_\bP\|_{2\to\infty}^2}{|z|^4}\right\}.
\end{align*}
Hence, we conclude that
\begin{align*}
&z\bu\transpose\{\bG(z) - \bF_k(z)\}\bu_k\\
&\quad =  - \frac{\bu\transpose\bE\bu_k}{z} - \frac{\bu\transpose(\bE^2 - \expect\bE^2)\bu_k}{z^2}\\
&\quad\quad + z\calR(\bu, \bu_k, z) - z\calR(\bu, \bV_k, z)\{\bS_k^{-1} + \calR(\bV_k, \bV_k, z)\}^{-1}\calR(\bV_k, \bu_k, z)\\
&\quad\quad + \frac{1}{z}\calR(\bu, \bV_k, z)\{\bS_k^{-1} + \calR(\bV_k, \bV_k, z)\}^{-1}\bV_k\transpose\bE\bu_k\\
&\quad\quad + \frac{1}{z^2}\calR(\bu, \bV_k, z)\{\bS_k^{-1} + \calR(\bV_k, \bV_k, z)\}^{-1}\bV_k\transpose(\bE^2 - \expect\bE^2)\bu_k\\
&\quad\quad + \widetilde{O}_{\prob}\left\{\frac{\alpha_n^2}{|z|^2} + \frac{n^2\rho_n^{3/2}(\log n)^{3\xi}\|\bU_\bP\|_{2\to\infty}^2}{|z|^3} + \frac{(n\rho_n)^2}{|z|^4}\right\}\\
&\quad = \calA(\bu, k, z) - \frac{\bb(\bu, k, z)\transpose\bE\bu_k}{z} - \frac{\bb(\bu, k, z)\transpose(\bE^2 - \expect\bE^2)\bu_k}{z^2}
\\&\quad\quad
 + \widetilde{O}_{\prob}\left\{\frac{\alpha_n^2}{|z|^2} + \frac{n^2\rho_n^{3/2}(\log n)^{3\xi}\|\bU_\bP\|_{2\to\infty}^2}{|z|^3} + \frac{(n\rho_n)^2}{|z|^4}\right\}.
\end{align*}
\end{proof}

\begin{lemma}\label{lemma:uk_G'_minus_Fk'_uk}
Suppose the conditions of Theorem \ref{thm:Eigenvector_Expansion} hold. For each fixed $k\in [p]$, any $\xi > 1$, and $z\in [a_k, b_k]$, we have
\begin{align*}
\bu_k\transpose\{\bG'(z) - \bF'_k(z)\}\bu_k
& = \frac{2\bu_k\transpose\bE\bu_k}{z^3} + \frac{3\bu_k\transpose(\bE^2 - \expect\bE^2)\bu_k}{z^4} + \calR'(\bu_k, \bu_k, z)\\
&\quad + \widetilde{O}_{\prob}\left\{\frac{\alpha_n^2}{|z|^4} + \frac{(n\rho_n)^2}{|z|^6} + \frac{n^2\rho_n^{3/2}(\log n)^{3\xi}\|\bU_\bP\|_{2\to\infty}^2}{|z|^5}\right\}.
\end{align*}
\end{lemma}
\begin{proof}
From \eqref{eqn:Vk_G'_Vk} and Result \ref{result:Noise_matrix_moment_bound}, it is immediate that
\begin{align*}
&\bV_k\transpose\bG'(z)\bV_k - \calR'(\bV_k, \bV_k, z)\\
&\quad = \frac{2\bu_k\transpose\bE\bV_k}{z^3} + \frac{3\bu_k\transpose(\bE^2 - \expect\bE^2)\bV_k}{z^4} + \widetilde{O}_{\prob}\left\{\frac{n^2\rho_n^{3/2}(\log n)^{3\xi}\|\bU_\bP\|_{2\to\infty}^2}{|z|^5}\right\}
\\
&\quad = \widetilde{O}_{\prob}\left(\frac{
\alpha_n}{|z|^3} + \frac{n\rho_n}{|z|^4}\right)
.
\end{align*}
Then by  \eqref{eqn:R'(Vk,Vk,z)}  and \eqref{eqn:Sk_plus_VkGVk_inverse}, we obtain
\bea\label{diff:sgsr}
&\left\|\frac{\mathrm{d}}{\mathrm{d}z}\{\bS_k^{-1} + \bV_k\transpose\bG(z)\bV_k\}^{-1} - \frac{\mathrm{d}}{\mathrm{d}z}\{\bS_k^{-1} + \calR(\bV_k, \bV_k, z)\}^{-1}\right\|_2\\
&\quad = \|- \{\bS_k^{-1} + \bV_k\transpose\bG(z)\bV_k\}^{-1}\bV_k\transpose\bG'(z)\bV_k\{\bS_k^{-1} + \bV_k\transpose\bG(z)\bV_k\}^{-1}\\
&\quad\quad + \{\bS_k^{-1} + \calR(\bV_k, \bV_k, z)\}^{-1}\calR'(\bV_k, \bV_k, z)\{\bS_k^{-1} + \calR(\bV_k, \bV_k, z)\}^{-1}\|_2\\
&\quad\leq \|\{\bS_k^{-1} + \bV_k\transpose\bG(z)\bV_k\}^{-1}\|_2^2\|\bV_k\transpose\bG'(z)\bV_k - \calR'(\bV_k, \bV_k, z)\|_2\\
&\quad\quad + \|\{\bS_k^{-1} + \bV_k\transpose\bG(z)\bV_k\}^{-1}\|_2\|\calR'(\bV_k, \bV_k, z)\|_2\\
&\quad\quad\quad\times \|\{\bS_k^{-1} + \bV_k\transpose\bG(z)\bV_k\}^{-1} - \{\bS_k^{-1} + \calR(\bV_k, \bV_k, z)\}^{-1}\|_2\\
&\quad\quad + \|\{\bS_k^{-1} + \bV_k\transpose\bG(z)\bV_k\}^{-1} - \{\bS_k^{-1} + \calR(\bV_k, \bV_k, z)\}^{-1}\|_2\\
&\quad\quad\quad\times \|\calR'(\bV_k, \bV_k, z)\|_2\|\{\bS_k^{-1} + \calR(\bV_k, \bV_k, z)\}^{-1}\|_2\\
&\quad
 = \widetilde{O}_{\prob}\left(\frac{\alpha_n}{|z|} + \frac{n\rho_n}{|z|^2}\right)
\eae
and
\begin{equation}\label{bound:sv}
\begin{aligned}
&\left\|\frac{\mathrm{d}}{\mathrm{d}z}\{\bS_k^{-1} + \calR(\bV_k, \bV_k, z)\}^{-1}\right\|_2\\
&\quad\leq \left\|\{\bS_k^{-1} + \calR(\bV_k, \bV_k, z)\}^{-1}\right\|_2^2\|\calR'(\bV_k, \bV_k, z)\|_2 = O(1).
\end{aligned}
\end{equation}
Combing \eqref{diff:sgsr} and \eqref{bound:sv} and using the triangle inequality, we have,
\begin{align*}
\frac{\mathrm{d}}{\mathrm{d}z}\{\bS_k^{-1} + \bV_k\transpose\bG(z)\bV_k\}^{-1}
& = \frac{\mathrm{d}}{\mathrm{d}z}\{\bS_k^{-1} + \calR(\bV_k, \bV_k, z)\}^{-1} + \widetilde{O}_{\prob}\left(\frac{\alpha_n}{|z|} + \frac{n\rho_n}{|z|^2}\right)
 = \widetilde{O}_{\prob}(1).
\end{align*}
We then further have
\begin{align*}
\bu_k\transpose\bF_k'(z)\bu_k
& = 2\bu_k\transpose\bG'(z)\bV_k\{\bS_k^{-1} + \bV_k\transpose\bG(z)\bV_k\}^{-1}\bV_k\transpose\bG(z)\bu_k
\\&\quad
 + \bu_k\transpose\bG(z)\bV_k\Big[\frac{\mathrm{d}}{\mathrm{d}z}\{\bS_k^{-1} + \bV_k\transpose\bG(z)\bV_k\}^{-1}\Big]\bV_k\transpose\bG(z)\bu_k\\
& = \widetilde{O}_{\prob}\left(\frac{\alpha_n}{|z|^3} + \frac{n\rho_n}{|z|^4}\right)\widetilde{O}(|z|)\widetilde{O}_{\prob}\left(\frac{\alpha_n}{|z|^2} + \frac{n\rho_n}{|z|^3}\right)\\
&\quad + \widetilde{O}_{\prob}\left(\frac{\alpha_n}{|z|^2} + \frac{n\rho_n}{|z|^3}\right)\widetilde{O}(1)\widetilde{O}_{\prob}\left(\frac{\alpha_n}{|z|^2} + \frac{n\rho_n}{|z|^3}\right)\\
& = \widetilde{O}_{\prob}\left\{\frac{\alpha_n^2}{|z|^4} + \frac{(n\rho_n)^2}{|z|^6}\right\},
\end{align*}
by \eqref{eqn:uk_G'_Vk}, \eqref{eqn:Sk_plus_VkGVk_inverse}, and \eqref{eqn:u_G_Vk_and_uk_G_Vk}. Thus, combining the above result with \eqref{eqn:uk_G'_uk}, we obtain
\begin{align*}
\bu_k\transpose\{\bG'(z) - \bF_k'(z)\}\bu_k
& = \frac{2\bu_k\transpose\bE\bu_k}{z^3} + \frac{3\bu_k\transpose(\bE^2 - \expect\bE^2)\bu_k}{z^4} + \calR'(\bu_k, \bu_k, z)\\
&\quad + \widetilde{O}_{\prob}\left\{\frac{\alpha_n^2}{|z|^4} + \frac{(n\rho_n)^2}{|z|^6} + \frac{n^2\rho_n^{3/2}(\log n)^{3\xi}\|\bU_\bP\|_{2\to\infty}^2}{|z|^5}\right\}.
\end{align*}
The proof is thus completed.
\end{proof}

\begin{lemma}\label{lemma:tkhat_concentration}
Suppose Assumptions \ref{assumption:Signal_strength}--\ref{assumption:Noise_matrix_distribution} in the manuscript hold. For each fixed $k\in [p]$, let $\widehat{t}_k$ be the solution to the (random) equation
\[
1 + \lambda_k\bu_k\transpose\{\bG(z) - \bF_k(z)\}\bu_k = 0
\]
over $z\in [a_k, b_k]$. Then $\widehat{t}_k$ exists, is unique w.h.p., and satisfies
\[
\widehat{t}_k = \lambda_k + \widetilde{O}_{\prob}\left(\alpha_n + \frac{n\rho_n}{\Delta_n}\right).
\]
\end{lemma}
\begin{proof}
Let $t_k$ be the solution to the (deterministic) equation $f_k(z) = 0$ over $z\in [a_k, b_k]$; Recall the definition of $f_k(z)$ in Table~\ref{def:sum}. Following the proof of Lemma 3 in \cite{doi:10.1080/01621459.2020.1840990}, we know that $t_k$ exists and is unique. Furthermore, $f_k'(z) = \lambda_k z^{-2}\{1 + o(1)\}$
for $z\in [a_k, b_k]$ and
\begin{align*}
f_k(\lambda_k)
& = 1 + \lambda_k\calR(\bu_k, \bu_k, \lambda_k) - \lambda_k\calR(\bu_k, \bV_k, \lambda_k)\{\bS_k^{-1} + \calR(\bV_k, \bV_k, \lambda_k)\}^{-1}\\
&\quad\times \calR(\bV_k, \bu_k, \lambda_k)\\
& = 1 + \lambda_k\calR(\bu_k, \bu_k, \lambda_k) - \lambda_kO\left(\frac{n\rho_n}{|\lambda_k|^3}\right)O(|\lambda_k|)O\left(\frac{n\rho_n}{|\lambda_k|^3}\right)\\
& = 1 - \sum_{l = 0}^L\frac{\bu_k\transpose\expect\big(\bE^l\big)\bu_k}{\lambda_k^l} + O\left(\frac{n^2\rho_n^2}{|\lambda_k|^4}\right)
= -\sum_{l = 2}^L\frac{\bu_k\transpose\expect\big(\bE^l\big)\bu_k}{\lambda_k^l} + O\left(\frac{n^2\rho_n^2}{|\lambda_k|^4}\right)
\\&
= O\left(\frac{n\rho_n}{\Delta_n^2}\right)
\end{align*}
by Result \ref{result:Noise_matrix_moment_bound}, \eqref{eqn:R(Vk,uk,z)}, and \eqref{eqn:Sk_plus_VkGVk_inverse}. It follows from mean-value theorem that
\[
f_k(\lambda_k) = f_k(\lambda_k) - f_k(t_k) = f_k'(\tilde{z})(\lambda_k - t_k),
\]
where $\tilde{z}$ is between $\lambda_k$ and $t_k$. Clearly, 
\bea\label{f'kz}
|f_k'(\tilde{z})| = \Theta(|\lambda_k|/|\tilde{z}|^2) = \Theta(1/\Delta_n),
\eae noting that $t_k\in[a_k,b_k]$ and thus $|\tilde{z}| = \Theta(\Delta_n)$. It then follows that,
\[
|t_k - \lambda_k| = \left|\frac{f_k(\lambda_k)}{f_k'(\tilde{z})}\right| = O\left(\frac{n\rho_n}{\Delta_n^2}\right)\Theta(\Delta_n) = O\left(\frac{n\rho_n}{\Delta_n}\right).
\]
Therefore, it is now sufficient to show that $\widehat{t}_k = t_k + \widetilde{O}_{\prob}\{\alpha_n + (n\rho_n)/\Delta_n\}$. 

The argument here is similar to that in (A.56)--(A.60) in Appendix A of \cite{doi:10.1080/01621459.2020.1840990} with slight modifications. For convenience, we let $h_k(\cdot)$ denote $z\mapsto 1 + \lambda_k\bu_k\transpose\{\bG(z) - \bF_k(z)\}\bu_k$. By Lemma \ref{lemma:uk_G'_minus_Fk'_uk} and Result \ref{result:Noise_matrix_moment_bound},
\begin{align*}
h'_k(z) & = \lambda_k\bu_k\transpose\{\bG'(z) - \bF_k'(z)\}\bu_k\\
& = \frac{2\lambda_k\bu_k\transpose\bE\bu_k}{z^3} + \frac{3\lambda_k\bu_k\transpose(\bE^2 - \expect\bE^2)\bu_k}{z^4} + \lambda_k\calR'(\bu_k, \bu_k, z)\\
&\quad + \frac{\lambda_k}{z^2}\widetilde{O}_{\prob}\left\{\frac{\alpha_n^2}{|z|^2} + \frac{(n\rho_n)^2}{|z|^4} + \frac{n^2\rho_n^{3/2}(\log n)^{3\xi}\|\bU_\bP\|_{2\to\infty}^2}{|z|^3}\right\}\\
& = \lambda_k\calR'(\bu_k, \bu_k, z) + \frac{\lambda_k}{z^2}\widetilde{O}_{\prob}\left(\frac{\alpha_n}{|z|}\right) + \frac{\lambda_k}{z^2}\widetilde{O}_{\prob}\left(\frac{\|\bu_k\transpose(\bE^2 - \expect\bE^2)\bu_k\|}{|z|^2}\right)\\
&\quad+ \frac{\lambda_k}{z^2}\widetilde{O}_{\prob}\left\{\frac{\alpha_n^2}{|z|^2} + \frac{(n\rho_n)^2}{|z|^4} + \frac{n^2\rho_n^{3/2}(\log n)^{3\xi}\|\bU_\bP\|_{2\to\infty}^2}{|z|^3}\right\}\\
& = \lambda_k\sum_{l = 0,l\neq 1}^L\frac{(l + 1)\bu_k\transpose\expect\bE^l\bu_k}{z^{l + 2}} + \frac{\lambda_k}{z^2}\widetilde{O}_{\prob}\left\{\frac{\alpha_n}{|z|} + \frac{n\rho_n}{|z|^2} + \frac{n^2\rho_n^{3/2}(\log n)^{3\xi}\|\bU_\bP\|_{2\to\infty}^2}{|z|^3}\right\}\\
& = \frac{\lambda_k}{z^2} + \frac{\lambda_k}{z^2}O\left(\frac{n\rho_n}{|z|^2}\right) +  \frac{\lambda_k}{z^2}\sqrt{n}\|\bU_\bP\|_{2\to\infty}\sum_{l = 3}^L(l + 1)\left[O\left\{\frac{(n\rho_n)^{1/2}}{z}\right\}\right]^l\\
&\quad + \frac{\lambda_k}{z^2}\widetilde{O}_{\prob}\left\{\frac{\alpha_n^2}{|z|^2} + \frac{(n\rho_n)^2}{|z|^4} + \frac{n^2\rho_n^{3/2}(\log n)^{3\xi}\|\bU_\bP\|_{2\to\infty}^2}{|z|^3}\right\}\\
& = \frac{\lambda_k}{z^2}\{1 + o(1)\} + \frac{\lambda_k}{z^2}\widetilde{O}_{\prob}\left\{\frac{\alpha_n^2}{|z|^2} + \frac{(n\rho_n)^2}{|z|^4} + \frac{n^2\rho_n^{3/2}(\log n)^{3\xi}\|\bU_\bP\|_{2\to\infty}^2}{|z|^3}\right\}.
\end{align*}
This shows that $h_k(z)$ is a monotone function over $z\in [a_k, b_k]$ w.h.p.. By Lemma \ref{lemma:uk_G_minus_Fk_uk}, we have
\begin{align*}
h_k(z)
& = 1 + \frac{\lambda_k}{z}z\bu_k\transpose\{\bG(z) - \bF_k(z)\}\bu_k\\
& = 1 + \frac{\lambda_k}{z}\calA(\bu_k, k, z) - \frac{\lambda_k}{z}\frac{\bb(\bu_k, k, z)\transpose\bE\bu_k}{z} - \frac{\lambda_k}{z}\frac{\bb(\bu_k, k, z)\transpose(\bE^2 - \expect\bE^2)\bu_k}{z^2}\\
&\quad  + \frac{\lambda_k}{z}\widetilde{O}_{\prob}\left\{\frac{\alpha_n^2}{|z|^2} + \frac{n^2\rho_n^{3/2}(\log n)^{3\xi}\|\bU_\bP\|_{2\to\infty}^2}{|z|^3} + \frac{(n\rho_n)^2}{|z|^4}\right\}.
\end{align*}
By Result \ref{result:Noise_matrix_concentration}, \eqref{boundSRVV}, Assumption \ref{assumption:Eigenvector_delocalization}, and Lemma \ref{lemma:Utranspose_E_cubic_U_concentration},
\begin{equation}
\label{eqn:A(uk,k,z)}
\begin{aligned}
\calA(\bu_k, k, z)
& = -1 - \sum_{l = 2}^L\frac{\bu_k\transpose\expect \bE^l\bu_k}{z^l}\\
&\quad - \left(\sum_{l = 2}^L\frac{\bu_k\transpose\expect \bE^l\bV_k}{z^l}\right)
\{\bS_k^{-1} + \calR(\bV_k, \bV_k, z)\}^{-1}
\left(\sum_{l = 2}^L\frac{\bV_k\transpose\expect \bE^l\bu_k}{z^{l + 1}}\right)\\
& = - 1 -\frac{\bu_k\transpose\expect \bE^2\bu_k}{z^2} + O\left\{\frac{(n\rho_n)^{3/2}}{|z|^3}\right\}
\end{aligned}
\end{equation}
and
\begin{equation}
\label{eqn:b(uk,k,z)}
\begin{aligned}
\bb(\bu_k, k, z)& = \bu_k - \bV_k\{\bS_k^{-1} + \calR(\bV_k, \bV_k, z)\}^{-1}\calR(\bV_k, \bu_k, z)\\
& = \bu_k - \bV_k\times O\left(\frac{n\rho_n}{|z|^2}\right).
\end{aligned}
\end{equation}
Therefore, over $z\in [a_k, b_k]$, we have
\begin{align*}
h_k(z)
& = 1 - \frac{\lambda_k}{z} -\frac{\lambda_k\bu_k\transpose\expect \bE^2\bu_k}{z^3} + \frac{\lambda_k}{z}O\left\{\frac{(n\rho_n)^{3/2}}{|z|^3}\right\}\\
&\quad - \frac{\lambda_k}{z}\left\{\bu_k - \bV_k\times O\left(\frac{n\rho_n}{|z|^2}\right)\right\}\transpose\frac{\bE\bu_k}{z}
\\&\quad
 - \frac{\lambda_k}{z}\left\{\bu_k - \bV_k\times O\left(\frac{n\rho_n}{|z|^2}\right)\right\}\transpose\frac{(\bE^2 - \expect\bE^2)\bu_k}{z^2}
 \\&\quad
 + \frac{\lambda_k}{z}\widetilde{O}_{\prob}\left\{\frac{\alpha_n^2}{|z|^2} + \frac{n^2\rho_n^{3/2}(\log n)^{3\xi}\|\bU_\bP\|_{2\to\infty}^2}{|z|^3} + \frac{(n\rho_n)^2}{|z|^4}\right\}\\
& = \frac{z - \lambda_k}{z} + \frac{\lambda_k}{z}\widetilde{O}_{\prob}\left\{\frac{\alpha_n}{|z|} + \frac{n\rho_n}{|z|^2} + \frac{n^2\rho_n^{3/2}(\log n)^{3\xi}\|\bU_\bP\|_{2\to\infty}^2}{|z|^3}\right\}
\end{align*}
This shows that $h_k(a_k)h_k(b_k) < 0$ w.h.p., and hence, there exists a unique $\widehat{t}_k\in [a_k, b_k]$ such that $h_k(\widehat{t}_k) = 0$. Following the above derivation, we also have
\begin{align*}
h_k(z)& = \frac{z - \lambda_k}{z} -\frac{\lambda_k\bu_k\transpose\expect \bE^2\bu_k}{z^3} + \frac{\lambda_k}{z} O\left\{\frac{(n\rho_n)^{3/2}}{|z|^3}\right\}
-\frac{\lambda_k}{z}\frac{\bu_k\transpose\bE\bu_k}{z}\\
&\quad + \frac{\lambda_k}{z}\widetilde{O}_{\prob}\left(\frac{\alpha_nn\rho_n}{|z|^3}\right)
 - \frac{\lambda_k}{z}\frac{\bu_k\transpose(\bE^2 - \expect\bE^2)\bu_k}{z^2}
 + \frac{\lambda_k}{z}\widetilde{O}_{\prob}\left\{\frac{(n\rho_n)^2}{|z|^4}\right\}\\
&\quad + \frac{\lambda_k}{z}\widetilde{O}_{\prob}\left\{\frac{\alpha_n^2}{|z|^2} + \frac{n^2\rho_n^{3/2}(\log n)^{3\xi}\|\bU_\bP\|_{2\to\infty}^2}{|z|^3} + \frac{(n\rho_n)^2}{|z|^4}\right\}\\
& = 1 - \frac{\lambda_k}{z} -\frac{\lambda_k\bu_k\transpose\bE\bu_k}{z^2}\\
&\quad +  \frac{\lambda_k}{z}\widetilde{O}_{\prob}\left\{\frac{\alpha_n^2}{|z|^2} + \frac{n\rho_n}{|z|^2} + \frac{n^2\rho_n^{3/2}(\log n)^{3\xi}\|\bU_\bP\|_{2\to\infty}^2}{|z|^3} + \frac{(n\rho_n)^2}{|z|^4}\right\}.
\end{align*}
Alternatively, for $f_k(z)$, by \eqref{eqn:R(Vk,uk,z)} and Result \ref{result:Noise_matrix_moment_bound}, we have
\begin{align*}
f_k(z)& = 1 + \lambda_k\calR(\bu_k, \bu_k, z) - \lambda_k\calR(\bu_k, \bV_k, z)\{\bS_k^{-1} + \calR(\bV_k, \bV_k, z)\}^{-1}\calR(\bV_k, \bu_k, z)\\
& = 1 - \frac{\lambda_k}{z}\sum_{l = 0,l\neq 1}^L\frac{\bu_k\transpose\expect\bE^l\bu_k}{z^l}\\
&\quad - \lambda_k\calR(\bu_k, \bV_k, z)\{\bS_k^{-1} + \calR(\bV_k, \bV_k, z)\}^{-1}\calR(\bV_k, \bu_k, z)\\
& = 1 - \frac{\lambda_k}{z} - \frac{\lambda_k}{z}\sum_{l = 2}^L\frac{\bu_k\transpose\expect\bE^l\bu_k}{z^l} + O\left\{\frac{(n\rho_n)^2\Delta_n}{|z|^5}\right\} = 1 - \frac{\lambda_k}{z} + O\left\{\frac{(n\rho_n)\Delta_n}{|z|^3}\right\}.
\end{align*}
Namely, for $z\in [a_k, b_k]$,
\begin{align*}
h_k(z) &= f_k(z) - \frac{\lambda_k}{z^2}\bu_k\transpose\bE\bu_k + \widetilde{O}_{\prob}\left(\frac{\alpha_n^2}{|z|^2} + \frac{n\rho_n}{|z|^2}\right) = f_k(z) + \widetilde{O}_{\prob}\left(\frac{\alpha_n}{|z|} + \frac{n\rho_n}{|z|^2}\right)
\end{align*}
by Result \ref{result:Noise_matrix_moment_bound}. Recall from the earlier part of the proof that $|f_k'(z)| = \Theta(|\lambda_k|/|z|^2)$ for any $z\in [a_k, b_k]$ by \eqref{f'kz}. It follows from the definitions of $\widehat{t}_k$, $t_k$, and mean-value theorem that
\begin{align*}
0& = h_k(\widehat{t}_k) = f_k(\widehat{t_k}) + \widetilde{O}_{\prob}\left(\frac{\alpha_n}{\Delta_n} + \frac{n\rho_n}{\Delta_n^2}\right)\\
& = f_k(t_k) + f_k'(\bar{t}_k)(\widehat{t}_k - t_k) + \widetilde{O}_{\prob}\left(\frac{\alpha_n}{\Delta_n} + \frac{n\rho_n}{\Delta_n^2}\right),
\end{align*}
where $\bar{t}_k$ is a random variable between $t_k$ and $\widehat{t}_k$. 
Since $f_k(t_k) = 0$, the above equation entails that
\begin{align*}
\widehat{t}_k - t_k = \frac{-1}{f_k'(\bar{t}_k)}\widetilde{O}_{\prob}\left(\frac{\alpha_n}{\Delta_n} + \frac{n\rho_n}{\Delta_n^2}\right) = \widetilde{O}_{\prob}\left(\alpha_n + \frac{n\rho_n}{\Delta_n}\right).
\end{align*}
The proof is thus completed.
\end{proof}

\subsection{Proof of Proposition \ref{prop:bilinear_form_expansion}}
\label{sub:proof_of_prop_bilinear_form_expansion}

By Lemma \ref{lemma:tkhat_concentration}, we know that 
\[
\widehat{t}_k = \lambda_k + \widetilde{O}_{\prob}\left(\alpha_n + \frac{n\rho_n}{\Delta_n}\right) = \lambda_k \{1 + \widetilde{o}_{\prob}(1)\}.
\]
By \eqref{eqn:bilinear_form_residual_formula}, we know that
\begin{align*}
\bu_k\transpose\widehat{\bu}_k\widehat{\bu}_k\transpose\bu_k
& = 
\frac{\widehat{t}_k^2\bu_k\transpose\{\bG(\widehat{t}_k) - \bF_k(\widehat{t}_k)\}\bu_k\bu_k\transpose\{\bG(\widehat{t}_k) - \bF_k(\widehat{t}_k)\}\bu_k}{\widehat{t}_k^2\bu_k\transpose\{\bG'(\widehat{t}_k) - \bF_k'(\widehat{t}_k)\}\bu_k},\\
\bu_m\transpose\widehat{\bu}_k\widehat{\bu}_k\transpose\bu_k
& = 
\frac{\widehat{t}_k^2\bu_m\transpose\{\bG(\widehat{t}_k) - \bF_k(\widehat{t}_k)\}\bu_k\bu_k\transpose\{\bG(\widehat{t}_k) - \bF_k(\widehat{t}_k)\}\bu_k}{\widehat{t}_k^2\bu_k\transpose\{\bG'(\widehat{t}_k) - \bF_k'(\widehat{t}_k)\}\bu_k}.
\end{align*}
Note that when $z\in[a_k,b_k]$ is of the same order as $\lambda_k$,
\begin{align*}
\frac{\mathrm{d}}{\mathrm{d}z}\{z^2\calR'(\bu_k, \bu_k, z)\}& = -\sum_{l = 2}^L\frac{(l + 1)l\bu_k\transpose\expect\bE^l\bu_k}{z^{l + 1}} = O\left(\frac{n\rho_n}{|z|^3}\right).
\end{align*}
 It follows from mean-value theorem that there exists some random variable $\bar{z}_k$ between $\lambda_k$ and $\widehat{t}_k$, such that
\begin{align*}
\widehat{t}_k^2\calR'(\bu_k, \bu_k, \widehat{t}_k)
& = \lambda_k^2\calR'(\bu_k, \bu_k, \lambda_k) + \frac{\mathrm{d}}{\mathrm{d}z}\{z^2\calR(\bu_k, \bu_k, z)\}\mathrel{\bigg|}_{z = \bar{z}_k}(\widehat{t}_k - \lambda_k)\\
& = \lambda_k^2\calR'(\bu_k, \bu_k, \lambda_k) + \widetilde{O}_{\prob}\left\{\frac{\alpha_nn\rho_n}{\Delta_n^3} + \frac{(n\rho_n)^2}{\Delta_n^4}\right\}.
\end{align*}
Now we apply Lemma \ref{lemma:uk_G'_minus_Fk'_uk} with $z = \widehat{t}_k$ and obtain
\begin{align*}
&\widehat{t}_k^2\bu_k\transpose\{\bG'(\widehat{t}_k) - \bF_k'(\widehat{t}_k)\}\bu_k\\
&\quad = 
\frac{2\bu_k\transpose\bE\bu_k}{\widehat{t}_k} + \frac{3\bu_k\transpose(\bE^2 - \expect\bE^2)\bu_k}{\widehat{t}_k^2} + \widehat{t}_k^2\calR'(\bu_k, \bu_k, \widehat{t}_k)
\\&\quad\quad
 + \widetilde{O}_{\prob}\left\{\frac{\alpha_n^2}{\Delta_n^2} + \frac{(n\rho_n)^2}{\Delta_n^4} + \frac{n^2\rho_n^{3/2}(\log n)^{3\xi}\|\bU_\bP\|_{2\to\infty}^2}{\Delta_n^3}\right\}\\
&\quad = \frac{2\bu_k\transpose\bE\bu_k}{\widehat{t}_k} + \frac{3\bu_k\transpose(\bE^2 - \expect\bE^2)\bu_k}{\widehat{t}_k^2} + \lambda_k^2\calR'(\bu_k, \bu_k, \lambda_k) + \widetilde{O}_{\prob}\left\{\frac{\alpha_n(n\rho_n)}{\Delta_n^3} + \frac{(n\rho_n)^2}{\Delta_n^4}\right\}
\\&\quad\quad
 + \widetilde{O}_{\prob}\left\{\frac{\alpha_n^2}{\Delta_n^2} + \frac{(n\rho_n)^2}{\Delta_n^4} + \frac{n^2\rho_n^{3/2}(\log n)^{3\xi}\|\bU_\bP\|_{2\to\infty}^2}{\Delta_n^3}\right\}\\
&\quad = \frac{2\bu_k\transpose\bE\bu_k}{\widehat{t}_k^3} + \frac{3\bu_k\transpose(\bE^2 - \expect\bE^2)\bu_k}{\widehat{t}_k^2} + \lambda_k^2\sum_{l = 0,l\neq 1}^L\frac{(l + 1)\bu_k\transpose\expect\bE^l\bu_k}{\lambda_k^{l + 2}}\\
&\quad\quad + \widetilde{O}_{\prob}\left\{\frac{\alpha_n^2}{\Delta_n^2} + \frac{(n\rho_n)^{3/2}}{q_n\Delta_n^3} + \frac{(n\rho_n)^2}{\Delta_n^4} + \frac{n^2\rho_n^{3/2}(\log n)^{3\xi}\|\bU_\bP\|_{2\to\infty}^2}{\Delta_n^3}\right\}\\
&\quad = \frac{2\bu_k\transpose\bE\bu_k}{\widehat{t}_k^3} + \frac{3\bu_k\transpose(\bE^2 - \expect\bE^2)\bu_k}{\widehat{t}_k^2} + 1 + \frac{3\bu_k\transpose\expect\bE^2\bu_k}{\lambda_k^2}\\
&\quad\quad + \widetilde{O}_{\prob}\left\{\frac{\alpha_n^2}{\Delta_n^2} + \frac{(n\rho_n)^{3/2}}{q_n\Delta_n^3} + \frac{(n\rho_n)^2}{\Delta_n^4} + \frac{n^2\rho_n^{3/2}(\log n)^{3\xi}\|\bU_\bP\|_{2\to\infty}^2}{\Delta_n^3}\right\}\\
&\quad = \frac{2\bu_k\transpose\bE\bu_k}{\lambda_k}\left(1 + \frac{\lambda_k - \widehat{t}_k}{\widehat{t}_k}\right) + \frac{3\bu_k\transpose(\bE^2 - \expect\bE^2)\bu_k}{\lambda_k^2}\left\{1 + \frac{(\lambda_k - \widehat{t}_k)(\widehat{t}_k + \lambda_k)}{\widehat{t_k}^2}\right\}\\
&\quad\quad + 1 + \frac{3\bu_k\transpose\expect\bE^2\bu_k}{\lambda_k^2} + \widetilde{O}_{\prob}\left\{\frac{\alpha_n^2}{\Delta_n^2} +\frac{(n\rho_n)^{3/2}}{q_n\Delta_n^3} + \frac{(n\rho_n)^2}{\Delta_n^4} + \frac{n^2\rho_n^{3/2}(\log n)^{3\xi}\|\bU_\bP\|_{2\to\infty}^2}{\Delta_n^3}\right\}\\
&\quad = 1 + \frac{2\bu_k\transpose\bE\bu_k}{\lambda_k} + \frac{3\bu_k\transpose\bE^2\bu_k}{\lambda_k^2}\\
&\quad\quad + \widetilde{O}_{\prob}\left\{\frac{\alpha_n^2}{\Delta_n^2} +  \frac{(n\rho_n)^{3/2}}{q_n\Delta_n^3} +  \frac{(n\rho_n)^2}{\Delta_n^4} + \frac{n^2\rho_n^{3/2}(\log n)^{3\xi}\|\bU_\bP\|_{2\to\infty}^2}{\Delta_n^3}\right\}
\end{align*}
by Result \ref{result:Noise_matrix_moment_bound} and Lemma \ref{lemma:Utranspose_E_cubic_U_concentration}; Note we consider the following decomposition in the fourth equality above,  
$$
\frac{(l + 1)\bu_k\transpose\expect\bE^l\bu_k}{\lambda_k^{l + 2}} = \frac{(l + 1)\bu_k\transpose\{\expect(\bE^l) - \bE^l\}\bu_k}{\lambda_k^{l + 2}} + \frac{(l + 1)\bu_k\transpose\bE^l\bu_k}{\lambda_k^{l + 2}}.
$$ In particular, we have 
\[
\widehat{t}_k^2\bu_k\transpose\{\bG'(\widehat{t}_k) - \bF_k'(\widehat{t}_k)\}\bu_k - 1 = \widetilde{O}_{\prob}\left(\frac{\alpha_n}{\Delta_n} + \frac{n\rho_n}{\Delta_n^2}\right)
\]
by Result \ref{result:Noise_matrix_concentration}. Using the formula
\[
\frac{1}{x} - \frac{1}{y} = -\frac{x - y}{y} + \frac{(x - y)^2}{xy^2},
\]
we further obtain
\begin{align*}
\frac{1}{\widehat{t}_k^2\bu_k\transpose\{\bG'(\widehat{t}_k) - \bF_k'(\widehat{t}_k)\}\bu_k}& = 1 - \frac{2\bu_k\transpose\bE\bu_k}{\lambda_k} - \frac{3\bu_k\transpose\bE^2\bu_k}{\lambda_k^2} + \widetilde{O}_{\prob}\left\{\frac{\alpha_n^2}{\Delta_n^2} + \frac{(n\rho_n)^2}{\Delta_n^4}\right\}.
\end{align*}
Applying Lemma \ref{lemma:uk_G_minus_Fk_uk} with $z = \widehat{t}_k$ and $\bu = \bu_k$, and invoking \eqref{eqn:A(uk,k,z)} and \eqref{eqn:b(uk,k,z)}, we also obtain
\begin{align*}
&\widehat{t}_k\bu_k\transpose\{\bG(\widehat{t}_k) - \bF_k(\widehat{t}_k)\}\bu_k\\
&\quad = \calA(\bu_k, k, \widehat{t}_k) - \frac{\bb(\bu_k, k, \widehat{t}_k)\transpose\bE\bu_k}{\widehat{t}_k} - \frac{\bb(\bu_k, k, \widehat{t}_k)\transpose(\bE^2 - \expect\bE^2)\bu_k}{\widehat{t}_k^2}\\
&\quad\quad  + \widetilde{O}_{\prob}\left\{\frac{\alpha_n^2}{|\widehat{t}_k|^2} + \frac{n^2\rho_n^{3/2}(\log n)^{3\xi}\|\bU_\bP\|_{2\to\infty}^2}{|\widehat{t}_k|^3} + \frac{(n\rho_n)^2}{|\widehat{t}_k|^4}\right\}\\
&\quad =  - 1 -\frac{\bu_k\transpose\expect \bE^2\bu_k}{\widehat{t}_k^2} + O\left\{\frac{(n\rho_n)^{3/2}}{q_n|\widehat{t}_k|^3} + \frac{(n\rho_n)^2}{|\widehat{t}_k|^4}\right\} - \left\{\bu_k - \bV_k\times O\left(\frac{n\rho_n}{|\widehat{t}_k|^2}\right)\right\}\transpose\frac{\bE\bu_k}{\widehat{t}_k}\\
&\quad\quad - \left\{\bu_k - \bV_k\times O\left(\frac{n\rho_n}{|\widehat{t}_k|^2}\right)\right\}\transpose\frac{(\bE^2 - \expect\bE^2)\bu_k}{\widehat{t}_k^2}\\
&\quad\quad  + \widetilde{O}_{\prob}\left\{\frac{\alpha_n^2}{|\widehat{t}_k|^2} + \frac{n^2\rho_n^{3/2}(\log n)^{3\xi}\|\bU_\bP\|_{2\to\infty}^2}{|\widehat{t}_k|^3} + \frac{(n\rho_n)^{3/2}}{|\widehat{t}_k|^3}\right\}\\
&\quad = - 1 - \frac{\bu_k\transpose\bE\bu_k}{\widehat{t}_k} - \frac{\bu_k\transpose\bE^2\bu_k}{\widehat{t}_k^2} + \widetilde{O}_{\prob}\left\{\frac{\alpha_n^2}{|\widehat{t}_k|^2} + \frac{n^2\rho_n^{3/2}(\log n)^{3\xi}\|\bU_\bP\|_{2\to\infty}^2}{|\widehat{t}_k|^3} + \frac{(n\rho_n)^{3/2}}{|\widehat{t}_k|^3}\right\}\\
&\quad = - 1 - \frac{\bu_k\transpose\bE\bu_k}{\lambda_k}\left(1 + \frac{\lambda_k - \widehat{t}_k}{\widehat{t}_k}\right) - \frac{\bu_k\transpose\bE^2\bu_k}{\lambda_k^2}\left\{1 + \frac{(\lambda_k - \widehat{t}_k)(\lambda_k + \widehat{t}_k)}{\widehat{t}_k^2}\right\}
\\&\quad\quad
+ \widetilde{O}_{\prob}\left\{\frac{\alpha_n^2}{|\widehat{t}_k|^2} + \frac{n^2\rho_n^{3/2}(\log n)^{3\xi}\|\bU_\bP\|_{2\to\infty}^2}{|\widehat{t}_k|^3} + \frac{(n\rho_n)^{3/2}}{|\widehat{t}_k|^3}\right\}\\
&\quad = - 1 - \frac{\bu_k\transpose\bE\bu_k}{\lambda_k} - \frac{\bu_k\transpose\bE^2\bu_k}{\lambda_k^2} + \widetilde{O}_{\prob}\left\{\frac{\alpha_n^2}{\Delta_n^2} + \frac{n^2\rho_n^{3/2}(\log n)^{3\xi}\|\bU_\bP\|_{2\to\infty}^2}{\Delta_n^3} + \frac{(n\rho_n)^{3/2}}{\Delta_n^3}\right\}
\end{align*}
by \eqref{eqn:A(uk,k,z)} and \eqref{eqn:b(uk,k,z)}. 
We therefore obtain from \eqref{eqn:bilinear_form_residual_formula} that
\begin{align*}
\bu_k\transpose\widehat{\bu}_k\widehat{\bu}_k\transpose\bu_k
& = \frac{\widehat{t}_k^2\bu_k\transpose\{\bG(\widehat{t}_k) - \bF_k(\widehat{t}_k)\}\bu_k\bu_k\transpose\{\bG(\widehat{t}_k) - \bF_k(\widehat{t}_k)\}\bu_k}{\widehat{t}_k^2\bu_k\transpose\{\bG'(\widehat{t}_k) - \bF_k'(\widehat{t}_k)\}\bu_k}\\
& = \left(1 - \frac{2\bu_k\transpose\bE\bu_k}{\lambda_k} - \frac{3\bu_k\transpose\bE^2\bu_k}{\lambda_k^2} + \mbox{rem}\right)
\left(1 + \frac{2\bu_k\transpose\bE\bu_k}{\lambda_k} + \frac{2\bu_k\transpose\bE^2\bu_k}{\lambda_k^2} + \mbox{rem}\right)\\
& = 1 - \frac{\bu_k\transpose\bE^2\bu_k}{\lambda_k^2} + \mbox{rem},
\end{align*}
where
\[
\mbox{rem} = \widetilde{O}_{\prob}\left\{\frac{\alpha_n^2}{\Delta_n^2} + \frac{n^2\rho_n^{3/2}(\log n)^{3\xi}\|\bU_\bP\|_{2\to\infty}^2}{\Delta_n^3} + \frac{(n\rho_n)^{3/2}}{\Delta_n^3}\right\}.
\]
This completes the proof of the first assertion.
\par
The case with $\bu = \bu_m$, $m\in [d]$, $m\neq k$ is slightly different. Observe that by Result \ref{result:Noise_matrix_moment_bound} and Lemma \ref{lemma:Utranspose_E_cubic_U_concentration}, uniformly over $z\in [a_k, b_k]$, we have
\begin{align*}
\calA(\bu_m, k, z)& = z\calR(\bu_m, \bu_k, z) - z\calR(\bu_m, \bV_k, z)\{\bS_k^{-1} + \calR(\bV_k, \bV_k, z)\}^{-1}\calR(\bV_k, \bu_k, z)\\
& = -\sum_{l = 2}^L\frac{\bu_m\transpose\expect \bE^l\bu_k}{z^l} - \left(\bu_m\transpose\bV_k + \sum_{l = 2}^L\frac{\bu_m\transpose\expect \bE^l\bV_k}{z^l}\right)\\
&\quad\times \mathrm{diag}\left(\frac{\lambda_rz}{z - \lambda_r}:r\neq k\right)\left\{\eye_{d - 1} + O\left(\frac{n\rho_n}{|z|^2}\right)\right\}\left(\sum_{l = 2}^L\frac{\bV_k\transpose\expect \bE^l\bu_k}{z^{l + 1}}\right)\\
& = -\frac{\bu_m\transpose\expect \bE^2\bu_k}{z^2} + O\left\{\frac{(n\rho_n)^{3/2}}{q_n|z|^3}\right\} + O\left\{\frac{\sqrt{n}(n\rho_n)^2\|\bU_\bP\|_{2\to\infty}}{|z|^4}\right\}\\
&\quad - \left\{\bu_m\transpose\bV_k + O\left(\frac{n\rho_n}{|z|^2}\right)\right\} \mathrm{diag}\left(\frac{\lambda_rz}{z - \lambda_r}:r\neq k\right)
\\
&\quad\times\left[\frac{\bV_k\transpose\expect \bE^2\bu_k}{z^3} + O\left\{\frac{(n\rho_n)^{3/2}}{q_n|z|^4}\right\} + O\left(\frac{n^{5/2}\rho_n^{2}}{|z|^5}\|\bU_\bP\|_{2\to\infty}\right)\right]\\
&\quad + O\left\{\frac{(n\rho_n)^2}{|z|^4}\right\}\\
& = -\frac{\bu_m\transpose\expect \bE^2\bu_k}{z^2} - \bu_m\transpose\bV_k\mathrm{diag}\left(\frac{\lambda_r}{z - \lambda_r}:r\neq k\right)
\left(\frac{\bV_k\transpose\expect \bE^2\bu_k}{z^2}\right)\\
&\quad + O\left\{\frac{(n\rho_n)^{3/2}}{q_n|z|^3}\right\} + O\left\{\frac{(n\rho_n)^2}{|z|^4}\max(1, \sqrt{n}\|\bU_\bP\|_{2\to\infty})\right\}\\
& = -\frac{\bu_m\transpose\expect \bE^2\bu_k}{z^2} - \frac{1}{z^2}\frac{\lambda_m}{z - \lambda_m}\bu_m\transpose\expect\bE^2\bu_k\\
&\quad + O\left\{\frac{(n\rho_n)^{3/2}}{|z|^3}\max\left(\frac{1}{q_n}, \frac{\sqrt{n\rho_n}}{|z|}, \frac{n\rho_n^{1/2}\|\bU_\bP\|_{2\to\infty}}{|z|}\right)\right\}\\
& = -\frac{\bu_m\transpose\expect \bE^2\bu_k}{z(z - \lambda_m)} + O\left\{\frac{(n\rho_n)^{3/2}}{|z|^3}\max\left(\frac{1}{q_n}, \frac{\sqrt{n\rho_n}}{|z|}, \frac{n\rho_n^{1/2}\|\bU_\bP\|_{2\to\infty}}{|z|}\right)\right\}
\end{align*}
and
\begin{align*}
&\bb(\bu_m, k, z)\\
&\quad = \bu_m - \bV_k\{\bS_k^{-1} + \calR(\bV_k, \bV_k, z)\}^{-1}\calR(\bV_k, \bu_m, z)\\
&\quad = \bu_m - \bV_k\mathrm{diag}\left(\frac{\lambda_rz}{z - \lambda_r}:r\neq k\right)\left\{\eye_{d - 1} + O\left(\frac{n\rho_n}{|z|^2}\right)\right\}\left(\frac{\bV_k\transpose\bu_m}{z} + \sum_{l = 2}^L\frac{\bV_k\transpose\expect \bE^l\bu_m}{z^{l + 1}}\right)\\
&\quad = \bu_m - \bV_k\mathrm{diag}\left(\frac{\lambda_rz}{z - \lambda_r}:r\neq k\right)\left(\frac{\bV_k\transpose\bu_m}{z} + \sum_{l = 2}^L\frac{\bV_k\transpose\expect \bE^l\bu_m}{z^{l + 1}}\right) + \bV_k\times O\left(\frac{n\rho_n}{|z|^2}\right)\\
&\quad = \bu_m - \bV_k\mathrm{diag}\left(\frac{\lambda_r}{z - \lambda_r}:r\neq k\right)\left\{\bV_k\transpose\bu_m + O\left(\frac{n\rho_n}{|z|^2}\right)\right\} +\bV_k\times O\left(\frac{n\rho_n}{|z|^2}\right)\\
&\quad = \bu_m -  \bV_k\mathrm{diag}\left(\frac{\lambda_r}{z - \lambda_r}:r\neq k\right)\bV_k\transpose\bu_m + \bV_k\times O\left(\frac{n\rho_n}{|z|^2}\right)\\
&\quad = \frac{z}{z - \lambda_m}\bu_m + \bV_k\times O\left(\frac{n\rho_n}{|z|^2}\right).
\end{align*}
We thus obtain
\begin{align*}
&\widehat{t}_k\bu_m\transpose\{\bG(\widehat{t}_k) - \bF_k(\widehat{t}_k)\}\bu_k\\
&\quad = - \frac{\bu_m\transpose\expect\bE^2\bu_k}{\widehat{t}_k(\widehat{t}_k - \lambda_m)} - \frac{1}{\widehat{t}_k}\left\{\frac{\widehat{t}_k}{\widehat{t}_k - \lambda_m}\bu_m + \bV_k\times \widetilde{O}_{\prob}\left(\frac{n\rho_n}{|\widehat{t}_k|^2}\right)\right\}\transpose\bE\bu_k\\
&\quad\quad + O\left\{\frac{(n\rho_n)^{3/2}}{|\widehat{t}_k|^3}\max\left(\frac{1}{q_n}, \frac{\sqrt{n\rho_n}}{|\widehat{t}_k|}, \frac{n\rho_n^{1/2}\|\bU_\bP\|_{2\to\infty}}{|\widehat{t}_k|}\right)\right\}\\
&\quad\quad - \frac{1}{\widehat{t}_k^2}\left\{\frac{\widehat{t}_k}{\widehat{t}_k - \lambda_m}\bu_m + \bV_k\times O\left(\frac{n\rho_n}{|\widehat{t}_k|^2}\right)\right\}\transpose(\bE^2 - \expect\bE^2)\bu_k\\
&\quad\quad + \widetilde{O}_{\prob}\left\{\frac{\alpha_n^2}{|\widehat{t}_k|^2} + \frac{n^2\rho_n^{3/2}(\log n)^{3\xi}\|\bU_\bP\|_{2\to\infty}^2}{|\widehat{t}_k|^3} + \frac{(n\rho_n)^2}{|\widehat{t}_k|^4}\right\}\\
&\quad = - \frac{\bu_m\transpose\bE\bu_k}{\widehat{t}_k - \lambda_m} - \frac{\bu_m\transpose\bE^2\bu_k}{\widehat{t}_k(\widehat{t}_k - \lambda_m)} +  \widetilde{O}_{\prob}\left\{\frac{\alpha_n^2}{|\widehat{t}_k|^2} + \frac{n^2\rho_n^{3/2}(\log n)^{3\xi}\|\bU_\bP\|_{2\to\infty}^2}{|\widehat{t}_k|^3} + \frac{(n\rho_n)^{3/2}}{|\widehat{t}_k|^3}\right\}\\
&\quad = - \frac{\bu_m\transpose\bE\bu_k}{\lambda_k - \lambda_m}\left(1 + \frac{\lambda_k - \widehat{t}_k}{\widehat{t}_k - \lambda_m}\right) - \frac{\bu_m\transpose\bE^2\bu_k}{\lambda_k(\lambda_k - \lambda_m)}\left\{1 + \frac{(\lambda_k - \widehat{t}_k)(\lambda_k + \widehat{t}_k - \lambda_m)}{\widehat{t}_k(\widehat{t}_k - \lambda_m)}\right\}\\
&\quad\quad + \widetilde{O}_{\prob}\left\{\frac{\alpha_n^2}{|\widehat{t}_k|^2} + \frac{n^2\rho_n^{3/2}(\log n)^{3\xi}\|\bU_\bP\|_{2\to\infty}^2}{|\widehat{t}_k|^3} + \frac{(n\rho_n)^{3/2}}{|\widehat{t}_k|^3}\right\}\\
&\quad = - \frac{\bu_m\transpose\bE\bu_k}{\lambda_k - \lambda_m} - \frac{\bu_m\transpose\bE^2\bu_k}{\lambda_k(\lambda_k - \lambda_m)}+  \widetilde{O}_{\prob}\left\{\frac{\alpha_n^2}{\Delta_n^2} + \frac{n^2\rho_n^{3/2}(\log n)^{3\xi}\|\bU_\bP\|_{2\to\infty}^2}{\Delta_n^3} + \frac{(n\rho_n)^{3/2}}{\Delta_n^3}\right\}.
\end{align*}
Thus, we finally have
\begin{align*}
\bu_m\transpose\widehat{\bu}_k\widehat{\bu}_k\transpose\bu_k
& = \frac{\widehat{t}_k^2\bu_m\transpose\{\bG(\widehat{t}_k) - \bF_k(\widehat{t}_k)\}\bu_k\bu_k\transpose\{\bG(\widehat{t}_k) - \bF_k(\widehat{t}_k)\}\bu_k}{\widehat{t}_k^2\bu_k\transpose\{\bG'(\widehat{t}_k) - \bF_k'(\widehat{t}_k)\}\bu_k}\\
& = \left(1 - \frac{2\bu_k\transpose\bE\bu_k}{\lambda_k} - \frac{3\bu_k\transpose\bE^2\bu_k}{\lambda_k^2} + \mbox{rem}\right)
\left(-\frac{\bu_m\transpose\bE\bu_k}{\lambda_k - \lambda_m} - \frac{\bu_m\transpose\bE^2\bu_k}{\lambda_k(\lambda_k - \lambda_m)} + \mbox{rem}\right)\\
&\quad\times \left( - 1 - \frac{\bu_k\transpose\bE\bu_k}{\lambda_k} - \frac{\bu_k\transpose\bE^2\bu_k}{\lambda_k^2} +\mbox{rem}\right)\\
& = \left(1 - \frac{2\bu_k\transpose\bE\bu_k}{\lambda_k} - \frac{3\bu_k\transpose\bE^2\bu_k}{\lambda_k^2} + \mbox{rem}\right)
\left(\frac{\bu_m\transpose\bE\bu_k}{\lambda_k - \lambda_m} + \frac{\bu_m\transpose\bE^2\bu_k}{\lambda_k(\lambda_k - \lambda_m)} + \mbox{rem}\right)
\\&
= \frac{\bu_m\transpose\bE\bu_k}{\lambda_k - \lambda_m} + \frac{\bu_m\transpose\bE^2\bu_k}{\lambda_k(\lambda_k - \lambda_m)} + \mbox{rem}.
\end{align*}
where
\[
\mbox{rem} = \widetilde{O}_{\prob}\left\{\frac{\alpha_n^2}{\Delta_n^2} + \frac{n^2\rho_n^{3/2}(\log n)^{3\xi}\|\bU_\bP\|_{2\to\infty}^2}{\Delta_n^3} + \frac{(n\rho_n)^{3/2}}{\Delta_n^3}\right\}.
\]
The proof is thus completed.

\section{Proofs for Section \ref{sec:eigenvector_expansion}}
\label{sec:proof_of_eigenvector_expansion}
In this section, we complete the proofs for Section \ref{sec:eigenvector_expansion}. We  list several technical lemmas that will be frequently used. Recall our proof architecture introduced in Section~\ref{sub:proof_sketch}. The key strategy of our proof is to expand $\bR_1$ through $\bR_4$ of $\bR_{\bU_+}$ in \eqref{def:RU+}, into the first, second and higher-order terms. Specifically, Lemma~\ref{lemma:Second_order_expansion_I} expands $\bR_2$. Lemma~\ref{lemma:Remainder_analysis_II} provides a high-probability bound for $\|\bE\bR_{\bU_+}\|_{2\rightarrow\infty}$, which will be used to bound several quantities including showing the negligibility of $\bR_4$. Lemma~\ref{lemma:Second_order_expansion_II} expands $\bR_1$ to the second and higher-order terms. The remaining term of $\bR_{\bU_+}$, i.e., $\bR_3$, is expanded by Proposition~\ref{prop:eigenvector_angle_expansion} as discussed in the manuscript.
 

\begin{lemma}\label{lemma:Second_order_expansion_I}
Suppose Assumptions \ref{assumption:Signal_strength}--\ref{assumption:Noise_matrix_distribution} in the main paper hold. Then
\begin{align*}
\bR_2
& = -\bU_{\bP_+}\bU_{\bP_+}\transpose\bE\bU_{\bP_+}\bS_{\bP_+}^{-1}\bW^*_+
 - \bU_{\bP_+}\bU_{\bP_+}\transpose\bE^2\bU_{\bP_+}\bS_{\bP_+}^{-2}\bW^*_+
 + \bT^{(1)},
\end{align*}
where $\bW^*_+ = \mathrm{diag}\{\mathrm{sgn}(\bu_1\transpose\widehat{\bu}_1),\ldots,\mathrm{sgn}(\bu_p\transpose\widehat{\bu}_p)\}$ and 
\[
\|\bT^{(1)}\|_{2\to\infty} = \widetilde{O}_{\prob}\left[\|\bU_\bP\|_{2\to\infty}\left\{\frac{(n\rho_n)^{3/2}}{\Delta_n^3} + \frac{\alpha_n^2}{\Delta_n^2}\right\}\right].
\] 
\end{lemma}
\begin{proof}
Write
\begin{align*}
\bR_2
& = \bU_{\bP_+}\bS_{\bP_+}(\bU_{\bP_+}\transpose\bU_{\bA_+}\bS_{\bA_+}^{-1} - \bS_{\bP_+}^{-1}\bU_{\bP_+}\transpose\bU_{\bA_+})\\
& = \bU_{\bP_+}(\bS_{\bP_+}\bU_{\bP_+}\transpose\bU_{\bA_+} - \bU_{\bP_+}\transpose\bU_{\bA_+}\bS_{\bA_+})\bS_{\bA_+}^{-1}\\
& = \bU_{\bP_+}(\bS_{\bP_+}\bU_{\bP_+}\transpose\bU_{\bA_+} - \bU_{\bP_+}\transpose\bU_{\bA_+}\bS_{\bA_+})(\bS_{\bA_+}^{-1} - \bS_{\bP_+}^{-1})
\\&\quad
 + \bU_{\bP_+}(\bS_{\bP_+}\bU_{\bP_+}\transpose\bU_{\bA_+} - \bU_{\bP_+}\transpose\bU_{\bA_+}\bS_{\bA_+})\bS_{\bP_+}^{-1}.
\end{align*}
By Result \ref{result:SU_exchange_concentration} and Lemma \ref{lemma:Remainder_analysis_I},
the first term in the decomposition has the following bound:
\begin{align*}
&\bU_{\bP_+}(\bS_{\bP_+}\bU_{\bP_+}\transpose\bU_{\bA_+} - \bU_{\bP_+}\transpose\bU_{\bA_+}\bS_{\bA_+})(\bS_{\bA_+}^{-1} - \bS_{\bP_+}^{-1})\\
&\quad = \bU_{\bP_+}\times O\left(
\|
\bS_{\bP_+}\bU_{\bP_+}\transpose\bU_{\bA_+} - \bU_{\bP_+}\transpose\bU_{\bA_+}\bS_{\bA_+}\|_2\right)\times O\left(\|\bW^*_+\bS_{\bA_+}^{-1} - \bS_{\bP_+}^{-1}\bW^*_+
\|_2\right)\\
&\quad = \bU_{\bP_+}\times \widetilde{O}_{\prob}\left\{\frac{\alpha_n^2}{\Delta_n^2} + \frac{(n\rho_n)^2}{\Delta_n^4}\right\}.
\end{align*}
We next focus on the second term. Following the derivation in Result \ref{result:SU_exchange_concentration}, we have
\begin{align*}
&\bU_{\bP_+}(\bS_{\bP_+}\bU_{\bP_+}\transpose\bU_{\bA_+} - \bU_{\bP_+}\transpose\bU_{\bA_+}\bS_{\bA_+})\bS_{\bP_+}^{-1}\\
&\quad = -\bU_{\bP_+}\{\bU_{\bP_+}\transpose\bE\bU_{\bP_+}\bW^*_+ + \bU_{\bP_+}\transpose\bE(\bU_{\bA_+} - \bU_{\bP_+}\bW^*_+)\}\bS_{\bP_+}^{-1}\\
&\quad = -\bU_{\bP_+}\bU_{\bP_+}\transpose\bE\bU_{\bP_+}\bS_{\bP_+}^{-1}\bW^*_+\\
&\quad\quad - \bU_{\bP_+}\bU_{\bP_+}\transpose\bE(\bE\bU_{\bP_+}\bS_{\bP_+}^{-1}\bW^*_+ + \bR_{\bU_+} + \bU_{\bP_-}\bS_{\bP_-}\bU_{\bP_-}\transpose\bU_{\bA_+}\bS_{\bA_+}^{-1})\bS_{\bP_+}^{-1}\\
&\quad = -\bU_{\bP_+}\bU_{\bP_+}\transpose\bE\bU_{\bP_+}\bS_{\bP_+}^{-1}\bW^*_+ - \bU_{\bP_+}\bU_{\bP_+}\transpose\bE^2\bU_{\bP_+}\bS_{\bP_+}^{-2}\bW^*_+ - \bU_{\bP_+}\bU_{\bP_+}\transpose\bE\bR_{\bU_+}\bS_{\bP_+}^{-1}\\
&\quad\quad - \bU_{\bP_+}\bU_{\bP_+}\transpose\bE\bU_{\bP_-}\bS_{\bP_-}\bU_{\bP_-}\transpose\bU_{\bA_+}\bS_{\bA_+}^{-1}\bS_{\bP_+}^{-1}.
\end{align*}
By Result \ref{result:Noise_matrix_concentration} and Lemma \ref{lemma:Eigenvector_angle_analysis}, the last term is
\begin{align*}
&\bU_{\bP_+}\bU_{\bP_+}\transpose\bE\bU_{\bP_-}\bS_{\bP_-}\bU_{\bP_-}\transpose\bU_{\bA_+}\bS_{\bA_+}^{-1}\bS_{\bP_+}^{-1}\\
&\quad = 
\bU_{\bP_+}\times O(\|\bU_{\bP_+}\transpose\bE\bU_{\bP_-}\|_2)\times O(\|\bS_{\bP_-}\|_2\|\bS_\bA^{-1}\|_2\|\bS_\bP^{-1}\|_2)\times O(\|\bU_{\bP_-}\transpose\bU_{\bA_+}\|_2)\\
&\quad = \bU_{\bP_+}\times \widetilde{O}_{\prob}\left\{\frac{\alpha_n^2}{\Delta_n^2} + \frac{(n\rho_n)^2}{\Delta_n^4}\right\}.
\end{align*}
It is sufficient to show that $\bU_{\bP_+}\transpose\bE\bR_{\bU_+}\bS_{\bP_+}^{-1} = \widetilde{O}_{\prob}\{\alpha_n^2/\Delta_n^2 + (n\rho_n)^{3/2}/\Delta_n^3\}$. We apply the formula in \eqref{eqn:first_order_eigenvector_expansion} for $\bR_{\bU_+}$ and write
\begin{align*}
\bU_{\bP_+}\transpose\bE\bR_{\bU_+}\bS_{\bP_+}^{-1}
& = \bU_{\bP_+}\transpose\bE\{\bU_{\bP_+}(\bS_{\bP_+}\bU_{\bP_+}\transpose\bU_{\bA_+} - \bU_{\bP_+}\transpose\bU_{\bA_+}\bS_{\bA_+})\bS_{\bA_+}^{-1}
\}\bS_{\bP_+}^{-1}\\
&\quad + \bU_{\bP_+}\transpose\bE\bU_{\bP_+}(\bU_{\bP_+}\transpose\bU_{\bA_+} - \bW^*_+)\bS_{\bP_+}^{-1}\\
&\quad + \bU_{\bP_+}\transpose\bE^2\bU_{\bP_+}(\bW^*_+\bS_{\bA_+}^{-1} - \bS_{\bP_+}^{-1}\bW^*_+)\bS_{\bP_+}^{-1}\\
&\quad + \bU_{\bP_+}\transpose\bE^2(\bU_{\bA_+} - \bU_{\bP_+}\bW^*_+)\bS_{\bA_+}^{-1}\bS_\bP^{-1}
\end{align*}
Observe that
\begin{align*}
&\|\bU_{\bP_+}\transpose\bE\bU_{\bP_+}(\bS_{\bP_+}\bU_{\bP_+}\transpose\bU_{\bA_+} - \bU_{\bP_+}\transpose\bU_{\bA_+}\bS_{\bA_+})\bS_{\bA_+}^{-1}\bS_{\bP_+}^{-1}\|_2\\
&\quad \leq \|\bU_{\bP_+}\transpose\bE\bU_{\bP_+}\|_2\|\bS_{\bP_+}\bU_{\bP_+}\transpose\bU_{\bA_+} - \bU_{\bP_+}\transpose\bU_{\bA_+}\bS_{\bA_+}\|_2\|\bS_{\bA_+}^{-1}\|_2\|\bS_{\bP_+}^{-1}\|_2
\\&\quad
 = \widetilde{O}_{\prob}\left\{\frac{\alpha_n^2}{\Delta_n^2} + \frac{(n\rho_n)^2}{\Delta_n^4}\right\}
\end{align*}
by Result \ref{result:Noise_matrix_concentration} and Result \ref{result:SU_exchange_concentration},
\begin{align*}
&\|\bU_{\bP_+}\transpose\bE^2\bU_{\bP_+}(\bW^*_+\bS_{\bA_+}^{-1} - \bS_{\bP_+}^{-1}\bW^*_+)\bS_{\bP_+}^{-1}\|_2\\
&\quad\leq\|\bU_{\bP_+}\transpose\bE^2\bU_{\bP_+}\|_2\|\bW^*_+\bS_{\bA_+}^{-1} - \bS_{\bP_+}^{-1}\bW^*_+\|_2\|\bS_{\bP_+}^{-1}\|_2
 = \widetilde{O}_{\prob}\left\{\frac{\alpha_n^2}{\Delta_n^2} + \frac{(n\rho_n)^2}{\Delta_n^4}\right\}
\end{align*}
by Lemma \ref{lemma:Remainder_analysis_I} and Result \ref{result:Noise_matrix_moment_bound}, 
\begin{align*}
\|\bU_{\bP_+}\transpose\bE\bU_{\bP_+}(\bU_{\bP_+}\transpose\bU_{\bA_+} - \bW^*_+)\bS_{\bP_+}^{-1}\|_2
&\leq \|\bU_{\bP_+}\transpose\bE\bU_{\bP_+}\|_2\|\bU_{\bP_+}\transpose\bU_{\bA_+} - \bW^*_+\|_2\|\bS_{\bP_+}^{-1}\|_2
\\&
 = \widetilde{O}_{\prob}\left\{\frac{\alpha_n^2}{\Delta_n^2} + \frac{(n\rho_n)^2}{\Delta_n^4}\right\}
\end{align*}
by Result \ref{result:Noise_matrix_concentration} and Lemma \ref{lemma:Eigenvector_angle_analysis}, and
\begin{align*}
\|\bU_{\bP_+}\transpose\bE^2(\bU_{\bA_+} - \bU_{\bP_+}\bW^*_+)\bS_{\bA_+}^{-1}\bS_{\bP_+}^{-1}\|_2
&\leq \|\bE\|_2^2\|\bU_{\bA_+} - \bU_{\bP_+}\bW^*_+\|_2\|\bS_{\bA_+}^{-1}\|_2\|\bS_{\bP_+}^{-1}\|_2
\\&
 = \widetilde{O}_{\prob}\left\{\frac{(n\rho_n)^{3/2}}{\Delta_n^3}\right\}
\end{align*}
by Result \ref{result:Noise_matrix_concentration} and Lemma \ref{lemma:Eigenvector_zero_order_deviation}.
The proof is thus completed.
\end{proof}

\begin{lemma}\label{lemma:Remainder_analysis_II}
Suppose Assumptions \ref{assumption:Signal_strength}--\ref{assumption:Noise_matrix_distribution} in the manuscript hold. Then for any $\xi > 1$,
\begin{align*}
&\|\bE\bR_{\bU_+}\|_{2\to\infty}\\
&\quad = \widetilde{O}_{\prob}\left[\|\bU_\bP\|_{2\to\infty}\left\{\frac{(n\rho_n)^{3/2}(\log n)^{3\xi}}{\Delta_n^2} + \frac{(n\rho_n)^{1/2}(\log n)^\xi}{\Delta_n}\left(\alpha_n + \frac{n\rho_n}{\Delta_n}\right)\right\}\right].
\end{align*}
\end{lemma}

\begin{proof}
Recall the definition of $\mathbf{R}_{\mathbf{U}_+}$ in \eqref{def:RU+}. We first introduce another representation of the remainder term $\mathbf{R}_{\mathbf{U}_+}$ based on the derivation in \cite{cape2019signal}:
\begin{align}
\label{eqn:eigenvector_first_order_remainder}
\begin{aligned}
&\bR_{\bU_+} = \bR_2  + \bR_3 + \bR_2^{(1)} + \bR_{2,\bW}^{(2)}   + \bE\bU_{\bP_-}\bS_{\bP_-}\bU_{\bP_-}\transpose \bU_{\bA_+}\bS_{\bA_+}^{-2}+ \bR_2^{(\infty)},\\
&\bR_2^{(1)} = \bE\bU_{\bP_+}\bS_{\bP_+}(\bU_{\bP_+}\transpose\bU_{\bA_+}\bS_{\bA_+}^{-2} - \bS_{\bP_+}^{-2}\bU_{\bP_+}\transpose\bU_{\bA_+}),\\
&{\color{black}\bR_{2,\bW}^{(2)} = \bE\bU_{\bP_+}\bS_{\bP_+}^{-1}(\bU_{\bP_+}\transpose\bU_{\bA_+} - \bW_+^*)} ,\\
&\bR_2^{(\infty)} = \sum_{m = 2}^\infty\bE^m\bP\bU_{\bA_+}\bS_{\bA_+}^{-(m + 1)}.
\end{aligned}
\end{align}
To be more specific, following a similar argument as \cite[$\mathsection$3.5]{cape2019signal}, we have the following facts:
\bea\label{ftsfrombmk}
\bU_{\bA_+} &= \sum_{m = 0 }^\infty \bE^m \bP \bU_{\bA_+}\bS_{\bA_+}^{-(m + 1)},
\\
\bE\bU_{\bA_+}\bS_{\bA_+}^{-1} &= \bE \bU_{\bP_+}\bS_{\bP_+}\bU_{\bP_+}\transpose \bU_{\bA_+}\bS_{\bA_+}^{-2}  +\bE \bU_{\bP_-}\bS_{\bP_-}\bU_{\bP_-}\transpose \bU_{\bA_+}\bS_{\bA_+}^{-2} 
\\
&\quad + \sum_{m = 2 }^\infty \bE^m \bP \bU_{\bA_+}\bS_{\bA_+}^{-(m + 1)}
\\
& = \bE \bU_{\bP_-}\bS_{\bP_-}\bU_{\bP_-}\transpose \bU_{\bA_+}\bS_{\bA_+}^{-2} + \bR_2^{(1)} +\bR_{2,\bW}^{(2)} + \bR_2^{(\infty)}  +  \bE\bU_{\bP_+}\bS_{\bP_+}^{-1}\bW_+^*.
\eae
Then by definitions in \eqref{def:RU+} and \eqref{ftsfrombmk}, we have
\bea\nonumber
\bR_{\bU_+} &= \bR_1 + \bR_2 + \bR_3 + \bR_4 = \bR_2 + \bR_3 +  \bE\bU_{\bA_+}\bS_{\bA_+}^{-1} - \bE\bU_{\bP_+}\bS_{\bP_+}^{-1}\bW_+^*
\\
& =  \bR_2 + \bR_3 + \bE \bU_{\bP_-}\bS_{\bP_-}\bU_{\bP_-}\transpose \bU_{\bA_+}\bS_{\bA_+}^{-2} + \bR_2^{(1)} +\bR_{2,\bW}^{(2)} + \bR_2^{(\infty)}  ,
\eae
and thus \eqref{eqn:eigenvector_first_order_remainder} holds.
\par
Applying the alternative form of $\bR_{\bU_+}$ in \eqref{eqn:eigenvector_first_order_remainder} with $\bW = \bW^*_+$, we have
\begin{align*}
\bE\bR_{\bU_+}
& = \bE\bU_{\bP_+}(\bS_{\bP_+}\bU_{\bP_+}\transpose\bU_{\bA_+} - \bU_{\bP_+}\transpose\bU_{\bA_+}\bS_{\bA_+})\bS_{\bA_+}^{-1}
 + \bE\bU_{\bP_+}(\bU_{\bP_+}\transpose\bU_{\bA_+} - \bW^*_+)\\
&\quad + \bE^2\bU_{\bP_+}\bS_{\bP_+}^{-1}(\bS_{\bP_+}^2\bU_{\bP_+}\transpose\bU_{\bA_+} - \bU_{\bP_+}\transpose\bU_{\bA_+}\bS_{\bA_+}^2)\bS_{\bA_+}^{-2}
\\&\quad
 + \bE^2\bU_{\bP_+}\bS_{\bP_+}^{-1}(\bU_{\bP_+}\transpose\bU_{\bA_+} - \bW^*_+) + \bE^2\bU_{\bP_-}\bS_{\bP_-}\bU_{\bP_-}\transpose\bU_{\bA_+}\bS_{\bA_+}^{-2}\\
&\quad + \sum_{m = 3}^\infty\bE^m\bU_{\bP}\bS_{\bP}\bU_{\bP}\transpose\bU_{\bA_+}\bS_{\bA_+}^{-m}.
\end{align*}
By Result \ref{result:Noise_matrix_rowwise_concentration}, Lemma \ref{lemma:Rowwise_higher_order_concentration}, and a union bound, we have 
\begin{align*}
\|\bE\bU_{\bP_+}\|_{2\to\infty} &= \widetilde{O}_{\prob}\{\|\bU_\bP\|_{2\to\infty}(n\rho_n)^{1/2}(\log n)^\xi\},\\
\|\bE^2\bU_{\bP_+}\|_{2\to\infty} &= \widetilde{O}_{\prob}\{\|\bU_\bP\|_{2\to\infty}(n\rho_n)(\log n)^{2\xi}\}.
\end{align*}
Then by Result \ref{result:SU_exchange_concentration} and Lemma \ref{lemma:Eigenvector_angle_analysis} and recalling $\alpha_n = \|\bU_{\bP}\transpose\bE\bU_{\bP}\|$, we have
\begin{align*}
&\|\bE\bU_{\bP_+}(\bS_{\bP_+}\bU_{\bP_+}\transpose\bU_{\bA_+} - \bU_{\bP_+}\transpose\bU_{\bA_+}\bS_{\bA_+})\bS_{\bA_+}^{-1}\|_{2\to\infty} \\
&\quad = \widetilde{O}_{\prob}\left\{\|\bU_\bP\|_{2\to\infty}\frac{(n\rho_n)^{1/2}(\log n)^{\xi} }{\Delta_n}\left(\alpha_n + \frac{n\rho_n}{\Delta_n}\right)\right\},\\
&\|\bE\bU_{\bP_+}(\bU_{\bP_+}\transpose\bU_{\bA_+} - \bW^*_+)\|_{2\to\infty} \\
&\quad = \widetilde{O}_{\prob}\left\{\|\bU_\bP\|_{2\to\infty}\frac{(n\rho_n)^{1/2}(\log n)^{\xi} }{\Delta_n}\left(\alpha_n + \frac{n\rho_n}{\Delta_n}\right)\right\},\\
&\|\bE^2\bU_{\bP_+}\bS_{\bP_+}^{-1}(\bU_{\bP_+}\transpose\bU_{\bA_+} - \bW^*_+)\|_{2\to\infty}\\
&\quad = \widetilde{O}_{\prob}\left\{\|\bU_\bP\|_{2\to\infty}\frac{(n\rho_n)(\log n)^{2\xi} }{\Delta_n^2}\left(\alpha_n + \frac{n\rho_n}{\Delta_n}\right)\right\},
\\
&\|\bE^2\bU_{\bP_-}\bS_{\bP_-}\bU_{\bP_-}\transpose\bU_{\bA_+}\bS_{\bA_+}^{-2}\|_{2\to\infty}\\
&\quad = \widetilde{O}_{\prob}\left\{\|\bU_\bP\|_{2\to\infty}\frac{(n\rho_n)(\log n)^{2\xi} }{\Delta_n^2}\left(\alpha_n + \frac{n\rho_n}{\Delta_n}\right)\right\}.
\end{align*}
Here we also use some preliminary properties of matrix perturbation and two-to-infinity norm; See e.g. \cite{cape2019signal}. By definition and Result \ref{result:Noise_matrix_concentration},
\begin{align*}
&\|\bS_{\bP_+}^2\bU_{\bP_+}\transpose\bU_{\bA_+} - \bU_{\bP_+}\transpose\bU_{\bA_+}\bS_{\bA_+}^2\|_2\\
&\quad = \|\bU_{\bP_+}\transpose\bP^2\bU_{\bA_+} - \bU_{\bP_+}\transpose\bA^2\bU_{\bA_+}\|_2\\
&\quad = \|\bU_{\bP_+}\transpose\bP^2\bU_{\bA_+} - \bU_{\bP_+}\transpose(\bP^2 + \bP\bE + \bE\bP + \bE^2)\bU_{\bA_+}\|_2\\
&\quad\leq 2\|\bP\|_2\|\bE\|_2 + \|\bE\|_2^2 = \widetilde{O}_{\prob}\{\Delta_n(n\rho_n)^{1/2}\}.
\end{align*}
It follows that 
\begin{align*}
&\|\bE^2\bU_{\bP_+}\bS_{\bP_+}^{-1}(\bS_{\bP_+}^2\bU_{\bP_+}\transpose\bU_{\bA_+} - \bU_{\bP_+}\transpose\bU_{\bA_+}\bS_{\bA_+}^2)\bS_{\bA_+}^{-2}\|_{2\to\infty}\\
&\quad\leq \|\bE^2\bU_{\bP_+}\|_{2\to\infty}\|\bS_{\bP_+}^{-1}\|_2\|\bS_{\bP_+}^2\bU_{\bP_+}\transpose\bU_{\bA_+} - \bU_{\bP_+}\transpose\bU_{\bA_+}\bS_{\bA_+}^2\|_2
\|\bS_{\bA_+}^{-2}\|_2
\\
&\quad = \widetilde{O}_{\prob}\left\{\|\bU_\bP\|_{2\to\infty}\frac{(n\rho_n)^{3/2}(\log n)^{2\xi}}{\Delta_n^2}\right\}.
\end{align*}
We now focus on the last term. Recalling Table~\ref{def:sum}, let $M$ be the smallest integer no less than $3$ such that $(\sqrt{n\rho_n}/\Delta_n)^{M - 2}\lesssim \|\bU_\bP\|_{2\to\infty}$. This is possible because there exists some $\beta_\Delta > 0$ such that $\sqrt{n\rho_n}/\Delta_n = O(n^{-\beta_\Delta})$ by Assumption \ref{assumption:Signal_strength}. Then by Lemma \ref{lemma:Rowwise_higher_order_concentration}, there exists some constant $C > 0$, such that
\begin{align*}
&\left\|\sum_{m = 3}^\infty\bE^m\bU_\bP\bS_\bP\bU_\bP\transpose\bU_{\bA_+}\bS_{\bA_+}^{-m}\right\|_{2\to\infty}\\
&\quad\leq \sum_{m = 3}^M\|\bE^m\bU_\bP\|_{2\to\infty}\|\bS_\bP\|_2\|\bS_{\bA_+}^{-1}\|_2^m + \|\bS_\bP\|_2\sum_{m = M + 1}^\infty \|\bE\|_2^m\|\bS_{\bA_+}^{-1}\|_2^m\\
&\quad= \|\bU_\bP\|_{2\to\infty}\sum_{m = 3}^M\widetilde{\mathcal{O}}_{\mathbb{P}}\left\{\frac{(n\rho_n)^{m/2}(\log n)^{m\xi}}{\Delta_n^{m - 1}}\right\} + \Delta_n\widetilde{\mathcal{O}}_{\mathbb{P}}\left[\sum_{m = M + 1}^\infty\left\{\frac{C(n\rho_n)^{1/2}}{\Delta_n}\right\}^m\right]\\
&\quad = \widetilde{O}_{\prob}\left\{\|\bU_\bP\|_{2\to\infty}\frac{(n\rho_n)^{3/2}(\log n)^{3\xi}}{\Delta_n^2}\right\} + \widetilde{O}_{\prob}\left\{
\frac{(n\rho_n)^{(M + 1)/2}}{\Delta_n^M}
\right\}\\
&\quad = \widetilde{O}_{\prob}\left\{\|\bU_\bP\|_{2\to\infty}\frac{(n\rho_n)^{3/2}(\log n)^{3\xi}}{\Delta_n^2}\right\}.
\end{align*}
The proof is thus completed.
\end{proof}

\begin{lemma}\label{lemma:Second_order_expansion_II}
Suppose Assumptions \ref{assumption:Signal_strength}--\ref{assumption:Noise_matrix_distribution} in the manuscript hold. Then for any $\xi > 1$,
\begin{align*}
\bR_1 = \bE(\bU_{\bA_+} - \bU_{\bP_+}\bW^*_+)\bS_{\bA_+}^{-1}
& = \bE^2\bU_{\bP_+}\bS_{\bP_+}^{-2}\bW^*_+ + \bT^{(2)},
\end{align*}
where $\bT^{(2)}$ satisfies
\begin{align*}
&\|\bT^{(2)}\|_{2\to\infty}\\
&\quad = \widetilde{O}_{\prob}\left[\|\bU_\bP\|_{2\to\infty}\left\{\frac{(n\rho_n)^{1/2}(\log n)^\xi}{\Delta_n^2}\left(\alpha_n + \frac{n\rho_n}{\Delta_n}\right)^2 + \frac{(n\rho_n)^{3/2}(\log n)^{3\xi}}{\Delta_n^3}\right\}\right].
\end{align*}
\end{lemma}

\begin{proof}
By Lemma \ref{lemma:Remainder_analysis_I} and Lemma \ref{lemma:Rowwise_higher_order_concentration}, we have
\begin{align*}
\bE(\bU_{\bA_+} - \bU_{\bP_+}\bW_+^*)\bS_{\bA_+}^{-1}
& = \bE^2\bU_{\bP_+}\bS_{\bP_+}^{-2}\bW^*_+ + \bE\bR_{\bU_+}\bS_{\bP_+}^{-1}\\
&\quad + \bE\bU_{\bP_-}\bS_{\bP_-}\bU_{\bP_-}\transpose\bU_{\bA_+}\bS_{\bA_+}^{-1}\bS_{\bP_+}^{-1}\\
&\quad + \bE(\bU_{\bA_+} - \bU_{\bP_+}\bW^*_+)(\bS_{\bA_+}^{-1} - \bS_{\bP_+}^{-1})\\
& = \bE^2\bU_{\bP_+}\bS_{\bP_+}^{-2}\bW^*_+ + \bE\bR_{\bU_+}\bS_{\bP_+}^{-1}\\
&\quad + \bE\bU_{\bP_-}\bS_{\bP_-}\bU_{\bP_-}\transpose\bU_{\bA_+}\bS_{\bA_+}^{-1}\bS_{\bP_+}^{-1}\\
&\quad + \bE^2\bU_{\bP_+}\bS_{\bP_+}^{-1}(\bW^*_+\bS_{\bA_+}^{-1} - \bS_{\bP_+}^{-1}\bW_+^*)\\
&\quad + \bE\bR_{\bU_+}(\bW_+^*\bS_{\bA_+}^{-1} - \bS_{\bP_+}^{-1}\bW_+^*)\bW_+^*\\
&\quad + \bE\bU_{\bP_-}\bS_{\bP_-}\bU_{\bP_-}\transpose\bU_{\bA_+}\bS_{\bA_+}^{-1}(\bW^*_+\bS_{\bA_+}^{-1} - \bS_{\bP_+}^{-1}\bW_+^*)\bW_+^*.
\end{align*}
For the second term, we have
\begin{align*}
&\|\bE\bR_{\bU_+}\bS_{\bP_+}^{-1}\|_{2\to\infty}\\
&\quad\leq \|\bE\bR_{\bU_+}\|_{2\to\infty}\|\bS_{\bP_+}^{-1}\|_2\\
&\quad = \widetilde{O}_{\prob}\left[\|\bU_\bP\|_{2\to\infty}\left\{
\frac{(n\rho_n)^{1/2}(\log n)^\xi}{\Delta_n^2}\left(\alpha_n + \frac{n\rho_n}{\Delta_n}\right) + \frac{(n\rho_n)^{3/2}(\log n)^{3\xi}}{\Delta_n^3}
\right\}\right]
\end{align*}
by Lemma \ref{lemma:Remainder_analysis_II}. For the third term, we have
\begin{align*}
&\|\bE\bU_{\bP_-}\bS_{\bP_-}\bU_{\bP_-}\transpose\bU_{\bA_+}\bS_{\bA_+}^{-1}\bS_{\bP_+}^{-1}\|_{2\to\infty}\\
&\quad\leq \|\bE\bU_{\bP_-}\|_{2\to\infty}\|\bS_{\bP_-}\|_2\|\bU_{\bP_-}\transpose\bU_{\bA_+}\|_2\|\bS_{\bA_+}^{-1}\|_2\|\bS_{\bP_+}^{-1}\|_2\\
&\quad = \widetilde{O}_{\prob}\left\{\|\bU_\bP\|_{2\to\infty}\frac{(n\rho_n)^{1/2}(\log n)^{\xi}}{\Delta_n^2}\left(\alpha_n + \frac{n\rho_n}{\Delta_n}\right)\right\}
\end{align*}
by Lemma \ref{lemma:Eigenvector_angle_analysis} and Result \ref{result:Noise_matrix_rowwise_concentration}.
For the fourth term, we have
\begin{align*}
&\|\bE^2\bU_{\bP_+}\bS_{\bP_+}^{-1}(\bW^*_+\bS_{\bA_+}^{-1} - \bS_{\bP_+}^{-1}\bW_+^*)\|_{2\to\infty}\\
&\quad\leq \|\bE^2\bU_{\bP_+}\|_{2\to\infty}\|\bS_{\bP_+}^{-1}\|_2\|\bW^*_+\bS_{\bA_+}^{-1} - \bS_{\bP_+}^{-1}\bW_+^*\|_2\\
&\quad = \widetilde{O}_{\prob}\left\{\|\bU_\bP\|_{2\to\infty}\frac{(n\rho_n)(\log n)^{2\xi}}{\Delta_n^3}\left(\alpha_n + \frac{n\rho_n}{\Delta_n}\right)\right\}\\
&\quad = \widetilde{O}_{\prob}\left\{\|\bU_\bP\|_{2\to\infty}\frac{(n\rho_n)(\log n)^{2\xi}}{\Delta_n^3}\|\bE\|_2 + \|\bU_\bP\|_{2\to\infty}\frac{(n\rho_n)^{3/2}(\log n)^{2\xi}}{\Delta_n^3}\times\frac{(n\rho_n)^{1/2}}{\Delta_n}\right\}\\
&\quad = \widetilde{O}_{\prob}\left\{\|\bU_\bP\|_{2\to\infty}\frac{(n\rho_n)^{3/2}(\log n)^{2\xi}}{\Delta_n^3}\right\}
\end{align*}
by Lemma \ref{lemma:Rowwise_higher_order_concentration} and Lemma \ref{lemma:Eigenvector_angle_analysis}. For the fifth term, we have
\begin{align*}
&\|\bE\bR_{\bU_+}(\bW^*_+\bS_{\bA_+}^{-1} - \bS_{\bP_+}^{-1}\bW^*_+)\bW^*_+\|_{2\to\infty}\\
&\quad\leq \|\bE\bR_{\bU_+}\|_{2\to\infty}\|\bW^*_+\bS_{\bA_+}^{-1} - \bS_{\bP_+}^{-1}\bW^*_+\|_2\\
&\quad = \widetilde{O}_{\prob}\left[\|\bU_\bP\|_{2\to\infty}\left\{
\frac{(n\rho_n)^{1/2}(\log n)^\xi}{\Delta_n^3}\left(\alpha_n + \frac{n\rho_n}{\Delta_n}\right) + \frac{(n\rho_n)^{3/2}(\log n)^{3\xi}}{\Delta_n^3}
\right\}\right]
\end{align*}
by Lemma \ref{lemma:Remainder_analysis_I} and Lemma \ref{lemma:Remainder_analysis_II}. For the last term, we have
\begin{align*}
&\|\bE\bU_{\bP_-}\bS_{\bP_-}\bU_{\bP_-}\transpose\bU_{\bA_+}\bS_{\bA_+}^{-1}(\bW^*_+\bS_{\bA_+} - \bS_{\bP_+}\bW_+^*)\bW_+^*\|_{2\to\infty}\\
&\quad\leq \|\bE\bU_{\bP_-}\|_{2\to\infty}\|\bS_{\bP_-}\|_2\|\bU_{\bP_-}\transpose\bU_{\bA_+}\|_2\|\bS_{\bA_+}^{-1}\|_2\|\bW^*_+\bS_{\bA_+}^{-1} - \bS_{\bP_+}^{-1}\bW_+^*\|_2\\
&\quad = \widetilde{O}_{\prob}\left\{\|\bU_\bP\|_{2\to\infty}\frac{(n\rho_n)^{1/2}(\log n)^\xi}{\Delta_n^2}\left(\alpha_n + \frac{n\rho_n}{\Delta_n}\right)^2\right\}.
\end{align*}
The proof is thus completed. 
\end{proof}

\begin{proof}[Proof of Theorem \ref{thm:Eigenvector_Expansion}]
By the decomposition \eqref{eqn:first_order_eigenvector_expansion}, we have
\begin{align*}
\bU_{\bA_+} - \bU_{\bP_+}\bW_+^*
& = \bE\bU_{\bP_+}\bS_{\bP_+}^{-1}\bW_+^* + \bR_{\bU_+} + \bU_{\bP_-}\bS_{\bP_-}\bU_{\bP_-}\transpose\bU_{\bA_+}\bS_{\bA_+}^{-1}\\
& = \bE\bU_{\bP_+}\bS_{\bP_+}^{-1}\bW_+^* + \bR_{\bU_+} + \bU_{\bP_-}\bS_{\bP_-}\bU_{\bP_-}\transpose\bU_{\bA_+}\bS_{\bP_+}^{-1} + \bT^{(3)},
\end{align*}
where, recalling that $\bW_+^*$ is a diagonal matrix and $(\bW_+^*) ^2$ is an identity matrix
, we have 
\begin{align*}
\bT^{(3)} = \bU_{\bP_-}\bS_{\bP_-}\bU_{\bP_-}\transpose\bU_{\bA_+}(\bW^*_+\bS_{\bA_+}^{-1} - \bS_{\bP_+}^{-1}\bW_+^*)\bW_+^*
\end{align*}
and by Lemma \ref{lemma:Eigenvector_angle_analysis} and Lemma \ref{lemma:Remainder_analysis_I},
\[
\|\bT^{(3)}\|_{2\to\infty} = \widetilde{O}_{\prob}\left\{\|\bU_\bP\|_{2\to\infty}\frac{1}{\Delta_n^2}\left(\alpha_n + \frac{n\rho_n}{\Delta_n}\right)^2\right\}.
\]
By Lemma \ref{lemma:Remainder_analysis_I}, Lemma \ref{lemma:Second_order_expansion_I}, and Lemma \ref{lemma:Second_order_expansion_II}, recalling \eqref{def:RU+}, we have
\begin{align*}
\bR_{\bU_+}& = \underbrace{\bE^2\bU_{\bP_+}\bS_{\bP_+}^{-2}\bW_+^* + \bT^{(2)}}_{\bR_1}\\
&\quad \underbrace{- \bU_{\bP_+}\bU_{\bP_+}\transpose\bE\bU_{\bP_+}\bS_{\bP_+}^{-1}\bW_+^* - \bU_{\bP_+}\bU_{\bP_+}\transpose\bE^2\bU_{\bP_+}\bS_{\bP_+}^{-2}\bW_+^* + \bT^{(1)}}_{\bR_2}\\
&\quad + \bR_{3} + \bR_{4},
\end{align*}
 and 
\begin{align*}
&\|\bT^{(1)} + \bT^{(2)} + \bR_4\|_{2\to\infty}\\
&\quad = \widetilde{O}_{\prob}\left[\|\bU_\bP\|_{2\to\infty}\left\{\frac{(n\rho_n)^{1/2}(\log n)^\xi}{\Delta_n^2}\left(\alpha_n + \frac{n\rho_n}{\Delta_n}\right)^2 + \frac{(n\rho_n)^{3/2}(\log n)^{3\xi}}{\Delta_n^3}\right\}\right].
\end{align*}
For $\be_k\in\mathbb{R}^d$ the $k$th standard basis vector in $\mathbb{R}^p$, we have, by Proposition \ref{prop:eigenvector_angle_expansion},
\begin{align*}
&\bU_{\bP_-}\bS_{\bP_-}\bU_{\bP_-}\transpose\bU_{\bA_+}\bS_{\bP_+}^{-1}\be_k\\
&\quad = \sum_{m = p + 1}^d\frac{\lambda_m}{\lambda_k}\bu_m\bu_m\transpose\widehat{\bu}_k\\
&\quad = \sum_{m = p + 1}^d\frac{\mathrm{sgn}(\bu_k\transpose\widehat{\bu}_k)\lambda_m\bu_m\bu_m\transpose\bE\bu_k}{\lambda_k(\lambda_k - \lambda_m)} + \sum_{m = p + 1}^d\frac{\mathrm{sgn}(\bu_k\transpose\widehat{\bu}_k)\lambda_m\bu_m\bu_m\transpose\bE^2\bu_k}{\lambda_k^{2}(\lambda_k - \lambda_m)} + \bT^{(5)},
\end{align*}and
\begin{align*}
\bR_3\be_k
& = \bu_{k}\{\bu_k\transpose\widehat{\bu}_k - \mathrm{sgn}(\bu_k\transpose\widehat{\bu}_k)\} + \sum_{m\in [p]\backslash\{k\}}\bu_{m}\bu_m\transpose\widehat{\bu}_k\\
& = -\mathrm{sgn}(\bu_k\transpose\widehat{\bu}_k)\frac{\bu_{k}\bu_k\transpose\bE^2\bu_k}{2\lambda_k^{2}} + \sum_{m\in [p]\backslash\{k\}}\mathrm{sgn}(\bu_k\transpose\widehat{\bu}_k)\frac{\bu_{m}\bu_m\transpose\bE\bu_k}{\lambda_k - \lambda_m}
\\&\quad
 + \sum_{m\in [p]\backslash\{k\}}\mathrm{sgn}(\bu_k\transpose\widehat{\bu}_k)\frac{\bu_{m}\bu_m\transpose \bE^2\bu_k}{\lambda_k(\lambda_k - \lambda_m)} + \bT^{(4)}
\end{align*}
for some $n\times d$ matrices $\bT^{(4)},\bT^{(5)}$, where, also by Proposition \ref{prop:eigenvector_angle_expansion}, we have
\[
\|\bT^{(4)} + \bT^{(5)}\|_{2\to\infty} = \widetilde{O}_{\prob}\left[\|\bU_\bP\|_{2\to\infty}\left\{\frac{\alpha_n^2}{\Delta_n^2} + \frac{n^2\rho_n^{3/2}(\log n)^{3\xi}\|\bU_\bP\|_{2\to\infty}^2}{\Delta_n^3} + \frac{(n\rho_n)^{3/2}}{\Delta_n^3}\right\}\right].
\] 
Let $[\bM]_{\star,k}$ denote the $k$th column of any matrix $\bM$ and summarize the above result. We conclude that
\begin{align*}
&\widehat{\bu}_k - \bu_k\mathrm{sgn}(\bu_k\transpose\widehat{\bu}_k)\\
&\quad = (\eye_n  - \bU_{\bP_+}\bU_{\bP_+})\bE\bU_{\bP_+}\bS_{\bP_+}^{-1}\bW_+^*\be_k +(\eye_n  - \bU_{\bP_+}\bU_{\bP_+}) \bE^2\bU_{\bP_+}\bS_{\bP_+}^{-2}\bW_+^*\be_k
\\
&\quad\quad + \bU_{\bP_-}\bS_{\bP_-}\bU_{\bP_-}\transpose\bU_{\bA_+}\bS_{\bP_+}^{-1}\be_k  + \bR_{3}\be_k + \bt'_k
\\
&\quad = (\eye_n  - \bu_k\bu_k\transpose)\bE\bU_{\bP_+}\bS_{\bP_+}^{-1}\bW_+^*\be_k +(\eye_n  - \bu_k\bu_k\transpose) \bE^2\bU_{\bP_+}\bS_{\bP_+}^{-2}\bW_+^*\be_k
\\
&\quad\quad - \Bigg(\sum_{m\in[p]/\{k\}}\bu_m\bu_m\transpose\Bigg)\bE\bU_{\bP_+}\bS_{\bP_+}^{-1}\bW_+^*\be_k -\Bigg(\sum_{m\in[p]/\{k\}}\bu_m\bu_m\transpose\Bigg) \bE^2\bU_{\bP_+}\bS_{\bP_+}^{-2}\bW_+^*\be_k
\\
&\quad\quad + \bU_{\bP_-}\bS_{\bP_-}\bU_{\bP_-}\transpose\bU_{\bA_+}\bS_{\bP_+}^{-1}\be_k  + \bR_{3}\be_k + \bt'_k
\\
&\quad = \frac{(\eye_n - \bu_k\bu_k\transpose)\bE\bu_k\mathrm{sgn}(\bu_k\transpose\widehat{\bu}_k)}{\lambda_k} + \frac{(\eye_n - \bu_k\bu_k\transpose)\bE^2\bu_k\mathrm{sgn}(\bu_k\transpose\widehat{\bu}_k)}{\lambda_k^{2}}\\
&\quad\quad + \sum_{m = p + 1}^d\frac{\mathrm{sgn}(\bu_k\transpose\widehat{\bu}_k)\lambda_m\bu_m\bu_m\transpose\bE\bu_k}{\lambda_k(\lambda_k - \lambda_m)} + \sum_{m = p + 1}^d\frac{\mathrm{sgn}(\bu_k\transpose\widehat{\bu}_k)\lambda_m\bu_m\bu_m\transpose\bE^2\bu_k}{\lambda_k^{2}(\lambda_k - \lambda_m)}
\\
&\quad\quad -\mathrm{sgn}(\bu_k\transpose\widehat{\bu}_k)\frac{\bu_{k}\bu_k\transpose\bE^2\bu_k}{2\lambda_k^{2}} + \sum_{m\in [p]\backslash\{k\}}\mathrm{sgn}(\bu_k\transpose\widehat{\bu}_k)\bu_{m}\bu_m\transpose\bE\bu_k\Big(\frac{1}{\lambda_k - \lambda_m} - \frac{1}{\lambda_k}\Big)
\\&\quad\quad
 + \sum_{m\in [p]\backslash\{k\}}\mathrm{sgn}(\bu_k\transpose\widehat{\bu}_k)\bu_{m}\bu_m\transpose \bE^2\bu_k\Big(\frac{1}{\lambda_k(\lambda_k - \lambda_m)} - \frac{1}{\lambda_k^2}\Big) + \mathbf{t}_k
\\
&\quad = \frac{(\eye_n - \bu_k\bu_k\transpose)\bE\bu_k\mathrm{sgn}(\bu_k\transpose\widehat{\bu}_k)}{\lambda_k} + \frac{(\eye_n - \bu_k\bu_k\transpose)\bE^2\bu_k\mathrm{sgn}(\bu_k\transpose\widehat{\bu}_k)}{\lambda_k^{2}}\\
&\quad\quad - \frac{\bu_k\bu_k\transpose\bE^2\bu_k\mathrm{sgn}(\bu_k\transpose\widehat{\bu}_k)}{2\lambda_k^{2}}  + \sum_{m\in [d]\backslash\{k\}}\frac{\lambda_m\bu_m\bu_m\transpose\bE\bu_k\mathrm{sgn}(\bu_k\transpose\widehat{\bu}_k)}{\lambda_k(\lambda_k - \lambda_m)}\\
&\quad\quad + \sum_{m\in[d]\backslash\{k\}}\frac{\lambda_m\bu_m\bu_m\transpose\bE^2\bu_k\mathrm{sgn}(\bu_k\transpose\widehat{\bu}_k)}{\lambda_k^{2}(\lambda_k - \lambda_m)} + \bt_k,
\end{align*}
where 
\begin{align*}
&\max\{\|\bt'_k\|_{2\to\infty},\|\bt_k\|_{2\to\infty}\}\\
&\quad = \widetilde{O}_{\prob}\left[\|\bU_\bP\|_{2\to\infty}\left\{\frac{(n\rho_n)^{1/2}(\log n)^\xi}{\Delta_n^2}\left(\alpha_n + \frac{n\rho_n}{\Delta_n}\right)^2 + \frac{(n\rho_n)^{3/2}(\log n)^{3\xi}}{\Delta_n^3}\right\}\right]\\
&\quad\quad + \widetilde{O}_{\prob}\left\{\|\bU_\bP\|_{2\to\infty}\frac{n^2\rho_n^{3/2}(\log n)^{3\xi}\|\bU_\bP\|_{2\to\infty}^2}{\Delta_n^3}\right\}.
\end{align*}
The proof is thus completed.
\end{proof}

\begin{proof}[Proof of Theorem \ref{thm:bias_corrected_ASE}]
It is sufficient to show that for any $m\in [d]$,
\begin{align*}
&\left\|\frac{\widehat{\bD}\widehat{\bu}_k\mathrm{sgn}(\bu_k\transpose\widehat{\bu}_k)}{\widehat{\lambda}_k^2} - \frac{\bD\bu_k}{\lambda_k^{2}}\right\|_\infty = \widetilde{O}_{\prob}\left\{\frac{(n\rho_n)(\log n)^{\xi}}{\sqrt{n}\Delta_n^2}\max\left(\frac{1}{q_n},\frac{\sqrt{n\rho_n}}{\Delta_n}\right)\right\},\\
&\left\|\frac{\widehat{\bu}_m\widehat{\bu}_m\transpose\widehat{\bD}\widehat{\bu}_k\mathrm{sgn}(\bu_k\transpose\widehat{\bu}_k)}{\widehat{\lambda}_k^2} - \frac{\bu_m\bu_m\transpose\bD\bu_k}{\lambda_k^{2}}\right\|_\infty = \widetilde{O}_{\prob}\left\{\frac{(n\rho_n)(\log n)^{\xi}}{\sqrt{n}\Delta_n^2}\max\left(\frac{1}{q_n},\frac{\sqrt{n\rho_n}}{\Delta_n}\right)\right\},\\
&\left\|\frac{\widehat{\lambda}_m\widehat{\bu}_m\widehat{\bu}_m\transpose\widehat{\bD}\widehat{\bu}_k\mathrm{sgn}(\bu_k\transpose\widehat{\bu}_k)}{\widehat{\lambda}_k^2(\widehat{\lambda}_k - \widehat{\lambda}_m)} - \frac{\lambda_m\bu_m\bu_m\transpose\bD\bu_k}{\lambda_k^{2}(\lambda_k - \lambda_m)}\right\|_\infty
\\&\quad
 = \widetilde{O}_{\prob}\left\{\frac{(n\rho_n)(\log n)^{\xi}}{\sqrt{n}\Delta_n^2}\max\left(\frac{1}{q_n}, \frac{\sqrt{n\rho_n}}{\Delta_n}\right)\right\}.
\end{align*}
We first argue that $\|\widehat{\bD} - \bD\|_\infty = \widetilde{O}_{\prob}\{(n\rho_n)(\log n)^{\xi}/q_n\}$. We claim that
\begin{align}\label{eqn:pij_zeroth_order_deviation}
\max_{i,j\in [n]}|\widehat{p}_{ij} - p_{ij}| = \widetilde{O}_{\prob}\left\{(\log n)^{\xi}\sqrt{\frac{\rho_n}{n}}\right\}.
\end{align}
In fact, by definition of $\widehat{\bP}$, Lemma \ref{lemma:Eigenvector_zero_order_deviation}, and Lemma \ref{lemma:Remainder_analysis_I}, we have
\begin{align*}
|\widehat{p}_{ij} - p_{ij}|
&\leq |\be_i\transpose(\bU_{\bA} - \bU_\bP\bW^*)\bS_\bA\bU_\bA\transpose\be_j| + |\be_i\transpose\bU_\bP\transpose(\bW^*\bS_\bA - \bS_\bP\bW^*)\bU_\bA\transpose\be_j|\\
&\quad + |\be_i\transpose\bU_\bP\bS_\bP(\bW^*\bU_\bA\transpose - \bU_\bP\transpose)\be_j|\\
&\leq \|\bU_{\bA} - \bU_{\bP}\bW^*\|_{2\to\infty}(\|\bS_\bA\|_2\|\bU_\bA\|_{2\to\infty} + \|\bS_\bP\|_2\|\bU_\bP\|_{2\to\infty})\\
&\quad + \|\bU_\bP\|_{2\to\infty}\|\bU_\bA\|_{2\to\infty}\|\bW^*\bS_\bA - \bS_\bP\bW^*\|_2\\
& = \widetilde{O}_{\prob}\left\{(\log n)^{\xi}\sqrt{\frac{\rho_n}{n}}\right\}.
\end{align*}
It follows from a union bound that 
\[
\max_{i,j\in [n]}|\widehat{p}_{ij} - p_{ij}| = \widetilde{O}_{\prob}\left\{(\log n)^{\xi}\sqrt{\frac{\rho_n}{n}}\right\}.
\]
Therefore, by triangle inequality and Cauchy-Schwarz inequality,
\begin{align*}
\|\widehat{\bD} - \bD\|_\infty
& = \max_{i\in [n]}\left|\sum_{j = 1}^n(A_{ij} - \widehat{p}_{ij})^2 - \sum_{j = 1}^n\sigma_{ij}^2\right|\\
&\leq \max_{i\in [n]}\left|\sum_{j = 1}^n(E_{ij}^2 - \sigma_{ij}^2)\right| + 2\max_{i\in [n]}\left|\sum_{j = 1}^nE_{ij}(p_{ij} - \widehat{p}_{ij})\right| + \max_{i\in [n]}\sum_{j = 1}^n(p_{ij} - \widehat{p}_{ij})^2\\
&\leq \max_{i\in [n]}\left|\sum_{j = 1}^n(E_{ij}^2 - \sigma_{ij}^2)\right| + 2\max_{i\in [n]}\left(\sum_{j = 1}^nE_{ij}^2\right)^{1/2}\left\{\sum_{j = 1}^n(p_{ij} - \widehat{p}_{ij})^2\right\}^{1/2}\\
&\quad + n\max_{i,j\in[n]}(p_{ij} - \widehat{p}_{ij})^2\\
& = \max_{i\in [n]}\left|\sum_{j = 1}^n(E_{ij}^2 - \sigma_{ij}^2)\right| + 2\|\bE\|_2\sqrt{n}\max_{i,j\in [n]}|p_{ij} - \widehat{p}_{ij}| + n\max_{i,j\in[n]}(p_{ij} - \widehat{p}_{ij})^2\\
& = \max_{i\in [n]}\left|\sum_{j = 1}^n(E_{ij}^2 - \sigma_{ij}^2)\right| + \widetilde{O}_{\prob}\left\{\sqrt{n}\rho_n(\log n)^\xi\right\}.
\end{align*}
The first term is $\widetilde{O}_{\prob}\{(n\rho_n)(\log n)^\xi/q_n\}$ by Result~\ref{result:concentration}, Equation~\eqref{result:concentration:res2}. This concludes the concentration bound $\|\widehat{\bD} - \bD\|_\infty = \widetilde{O}_{\prob}\{(n\rho_n)(\log n)^\xi/q_n\}$ because $q_n = O(\sqrt{n})$ by Assumption \ref{assumption:Noise_matrix_distribution}.

For the first claim, by Lemma \ref{lemma:Eigenvector_zero_order_deviation} and Lemma \ref{lemma:Remainder_analysis_I}, we have
\begin{align*}
\left\|\frac{\widehat{\bD}\widehat{\bu}_k\mathrm{sgn}(\bu_k\transpose\widehat{\bu}_k)}{\widehat{\lambda}_k^2} - \frac{\bD\bu_k}{\lambda_k^{2}}\right\|_\infty
&\leq \left\|\frac{\widehat{\bD}}{\widehat{\lambda}_k^2}\right\|_\infty \|\widehat{\bu}_k\mathrm{sgn}(\bu_k\transpose\widehat{\bu}_k) - \bu_k\|_\infty + \left\|\frac{\widehat{\bD} - \bD}{\widehat{\lambda}_k^2}\right\|_\infty\|\bu_k\|_\infty\\
&\quad + \|\bD\|_\infty\|\bu_k\|_\infty\left|\frac{\lambda_k + \widehat{\lambda}_k}{\widehat{\lambda}_k^2\lambda_k^{2}}\right|\|\bW^*\bS_\bA - \bS_\bP\bW^*\|_2\\
& = \widetilde{O}_{\prob}\left\{\frac{(n\rho_n)(\log n)^{\xi}}{\sqrt{n}\Delta_n^2}\max\left(\frac{1}{q_n},\frac{\sqrt{n\rho_n}}{\Delta_n}\right)\right\}.
\end{align*}
For the second claim, by Lemma \ref{lemma:Eigenvector_zero_order_deviation}, Assumption \ref{assumption:Eigenvector_delocalization}, and the first claim, we have
\begin{align*}
&\left\|\frac{\widehat{\bu}_m\widehat{\bu}_m\transpose\widehat{\bD}\widehat{\bu}_k\mathrm{sgn}(\bu_k\transpose\widehat{\bu}_k)}{\widehat{\lambda}_k^2} - \frac{\bu_m\bu_m\transpose\bD\bu_k}{\lambda_k^{2}}\right\|_\infty\\
&\quad\leq \left\|\widehat{\bu}_m\widehat{\bu}_m\transpose - \bu_m\bu_m\transpose\right\|_\infty\left\|\frac{\widehat{\bD}\widehat{\bu}_k\mathrm{sgn}(\bu_k\transpose\widehat{\bu}_k)}{\widehat{\lambda}_k^2}\right\|_\infty
\\&\quad\quad
 + \left\|\bu_m\bu_m\transpose\right\|_\infty\left\|\frac{\widehat{\bD}\widehat{\bu}_k\mathrm{sgn}(\bu_k\transpose\widehat{\bu}_k)}{\widehat{\lambda}_k^2} - \frac{\bD\bu_k}{\lambda_k^{2}}\right\|_\infty\\
&\quad\leq \Big\{\|\widehat{\bu}_m\mathrm{sgn}(\bu_m\transpose\widehat{\bu}_m) - \bu_m\|_\infty\|\widehat{\bu}_m\|_1 + \sqrt{n}\|\bu_m\|_\infty\|\widehat{\bu}_m\mathrm{sgn}(\bu_m\transpose\widehat{\bu}_m) - \bu_m\|_2\Big\}\\ 
&\quad\quad\times \frac{\|\widehat{\bD}\|_\infty\|\widehat{\bu}_k\|_\infty}{\widehat{\lambda}_k^2}
 + \|\bu_m\|_\infty\|\bu_m\|_1\left\|\frac{\widehat{\bD}\widehat{\bu}_k\mathrm{sgn}(\bu_k\transpose\widehat{\bu}_k)}{\widehat{\lambda}_k^2} - \frac{\bD\bu_k}{\lambda_k^{2}}\right\|_\infty\\
&\quad = \widetilde{O}_{\prob}\left\{\frac{(n\rho_n)(\log n)^{\xi}}{\sqrt{n}\Delta_n^2}\max\left(\frac{1}{q_n},\frac{\sqrt{n\rho_n}}{\Delta_n}\right)\right\}.
\end{align*}
Finally, for the third claim, by Lemma \ref{lemma:Eigenvector_zero_order_deviation}, Lemma \ref{lemma:Remainder_analysis_I} and the second claim, we have
\begin{align*}
&\left\|\frac{\widehat{\lambda}_m\widehat{\bu}_m\widehat{\bu}_m\transpose\widehat{\bD}\widehat{\bu}_k\mathrm{sgn}(\bu_k\transpose\widehat{\bu}_k)}{\widehat{\lambda}_k^2(\widehat{\lambda}_k - \widehat{\lambda}_m)} - \frac{\lambda_m\bu_m\bu_m\transpose\bD\bu_k}{\lambda_k^{2}(\lambda_k - \lambda_m)}\right\|_\infty\\
&\quad\leq \left|\frac{\widehat{\lambda}_m}{\widehat{\lambda}_k - \widehat{\lambda}_m} - \frac{\lambda_m}{\lambda_k - \lambda_m}\right|\left\|\frac{\widehat{\bu}_m\widehat{\bu}_m\transpose\widehat{\bD}\widehat{\bu}_k\mathrm{sgn}(\bu_k\transpose\widehat{\bu}_k)}{\widehat{\lambda}_k^2}\right\|_\infty\\
&\quad\quad + \left|\frac{\lambda_m}{\lambda_k - \lambda_m}\right|\left\|\frac{\widehat{\bu}_m\widehat{\bu}_m\transpose\widehat{\bD}\widehat{\bu}_k\mathrm{sgn}(\bu_k\transpose\widehat{\bu}_k)}{\widehat{\lambda}_k^2} - \frac{\bu_k\bu_k\transpose\bD\bu_k}{\lambda_k^{2}}\right\|_\infty\\
&\quad\leq \frac{2\|\bW^*\bS_\bA - \bS_\bP\bW^*\|_2\|\bS_\bP\|_2}{|(\widehat{\lambda}_k - \widehat{\lambda}_m)(\lambda_k - \lambda_m)|}\frac{\sqrt{n}\|\widehat{\bu}_m\|_\infty\|\widehat{\bD}\|_\infty\|\widehat{\bu}_k\|_\infty}{\widehat{\lambda}_k^2}\\
&\quad\quad + \widetilde{O}_{\prob}\left\{\frac{(n\rho_n)(\log n)^{\xi}}{\sqrt{n}\Delta_n^2}\max\left(\frac{1}{q_n},\frac{\sqrt{n\rho_n}}{\Delta_n}\right)\right\}\\
&\quad = \widetilde{O}_{\prob}\left\{\frac{(n\rho_n)(\log n)^{\xi}}{\sqrt{n}\Delta_n^2}\max\left(\frac{1}{q_n},\frac{\sqrt{n\rho_n}}{\Delta_n}\right)\right\}.
\end{align*}
The proof is thus completed.
\end{proof}

\section{Proofs of Theorem \ref{thm:edgeworth_expansion}}
\label{sec:proof_of_edgeworth_expansion}

This section elaborates on the proof of Theorem \ref{thm:edgeworth_expansion}, which is somewhat technical. We first establish the variance expansion result, namely, Lemma \ref{lemma:variance_expansion}, in Section \ref{sub:proof_of_variance_expansion}. Section \ref{sub:proof_sketch_edgeworth_expansion} then outlines the proof sketch of Theorem \ref{thm:edgeworth_expansion}. We provide some preparatory technical lemmas in Section \ref{sub:technical_lemmas_for_edgeworth_expansion} and present the complete proof in Section \ref{sub:proof_of_edgeworth_expansion_smoothed_version} and Section \ref{sub:finishing_proof_of_edgeworth_expansion} via Esseen's smoothing lemma (Theorem \ref{thm:esseen_smoothing_lemma}, \cite{10.1007/BF02392223}) and delicate analyses of characteristic functions of related random variables. 

\subsection{Proof of Lemma \ref{lemma:variance_expansion}}
\label{sub:proof_of_variance_expansion}

We first claim the following useful expansion for $\widehat{p}_{ij} - p_{ij}$, where $\hat{p}_{ij}$ and $p_{ij}$ are $ij$th entries of $\bU_\bA\bS_\bA\bU_\bA\transpose$ and $\bP$, respectively. In particular, we have,
\begin{align*}
&\widehat{p}_{ij} - p_{ij}\\
&\quad = \be_i\transpose(\bU_\bA\bS_\bA\bU_\bA\transpose - \bU_\bP\bS_\bP\bU_\bP\transpose)\be_j\\
&\quad = \be_i\transpose(\bU_\bA - \bU_\bP\bW^*)\bS_\bA\bU_\bA\transpose\be_j + \be_i\transpose\bU_\bP(\bW^*\bS_\bA - \bS_\bP\bW^*)\bU_\bA\transpose\be_j\\
&\quad\quad + \be_i\transpose\bU_\bP\bS_\bP(\bW^*\bU_\bA\transpose - \bU_\bP\transpose)\be_j\\
&\quad = \be_i\transpose(\bU_\bA - \bU_\bP\bW^*)(\bS_\bA - \bS_\bP)\bU_\bA\transpose\be_j + \be_i\transpose(\bU_\bA - \bU_\bP\bW^*)\bS_\bP(\bU_\bA - \bU_\bP\bW^*)\transpose\be_j\\
&\quad\quad + \be_i\transpose(\bU_\bA - \bU_\bP\bW^*)\bS_\bP\bW^*\bU_\bP\transpose\be_j + \be_i\transpose\bU_\bP(\bW^*\bS_\bA - \bS_\bP\bW^*)\bU_\bA\transpose\be_j\\
&\quad\quad + \be_i\transpose\bU_\bP\bS_\bP(\bW^*\bU_\bA\transpose - \bU_\bP\transpose)\be_j\\
&\quad = \be_i\transpose(\bU_\bA - \bU_\bP\bW^*)(\bS_\bA - \bS_\bP)\bU_\bA\transpose\be_j + \be_i\transpose(\bU_\bA - \bU_\bP\bW^*)\bS_\bP(\bU_\bA - \bU_\bP\bW^*)\transpose\be_j\\
&\quad\quad + \be_i\transpose(\bE\bU_\bP\bS_\bP^{-1}\bW^* + \bR_\bU)\bS_\bP\bW^*\bU_\bP\transpose\be_j + \be_i\transpose\bU_\bP(\bW^*\bS_\bA - \bS_\bP\bW^*)\bU_\bA\transpose\be_j\\
&\quad\quad + \be_i\transpose\bU_\bP\bS_\bP(\bE\bU_\bP\bS_\bP^{-1} + \bR_\bU\bW^*)\transpose\be_j\\
&\quad = \be_i\transpose\bE\bU_\bP\bU_\bP\transpose\be_j + \be_i\transpose\bU_\bP\bU_\bP\transpose\bE\be_j + r_{p_{ij}},
\end{align*}
where the fourth equality is by Lemma \ref{lemma:Eigenvector_zero_order_deviation}, and we denote
\begin{equation}\label{defrij}
\begin{aligned}
r_{p_{ij}}& = \be_i\transpose(\bU_\bA - \bU_\bP\bW^*)(\bS_\bA - \bS_\bP)\bU_\bA\transpose\be_j + \be_i\transpose(\bU_\bA - \bU_\bP\bW^*)\bS_\bP(\bU_\bA - \bU_\bP\bW^*)\transpose\be_j\\
&\quad + \be_i\transpose\bU_\bP(\bW^*\bS_\bA - \bS_\bP\bW^*)\bU_\bA\transpose\be_j + \be_i\transpose\bR_\bU\bS_\bP\bW^*\bU_\bP\transpose\be_j + \be_i\transpose\bU_\bP\bS_\bP\bW^*\bR_\bU\transpose\be_j.
\end{aligned}
\end{equation}
Now by Lemma \ref{lemma:Eigenvector_zero_order_deviation} and Lemma \ref{lemma:Remainder_analysis_I}, we have
\begin{align*}
\max_{i,j\in [n]}|r_{p_{ij}}|&\leq \|\bU_\bA - \bU_\bP\bW^*\|_{2\to\infty}\|\bW^*\bS_\bA - \bS_\bP\bW^*\|_2\|\bU_\bA\|_{2\to\infty}
\\
&\quad + \|\bU_\bA - \bU_\bP\bW^*\|_{2\to\infty}^2\|\bS_\bP\|_2 + \|\bU_\bP\|_{2\to\infty}{\|\bW^*\bS_\bA - \bS_\bP\bW^*\|_2}\|\bU_\bA\|_{2\to\infty}\\
&\quad + 2\|\bR_\bU\|_{2\to\infty}\|\bS_\bP\|_2\|\bU_\bP\|_{2\to\infty}\\
& = \widetilde{O}_{\prob}\left\{\frac{\rho_n^{1/2}(\log n)^\xi}{\Delta_n}\right\}\widetilde{O}_{\prob}\left\{(\log n)^\xi + \frac{n\rho_n}{\Delta_n}\right\}\widetilde{O}_{\prob}\left(\frac{1}{\sqrt{n}}\right) + \widetilde{O}_{\prob}\left\{\frac{\rho_n(\log n)^{2\xi}}{\Delta_n}\right\}\\
&\quad + O\left(\frac{1}{\sqrt{n}}\right)\widetilde{O}_{\prob}\left\{(\log n)^\xi + \frac{n\rho_n}{\Delta_n}\right\}\widetilde{O}_{\prob}\left(\frac{1}{\sqrt{n}}\right)\\
&\quad + \widetilde{O}_{\prob}\left\{\frac{\rho_n(\log n)^{2\xi}}{\Delta_n} + \frac{(\log n)^\xi}{n}\right\}\\
& = \widetilde{O}_{\prob}\left\{\frac{\rho_n(\log n)^{2\xi}}{\Delta_n} + \frac{(\log n)^\xi}{n}\right\}.
\end{align*}
Recall that by Result \ref{result:Noise_matrix_concentration}, under Assumption \ref{assumption:Eigenvector_delocalization}, $\alpha_n = \widetilde{O}_{\prob}\{(\log n)^\xi\}$. Recalling the definition in \eqref{eqn:plug_in_estimate_variance}, we use $(A_{ij} - \widehat{p}_{ij})^2 = E_{ij}^2 - 2E_{ij}(\hat{p}_{ij} - p_{ij}) + (\hat{p}_{ij} - p_{ij})^2$ and we can write
\bea\label{ssdecompose}
\widehat{s}_{ik}^2 - s_{ik}^2
&
= \sum_{j = 1}^n\left(\frac{E_{ij}^2\widehat{u}_{jk}^2}{\widehat{\lambda}_k^2} - \frac{\sigma_{ij}^2u_{jk}^2}{\lambda_k^{2}}\right) - 2\sum_{j = 1}^n\frac{E_{ij}(\widehat{p}_{ij} - p_{ij})\widehat{u}_{jk}^2}{\widehat{\lambda}_k^2} + \sum_{j = 1}^n\frac{(\widehat{p}_{ij} - p_{ij})^2\widehat{u}_{jk}^2}{\widehat{\lambda}_k^2}
\\
&\quad = \sum_{j = 1}^n\left(E_{ij}^2 - \sigma_{ij}^2\right)\frac{u_{jk}^2}{\lambda_k^{2}} + \sum_{j = 1}^nE_{ij}^2\left(\frac{\widehat{u}_{jk}^2}{\widehat{\lambda}_k^2} - \frac{u_{jk}^2}{\lambda_k^{2}}\right) - 2\sum_{j = 1}^n\frac{E_{ij}(\widehat{p}_{ij} - p_{ij})\widehat{u}_{jk}^2}{\widehat{\lambda}_k^2}\\
&\quad\quad  + \sum_{j = 1}^n\frac{(\widehat{p}_{ij} - p_{ij})^2\widehat{u}_{jk}^2}{\widehat{\lambda}_k^2}
\\
&\quad = \sum_{j = 1}^n\left(E_{ij}^2 - \sigma_{ij}^2\right)\frac{u_{jk}^2}{\lambda_k^{2}} + \sum_{j = 1}^nE_{ij}^2\left(\frac{\widehat{u}_{jk}^2}{\widehat{\lambda}_k^2} - \frac{u_{jk}^2}{\lambda_k^{2}}\right) - 2\sum_{j = 1}^n\frac{E_{ij}(\widehat{p}_{ij} - p_{ij})\widehat{u}_{jk}^2}{\widehat{\lambda}_k^2}\\
&\quad\quad  +  \widetilde{O}_{\prob}\left\{ \frac{\rho_n(\log n)^{2\xi}}{n\Delta_n^2}\right\},
\eae
where the last equality is by \eqref{eqn:pij_zeroth_order_deviation}. For the second term on the right-hand side of \eqref{ssdecompose}, by Lemma \ref{lemma:Eigenvector_zero_order_deviation}, we have
\begin{align*}
\widehat{u}_{jk}^2 - u_{jk}^2& = \{\widehat{u}_{jk}\mathrm{sgn}(\widehat{\bu}_k\transpose\bu_k) - u_{jk}\}\{\widehat{u}_{jk}\mathrm{sgn}(\widehat{\bu}_k\transpose\bu_k) + u_{jk}\}\\
& = \{\widehat{u}_{jk}\mathrm{sgn}(\widehat{\bu}_k\transpose\bu_k) - u_{jk}\}\{\widehat{u}_{jk}\mathrm{sgn}(\widehat{\bu}_k\transpose\bu_k) - u_{jk} + 2u_{jk}\}\\
& = 2u_{jk}\{\widehat{u}_{jk}\mathrm{sgn}(\widehat{\bu}_k\transpose\bu_k) - u_{jk}\} + \{\widehat{u}_{jk}\mathrm{sgn}(\widehat{\bu}_k\transpose\bu_k) - u_{jk}\}^2\\
& = \frac{2u_{jk}\be_j\transpose\bE\bu_k}{\lambda_k} + 2u_{jk}[\bR_\bU]_{jk}\mathrm{sgn}(\widehat{\bu}_k\transpose\bu_k) + \{\widehat{u}_{jk}\mathrm{sgn}(\widehat{\bu}_k\transpose\bu_k) - u_{jk}\}^2\\
& = \frac{2u_{jk}\be_j\transpose\bE\bu_k}{\lambda_k} + \widetilde{O}_{\prob}\left\{\frac{\rho_n(\log n)^{2\xi}}{\Delta_n^2} + \frac{(\log n)^\xi}{n\Delta_n}\right\},
\end{align*}
where $[\bR_\bU]_{jk}$ is the $(j, k)$th element of $\bR_\bU = \bU_\bA - \bU_\bP\bW^* - \bE\bU_\bP\bS_\bP^{-1}\bW^*$. This entails,
\begin{align*}
\frac{\widehat{u}_{jk}^2}{\widehat{\lambda}_k^2} - \frac{u_{jk}^2}{\lambda_k^2}
& = \frac{\widehat{u}_{jk}^2 - u_{jk}^2}{\lambda_k^2} + \widehat{u}_{jk}^2\left(\frac{1}{\widehat{\lambda}_k^2} - \frac{1}{\lambda_k^2}\right)\\
& = \frac{2u_{jk}\be_j\transpose\bE\bu_k}{\lambda_k^3} + \widehat{u}_{jk}^2\left(\frac{1}{\widehat{\lambda}_k} + \frac{1}{\lambda_k}\right)\left(\frac{\lambda_k - \widehat{\lambda}_k}{\widehat{\lambda}_k\lambda_k}\right)\\
&\quad + \widetilde{O}_{\prob}\left\{\frac{\rho_n(\log n)^{2\xi}}{\Delta_n^4} + \frac{(\log n)^\xi}{n\Delta_n^3}\right\}\\
& = \frac{2u_{jk}\be_j\transpose\bE\bu_k}{\lambda_k^3} + \widetilde{O}_{\prob}\left\{\frac{\rho_n(\log n)^{2\xi}}{\Delta_n^4} + \frac{(\log n)^\xi}{n\Delta_n^3} + \frac{\rho_n}{\Delta_n^4}\right\}
\\
& = \frac{2u_{jk}\be_j\transpose\bE\bu_k}{\lambda_k^3} + \widetilde{O}_{\prob}\left\{\frac{(\log n)^\xi}{n\Delta_n^3} + \frac{\rho_n(\log n)^{2\xi}}{\Delta_n^4}\right\};
\end{align*}
Note here we have $|\lambda_k - \hat{\lambda}_k|\leq\|\bW^*\bS_\bA - \bS_\bP\bW^*\|_2 = \widetilde{O}_{\prob}\{\rho_n^{1/2}(\log n)^\xi + n\rho_n/\Delta_n\}$ by Result \ref{result:Noise_matrix_concentration} and Lemma \ref{lemma:Remainder_analysis_I}. 

Subsequently, uniformly over $i\in [n]$, we have
\begin{align*}
\sum_{j = 1}^nE_{ij}^2\left(\frac{\widehat{u}_{jk}^2}{\widehat{\lambda}_k^2} - \frac{u_{jk}^2}{\lambda_k^{2}}\right)
& = \sum_{j = 1}^nE_{ij}^2\frac{2u_{jk}\be_j\transpose\bE\bu_k}{\lambda_k^3} + \|\bE\|_{2\to\infty}^2\widetilde{O}_{\prob}\left\{\frac{(\log n)^\xi}{n\Delta_n^3} + \frac{\rho_n (\log n)^{2\xi}}{\Delta_n^4}\right\}\\
& = 2\sum_{j = 1}^n\sum_{a = 1}^n\frac{E_{ij}^2E_{ja}u_{ak}u_{jk}}{\lambda_k^3} +\widetilde{O}_{\prob}\left\{
\frac{\rho_n(\log n)^{\xi}}{\Delta_n^3} + \frac{n\rho_n^2 (\log n)^{2\xi}}{\Delta_n^4}
\right\};
\end{align*}
Here, we use  $\|\bE\|_{2\to\infty}\leq \|\bE\|_2 = \widetilde{O}_{\prob}(\sqrt{n\rho_n})$ by Result \ref{result:Noise_matrix_concentration}. 
For the third term of \eqref{ssdecompose}, we write it as follows using \eqref{defrij},
\begin{align*}
&\sum_{j = 1}^nE_{ij}(\widehat{p}_{ij} - p_{ij})\frac{\widehat{u}_{jk}^2}{\widehat{\lambda}_k^{2}}
\\
&\quad = \sum_{j = 1}^nE_{ij}(\widehat{p}_{ij} - p_{ij})\frac{{u}_{jk}^2}{\lambda_k^{2}} + \sum_{j = 1}^nE_{ij}(\widehat{p}_{ij} - p_{ij})\left(\frac{\widehat{u}_{jk}^2}{\widehat{\lambda}_k^{2}} - \frac{u_{jk}^2}{\lambda_k^{2}}\right)\\
&\quad = \sum_{j = 1}^nE_{ij}\left(\be_i\transpose\bU_\bP\bU_\bP\transpose\bE\be_j + \be_i\transpose\bE\bU_\bP\bU_\bP\transpose\be_j + r_{p_{ij}}\right)\frac{u_{jk}^2}{\lambda_k^{2}}\\
&\quad\quad + \sum_{j = 1}^nE_{ij}(\widehat{p}_{ij} - p_{ij})\left(\frac{\widehat{u}_{jk}^2}{\widehat{\lambda}_k^{2}} - \frac{u_{jk}^2}{\lambda_k^{2}}\right)\\
&\quad = \sum_{j = 1}^n\sum_{a = 1}^n\frac{E_{ij}E_{ja}u_{jk}^2\be_i\transpose\bU_\bP\bU_\bP\transpose\be_a}{\lambda_k^{2}} + \left(\sum_{j = 1}^n\frac{E_{ij}u_{jk}^2\be_j\transpose\bU_\bP}{\lambda_k^2}\right)\left(\sum_{a = 1}^nE_{ia}\bU_\bP\transpose\be_a\right)\\
&\quad\quad + \sum_{j = 1}^n\frac{E_{ij}r_{p_{ij}}u_{jk}^2}{\lambda_k^{2}} + \sum_{j = 1}^nE_{ij}(\widehat{p}_{ij} - p_{ij})\left(\frac{\widehat{u}_{jk}^2}{\widehat{\lambda}_k^{2}} - \frac{u_{jk}^2}{\lambda_k^{2}}\right).
\end{align*}
 The first term above is $\widetilde{O}_{\prob}\{\rho_n(\log n)^{2\xi}/(n\Delta_n^2)\}$, uniformly over $i\in[n]$, by following the argument for Lemma \ref{lemma:Rowwise_higher_order_concentration} and a union bound over $i\in[n]$. The second term above is $\widetilde{O}_{\prob}\{\rho_n(\log n)^{2\xi}/(n\Delta_n^2)\}$, uniformly over $i\in[n]$, by Result~\ref{result:concentration}, Equation~\eqref{result:concentration:res2} and a union bound over $i\in[n]$. For the third term,
we apply Cauchy-Schwarz inequality, Result \ref{result:Noise_matrix_concentration}, and the expansion of $\widehat{p}_{ij} - p_{ij}$ to obtain
\begin{align*}
\left|\sum_{j = 1}^n\frac{E_{ij}r_{p_{ij}}{u}_{jk}^2}{\widehat{\lambda}_k^2}\right|
&\leq \left(\sum_{j = 1}^nE_{ij}^2\right)^{1/2}\left(\sum_{j = 1}^n\frac{r_{p_{ij}}^2u_{jk}^4}{\lambda_k^{4}}\right)^{1/2}
 \\
&\quad\leq \|\bE\|_{2\to\infty}\sqrt{n}\max_{i,j\in [n]}|r_{p_{ij}}|\|\bu_k\|_\infty^2\frac{1}{\lambda_k^{2}}\\
&\quad
= \widetilde{O}_{\prob}\left\{\frac{\rho_n^{1/2}(\log n)^{\xi}}{n\Delta_n^2} + \frac{\rho_n^{3/2}(\log n)^{2\xi}}{\Delta_n^3}\right\},
\end{align*}
uniformly over $i\in[n]$. 
The fourth term can also be bounded using Cauchy-Schwarz inequality, Result \ref{result:Noise_matrix_concentration}, and \eqref{eqn:pij_zeroth_order_deviation}:
\begin{align*}
\left|\sum_{j = 1}^nE_{ij}(\widehat{p}_{ij} - p_{ij})\left(\frac{\widehat{u}_{jk}^2}{\widehat{\lambda}_k^2} - \frac{u_{jk}^2}{\widehat{\lambda}_k^2}\right)\right|
&\leq \left(\sum_{j = 1}^nE_{ij}^2\right)^{1/2}\sqrt{n}\max_{i,j\in [n]}|\widehat{p}_{ij} - p_{ij}|\max_{i,j\in [n]}\left|\frac{\widehat{u}_{jk}^2}{\widehat{\lambda}_k^2} - \frac{u_{jk}^2}{\widehat{\lambda}_k^2}\right|\\
&\leq \|\bE\|_{2\to\infty}\sqrt{n}\max_{i,j\in [n]}|\widehat{p}_{ij} - p_{ij}|\max_{i,j\in [n]}\left|\frac{\widehat{u}_{jk}^2}{\widehat{\lambda}_k^2} - \frac{u_{jk}^2}{\widehat{\lambda}_k^2}\right|\\
&\leq \|\bE\|_{2}\sqrt{n}\max_{i,j\in [n]}|\widehat{p}_{ij} - p_{ij}|\max_{i,j\in [n]}\left|\frac{\widehat{u}_{jk}^2}{\widehat{\lambda}_k^2} - \frac{u_{jk}^2}{\widehat{\lambda}_k^2}\right|\\
& = \widetilde{O}_{\prob}\left\{\frac{\rho_n^{3/2}(\log n)^{2\xi}}{\Delta_n^3}\right\}
.
\end{align*}
Summarizing all results above and by our standard assumptions, we conclude that
\begin{align*}\color{black}
&\widehat{s}_{ik}^2 - s_{ik}^2\\
&\quad = \sum_{j = 1}^n\frac{(E_{ij}^2 - \sigma_{ij}^2)u_{jk}^2}{\lambda_k^{2}} + \sum_{a = 1}^n\sum_{b = 1}^n\frac{2E_{ia}^2E_{ab}u_{ak}u_{bk}}{\lambda_k^3}\\
&\quad\quad + \widetilde{O}_{\prob}\Bigg\{
\frac{\rho_n(\log n)^{\xi}}{\Delta_n^3} + \frac{n\rho_n^2 (\log n)^{2\xi}}{\Delta_n^4}
+ {\frac{\rho_n(\log n)^{2\xi}}{n\Delta_n^2}}  + \frac{\rho_n^{1/2}(\log n)^{\xi}}{n\Delta_n^2} +{ \frac{\rho_n^{3/2}(\log n)^{2\xi}}{\Delta_n^3}}\Bigg\}
\\
&\quad\quad + \widetilde{O}_{\prob}\Bigg\{
\frac{\rho_n^{3/2}(\log n)^{2\xi}}{\Delta_n^3}
+{\frac{\rho_n(\log n)^{2\xi}}{n\Delta_n^3} }\Bigg\}
\\
&\quad = \sum_{j = 1}^n\frac{(E_{ij}^2 - \sigma_{ij}^2)u_{jk}^2}{\lambda_k^{2}} + \sum_{a = 1}^n\sum_{b = 1}^n\frac{2E_{ia}^2E_{ab}u_{ak}u_{bk}}{\lambda_k^3}\\
&\quad\quad + 
\widetilde{O}_{\prob}\Bigg\{
\frac{\rho_n^{1/2}(\log n)^{2\xi}}{n\Delta_n^2} + \frac{\rho_n(\log n)^{2\xi}}{\Delta_n^3} + \frac{n\rho_n^2(\log n)^{2\xi}}{\Delta_n^4}
\Bigg\},
\end{align*}
uniformly over $i\in[n]$. 
The proof is thus completed by dividing the above equation by $s_{ik}^2$ and observing that $\min_{i\in[n]}s_{ik}^2 = \Theta(\rho_n/\Delta_n^2)$.  

\subsection{Proof Sketch of Theorem \ref{thm:edgeworth_expansion}}
\label{sub:proof_sketch_edgeworth_expansion}
Fix $\xi > 1$. For simplicity, we denote the following quantities, 
\begin{align*}
\eps_n &:=\frac{(\log n)^{3\xi}}{\Delta_n}\max\left\{1,\frac{\sqrt{n\rho_n}}{q_n},\frac{n\rho_n}{\Delta_n},\frac{(n\rho_n)^2}{\Delta_n^2}\right\},\\
\vartheta_n &:= 
\frac{(\log n)^{2\xi}}{\Delta_n} + \frac{n\rho_n(\log n)^{2\xi}}{\Delta_n^2}
+ \frac{(\log n)^{2\xi}}{n\rho_n^{1/2}} + \frac{(\log n)^{2\xi}}{q_n^2}.
\end{align*}
By Results \ref{result:Noise_matrix_moment_bound} and \ref{result:concentration}, and Lemma \ref{lemma:Rowwise_higher_order_concentration}, we have the following results,
\begin{equation}
\label{eqn:sik_decomposition_error_bound}
\begin{aligned}
\sum_{j = 1}^n\frac{(E_{ij}^2 - \sigma_{ij}^2)u_{jk}^2}{\lambda_k^2s_{ik}^2} & = \widetilde{O}_{\prob}\left\{\frac{(\log n)^\xi}{q_n}\right\},\\
\sum_{a, b = 1}^n\frac{E_{ia}^2E_{ab}u_{ak}u_{bk}}{\lambda_k^3s_{ik}^2} & = \widetilde{O}_{\prob}\left\{\frac{(n\rho_n)^{1/2}(\log n)^\xi}{\Delta_n}\right\},\\
\frac{\bv_{ik}\transpose\bE\bu_k}{s_{ik}\lambda_k} & = \widetilde{O}_{\prob}\left\{\frac{(\log n)^\xi}{\sqrt{n}}\right\},\\
\frac{\be_i\transpose(\bE^2 - \expect\bE^2)\bu_k}{\lambda_k^2s_{ik}}& = \widetilde{O}_{\prob}\left\{\frac{(n\rho_n)^{1/2}(\log n)^{2\xi}}{\Delta_n}\right\}.
\end{aligned}
\end{equation}
Note that the second equation in \eqref{eqn:sik_decomposition_error_bound} can be derived as follows:
\begin{align*}
\sum_{a, b = 1}^n\frac{E_{ia}^2E_{ab}u_{ak}u_{bk}}{\lambda_k^3s_{ik}^2}
& = \sum_{a,b\in[n]\backslash\{i\}}^n\frac{E_{ia}^2E_{ab}u_{ak}u_{bk}}{\lambda_k^3s_{ik}^2} + \sum_{b\in[n]/\{i\}}\frac{E_{ii}^2E_{ib}u_{ik}u_{bk}}{\lambda_k^3s_{ik}^2} + \sum_{a = 1}^n\frac{E_{ia}^3u_{ak}u_{ik}}{\lambda_k^3s_{ik}^2},
\end{align*}
where, by Result~\ref{result:concentration}, Equations~\eqref{result:concentration:res1} and \eqref{result:concentration:res2}, we have
\begin{align*}
\sum_{b\in[n]/\{i\}}\frac{E_{ii}^2E_{ib}u_{ik}u_{bk}}{\lambda_k^3s_{ik}^2} &= \widetilde{O}_{\prob}\left\{\frac{(n\rho_n)^{1/2}(\log n)^\xi}{q_n^2\Delta_n}\right\} = \widetilde{O}_{\prob}\left\{\frac{(n\rho_n)^{1/2}}{q_n\Delta_n}\right\},\\
\left|\sum_{a = 1}^n\frac{E_{ia}^3u_{ak}u_{ik}}{\lambda_k^3s_{ik}^2}\right|
& \leq \max_{i,j\in[n]}|E_{ij}|\left\{\sum_{a = 1}^n\frac{(E_{ia}^2 - \expect E_{ia}^2)\|\bu_k\|_\infty^2}{\lambda_k^3s_{ik}^2} + \sum_{a = 1}^n\frac{\expect E_{ia}^2\|\bu_k\|_\infty^2}{\lambda_k^3s_{ik}^2}\right\}\\
& = \widetilde{O}_{\prob}\left\{\frac{(n\rho_n)^{1/2}}{q_n}\right\}\left[O\left\{\frac{(\log n)^\xi}{\Delta_n q_n} + \frac{1}{\Delta_n}\right\}\right]
 = \widetilde{O}_{\prob}\left\{\frac{(n\rho_n)^{1/2}}{q_n\Delta_n}\right\},
\end{align*}
and by the independence between $(E_{ia}:a\in[n]\backslash\{i\})$ and $(E_{ab}:a,b\in[n]\backslash\{i\})$, Result \ref{result:Noise_matrix_concentration}, Result~\ref{result:concentration}, Equations~\eqref{result:concentration:res1} and \eqref{result:concentration:res2} (with $l=1$ and $m = (1+n)n/2$ therein), and the condition that $\sqrt{n}/q_n^2 = n^{1/2 - 2\eps} = o(1)$, we have
\begin{align*}
\sum_{a,b\in[n]\backslash\{i\}}\frac{E_{ia}^2E_{ab}u_{ak}u_{bk}}{\lambda_k^3s_{ik}^2}
& = \widetilde{O}_{\prob}\left\{\frac{(\log n)^\xi}{\rho_n^{1/2}\Delta_n}\max_{i,j\in[n]}E_{ij}^2\right\} = \widetilde{O}_{\prob}\left\{\frac{(n\rho_n)^{1/2}(\log n)^\xi}{\Delta_n}\right\}.
\end{align*}
For the fourth equation in \eqref{eqn:sik_decomposition_error_bound}, 
we used the result $\expect\bE^2$ is a diagonal matrix with $\|\expect\bE^2\|_2 = \|\expect\bE^2\|_\infty = O(n\rho_n)$. 
By Theorem \ref{thm:bias_corrected_ASE} and Lemma \ref{lemma:variance_expansion}, we have
\begin{align*}
T_{ik}& = \left\{\frac{\be_i\transpose\bE\bu_k}{s_{ik}\lambda_k} + \frac{\bv_{ik}\transpose\bE\bu_k}{s_{ik}\lambda_k} + \frac{\be_i\transpose(\bE^2 -\expect\bE^2)\bu_k}{s_{ik}\lambda_k^{2}} + \widetilde{O}_{\prob}(\eps_n)\right\}
\left(1 + \frac{\widehat{s}_{ik}^2 - s_{ik}^2}{s_{ik}^2}\right)^{-1/2}\\
& = \left\{\frac{\be_i\transpose\bE\bu_k}{s_{ik}\lambda_k} + \frac{\bv_{ik}\transpose\bE\bu_k}{s_{ik}\lambda_k} + \frac{\be_i\transpose(\bE^2 -\expect\bE^2)\bu_k}{s_{ik}\lambda_k^{2}} + \widetilde{O}_{\prob}(\eps_n)\right\}
\\
&\quad\times\left[1 - \frac{\widehat{s}_{ik}^2 - s_{ik}^2}{2s_{ik}^2} + \widetilde{O}_{\prob}\left\{\frac{n\rho_n(\log n)^{2\xi}}{\Delta_n^2} + \frac{(\log n)^{2\xi}}{q_n^2}\right\}\right]\\
& = \left\{\frac{\be_i\transpose\bE\bu_k}{s_{ik}\lambda_k} + \frac{\bv_{ik}\transpose\bE\bu_k}{s_{ik}\lambda_k} + \frac{\be_i\transpose(\bE^2 -\expect\bE^2)\bu_k}{s_{ik}\lambda_k^{2}} + \widetilde{O}_{\prob}(\eps_n)\right\}\\
&\quad\times\left\{1 - \frac{1}{2}\sum_{j = 1}^n\frac{(E_{ij}^2 - \sigma_{ij}^2)u_{jk}^2}{s_{ik}^2\lambda_k^{2}} - \sum_{a, b = 1}^n\frac{E_{ia}^2E_{ab}u_{ak}u_{bk}}{\lambda_k^3s_{ik}^2} + \widetilde{O}_{\prob}(\vartheta_n)\right\}\\
& = \frac{\be_i\transpose\bE\bu_k}{s_{ik}\lambda_k} + \frac{\bv_{ik}\transpose\bE\bu_k}{s_{ik}\lambda_k} + \frac{\be_i\transpose(\bE^2 -\expect\bE^2)\bu_k}{s_{ik}\lambda_k^{2}} - \frac{1}{2}\frac{\be_i\transpose\bE\bu_k}{s_{ik}\lambda_k}\sum_{j = 1}^n\frac{(E_{ij}^2 - \sigma_{ij}^2)u_{jk}^2}{s_{ik}^2\lambda_k^{2}}\\
&\quad - \frac{\be_i\transpose\bE\bu_k}{s_{ik}\lambda_k}\sum_{a, b = 1}^n\frac{E_{ia}^2E_{ab}u_{ak}u_{bk}}{\lambda_k^3s_{ik}^2}\\
&\quad + \widetilde{O}_{\prob}\left\{\frac{(\log n)^\xi}{\sqrt{n}}\right\}\widetilde{O}_{\prob}\left\{\frac{(n\rho_n)^{1/2}(\log n)^\xi}{\Delta_n} + \frac{(\log n)^\xi}{q_n}\right\}\\
&\quad + \widetilde{O}_{\prob}\left\{\frac{(n\rho_n)^{1/2}(\log n)^{2\xi}}{\Delta_n}\right\}\widetilde{O}_{\prob}\left\{\frac{(n\rho_n)^{1/2}(\log n)^\xi}{\Delta_n} + \frac{(\log n)^\xi}{q_n}\right\}\\
&\quad + \widetilde{O}_{\prob}\{\eps_n + (\log n)^\xi\vartheta_n\}\\
&= \frac{\be_i\transpose\bE\bu_k}{s_{ik}\lambda_k} + \frac{\bv_{ik}\transpose\bE\bu_k}{s_{ik}\lambda_k} + \frac{\be_i\transpose(\bE^2 -\expect\bE^2)\bu_k}{s_{ik}\lambda_k^{2}} - \frac{\be_i\transpose\bE\bu_k}{s_{ik}\lambda_k}\sum_{a, b = 1}^n\frac{E_{ia}^2E_{ab}u_{ak}u_{bk}}{\lambda_k^3s_{ik}^2}\\
&\quad - \frac{1}{2}\frac{\be_i\transpose\bE\bu_k}{s_{ik}\lambda_k}\sum_{j = 1}^n\frac{(E_{ij}^2 - \sigma_{ij}^2)u_{jk}^2}{s_{ik}^2\lambda_k^{2}} + \widetilde{O}_{\prob}\{\eps_n + (\log n)^\xi\vartheta_n\}.
\end{align*}
We write $\bv = \bv_{ik}$ for simplicity and let $v_j$ be the $j$th entry of $\bv$. Then we have
\begin{align*}
&\frac{\bv\transpose\bE\bu_k}{s_{ik}\lambda_k}\\
&\quad = \sum_{a, b \in [n]\backslash\{i\}}\frac{E_{ab}v_au_{bk}}{s_{ik}\lambda_k} + \left(\sum_{b\in [n]\backslash\{i\}}\frac{E_{ib}v_iu_{bk}}{s_{ik}\lambda_k} + \sum_{a\in [n]\backslash\{i\}}\frac{E_{ai}v_au_{ik}}{s_{ik}\lambda_k} + \frac{v_iE_{ii}u_{ik}}{s_{ik}\lambda_k}\right)\\
&\quad
= \sum_{a, b \in [n]\backslash\{i\}}\frac{E_{ab}v_au_{bk}}{s_{ik}\lambda_k} + \widetilde{O}_{\prob}\left\{\frac{(\log n)^\xi}{n}\right\}
\end{align*}
by Result \ref{result:concentration}, Equation \eqref{result:concentration:res2} since $\max_{a,b\in [n]}|v_au_{bk}|/|s_{ik}\lambda_k| = O(n^{-3/2}\rho_n^{-1/2})$ and the second term in the parenthesis is a sum of independent random variables. Also, we have
\begin{align*}
\frac{\be_i\transpose(\bE^2 - \expect\bE^2)\bu_k}{s_{ik}\lambda_k^{2}}
& = \sum_{a,b = 1}^n\frac{E_{ia}E_{ab}u_{bk}}{s_{ik}\lambda_k^{2}} - \sum_{a = 1}^n\frac{ \expect E_{ia}^2 u_{ik}}{s_{ik}\lambda_k^2}\\
& = \sum_{a,b\in [n]\backslash\{i\}}\frac{E_{ia}E_{ab}u_{bk}}{s_{ik}\lambda_k^{2}} + \sum_{b\neq i}\frac{E_{ii}E_{ib}u_{bk}}{s_{ik}\lambda_k^{2}} + \sum_{a = 1}^n\frac{(E_{ia}^2 - \expect E_{ia}^2)u_{ik}}{s_{ik}\lambda_k^2}\\
& = \sum_{a,b\in [n]\backslash\{i\}}\frac{E_{ia}E_{ab}u_{bk}}{s_{ik}\lambda_k^{2}} + \sum_{b\neq i}\frac{E_{ii}E_{ib}u_{bk}}{s_{ik}\lambda_k^{2}} + \widetilde{O}_{\prob}\left\{\frac{(n\rho_n)^{1/2}(\log n)^\xi}{q_n\Delta_n}\right\}
\\
& = \sum_{a,b\in [n]\backslash\{i\}}\frac{E_{ia}E_{ab}u_{bk}}{s_{ik}\lambda_k^{2}}  + \widetilde{O}_{\prob}\left\{\frac{(n\rho_n)^{1/2}(\log n)^\xi}{q_n\Delta_n}\right\}
\end{align*}
where the third equality is by Equation \eqref{result:concentration:res2} of Result~\ref{result:concentration}, and for the last equality, since $|E_{ii}| = \widetilde{O}(\sqrt{n\rho_n}/q_n)$ by Equation \eqref{result:concentration:res1} of Result~\ref{result:concentration}, it follows from Equation \eqref{result:concentration:res2} of Result~\ref{result:concentration} that $\sum_{b\neq i}E_{ii}E_{ib}u_{bk}/(s_{ik}\lambda_k^{2}) = \widetilde{O}_{\prob}\{\sqrt{n\rho_n}(\log n)^{\xi}/(q_n\Delta_n)\}$. Also, 
\begin{align*}
&\frac{\be_i\transpose\bE\bu_k}{s_{ik}\lambda_k}\sum_{a, b = 1}^n\frac{E_{ia}^2E_{ab}u_{ak}u_{bk}}{\lambda_k^3s_{ik}^2}\\
&\quad = \frac{\be_i\transpose\bE\bu_k}{s_{ik}\lambda_k}\left(\sum_{a = 1}^n\sum_{b\in[n]/\{i\}}\frac{E_{ia}^2E_{ab}u_{ak}u_{bk}}{\lambda_k^3s_{ik}^2} + \sum_{a = 1}^n\frac{E_{ia}^3u_{ak}u_{ik}}{\lambda_k^3s_{ik}^2}\right)\\
&\quad = \frac{\be_i\transpose\bE\bu_k}{s_{ik}\lambda_k}\left(\sum_{a,b\in[n]/\{i\}}\frac{E_{ia}^2E_{ab}u_{ak}u_{bk}}{\lambda_k^3s_{ik}^2} + \sum_{b\in[n]/\{i\}}\frac{E_{ii}^2E_{ib}u_{ik}u_{bk}}{\lambda_k^3s_{ik}^2} + \sum_{a = 1}^n\frac{E_{ia}^3u_{ak}u_{ik}}{\lambda_k^3s_{ik}^2}\right).
\end{align*}
By Result~\ref{result:concentration}, we have
\begin{align*}
\sum_{b\in[n]/\{i\}}\frac{E_{ii}^2E_{ib}u_{ik}u_{bk}}{\lambda_k^3s_{ik}^2} &= \widetilde{O}_{\prob}\left\{\frac{(n\rho_n)^{1/2}(\log n)^\xi}{q_n^2\Delta_n}\right\} = \widetilde{O}_{\prob}\left\{\frac{(n\rho_n)^{1/2}}{q_n\Delta_n}\right\},
\end{align*} 

\begin{align*}
\left|\sum_{a = 1}^n\frac{E_{ia}^3u_{ak}u_{ik}}{\lambda_k^3s_{ik}^2}\right|
& \leq \max_{i,j\in[n]}|E_{ij}|\left\{\sum_{a = 1}^n\frac{(E_{ia}^2 - \expect E_{ia}^2)\|\bu_k\|_\infty^2}{\lambda_k^3s_{ik}^2} + \sum_{a = 1}^n\frac{\expect E_{ia}^2\|\bu_k\|_\infty^2}{\lambda_k^3s_{ik}^2}\right\}\\
& = \widetilde{O}_{\prob}\left\{\frac{(n\rho_n)^{1/2}}{q_n}\right\}\left[O\left\{\frac{(\log n)^\xi}{\Delta_n q_n}\right\} + O\left(\frac{1}{\Delta_n}\right)\right]
 = \widetilde{O}_{\prob}\left\{\frac{(n\rho_n)^{1/2}}{q_n\Delta_n}\right\}.
\end{align*}

This entails that
\begin{align*}
&\frac{\be_i\transpose\bE\bu_k}{s_{ik}\lambda_k}\sum_{a, b = 1}^n\frac{E_{ia}^2E_{ab}u_{ak}u_{bk}}{\lambda_k^3s_{ik}^2}\\
&\quad = \left(\sum_{a,b\in[n]/\{i\}}\frac{\be_i\transpose\bE\bu_kE_{ia}^2E_{ab}u_{ak}u_{bk}}{\lambda_k^4s_{ik}^3}\right) + \widetilde{O}_{\prob}\left\{\frac{(n\rho_n)^{1/2}(\log n)^\xi}{q_n\Delta_n}\right\},
\end{align*}
 where ${\be_i\transpose\bE\bu_k}/({s_{ik}\lambda_k})$ is bounded by Result~\ref{result:Noise_matrix_rowwise_concentration}.  Recall that
\begin{align*}
T_{ik}^\sharp &= \frac{\be_i\transpose\bE\bu_k}{s_{ik}\lambda_k} = \sum_{j = 1}^n\frac{E_{ij}u_{jk}}{s_{ik}\lambda_k},\quad
\delta_{ik} = \sum_{j = 1}^n\frac{(E_{ij}^2 - \sigma_{ij}^2)u_{jk}^2}{s_{ik}^2\lambda_k^{2}},\\
\Delta_{ik} &= \sum_{a, b\in [n]\backslash\{i\}}\left\{\frac{(E_{ia}/\lambda_k {+} v_a)u_{bk}}{s_{ik}\lambda_k} -
\frac{\be_i\transpose\bE\bu_k E_{ia}^2u_{ak}u_{bk}}{\lambda_k^4s_{ik}^3}\right\}E_{ab}.
\end{align*}
We then obtain
\begin{align}\label{eqn:Tik_decomposition}
T_{ik}
& = T_{ik}^\sharp + \Delta_{ik} - \frac{1}{2}T_{ik}^\sharp\delta_{ik} + \widetilde{O}_{\prob}\left\{\eps_n + (\log n)^\xi\vartheta_n\right\}.
\end{align}
Observe that $T_{ik}^\sharp$ and $\delta_{ik}$ are functions of $\be_i\transpose\bE$, and given $\be_i\transpose\bE$, $\Delta_{ik}$ is a sum of independent mean-zero random variables. Write
\begin{align*}
\Delta_{ik}& = \sum_{a < b,a,b\neq i}\left(v_a + \frac{E_{ia}}{\lambda_k} 
-
\frac{\be_i\transpose\bE\bu_k E_{ia}^2u_{ak}}{\lambda_k^3s_{ik}^2}\right)\frac{E_{ab}u_{bk}}{s_{ik}\lambda_k}\\
&\quad + \sum_{a < b,a,b\neq i}\left(v_b + \frac{E_{ib}}{\lambda_k} 
-
\frac{\be_i\transpose\bE\bu_k E_{ib}^2u_{bk}}{\lambda_k^3s_{ik}^2}\right)\frac{E_{ab}u_{ak}}{s_{ik}\lambda_k}\\
&\quad + \sum_{a = 1,a\neq i}^n\left(v_a + \frac{E_{ia}}{\lambda_k} 
-
\frac{\be_i\transpose\bE\bu_k E_{ia}^2u_{ak}}{\lambda_k^3s_{ik}^2}\right)\frac{E_{aa}u_{ak}}{s_{ik}\lambda_k}\\
& = \sum_{a < b,a,b\neq i}\left\{v_au_{bk} + v_bu_{ak} + \frac{E_{ia}u_{bk} + E_{ib}u_{ak}}{\lambda_k}
-
\frac{\be_i\transpose\bE\bu_k (E_{ia}^2 + E_{ib}^2)u_{ak}u_{bk}}{\lambda_k^3s_{ik}^2}\right\}\frac{E_{ab}}{s_{ik}\lambda_k}\\
&\quad + \sum_{a\neq i}\left(v_a + \frac{E_{ia}}{\lambda_k} 
-
\frac{\be_i\transpose\bE\bu_k E_{ia}^2u_{ak}}{\lambda_k^3s_{ik}^2}\right)\frac{E_{aa}u_{ak}}{s_{ik}\lambda_k}.
\end{align*}
We further define
\begin{equation}\label{def:tildeT}
\widetilde{T}_{ik} = T_{ik}^\sharp - \frac{1}{2}T_{ik}^\sharp\delta_{ik},\quad \sigma_{ik}^2(\be_i\transpose\bE) = \var(\Delta_{ik}\mid\be_i\transpose\bE),
\end{equation}
and an independent random variable 
$$
\widetilde{\Delta}_{ik}\sim\mathrm{N}\big\{0, \sigma_{ik}^2(\be_i\transpose\bE)\big\}\mid \mathrm{given}\,\,\be_i\transpose\bE.
$$ We thus have
\begin{equation}\label{decom:T}
T_{ik}
 = \widetilde{T}_{ik} +  \Delta_{ik} + \widetilde{O}_{\prob}\left\{\eps_n + (\log n)^\xi\vartheta_n\right\},
\end{equation}whose main part's distribution (i.e., the distribution of $\tilde{T}_{ik} + \Delta_{ik}$) will be connected to the distribution of $\tilde{T}_{ik} + \tilde{\Delta}_{ik}$ (c.f.~Section~\ref{sub:finishing_proof_of_edgeworth_expansion}). Now let
\begin{align*}
\bar{\sigma}_{ik}^2(\be_i\transpose\bE)
& = \sum_{a < b,a,b\neq i}\left(v_au_{bk} + v_bu_{ak} + \frac{E_{ia}u_{bk} + E_{ib}u_{ak}}{\lambda_k}\right)^2\frac{\var(E_{ab})}{s_{ik}^2\lambda_k^2}
\\&\quad
 + \sum_{a\neq i}\left(v_a + \frac{E_{ia}}{\lambda_k}\right)^2\frac{\var(E_{aa})u_{ak}^2}{s_{ik}^2\lambda_k^2}.
\end{align*}
By the inequality $3a^2/4 - 3b^2 \leq (a + b)^2\leq 2a^2 + 2b^2$ for any $a, b\in\mathbb{R}$, we obtain
\begin{align*}
\sigma_{ik}^2(\be_i\transpose\bE)
&\leq 2\sum_{a < b,a,b\neq i}\left(v_au_{bk} + v_bu_{ak} + \frac{E_{ia}u_{bk} + E_{ib}u_{ak}}{\lambda_k}\right)^2\frac{\var(E_{ab})}{s_{ik}^2\lambda_k^2}
\\&\quad
 + 2\sum_{a\neq i}\left(v_a + \frac{E_{ia}}{\lambda_k}\right)^2\frac{\var(E_{aa})u_{ak}^2}{s_{ik}^2\lambda_k^2}\\
&\quad + 2\sum_{a < b,a,b\neq i}\left\{\frac{\be_i\transpose\bE\bu_k (E_{ia}^2 + E_{ib}^2)u_{ak}u_{bk}}{\lambda_k^3s_{ik}^2}\right\}^2\frac{\var(E_{ab})}{s_{ik}^2\lambda_k^2}\\
&\quad + 2\sum_{a\neq i}\left(\frac{\be_i\transpose\bE\bu_k E_{ia}^2u_{ak}}{\lambda_k^3s_{ik}^2}\right)^2\frac{\var(E_{aa})u_{ak}^2}{s_{ik}^2\lambda_k^2}\\
& = 2\bar{\sigma}_{ik}^2(\be_i\transpose\bE) + 2\sum_{a < b,a,b\neq i}\left\{\frac{\be_i\transpose\bE\bu_k (E_{ia}^2 + E_{ib}^2)u_{ak}u_{bk}}{\lambda_k^3s_{ik}^2}\right\}^2\frac{\var(E_{ab})}{s_{ik}^2\lambda_k^2}\\
&\quad + 2\sum_{a\neq i}\left(\frac{\be_i\transpose\bE\bu_k E_{ia}^2u_{ak}}{\lambda_k^3s_{ik}^2}\right)^2\frac{\var(E_{aa})u_{ak}^2}{s_{ik}^2\lambda_k^2},
\end{align*}
and
\begin{align*}
\sigma_{ik}^2(\be_i\transpose\bE)
&\geq \frac{3}{4}\bar{\sigma}_{ik}(\be_i\transpose\bE) - 3\sum_{a < b,a,b\neq i}\left\{\frac{\be_i\transpose\bE\bu_k (E_{ia}^2 + E_{ib}^2)u_{ak}u_{bk}}{\lambda_k^3s_{ik}^2}\right\}^2\frac{\var(E_{ab})}{s_{ik}^2\lambda_k^2}\\
&\quad - 3\sum_{a\neq i}\left(\frac{\be_i\transpose\bE\bu_k E_{ia}^2u_{ak}}{\lambda_k^3s_{ik}^2}\right)^2\frac{\var(E_{aa})u_{ak}^2}{s_{ik}^2\lambda_k^2}.
\end{align*}
By Result \ref{result:Noise_matrix_rowwise_concentration},  Result~\ref{result:concentration}, Lemma 3.7, we have
\begin{align*}
&\sum_{a < b,a,b\neq i}\left\{\frac{\be_i\transpose\bE\bu_k (E_{ia}^2 + E_{ib}^2)u_{ak}u_{bk}}{\lambda_k^3s_{ik}^2}\right\}^2\frac{\var(E_{ab})}{s_{ik}^2\lambda_k^2}
 + \sum_{a\neq i}\left(\frac{\be_i\transpose\bE\bu_k E_{ia}^2u_{ak}}{\lambda_k^3s_{ik}^2}\right)^2\frac{\var(E_{aa})u_{ak}^2}{s_{ik}^2\lambda_k^2}\\
&\quad\lesssim (\be_i\transpose\bE\bu_k)^2\sum_{a = 1}^n\sum_{b = 1}^n\frac{\rho_n(E_{ia}^4 + E_{ib}^4)\|\bu_k\|_{\infty}^4}{s_{ik}^6\lambda_k^8}\\
&\quad\lesssim (\be_i\transpose\bE\bu_k)^2\max_{i,j\in[n]}E_{ij}^2\left\{\sum_{a = 1}^n\frac{(E_{ia}^2 - \expect E_{ia}^2)}{\Delta_n^2n\rho_n^2} + \sum_{a = 1}^n\frac{\expect E_{ia}^2}{\Delta_n^2n\rho_n^2}\right\}\\
&\quad = \widetilde{O}_{\prob}\left\{\frac{(\log n)^{2\xi}}{\Delta_n^2q_n^2}\right\}\left\{\sum_{a = 1}^n(E_{ia}^2 - \expect E_{ia}^2) + O(n\rho_n)\right\} = \widetilde{O}_{\prob}\left\{\frac{(n\rho_n)(\log n)^{2\xi}}{\Delta_n^2q_n^2}\right\}.
\end{align*}
Therefore,
\begin{align}\label{bound:sigmaik}
\frac{3}{4}\bar{\sigma}_{ik}^2(\be_i\transpose\bE) - \left|\widetilde{O}_{\prob}\left\{\frac{(n\rho_n)(\log n)^{2\xi}}{\Delta_n^2q_n^2}\right\}\right|\leq \sigma_{ik}^2(\be_i\transpose\bE)\leq 2\bar{\sigma}_{ik}^2(\be_i\transpose\bE) + \widetilde{O}_{\prob}\left\{\frac{(n\rho_n)(\log n)^{2\xi}}{\Delta_n^2q_n^2}\right\}
\end{align}
Also, we have the following expression for $T_{ik}^\sharp\delta_{ik}$:
\[
T_{ik}^\sharp\delta_{ik} = \sum_{a < b}g_{ab}^{(ik)} + \sum_{a = 1}^ng_{aa}^{(ik)},
\]
where
\begin{equation}
\label{eqn:gab_formula}
\begin{aligned}
g_{ab}^{(ik)} = \left\{\begin{aligned}
&\frac{E_{ia}\{E_{ib}^2 - \var(E_{ib})\}u_{ak}u_{bk}^2 + E_{ib}\{E_{ia}^2 - \var(E_{ia})\}u_{ak}^2u_{bk}}{s_{ik}^3\lambda_k^3},&\quad&\mbox{if }a\neq b,\\
&\frac{E_{ia}\{E_{ia}^2 - \var(E_{ia})\}u_{ak}^3}{s_{ik}^3\lambda_k^3},&\quad&\mbox{if }a = b.
\end{aligned}\right.
\end{aligned}
\end{equation}
\begin{theorem}[Esseen's Smoothing Lemma \cite{10.1007/BF02392223}]
\label{thm:esseen_smoothing_lemma}
Let $F(\cdot)$ be the cumulative distribution function of a random variable and $G(\cdot)$ be a function of bounded variation over $\mathbb{R}$ with $G(-\infty) = 0$, $G(+\infty) = 1$. Then there exists constants $C_1, C_2 > 0$, such that for any $T > 0$, we have
\begin{align*}
\|F(x) - G(x)\|_\infty\leq C_1\int_{-T}^T\left|\frac{f(t) - g(t)}{t}\right|\mathrm{d}t + \frac{C_2\sup_{x\in\mathbb{R}}|G'(x)|}{T},
\end{align*}
where $f$ and $g$ are the Fourier-Stieltjes transform of $F$ and $G$, i.e.,
\[
f(t) = \int_{\mathbb{R}}e^{\mathbbm{i}tx}\mathrm{d}F(x),\quad
g(t) = \int_{\mathbb{R}}e^{\mathbbm{i}tx}\mathrm{d}G(x).
\]
\end{theorem}
Below, we will primarily focus on establishing 
\begin{align}\label{eqn:edgeworth_expansion_self_smoothed_version}
\left\|\prob\left(\widetilde{T}_{ik} + \widetilde{\Delta}_{ik}\leq x\right) - G_n^{(ik)}(x)\right\|_\infty = O\left\{\frac{1}{n} + \frac{n\rho_n\log n}{\Delta_n^2} + \frac{(\log n)^4}{q_n^2}\right\},
\end{align}
with
\begin{align}\label{def:Gnikx}
G_n^{(ik)}(x) = \Phi(x) + \frac{(2x^2 + 1)}{6}\phi(x)\sum_{j = 1}^n\frac{\expect E_{ij}^3u_{jk}^3}{s_{ik}^3\lambda_k^3},
\end{align}
using Esseen's smoothing lemma above, which is the most challenging part. 

\subsection{Technical Lemmas for Theorem \ref{thm:edgeworth_expansion}}
\label{sub:technical_lemmas_for_edgeworth_expansion}

\begin{lemma}\label{lemma:Delta_conditional_variance}
Suppose Assumptions \ref{assumption:Signal_strength}--\ref{assumption:Eigenvector_delocalization} hold and $\min_{i\in [n]}s_{ik}^2 = \Theta(\rho_n/\Delta_n^2)$. Then for any $\xi > 1$,
\begin{align*}
\bar{\sigma}_{ik}^2(\be_i\transpose\bE)
& =  \sum_{a = 1}^n\sum_{b = 1}^n\frac{\sigma_{ia}^2\sigma^2_{ab}u_{bk}^2}{s_{ik}^2\lambda_k^{4}}
 + \sum_{a\leq b,a,b\neq i}\frac{(v_au_{bk} + v_bu_{ak})^2z_{ab}^2\sigma^2_{ab}}{s_{ik}^2\lambda_k^{2}}
\\&\quad
  + \widetilde{O}_{\prob}\left\{\frac{\rho_n^{1/2}(\log n)^\xi}{\sqrt{n}\Delta_n} + \frac{(n\rho_n)(\log n)^\xi}{q_n\Delta_n^2}\right\},
\end{align*}
where $z_{ab} = 1$ if $a\neq b$ and $z_{ab} = 1/2$ if $a = b$.
\end{lemma}

\begin{proof}
Recall that
\begin{align*}
\bar{\sigma}_{ik}^2(\be_i\transpose\bE)
& = \sum_{a < b,a,b\neq i}\left(v_au_{bk} + v_bu_{ak} + \frac{E_{ia}u_{bk} + E_{ib}u_{ak}}{\lambda_k}\right)^2\frac{\sigma^2_{ab}}{s_{ik}^2\lambda_k^{2}}\\
&\quad + \sum_{a\neq i}\left(v_au_{ak} + \frac{E_{ia}u_{ak}}{\lambda_k}\right)^2\frac{\sigma^2_{aa}}{s_{ik}^2\lambda_k^{2}}\\
& = \sum_{a \leq b,a,b\neq i}\left(v_au_{bk} + v_bu_{ak} + \frac{E_{ia}u_{bk} + E_{ib}u_{ak}}{\lambda_k}\right)^2z_{ab}^2\frac{\sigma^2_{ab}}{s_{ik}^2\lambda_k^{2}},
\end{align*}
where we define $z_{ab} = 1$ if $a < b$ and $z_{aa} = 1/2$ for all $a,b\in [n]\backslash\{i\},a < b$. Expanding $\sigma_{ik}^2(\be_i\transpose\bE)$, we obtain
\begin{align*}
\bar{\sigma}_{ik}^2(\be_i\transpose\bE)
& = \sum_{a\leq b,a,b\neq i}\frac{z_{ab}^2\sigma^2_{ab}}{\lambda_k^{2}}\left(\frac{E_{ia}^2u_{bk}^2 + E_{ib}^2u_{ak}^2 + 2E_{ia}E_{ib}u_{ak}u_{bk}}{s_{ik}^2\lambda_k^{2}}\right)\\
&\quad + 2\sum_{a\leq b,a,b\neq i}\frac{z_{ab}^2\sigma^2_{ab}}{\lambda_k^{2}}\left\{\frac{(E_{ia}u_{bk} + E_{ib}u_{ak})(v_au_{bk} + v_bu_{ak})}{s_{ik}^2\lambda_k^{}}\right\}\\
&\quad + \sum_{a\leq b,a,b\neq i}\frac{(v_au_{bk} + v_bu_{ak})^2z_{ab}^2\sigma^2_{ab}}{s_{ik}^2\lambda_k^{2}};
\end{align*}
Here we note $\|\mathbf{v}\|_{\infty} = O(n^{-1})$ recalling the definition of $\mathbf{v}$ in \eqref{def:v}. 
For the second term, we can re-write it as
\begin{align*}
&2\sum_{a\leq b,a,b\neq i}\frac{z_{ab}^2\sigma^2_{ab}}{\lambda_k^{2}}\left\{\frac{(E_{ia}u_{bk} + E_{ib}u_{ak})(v_au_{bk} + v_bu_{ak})}{s_{ik}^2\lambda_k}\right\}\\
&\quad = \sum_{a \neq b,a,b\neq i}\frac{\sigma^2_{ab}}{\lambda_k^{2}}\left\{\frac{(E_{ia}u_{bk} + E_{ib}u_{ak})(v_au_{bk} + v_bu_{ak})}{s_{ik}^2\lambda_k}\right\} + 2\sum_{a\neq i}\frac{\sigma^2_{aa}}{\lambda_k^{2}}\frac{E_{ia}v_au_{ak}^2}{s_{ik}^2\lambda_k}\\
&\quad = \sum_{a,b\neq i}\frac{\sigma^2_{ab}}{\lambda_k^{2}}\left\{\frac{(E_{ia}u_{bk} + E_{ib}u_{ak})(v_au_{bk} + v_bu_{ak})}{s_{ik}^2\lambda_k}\right\} - 2\sum_{a\neq i}\frac{\sigma^2_{aa}}{\lambda_k^{2}}\frac{E_{ia}v_au_{ak}^2}{s_{ik}^2\lambda_k^{}}\\
&\quad = \sum_{a\neq i}E_{ia}\sum_{b\neq i}\frac{\sigma^2_{ab}}{\lambda_k^{2}}\left\{\frac{u_{bk}(v_au_{bk} + v_bu_{ak})}{s_{ik}^2\lambda_k}\right\} + \sum_{b\neq i}E_{ib}\sum_{a\neq i}\frac{\sigma^2_{ab}}{\lambda_k^{2}}\left\{\frac{u_{ak}(v_au_{bk} + v_bu_{ak})}{s_{ik}^2\lambda_k}\right\}\\
&\quad\quad - 2\sum_{a\neq i}\frac{\sigma^2_{aa}}{\lambda_k^{2}}\frac{E_{ia}v_au_{ak}^2}{s_{ik}^2\lambda_k}\\
&\quad = 
\widetilde{O}_{\prob}\left\{\sqrt{n\rho_n}(\log n)^\xi\max_{a\in [n]}\left|\sum_{b\neq i}\frac{\sigma^2_{ab}}{\lambda_k^{2}}\left\{\frac{u_{bk}(v_au_{bk} + v_bu_{ak})}{s_{ik}^2\lambda_k}\right\}\right|\right\}\\
&\quad\quad + \widetilde{O}_{\prob}\left\{\sqrt{n\rho_n}(\log n)^\xi\max_{a\in [n]}\left|\frac{\sigma^2_{aa}v_au_{ak}^2}{s_{ik}^2\lambda_k^3}\right|\right\}\\
&\quad = \widetilde{O}_{\prob}\left\{\frac{\rho_n^{1/2}(\log n)^\xi}{\sqrt{n}\Delta_n}\right\},
\end{align*}
where the fourth equality holds by Equation \eqref{result:concentration:res2} of Result~\ref{result:concentration}. 
We now focus on the first term and write
\begin{align*}
&\sum_{a\leq b,a,b\neq i}\frac{z_{ab}^2\sigma^2_{ab}}{\lambda_k^{2}}\left(\frac{E_{ia}^2u_{bk}^2 + E_{ib}^2u_{ak}^2 + 2E_{ia}E_{ib}u_{ak}u_{bk}}{s_{ik}^2\lambda_k^{2}}\right)\\
&\quad = \sum_{a < b,a,b\neq i}\frac{z_{ab}^2\sigma^2_{ab}}{2\lambda_k^{2}}\left(\frac{E_{ia}^2u_{bk}^2 + E_{ib}^2u_{ak}^2 + 2E_{ia}E_{ib}u_{ak}u_{bk}}{s_{ik}^2\lambda_k^{2}}\right)\\
&\quad\quad + \sum_{a > b,a,b\neq i}\frac{z_{ab}^2\sigma^2_{ab}}{2\lambda_k^{2}}\left(\frac{E_{ia}^2u_{bk}^2 + E_{ib}^2u_{ak}^2 + 2E_{ia}E_{ib}u_{ak}u_{bk}}{s_{ik}^2\lambda_k^{2}}\right)
 + \sum_{a\neq i}\frac{z_{aa}^2\sigma^2_{aa}}{\lambda_k^{2}}\left(\frac{4E_{ia}^2u_{ak}^2}{s_{ik}^2\lambda_k^{2}}\right)\\
&\quad = \sum_{a,b\neq i}\frac{z_{ab}^2\sigma^2_{ab}}{2\lambda_k^{2}}\left(\frac{E_{ia}^2u_{bk}^2 + E_{ib}^2u_{ak}^2 + 2E_{ia}E_{ib}u_{ak}u_{bk}}{s_{ik}^2\lambda_k^{2}}\right) + \sum_{a\neq i}\frac{z_{aa}^2\sigma^2_{aa}}{2\lambda_k^{2}}\left(\frac{4E_{ia}^2u_{ak}^2}{s_{ik}^2\lambda_k^{2}}\right)\\
&\quad = 2\sum_{a\neq i}E_{ia}^2\sum_{b\neq i}\frac{z_{ab}^2\sigma^2_{ab}u_{bk}^2}{2s_{ik}^2\lambda_k^{4}} + \sum_{a,b\neq i}\frac{z_{ab}^2\sigma^2_{ab}}{\lambda_k^{2}}\left(\frac{E_{ia}E_{ib}u_{ak}u_{bk}}{s_{ik}^2\lambda_k^{2}}\right)  + \sum_{a\neq i}\frac{z_{aa}^2\sigma^2_{aa}}{2\lambda_k^{2}}\left(\frac{4E_{ia}^2u_{ak}^2}{s_{ik}^2\lambda_k^{2}}\right)\\
&\quad = \sum_{a\neq i}\expect E_{ia}^2\sum_{b\neq i}\frac{z_{ab}^2\sigma^2_{ab}u_{bk}^2}{s_{ik}^2\lambda_k^{4}}  + \sum_{a\neq i}\frac{z_{aa}^2\sigma^2_{aa}}{\lambda_k^{2}}\left(\frac{2\expect E_{ia}^2u_{ak}^2}{s_{ik}^2\lambda_k^{2}}\right)\\
&\quad\quad + \sum_{a\neq i}(E_{ia}^2 - \expect E_{ia}^2)\sum_{b\neq i}\frac{z_{ab}^2\sigma^2_{ab}u_{bk}^2}{s_{ik}^2\lambda_k^{4}} + \sum_{a\neq i}\frac{z_{aa}^2\sigma^2_{aa}}{\lambda_k^{2}}\left\{\frac{2(E_{ia}^2 - \expect E_{ia}^2)u_{ak}^2}{s_{ik}^2\lambda_k^{2}}\right\}\\
&\quad\quad + \sum_{a,b\neq i}\frac{z_{ab}^2\sigma^2_{ab}}{\lambda_k^{2}}\left(\frac{E_{ia}E_{ib}u_{ak}u_{bk}}{s_{ik}^2\lambda_k^{2}}\right)\\
&\quad = \sum_{a\neq i}\expect E_{ia}^2\sum_{b\neq i}\frac{z_{ab}^2\sigma^2_{ab}u_{bk}^2}{s_{ik}^2\lambda_k^{4}}  + \sum_{a\neq i}\frac{z_{aa}^2\sigma^2_{aa}}{\lambda_k^{2}}\left(\frac{2\expect E_{ia}^2u_{ak}^2}{s_{ik}^2\lambda_k^{2}}\right)\\
&\quad\quad + \widetilde{O}_{\prob}\left\{\frac{(n\rho_n)(\log n)^\xi}{q_n}\max_{a\in[n]}\left|\sum_{b\neq i}\frac{z_{ab}^2\sigma^2_{ab}u_{bk}^2}{s_{ik}^2\lambda_k^{4}}\right|\right\}\\
&\quad\quad + \widetilde{O}_{\prob}\left\{\frac{(n\rho_n)(\log n)^\xi}{q_n}\max_{a\in[n]}\left|
\frac{z_{aa}^2\sigma^2_{aa}u_{ak}^2}{s_{ik}^2\lambda_k^{4}}
\right|\right\}\\
&\quad\quad + \sum_{a,b\neq i}\frac{z_{ab}^2\sigma^2_{ab}}{\lambda_k^{2}}\left(\frac{E_{ia}E_{ib}u_{ak}u_{bk}}{s_{ik}^2\lambda_k^{2}}\right)\\
&\quad = \sum_{a\neq i}\expect E_{ia}^2\sum_{b\neq i}\frac{z_{ab}^2\sigma^2_{ab}u_{bk}^2}{s_{ik}^2\lambda_k^{4}}  + \sum_{a\neq i}\frac{z_{aa}^2\sigma^2_{aa}}{\lambda_k^{2}}\left(\frac{2\expect E_{ia}^2u_{ak}^2}{s_{ik}^2\lambda_k^{2}}\right) + \sum_{a,b\neq i}\frac{z_{ab}^2\sigma^2_{ab}}{\lambda_k^{2}}\left(\frac{E_{ia}E_{ib}u_{ak}u_{bk}}{s_{ik}^2\lambda_k^{2}}\right)\\
&\quad\quad
 + \widetilde{O}_{\prob}\left\{\frac{(n\rho_n)(\log n)^\xi}{q_n\Delta_n^2}\right\}\\
&\quad = \sum_{a\neq i} \sigma_{ia}^2 \sum_{b\neq i}\frac{z_{ab}^2\sigma^2_{ab}u_{bk}^2}{s_{ik}^2\lambda_k^{4}} + \sum_{a,b\neq i}\frac{z_{ab}^2\sigma^2_{ab}}{\lambda_k^{2}}\left(\frac{E_{ia}E_{ib}u_{ak}u_{bk}}{s_{ik}^2\lambda_k^{2}}\right)
 + \widetilde{O}_{\prob}\left\{\frac{(n\rho_n)(\log n)^\xi}{q_n\Delta_n^2}\right\},
\end{align*}
where the fifth equality above holds by Equation \eqref{result:concentration:res2} of Result~\ref{result:concentration}. 
By Assumption \ref{assumption:Eigenvector_delocalization} and Assumption \ref{assumption:Noise_matrix_distribution}, we can write the third term above as
\begin{align*}
&\sum_{a,b\neq i}\frac{z_{ab}^2\sigma^2_{ab}}{\lambda_k^{2}}\left(\frac{E_{ia}E_{ib}u_{ak}u_{bk}}{s_{ik}^2\lambda_k^{2}}\right)\\
&\quad = \sum_{a\neq b,a,b\neq i}\frac{\sigma^2_{ab}}{\lambda_k^{2}}\left(\frac{E_{ia}E_{ib}u_{ak}u_{bk}}{s_{ik}^2\lambda_k^{2}}\right)+ \sum_{a\neq i}\frac{\sigma^2_{aa}}{4\lambda_k^{2}}\left(\frac{\expect E_{ia}^2u_{ak}^2}{s_{ik}^2\lambda_k^{2}}\right)
\\&\quad\quad
 + \sum_{a\neq i}\frac{\sigma^2_{aa}}{4\lambda_k^{2}}\left\{\frac{(E_{ia}^2 - \expect E_{ia}^2)u_{ak}^2}{s_{ik}^2\lambda_k^{2}}\right\}\\
&\quad = \widetilde{O}_{\prob}\left\{(n\rho_n)(\log n)^{2\xi}\max_{a,b\in[n]}\left|\frac{\sigma^2_{ab}}{\lambda_k^{2}}\left(\frac{u_{ak}u_{bk}}{s_{ik}^2\lambda_k^{2}}\right)\right|\right\}+ \sum_{a\neq i}\frac{\sigma^2_{aa}}{4\lambda_k^{2}}\left(\frac{\expect E_{ia}^2u_{ak}^2}{s_{ik}^2\lambda_k^{2}}\right)\\
&\quad\quad + \widetilde{O}_{\prob}\left\{\frac{(n\rho_n)(\log n)^\xi}{q_n}\max_{a\in[n]}\left|\frac{\sigma^2_{aa}}{4\lambda_k^{2}}\frac{u_{ak}^2}{s_{ik}^2\lambda_k^{2}}\right|\right\}\\
&\quad = \sum_{a\neq i}\frac{\sigma^2_{aa}}{4\lambda_k^{2}}\left(\frac{\expect E_{ia}^2u_{ak}^2}{s_{ik}^2\lambda_k^{2}}\right) + \widetilde{O}_{\prob}\left\{\frac{\rho_n(\log n)^{2\xi}}{\Delta_n^2}\right\} = \widetilde{O}_{\prob}\left\{\frac{\rho_n(\log n)^{2\xi}}{\Delta_n^2}\right\},
\end{align*}
where the second equality above is due to Equations \eqref{result:concentration:res2} and \eqref{result:concentration:res3} of Result~\ref{result:concentration}. 
Hence, we conclude that
\begin{align}\label{lemmaD3:final}
\bar{\sigma}
_{ik}^2(\be_i\transpose\bE)
& = \sum_{a\leq b,a,b\neq i}\frac{(v_au_{bk} + v_bu_{ak})^2z_{ab}^2\sigma^2_{ab}}{s_{ik}^2\lambda_k^{2}}
 + \sum_{a\neq i}\sum_{b\neq i}\frac{z_{ab}^2\sigma_{ia}^2\sigma^2_{ab}u_{bk}^2}{s_{ik}^2\lambda_k^{4}}
 \\&\nonumber\quad 
  + \widetilde{O}_{\prob}\left\{\frac{(n\rho_n)(\log n)^\xi}{q_n\Delta_n^2} + \frac{\rho_n^{1/2}(\log n)^\xi}{\sqrt{n}\Delta_n}\right\}.
\end{align}
We then observe that 
\begin{align*}
&\sum_{a\neq i}\frac{z_{aa}^2\sigma_{ia}^2\sigma_{aa}^2u_{ak}^2}{s_{ik}^2\lambda_k^4} = O\left(\frac{\rho_n}{\Delta_n^2}\right) = O\left(\frac{\rho^{1/2}_n(\log n)^\xi}{\sqrt{n}\Delta_n}\right),
\\
&\sum_{a\in [n]}\frac{\sigma_{ia}^2\sigma_{ai}^2u_{ik}^2}{s_{ik}^2\lambda_k^4} = O\left(\frac{\rho_n}{\Delta_n^2}\right)= O\left(\frac{\rho^{1/2}_n(\log n)^\xi}{\sqrt{n}\Delta_n}\right),
\\&
\sum_{b\neq i}\frac{\sigma_{ii}^2\sigma_{ib}^2u_{bk}^2}{s_{ik}^2\lambda_k^4} = O\left(\frac{\rho_n}{\Delta_n^2}\right)= O\left(\frac{\rho^{1/2}_n(\log n)^\xi}{\sqrt{n}\Delta_n}\right).
\end{align*}
This entails that 
\begin{align*}
&\sum_{a\neq i}\sum_{b\neq i}\frac{z_{ab}^2\sigma_{ia}^2\sigma^2_{ab}u_{bk}^2}{s_{ik}^2\lambda_k^{4}} 
\\
&= \sum_{a\neq i}\sum_{b\neq i, b\neq a}\frac{\sigma_{ia}^2\sigma^2_{ab}u_{bk}^2}{s_{ik}^2\lambda_k^{4}} + \sum_{a\neq i}\frac{z_{aa}^2\sigma_{ia}^2\sigma^2_{aa}u_{ak}^2}{s_{ik}^2\lambda_k^{4}}
\\
&= \sum_{a\neq i}\sum_{b\neq i}\frac{\sigma_{ia}^2\sigma^2_{ab}u_{bk}^2}{s_{ik}^2\lambda_k^{4}} - \sum_{a\neq i}\frac{3\sigma_{ia}^2\sigma^2_{aa}u_{ak}^2}{4s_{ik}^2\lambda_k^{4}}
\\
&= \sum_{a\in[n]}\sum_{b\in [n]}\frac{\sigma_{ia}^2\sigma^2_{ab}u_{bk}^2}{s_{ik}^2\lambda_k^{4}} - \sum_{a\neq i}\frac{3\sigma_{ia}^2\sigma^2_{aa}u_{ak}^2}{4s_{ik}^2\lambda_k^{4}} - \sum_{b\neq i}\frac{\sigma_{ii}^2\sigma^2_{ib}u_{bk}^2}{s_{ik}^2\lambda_k^{4}} - \sum_{a\in [n]}\frac{\sigma_{ia}^2\sigma_{ai}^2u_{ik}^2}{s_{ik}^2\lambda_k^4}  
\\
& = \sum_{a\in[n]}\sum_{b\in [n]}\frac{\sigma_{ia}^2\sigma^2_{ab}u_{bk}^2}{s_{ik}^2\lambda_k^{4}} + O\left(\frac{\rho^{1/2}_n(\log n)^\xi}{\sqrt{n}\Delta_n}\right),
\end{align*}
which combines with \eqref{lemmaD3:final} results in our target result.
\end{proof}

Lemma \ref{lemma:conditional_Berry_Esseen} below connects the Edgeworth expansion of $\widetilde{T}_{ik} + \Delta_{ik}$ with that of $\widetilde{T}_{ik} + \widetilde{\Delta}_{ik}$. 
\begin{lemma}\label{lemma:conditional_Berry_Esseen}
Suppose Assumptions \ref{assumption:Signal_strength}--\ref{assumption:Eigenvector_delocalization} hold and $\min_{i\in[n]}s_{ik}^2 = \Theta(\rho_n/\Delta_n^2)$. 
If $\beta_\Delta\leq 1/2$, then with $\sigma_{ik}^2(\be_i\transpose\bE) = \var(\Delta_{ik}\mid\be_i\transpose\bE)$, we have
\begin{align*}
\left\|\prob(\Delta_{ik}\leq z\mid\be_i\transpose\bE) - \prob\{\sigma_{ik}(\be_i\transpose\bE)Z\leq z\}\right\|_\infty &= \widetilde{O}_{\prob}\left(\frac{1}{q^2_n}\right),\\
\|F_{\Delta_{ik}}(z) - F_{\widetilde{\Delta}_{ik}}(z)\|_\infty &= O\left(\frac{1}{q_n^2}\right),
\end{align*}
where $Z\sim\mathrm{N}(0, 1)$ is a random variable independent of $\bE$. 
\end{lemma}
\begin{proof}
Recalling that given $\be_i\transpose\bE$, $\Delta_{ik}$ can be written as a sum of independent mean-zero random variables
\begin{align*}
\Delta_{ik} & = \sum_{a < b,a,b\neq i}\left\{v_au_{bk} + v_bu_{ak} + \frac{E_{ia}u_{bk} + E_{ib}u_{ak}}{\lambda_k} - \frac{\be_i\transpose\bE\bu_k (E_{ia}^2 + E_{ib}^2)u_{ak}u_{bk}}{\lambda_k^3s_{ik}^2}\right\}\frac{E_{ab}}{s_{ik}\lambda_k}\\
&\quad + \sum_{a\neq i}\left(v_a + \frac{E_{ia}}{\lambda_k} - \frac{\be_i\transpose\bE\bu_k E_{ia}^2u_{ak}}{\lambda_k^3s_{ik}^2}\right)\frac{E_{aa}u_{ak}}{s_{ik}\lambda_k}.
\end{align*}
Now we consider
\begin{align*}
\frac{\Delta_{ik}}{\sigma_{ik}(\be_i\transpose\bE)}
& = \sum_{a < b,a,b\neq i}\left\{v_au_{bk} + v_bu_{ak} + \frac{E_{ia}u_{bk} + E_{ib}u_{ak}}{\lambda_k} - \frac{\be_i\transpose\bE\bu_k (E_{ia}^2 + E_{ib}^2)u_{ak}u_{bk}}{\lambda_k^3s_{ik}^2}\right\}\\
&\quad\times \frac{E_{ab}}{\sigma_{ik}(\be_i\transpose\bE)s_{ik}\lambda_k}
 + \sum_{a\neq i}\left(v_a + \frac{E_{ia}}{\lambda_k} -\frac{\be_i\transpose\bE\bu_k E_{ia}^2u_{ak}}{\lambda_k^3s_{ik}^2}\right)\frac{E_{aa}u_{ak}}{\sigma_{ik}(\be_i\transpose\bE)s_{ik}\lambda_k}.
\end{align*}
Observe that by Result~\ref{result:concentration}, 
\begin{align*}
\Gamma_n^{(ik)}
& = \sum_{a < b,a,b\neq i}\left|v_au_{bk} + v_bu_{ak} + \frac{E_{ia}u_{bk} + E_{ib}u_{ak}}{\lambda_k} - \frac{\be_i\transpose\bE\bu_k (E_{ia}^2 + E_{ib}^2)u_{ak}u_{bk}}{\lambda_k^3s_{ik}^2}\right|^3
\\&\quad\times 
\expect\left|\frac{E_{ab}}{\sigma_{ik}(\be_i\transpose\bE)s_{ik}\lambda_k}\right|^3
\\&\quad + \sum_{a\neq i}\left|v_a + \frac{E_{ia}}{\lambda_k} - \frac{\be_i\transpose\bE\bu_k E_{ia}^2u_{ak}}{\lambda_k^3s_{ik}^2}\right|^3\expect\left|\frac{E_{aa}u_{ak}}{\sigma_{ik}(\be_i\transpose\bE)s_{ik}\lambda_k}\right|^3\\
&\lesssim 
\sum_{a,b = 1}^n\left\{\|\bv\|_\infty^3\|\bu_k\|_\infty^3 + \frac{(|E_{ia}|^3 + |E_{ib}|^3)\|\bu_k\|_\infty^3}{\lambda_k^3} + |\be_i\transpose\bE\bu_k|^3\frac{(E_{ia}^6 + E_{ib}^6)\|\bu_k\|_\infty^6}{\lambda_k^9s_{ik}^6}\right\}\\
&\quad\times
\frac{1}{\sigma_{ik}(\be_i\transpose\bE)^3}\expect\left|\frac{E_{ab}}{s_{ik}\lambda_k}\right|^3\\
&\lesssim\Bigg\{\frac{n^{5/2}(\|\bv\|_\infty\|\bu_k\|_\infty)^3}{q_n} + \sum_{a,b = 1}^n\frac{|E_{ia}|^3}{\Delta_n^3q_n} + |\be_i\transpose\bE\bu_k|^3\sum_{a,b = 1}^n\frac{E_{ia}^6}{\Delta_n^3\rho_n^3n^{5/2}q_n}\Bigg\}\frac{1}{\sigma_{ik}(\be_i\transpose\bE)^3}\\
& = O\left(\frac{1}{n^2q_n}\right)\frac{1}{\sigma_{ik}(\be_i\transpose\bE)^3} + \frac{\max_{i,j\in[n]}|E_{ij}|}{\sigma_{ik}(\be_i\transpose\bE)^3}\left\{\sum_{a = 1}^n\frac{(E_{ia}^2 - \expect E_{ia}^{2})}{\Delta_n^3q_n} + \sum_{a = 1}^n\frac{\expect E_{ia}^2}{\Delta_n^3q_n}\right\}\\
&\quad + \frac{\max_{i,j\in[n]}|E_{ij}|^4|\be_i\transpose\bE\bu_k|^3}{\sigma_{ik}(\be_i\transpose\bE)^3}\left\{\sum_{a = 1}^n\frac{(E_{ia}^2 - \expect E_{ia}^2)}{\Delta_n^3\rho_n^3n^{5/2}q_n} + \sum_{a = 1}^n\frac{\expect E_{ia}^2}{\Delta_n^3\rho_n^3n^{5/2}q_n}\right\}\\
& = 
O\left\{\frac{(n\rho_n)^{3/2}}{\Delta_n^3}\times\left(\frac{\Delta_n^2}{n^2\rho_n}\right)^{3/2}\times\frac{1}{\sqrt{n}q_n}\right\}
\frac{1}{\sigma_{ik}(\be_i\transpose\bE)^3}\\
&\quad + \widetilde{O}_{\prob}\left\{\frac{(n\rho_n)^{1/2}}{q_n}\times \frac{n\rho_n}{\Delta_n^3q_n}\right\}\frac{1}{\sigma_{ik}(\be_i\transpose\bE)^3}\\
&\quad + \widetilde{O}_{\prob}\left\{\frac{(n\rho_n)^2}{q_n^4}\times\rho_n^{3/2}(\log n)^{3\xi}\times \frac{n\rho_n}{\Delta_n^3\rho_n^3n^{5/2}q_n}\right\}\frac{1}{\sigma_{ik}(\be_i\transpose\bE)^3}\\
& = \widetilde{O}_{\prob}\left\{\frac{(n\rho_n)^{3/2}}{\Delta_n^3q_n^2}\right\}\frac{1}{\sigma_{ik}(\be_i\transpose\bE)^3}.
\end{align*}
By Lemma \ref{lemma:Delta_conditional_variance}, 
\begin{align*}
\bar{\sigma}_{ik}^2(\be_i\transpose\bE)&\geq\sum_{a,b = 1}^n\frac{\sigma_{ia}^2\sigma_{ab}^2u_{bk}^2}{s_{ik}^2\lambda_k^4} + \widetilde{O}_{\prob}\left\{\frac{\rho_n^{1/2}(\log n)^\xi}{\sqrt{n}\Delta_n} + \frac{n\rho_n(\log n)^{\xi}}{q_n\Delta_n^2}\right\}\\
& = \sum_{a = 1}^n\frac{\sigma_{ia}^2}{\lambda_k^2}\sum_{b = 1}^n\frac{\sigma_{ab}^2u_{bk}^2}{s_{ik}^2\lambda_k^2} + \widetilde{O}_{\prob}\left\{\frac{\rho_n^{1/2}(\log n)^\xi}{\sqrt{n}\Delta_n} + \frac{n\rho_n(\log n)^{\xi}}{q_n\Delta_n^2}\right\}\\
& = \sum_{a = 1}^n\frac{\sigma_{ia}^2}{\lambda_k^2}\frac{s_{ak}^2}{s_{ik}^2} + \widetilde{O}_{\prob}\left\{\frac{\rho_n^{1/2}(\log n)^\xi}{\sqrt{n}\Delta_n} + \frac{n\rho_n(\log n)^{\xi}}{q_n\Delta_n^2}\right\}\\
& \geq \min_{a\in [n]}\frac{s_{ak}^2}{s_{ik}^2}\sum_{a = 1}^n\frac{\sigma_{ia}^2}{\lambda_k^2} + \widetilde{O}_{\prob}\left\{\frac{\rho_n^{1/2}(\log n)^\xi}{\sqrt{n}\Delta_n} + \frac{n\rho_n(\log n)^{\xi}}{q_n\Delta_n^2}\right\}\\
& \geq \min_{a\in [n]}\frac{s_{ak}^2}{s_{ik}^2\|\bu_k\|_\infty^2}\sum_{a = 1}^n\frac{\sigma_{ia}^2u_{ak}^2}{\lambda_k^2} + \widetilde{O}_{\prob}\left\{\frac{\rho_n^{1/2}(\log n)^\xi}{\sqrt{n}\Delta_n} + \frac{n\rho_n(\log n)^{\xi}}{q_n\Delta_n^2}\right\}\\
& = \min_{a\in [n]}\frac{s_{ak}^2}{\|\bu_k\|_\infty^2} + \widetilde{O}_{\prob}\left\{\frac{\rho_n^{1/2}(\log n)^\xi}{\sqrt{n}\Delta_n} + \frac{n\rho_n(\log n)^{\xi}}{q_n\Delta_n^2}\right\}\\
& = \Theta\left(\frac{n\rho_n}{\Delta_n^2}\right) + \widetilde{O}_{\prob}\left\{\frac{\rho_n^{1/2}(\log n)^\xi}{\sqrt{n}\Delta_n} + \frac{n\rho_n(\log n)^{\xi}}{q_n\Delta_n^2}\right\}\\
& = \Theta\left(\frac{n\rho_n}{\Delta_n^2}\right) + \widetilde{O}_{\prob}\left\{\frac{n\rho_n}{\Delta_n^2}\left(\frac{\Delta_n^2}{n^2\rho_n}\right)^{1/2}\frac{(\log n)^\xi}{\sqrt{n}} + \frac{n\rho_n(\log n)^{\xi}}{q_n\Delta_n^2}\right\}\\
& = \Theta\left(\frac{n\rho_n}{\Delta_n^2}\right) \quad\mbox{w.h.p.}.
\end{align*}
Therefore, by the Berry-Esseen theorem, we have
\[
\left\|\prob\left\{\frac{\Delta_{ik}}{\sigma_{ik}(\be_i\transpose\bE)}\leq x{\mathrel{\Big|}}\be_i\transpose\bE\right\} - \prob(Z\leq x)\right\|_\infty\leq C\Gamma_n^{(ik)}\quad\mbox{w.h.p.}.
\]
Since 
\[
\Gamma_n^{(ik)} = \widetilde{O}_{\prob}\left\{\frac{(n\rho_n)^{3/2}}{\sigma_{ik}(\be_i\transpose\bE)^3\Delta_n^3q_n^4}\right\} = \widetilde{O}_{\prob}\left(\frac{1}{q_n^2}\right),
\] 
 the proof of the first assertion is completed by replacing $x$ with $z/\sigma_{ik}(\be_i\transpose\bE)$. For the second assertion, we have
\begin{align*}
&\|F_{\Delta_{ik}}(z) - F_{\widetilde{\Delta}_{ik}}(z)\|_\infty\\
&\quad = \sup_{z\in\mathbb{R}}\left|\prob(\Delta_{ik}\leq z) - \prob(\widetilde{\Delta}_{ik}\leq z)\right|\\
&\quad = \sup_{z\in\mathbb{R}}\left|\int\{\prob(\Delta_{ik}\leq z\mid\be_i\transpose\bE)  - \prob(\widetilde{\Delta}_{ik}\leq z\mid\be_i\transpose\bE)\}\prob(\mathrm{d}\be_i\transpose\bE)\right|\\
&\quad\leq \int \sup_{z\in\mathbb{R}}\left|\{\prob(\Delta_{ik}\leq z\mid\be_i\transpose\bE)  - \prob(\widetilde{\Delta}_{ik}\leq z\mid\be_i\transpose\bE)\}\right|\prob(\mathrm{d}\be_i\transpose\bE) = O\left(\frac{1}{q_n^2}\right).
\end{align*}
The proof is therefore completed.
\end{proof}

\begin{lemma}\label{lemma:CHF_local_expansion}
Suppose Assumptions \ref{assumption:Signal_strength}--\ref{assumption:Noise_matrix_distribution} hold and $\min_{i\in [n]}s_{ik}^2 = \Theta(\rho_n/\Delta_n^2)$. Let 
\[
\kappa_n^{(ik)} = \sum_{j = 1}^n\frac{\expect E_{ij}^3u_{jk}^3}{(s_{ik}\lambda_k)^3},\quad
\Xi_{ijk} = \frac{E_{ij}u_{jk}}{s_{ik}\lambda_k},
\]
and $\mathfrak{S} = \{\calJ\subset[n]:|\calJ|\geq n - 4\}$ be the collection of subsets of $[n]$ whose cardinalities are at least $n - 4$. Then given $i$ and $k$, for every $t\in [-n^{\eps/4}, n^{\eps/4}]$ (recall that $\eps$ is given in Assumption \ref{assumption:Signal_strength} and Assumption \ref{assumption:Noise_matrix_distribution}),
\begin{align*}
\max_{\calJ\in \mathfrak{S}}\left|\prod_{j\in \calJ}\expect e^{\mathbbm{i}t\Xi_{ijk}} - e^{-t^2/2}\left(1 - \frac{\mathbbm{i}t^3}{6}\kappa_n^{(ik)}\right)\right| = e^{-t^2/4}(t^6 + t^4 + t^2)O\left(\frac{1}{q_n^2}\right).
\end{align*}
In particular, when $n$ is sufficiently large,
\[
\max_{\calJ\in \mathfrak{S}}\left|\prod_{j\in \calJ}\expect e^{\mathbbm{i}t\Xi_{ijk}}\right|\leq 2e^{-t^2/4}.
\]
\end{lemma}
\begin{proof}
By Taylor's theorem, there exists a complex number $\theta_{ijk}$ with $|\theta_{ijk}|\leq 1$, such that
\[
\expect e^{\mathbbm{i}t\Xi_{ijk}} = 1 - \frac{t^2}{2}\expect \Xi_{ijk}^2 - \frac{\mathbbm{i}t^3}{6}\expect \Xi_{ijk}^3 + \frac{t^4}{24}\theta_{ijk}\expect\Xi_{ijk}^4.
\]
Note that 
\bea\nonumber
\max_{j\in [n]}\expect \Xi_{ijk}^2 &= O(1/n)
\\
\max_{j\in [n]}\expect \Xi_{ijk}^3 &= O\{1/(nq_n)\}
\\
\max_{j\in [n]}\expect \Xi_{ijk}^4 &= O\{1/(nq_n^2)\}
\eae by Assumption \ref{assumption:Eigenvector_delocalization} and Assumption \ref{assumption:Noise_matrix_distribution}. For any $z\in\mathbb{C}$ with $|z|\leq 1/2$, we have $$\log(1 + z) = z + \theta_z|z|^2,$$ where $\log(\cdot):\mathbb{C}\to\mathbb{C}$ is the complex logarithm. Here for $\theta_z$, an angle of $\theta_z$ is decided in $(-\pi, \pi]$ depending on $z$ with $|\theta_z|\leq C$ for some absolute constant $C > 0$; See, e.g., Equation~(8) on page~208 of \cite{chung2001course}. Then with
\[
z = -\frac{t^2}{2}\expect\Xi_{ijk}^2 - \frac{\mathbbm{i}t^3}{6}\expect\Xi_{ijk}^3 + \frac{t^4}{24}\theta_{ijk}\expect\Xi_{ijk}^4
\]
and $t\leq n^{\eps/4} = o\{(n\rho_n)^{1/4}\}$ and $t\leq n^{\eps/4}\leq q_n^{1/4}$ (see Assumptions \ref{assumption:Signal_strength} and \ref{assumption:Noise_matrix_distribution}),
we obtain
\[
|z|\leq \frac{t^2}{2}O\left(\frac{1}{n}\right) + \frac{t^3}{6}O\left(\frac{1}{nq_n}\right) + \frac{t^4}{24}O\left(\frac{1}{nq_n^2}\right)\leq 1/2
\]
as $n\rightarrow \infty$, and
\begin{align*}
\log\expect e^{\mathbbm{i}t\Xi_{ijk}}& = -\frac{t^2}{2}\expect\Xi_{ijk}^2 - \frac{\mathbbm{i}t^3}{6}\expect\Xi_{ijk}^3 + \frac{t^4}{24}\theta_{ijk}\expect\Xi_{ijk}^4\\
&\quad + \theta_z\left|-\frac{t^2}{2}\expect\Xi_{ijk}^2 - \frac{\mathbbm{i}t^3}{6}\expect\Xi_{ijk}^3 + \frac{t^4}{24}\theta_{ijk}\expect\Xi_{ijk}^4\right|^2.
\end{align*}
Furthermore,
\begin{align*}
&\max_{j\in [n]}\left|-\frac{t^2}{2}\expect\Xi_{ijk}^2 - \frac{\mathbbm{i}t^3}{6}\expect\Xi_{ijk}^3 + \frac{t^4}{24}\theta_{ijk}\expect\Xi_{ijk}^4\right|\\
&\quad = t^2O\left(\frac{1}{n}\right) + t^3O\left(\frac{1}{nq_n}\right) + t^4O\left(\frac{1}{nq_n^2}\right)\\
&\quad = \left(t^2 + \frac{t^3}{q_n} + \frac{t^4}{q_n^2}\right)O\left(\frac{1}{n}\right)\leq t^2O\left(\frac{1}{n}\right).
\end{align*}
Similarly, we can also show that,
\begin{align*}
&\max_{j\in [n]}\left|\log\expect e^{\mathbbm{i}t\Xi_{ijk}} - \left(-\frac{t^2}{2}\expect\Xi_{ijk}^2 - \frac{\mathbbm{i}t^3}{6}\expect\Xi_{ijk}^3\right)\right|\leq t^4O\left(\frac{1}{nq_n^2}\right),\\
&\max_{j\in [n]}\left|\log\expect e^{\mathbbm{i}t\Xi_{ijk}}\right| \leq \max_{j\in [n]}\left|\frac{t^2}{2}\expect\Xi_{ijk}^2 + \frac{\mathbbm{i}t^3}{6}\expect\Xi_{ijk}^3\right| + t^4O\left(\frac{1}{nq_n^2}\right)\leq t^2O\left(\frac{1}{n}\right).
\end{align*}
It follows that
\begin{align*}
&
\sum_{j = 1}^n\log\expect e^{\mathbbm{i}t\Xi_{ijk}} - \left(-\frac{t^2}{2}\sum_{j = 1}^n \expect \Xi_{ijk}^2- \frac{\mathbbm{i}t^3}{6}\kappa_n^{(ik)}\right) 
\\
&= \sum_{j = 1}^n\log\expect e^{\mathbbm{i}t\Xi_{ijk}} - \left(-\frac{t^2}{2}- \frac{\mathbbm{i}t^3}{6}\kappa_n^{(ik)}\right)
\\ & = t^4O\left(\frac{1}{q_n^2}\right);
\end{align*}
Here we note, by definition, 
$$
\sum_{j = 1}^n \expect \Xi_{ijk}^2 =  \sum_{ j =1}^n \expect\left(\frac{E^2_{ij}u^2_{jk}}{s^2_{ik}\lambda^2_k}\right)\Big/s_{ik}^2 = 1.
$$
This shows that
\begin{align*}
&\max_{\calJ\in\mathfrak{S}}\left|\sum_{j\in\calJ}\log\expect e^{\mathbbm{i}t\Xi_{ijk}} - \left(-\frac{t^2}{2} - \frac{\mathbbm{i}t^3}{6}\kappa_n^{(ik)}\right)\right|\\
&\quad\leq \left|\sum_{j = 1}^n\log\expect e^{\mathbbm{i}t\Xi_{ijk}} - \left(-\frac{t^2}{2} - \frac{\mathbbm{i}t^3}{6}\kappa_n^{(ik)}\right)\right| + \max_{\calJ\in\mathfrak{S}}\sum_{j\notin\calJ}\left|\log\expect e^{\mathbbm{i}t\Xi_{ijk}}\right|\\
&\quad\leq t^4O\left(\frac{1}{q_n^2}\right) + 4\max_{j\in[n]}\left|\log\expect e^{\mathbbm{i}t\Xi_{ijk}}\right|\\
&\quad\leq t^4O\left(\frac{1}{q_n^2}\right) + t^2O\left(\frac{1}{n}\right)\\
&\quad\leq (t^2 + t^4) O\left(\frac{1}{q_n^2}\right).
\end{align*}
This is equivalent to
\begin{align*}
\prod_{j \in\calJ}\expect e^{\mathbbm{i}t\Xi_{ijk}} = e^{-t^2/2}\exp\left\{ - \frac{\mathbbm{i}t^3}{6}\kappa_n^{(ik)} + (t^2 + t^4)O\left(\frac{1}{q_n^2}\right)\right\},
\end{align*}
uniformly over $\calJ\in\mathfrak{S}$. 
Since
\[
\left| - \frac{\mathbbm{i}t^3}{6}\kappa_n^{(ik)} + (t^2 + t^4)O\left(\frac{1}{q_n^2}\right)\right| = t^3\left(\frac{1}{q_n}\right) + (t^2 + t^4)\left(\frac{1}{q_n^2}\right)  = o(1)
\]
because $t\leq q_n^{1/4}$, it follows from Taylor's theorem that
\begin{align*}
&\exp\left\{ - \frac{\mathbbm{i}t^3}{6}\kappa_n^{(ik)} + (t^2 + t^4)O\left(\frac{1}{q_n^2}\right)\right\}\\
&\quad = 1  - \frac{\mathbbm{i}t^3}{6}\kappa_n^{(ik)} + (t^2 + t^4)O\left(\frac{1}{q_n^2}\right)
 + \frac{\theta}{2}\left| - \frac{\mathbbm{i}t^3}{6}\kappa_n^{(ik)} + (t^2 + t^4)O\left(\frac{1}{q_n^2}\right)\right|^2\\
&\quad = 1  - \frac{\mathbbm{i}t^3}{6}\kappa_n^{(ik)} + (t^2 + t^4)O\left(\frac{1}{q_n^2}\right) + t^6O\left(\frac{1}{q_n^2}\right) + (t^2 + t^4)^2O\left(\frac{1}{q_n^4}\right)\\
&\quad = 1  - \frac{\mathbbm{i}t^3}{6}\kappa_n^{(ik)} + \left(t^6 + t^4 + t^2\right)O\left(\frac{1}{q_n^2}\right),
\end{align*}
where $\theta\in\mathbb{C}$ and $|\theta|\leq C$ for some absolute constant $C > 0$. Hence, we obtain
\begin{align*}
\max_{J\in\mathfrak{S}}\left|\prod_{j \in\calJ}\expect e^{\mathbbm{i}t\Xi_{ijk}} - e^{-t^2/2}\left(1  - \frac{\mathbbm{i}t^3}{6}\kappa_n^{(ik)}\right)\right|\leq e^{-t^2/2}\left(t^6 + t^4 + t^2\right)O\left(\frac{1}{q_n^2}\right).
\end{align*}
The proof of the first assertion is therefore completed. The second assertion follows from the observations that $(t^6 + t^4 + t^2)/q_n^2 = o(1)$ and $t^3\kappa_n^{(ik)} = O(n^{3\eps/4}/q_n) = o(1)$. 
\end{proof}

\subsection{Proof of the Self-Smoothed Edgeworth Expansion \eqref{eqn:edgeworth_expansion_self_smoothed_version}}
\label{sub:proof_of_edgeworth_expansion_smoothed_version}

In this subsection, we prove the Edgeworth expansion \eqref{eqn:edgeworth_expansion_self_smoothed_version} for $\widetilde{T}_{ik} + \widetilde{\Delta}_{ik}$ using Esseen's smoothing lemma (Theorem \ref{thm:esseen_smoothing_lemma}). Specifically, we are going to show that
\begin{align*}
&\int_{-\Delta_n^2/(n\rho_n)}^{\Delta_n^2/(n\rho_n)}\left|\frac{\expect e^{\mathbbm{i}t(\widetilde{T}_{ik} + \widetilde{\Delta}_{ik})} - \mathrm{ch.f.}(t, G_n^{(ik)})}{t}\right|\mathrm{d}t + \frac{\sup_{x\in\mathbb{R}}|\mathrm{d}G_n^{(ik)}(x)/\mathrm{d}x|}{\Delta_n^2/(n\rho_n)} 
\\
&\quad= O\left\{\frac{1}{n} + \frac{n\rho_n\log n}{\Delta_n^2} + \frac{(\log n)^4}{q_n^2}\right\}; 
 \end{align*}
Here $\mathrm{ch.f.}(t,G_n^{(ik)})$ is the characteristic function of $G_n^{(ik)}$  in \eqref{def:Gnikx}, which is defined as,
$$
\mathrm{ch.f.}(t,G_n^{(ik)}) = \int_{\mathbb{R}}e^{\mathbbm{i}tx}dG_n^{(ik)}(x).
$$
It is clear that $\sup_x|\mathrm{d}G_n^{(ik)}(x)/\mathrm{d}x| = O(1)$ with a similar argument as before. Therefore, it is sufficient to establish the following results regarding the first integral:
\begin{align}
\label{eqn:CHF_II}
\int_{\{t:M_n(n^{2\beta_\Delta}\log n)^{1/2}\leq|t|\leq n^{2\beta_\Delta}\}}\left|\frac{\expect \exp\{\mathbbm{i}t(\widetilde{T}_{ik} + \widetilde{\Delta}_{ik})\}}{t}\right|\mathrm{d}t &= O\left(\frac{1}{n}\right),\\
\label{eqn:CHF_III}
\int_{\{t:n^\gamma\leq|t|\leq M_n(n^{2\beta_\Delta}\log n)^{1/2}\}}\left|\frac{\expect\exp\{\mathbbm{i}t(\widetilde{T}_{ik} + \widetilde{\Delta}_{ik})\}}{t}\right|\mathrm{d}t &= O\left\{\frac{(\log n)^4}{q_n^2} + \frac{n\rho_n\log n}{\Delta_n^2}\right\},\\
\label{eqn:CHF_IV}
\int_{-n^\gamma}^{n^\gamma}\left|\frac{\expect e^{\mathbbm{i}t(\widetilde{T}_{ik} + \widetilde{\Delta}_{ik})} - \mathrm{ch.f.}(t; G_n^{(ik)})}{t}\right|\mathrm{d}t &= O\left(\frac{n\rho_n}{\Delta_n^2} + \frac{1}{q_n^2}\right),
\end{align}
where $M_n\to\infty$ is a slowly growing sequence and $\gamma > 0$ is a constant to be decided later. Below, Lemmas \ref{lemma:CHF_II}, \ref{lemma:CHF_III}, and \ref{lemma:CHF_IV} will establish \eqref{eqn:CHF_II}, \eqref{eqn:CHF_III}, and \eqref{eqn:CHF_IV}, respectively.

\begin{lemma}\label{lemma:CHF_II}
Suppose Assumptions \ref{assumption:Signal_strength}--\ref{assumption:Noise_matrix_distribution} hold and $\min_{i\in[n]}s_{ik}^2 = \Theta(\rho_n/\Delta_n^2)$. If $\beta_\Delta\leq 1/2$, then for any sequence $(M_n)_{n = 1}^\infty$, $M_n\to\infty$,
\[
\int_{\{t:M_n(n^{2\beta_\Delta}\log n)^{1/2}\leq|t|\leq n^{2\beta_\Delta}}\}\left|\frac{\expect \exp\{\mathbbm{i}t(\widetilde{T}_{ik} + \widetilde{\Delta}_{ik})\}}{t}\right|\mathrm{d}t = O\left(\frac{1}{n}\right).
\]
\end{lemma}
\begin{proof}
By Lemma \ref{lemma:Delta_conditional_variance}, we know that $\sigma_{ik}^2(\be_i\transpose\bE) \geq cn^{-2\beta_\Delta}$
with probability at least $1 - n^{-2}$ for some constant $c > 0$. It follows that for any $t$, $|t|\geq M_n(n^{2\beta_{\Delta}}\log n)^{1/2}$,
\begin{align*}
&|\expect \exp\{\mathbbm{i}t(\widetilde{T}_{ik} + \widetilde{\Delta}_{ik})\}|\\
&\quad\leq \left|\expect \exp\{\mathbbm{i}t(\widetilde{T}_{ik} + \widetilde{\Delta}_{ik})\}\mathbbm{1}\left\{\sigma_{ik}^2(\be_i\transpose\bE) \geq \frac{c}{n^{2\beta_\Delta}}\right\}\right|+ \prob\left\{\sigma_{ik}^2(\be_i\transpose\bE) < \frac{c}{n^{2\beta_\Delta}}\right\}\\
&\quad\leq \left|\expect \exp\left\{\mathbbm{i}t\widetilde{T}_{ik} - \frac{t^2}{2}\sigma_{ik}^2(\be_i\transpose\bE)\right\}\mathbbm{1}\left\{\sigma_{ik}^2(\be_i\transpose\bE) \geq \frac{c}{n^{2\beta_\Delta}}\right\}\right| + \frac{1}{n^2}\\
&\quad\leq \exp\left(-\frac{ct^2}{2n^{2\beta_\Delta}}\right) +  \frac{1}{n^2}\leq \exp\left(-
\frac{c}{2}M_n^2\log n\right) + \frac{1}{n^2}.
\end{align*}
Therefore,
\begin{align*}
&\int_{\{t:M_n(n^{2\beta_\Delta}\log n)^{1/2}\leq|t|\leq n^{2\beta_\Delta}\}}\left|\frac{\expect \exp\{\mathbbm{i}t(\widetilde{T}_{ik} + \widetilde{\Delta}_{ik})\}}{t}\right|\mathrm{d}t\\
&\quad\leq\left\{\exp\left(-\frac{c}{2}M_n^2\log n\right) + \frac{1}{n^2}\right\}\int_{\{t:M_n(n^{2\beta_\Delta}\log n)^{1/2}\leq|t|\leq n^{2\beta_\Delta}\}}\frac{\mathrm{d}t}{|t|}\\
&\quad\leq \frac{4\beta_\Delta\log n}{n^{cM_n^2/2}} + \frac{4\beta_\Delta\log n}{n^2} = O\left(\frac{1}{n}\right).
\end{align*}
The proof is thus completed.
\end{proof}

\begin{lemma}\label{lemma:CHF_intermediate}
Suppose Assumptions \ref{assumption:Signal_strength}--\ref{assumption:Eigenvector_delocalization} hold and $\min_{i\in [n]}s_{ik}^2 = \Theta(\rho_n/\Delta_n^2)$. Let $\beta = \max_{a,b\in [n]}\expect|E_{ij}|^3$. If $\beta_\Delta \in [1/2 - \eps,1/2]$, then there exists a sufficiently small constants $\bar{\eps} > 0$, such that for any $\gamma\in (0, \eps)$ (recall that $\eps > 0$ is given by Assumption \ref{assumption:Noise_matrix_distribution}), 
\begin{align*}
&\int_{\{t:{n^\gamma}\leq|t|\leq \bar{\eps}\sqrt{n\rho_n}(\rho_n/\beta)\}}\left|\frac{\expect\exp\{\mathbbm{i}t(\widetilde{T}_{ik} + \widetilde{\Delta}_{ik})\}}{t}\right|\mathrm{d}t
 = O\left\{\frac{(\log n)^2}{q_n^2} + \frac{n\rho_n(\log n)}{\Delta_n^2}
\right\}.
\end{align*}

\end{lemma}
\begin{proof}
First note that $\beta = O(\sqrt{n}\rho_n^{3/2}/q_n)$ so that $\bar{\eps}\sqrt{n}\rho_n^{3/2}/\beta = \Omega(q_n) = \omega(n^\gamma)$ for any $\gamma\in (0, \eps)$. 
Since $s_{ik}^2 = \Theta(\rho_n/\Delta_n^2)$ and $|\lambda_k| = \Theta(\Delta_n)$ by Assumption \ref{assumption:Signal_strength}, then there exists some constant $c_1 \in(0,1)$ such that $s_{ik}^2\lambda_k^2/\rho_n\geq c_1$. Now let 
\[
\calJ_{ik} = \left\{j\in [n]:\frac{\var(E_{ij})u_{jk}^2}{s_{ik}^2\lambda_k^2}\geq \frac{c_1\rho_n}{2ns_{ik}^2\lambda_k^2}\right\}.
\]
By Assumption \ref{assumption:Noise_matrix_distribution}, $\var(E_{ij})\leq \rho_n$ for all $i,j\in [n]$. By Assumption \ref{assumption:Eigenvector_delocalization}, $\|\bu_k\|_\infty^2\leq C/n$ for some constant $C \geq 1$.  This entails that
\begin{align*}
1 &= \sum_{j = 1}^n\frac{\var(E_{ij})u_{jk}^2}{s_{ik}^2\lambda_k^2} = \frac{1}{s_{ik}^2}\sum_{j\in\calJ_{ik}}\frac{\var(E_{ij})u_{jk}^2}{\lambda_k^2} + \frac{1}{s_{ik}^2}\sum_{j\notin \calJ_{ik}}\frac{\var(E_{ij})u_{jk}^2}{\lambda_k^2}\\
&\leq \frac{\rho_nC^2|\calJ_{ik}|}{ns_{ik}^2\lambda_k^2} + \frac{c_1\rho_n(n - |\calJ_{ik}|)}{2ns_{ik}^2\lambda_k^2}\leq \frac{C^2|\calJ_{ik}|}{c_1n} + \frac{{1}}{2},
\end{align*}
where the first equality is by $s_{ik}^2 = \var\left({\be_i\transpose\bE\bu_k}/{\lambda_k}\right)$. Namely, $|\calJ_{ik}|\geq c_1n/(2C^2)$. This implies that there exists a permutation $\chi:[n]\to[n]$, such that 
\[
\min_{j\in 1,\ldots,\lfloor c_1n/(2C^2)\rfloor } \frac{\var(E_{i\chi(j)})u_{\chi(j)k}^2}{s_{ik}^2\lambda_k^2}\geq \frac{c_1\rho_n}{2ns_{ik}^2\lambda_k^2}\geq \frac{c}{n}
\] 
for some constant $c > 0$ because $\rho_n/(s_{ik}^2\lambda_k^2) = \Theta(1)$. 
For any $a,b\in[n]$, let
\begin{equation}
\label{eqn:hab_formula}
\begin{aligned}
h_{ab}^{(ik)}
&=
\left\{\left(v_au_{bk} + v_bu_{ak} + \frac{E_{ia}u_{bk} + E_{ib}u_{ak}}{\lambda_k}\right)z_{ab} + \frac{\be_i\transpose\bE\bu_k(E_{ia}^2 + E_{ib}^2)z_{ab}u_{ak}u_{bk}}{\lambda_k^3s_{ik}^2}\right\}^2\\
&\quad\times\frac{\var(E_{ab})}{s_{ik}^2\lambda_k^2}
\end{aligned}
\end{equation}
This allows us to write
\begin{align*}
\sigma_{ik}^2(\be_i\transpose\bE)& = \sum_{a < b}h_{\chi(a)\chi(b)}^{(ik)} + \sum_{a = 1}^nh_{\chi(a)\chi(a)}^{(ik)},\\
\expect \exp\{\mathbbm{i}t(\widetilde{T}_{ik} + \widetilde{\Delta}_{ik})\} &= \expect\big[\expect\big\{ \exp\{\mathbbm{i}t(\widetilde{T}_{ik})\}\mid \be_i\transpose\bE\big\}\big] \cdot \big[\expect \exp(\mathbbm{i}t\widetilde{\Delta}_{ik})\big]
\\
& = \expect\exp\left\{\sum_{j = 1}^n\Lambda_{i\chi(j)k}(t) - \frac{\mathbbm{i}t}{2}\sum_{a < b}g_{\chi(a)\chi(b)}^{(ik)} - \frac{t^2}{2}\sum_{a < b}h_{\chi(a)\chi(b)}^{(ik)}\right\},
\end{align*}
where 
\[
\Lambda_{ijk}(t) = \frac{\mathbbm{i}tE_{ij}u_{jk}}{s_{ik}\lambda_k} - \frac{\mathbbm{i}t}{2}g_{jj}^{(ik)} - \frac{t^2}{2}h_{jj}^{(ik)}
\]
and $g_{ab}^{(ik)}$ is defined in \eqref{eqn:gab_formula}. 
By Lemma A.1 in \cite{erdos2013} and Assumption \ref{assumption:Eigenvector_delocalization}, we know that $$\expect |\be_i\transpose\bE\bu_k|^p = O(\rho_n^{p/2}),$$ for any finite $p$.
By Assumption \ref{assumption:Eigenvector_delocalization} and Assumption \ref{assumption:Noise_matrix_distribution}, we have
\bea\label{firstorder:h}
\max_{a,b\in [n]}\expect h_{ab}^{(ik)}
&\lesssim \max_{a,b\in[n]}\frac{\sigma_{ab}^2}{s_{ik}^2\lambda_k^2}\left\{\|\bv\|_\infty^2\|\bu_k\|_\infty^2 + \frac{\expect E_{ia}^2\|\bu_k\|_\infty^2}{\lambda_k^2} + \frac{\expect(\be_i\transpose\bE\bu_k)^2E_{ia}^4\|\bu_k\|_\infty^4}{s_{ik}^4\lambda_k^6}\right\}\\
&\lesssim \max_{a\in[n]}\left[\frac{1}{n^3} + \frac{\rho_n}{n\Delta_n^2} + \frac{\left\{\expect(\be_i\transpose\bE\bu_k)^4\right\}^{1/2}\expect(E_{ia}^8)^{1/2}}{n^2\Delta_n^2\rho_n^2}\right]\\
& = O\left\{\frac{\rho_n}{\Delta_n^2}\left(\frac{1}{n} + \frac{1}{\sqrt{n}q_n^3}\right)\right\},
\eae
noting here, we use the fact that $\beta_{\Delta} \leq 1/2$. Similarly, we have 
\begin{align*}
\max_{a,b\in [n]}\expect (h_{ab}^{(ik)})^2
&\lesssim \max_{a,b\in[n]}\frac{\sigma_{ab}^4}{s_{ik}^4\lambda_k^4}\left\{\|\bv\|_\infty^4\|\bu_k\|_\infty^4 + \frac{\expect E_{ia}^4\|\bu_k\|_\infty^4}{\lambda_k^4} + \frac{\expect(\be_i\transpose\bE\bu_k)^4E_{ia}^8\|\bu_k\|_\infty^8}{s_{ik}^8\lambda_k^{12}}\right\}\\
&\lesssim \max_{a\in[n]}\left[\frac{1}{n^6} + \frac{\rho_n^2}{nq_n^2\Delta_n^4} + \frac{\left\{\expect(\be_i\transpose\bE\bu_k)^8\right\}^{1/2}\expect(E_{ia}^{16})^{1/2}}{n^4\Delta_n^4\rho_n^4}\right]\\
& = O\left\{\frac{\rho_n^2}{\Delta_n^4}\left(\frac{1}{nq_n^2} + \frac{1}{\sqrt{n}q_n^7}\right)\right\}\\
\max_{a,b\in [n]}\expect |h_{ab}^{(ik)}|^3
&\lesssim \max_{a,b\in[n]}\frac{\sigma_{ab}^6}{s_{ik}^6\lambda_k^6}\left\{\|\bv\|_\infty^6\|\bu_k\|_\infty^6 + \frac{\expect E_{ia}^6\|\bu_k\|_\infty^6}{\lambda_k^6} + \frac{\expect(\be_i\transpose\bE\bu_k)^6E_{ia}^{12}\|\bu_k\|_\infty^{12}}{s_{ik}^{12}\lambda_k^{18}}\right\}\\
&\lesssim \max_{a\in[n]}\left[\frac{1}{n^9} + \frac{\rho_n^3}{nq_n^4\Delta_n^6} + \frac{\left\{\expect(\be_i\transpose\bE\bu_k)^{12}\right\}^{1/2}\expect(E_{ia}^{24})^{1/2}}{n^6\Delta_n^6\rho_n^6}\right]\\
& =  O\left\{\frac{\rho_n^3}{\Delta_n^6}\left(\frac{1}{nq_n^4} + \frac{1}{\sqrt{n}q_n^{11}}\right)\right\}.
\end{align*}
Now let $m\in [n - 1]$ be an integer to be selected later and define
\[
\Gamma_{ik}(m) = -\frac{it}{2}\sum_{a = 1}^m\sum_{b = a + 1}^ng_{\chi(a)\chi(b)}^{(ik)} - \frac{t^2}{2}\sum_{a = 1}^m\sum_{b = a + 1}^nh_{\chi(a)\chi(b)}^{(ik)}. 
\]
By Taylor's theorem, for any $z\in\mathbb{C}$ with $\mathbf{Re}(z) \leq 0$, $|e^z - 1 - z|\leq |z|^2/2$. This entails that
\[
\exp\{\Gamma_{ik}(m)\} = 1 + \Gamma_{ik}(m) + \frac{\theta}{2}\Gamma_{ik}(m)^2,
\]
where $\theta$ is a complex-valued random variable depending on $\Gamma_{ik}(m)$ such that $|\theta| \leq 1$ with probability one, as $\mathbf{Re}\{\Gamma_{ik}(m)\} \leq 0$. It follows that
\begin{align}
&\expect \exp\{\mathbbm{i}t(\widetilde{T}_{ik} + \widetilde{\Delta}_{ik})\}\nonumber\\
&\quad = \expect \exp\left\{\sum_{j = 1}^n\Lambda_{i\chi(j)k}(t) - \frac{it}{2}\sum_{a < b}g_{\chi(a)\chi(b)}^{(ik)} - \frac{t^2}{2}\sum_{a < b}h_{\chi(a)\chi(b)}^{(ik)}\right\}\nonumber\\
&\quad = \expect \exp\left\{\sum_{j = 1}^n\Lambda_{i\chi(j)k}(t) - \sum_{a = m + 1}^{n - 1}\sum_{b = a + 1}^n\left(\frac{\mathbbm{i}t}{2}g_{\chi(a)\chi(b)}^{(ik)} + \frac{t^2}{2}h_{\chi(a)\chi(b)}^{(ik)}\right)\right\}\exp\{\Gamma_{ik}(m)\}\nonumber\\
\label{eqn:CHF_intermediate_term_I}
&\quad = \expect \exp\left\{\sum_{j = 1}^n\Lambda_{i\chi(j)k}(t) - \sum_{a = m + 1}^{n - 1}\sum_{b = a + 1}^n\left(\frac{\mathbbm{i}t}{2}g_{\chi(a)\chi(b)}^{(ik)} + \frac{t^2}{2}h_{\chi(a)\chi(b)}^{(ik)}\right)\right\}\\
\label{eqn:CHF_intermediate_term_II}
&\quad\quad + \expect \exp\left\{\sum_{j = 1}^n\Lambda_{i\chi(j)k}(t) - \sum_{a = m + 1}^{n - 1}\sum_{b = a + 1}^n\left(\frac{\mathbbm{i}t}{2}g_{\chi(a)\chi(b)}^{(ik)} + \frac{t^2}{2}h_{\chi(a)\chi(b)}^{(ik)}\right)\right\}\Gamma_{ik}(m)\\
\label{eqn:CHF_intermediate_term_III}
&\quad\quad + \frac{1}{2}\expect \exp\left\{\sum_{j = 1}^n\Lambda_{i\chi(j)k}(t) - \sum_{a = m + 1}^{n - 1}\sum_{b = a + 1}^n\left(\frac{\mathbbm{i}t}{2}g_{\chi(a)\chi(b)}^{(ik)} + \frac{t^2}{2}h_{\chi(a)\chi(b)}^{(ik)}\right)\right\}\theta\Gamma_{ik}(m)^2.
\end{align}
Note that the modulus of the term in line \eqref{eqn:CHF_intermediate_term_III} is bounded above by $(1/2)\expect \Gamma_{ik}(m)^2$ because the real part of the exponent inside the expected value is negative and $|e^z| = e^{\mathbf{Re}(z)}$ for any $z\in\mathbb{C}$. 
\par
Observe that for any $a_1,b_1,a_2,b_2\in [n]$ with $a_1 < b_1$, $a_2 < b_2$, it is clear that $\expect g_{a_1b_1}^{(ik)}g_{a_2b_2}^{(ik)}\neq 0$ occurs only if $a_1 = a_2$ and $b_1 = b_2$. By Assumption \ref{assumption:Noise_matrix_distribution}, this entails that
\begin{align*}
\expect\left\{\sum_{a = 1}^m\sum_{b = a + 1}^ng_{\chi(a)\chi(b)}^{(ik)}\right\}^2
& = \sum_{a = 1}^m\sum_{b = a + 1}^n\expect (g_{\chi(a)\chi(b)}^{(ik)})^2
 = mO\left(\frac{1}{nq_n^2}\right).
\end{align*}
The analysis of the second moment of $\sum_{a = 1}^m\sum_{b = a + 1}^nh_{\chi(a)\chi(b)}^{(ik)}$ is more involved. First observe that for any $a,b\in [n]$, 
\[
h_{ab}^{(ik)}\leq C\left\{\frac{1}{n^3} + \frac{1}{n\Delta_n^2}(E_{ia}^2 + E_{ib}^2) + \frac{1}{n^2\rho_n^2\Delta_n^2}(\be_i\transpose\bE\bu_k)^2(E_{ia}^4 + E_{ib}^4)\right\}.
\]
It follows that
\begin{align*}
\expect\left\{\sum_{a = 1}^m\sum_{b = a + 1}^nh_{\chi(a)\chi(b)}^{(ik)}\right\}^2
& = \sum_{a_1 = 1}^m\sum_{b_1 = a_1 + 1}^n\sum_{a_2 = 1}^m\sum_{b_2 = a_2 + 1}^n
\expect h_{\chi(a_1)\chi(b_1)}^{(ik)}h_{\chi(a_2)\chi(b_2)}^{(ik)}\\
&\leq CH_1 + CH_2 + CH_3,
\end{align*}
where
\begin{align*}
H_1 & = \sum_{a_1 = 1}^m\sum_{b_1 = a_1 + 1}^n\sum_{a_2 = 1}^m\sum_{b_2 = a_2 + 1}^n
\expect \left(\frac{1}{n^3} + \frac{E_{ia_1}^2 + E_{ib_1}^2}{n\Delta_n^2}\right)\left(\frac{1}{n^3} + \frac{E_{ia_2}^2 + E_{ib_2}^2}{n\Delta_n^2}\right),\\
H_2 & = \sum_{a_1 = 1}^m\sum_{b_1 = a_1 + 1}^n\sum_{a_2 = 1}^m\sum_{b_2 = a_2 + 1}^n
\expect \left(\frac{1}{n^3} + \frac{E_{ia_1}^2 + E_{ib_1}^2}{n\Delta_n^2}\right)\left\{\frac{(\be_i\transpose\bE\bu_k)^2(E_{ia_2}^4 + E_{ib_2}^4)}{n^2\rho_n^2\Delta_n^2}\right\},\\
H_3 & = \sum_{a_1 = 1}^m\sum_{b_1 = a_1 + 1}^n\sum_{a_2 = 1}^m\sum_{b_2 = a_2 + 1}^n
\expect \left\{\frac{(\be_i\transpose\bE\bu_k)^4(E_{ia_1}^4 + E_{ib_1}^4)(E_{ia_2}^4 + E_{ib_2}^4)}{n^4\rho_n^4\Delta_n^4}\right\}.
\end{align*}
Note that $2H_2\leq H_1 + H_3$ by Cauchy-Schwarz inequality, it is therefore sufficient to bound $H_1$ and $H_3$. 
 For $H_1$, we have
\begin{align*}
H_1 & \leq \frac{m^2}{n^4} + \sum_{a_1 = 1}^m\sum_{b_1 = a_1 + 1}^n\sum_{a_2 = 1}^m\sum_{b_2 = a_2 + 1}^n
\expect \left(\frac{E_{ia_1}^2 + E_{ib_1}^2 + E_{ia_2}^2 + E_{ib_2}^2}{n^4\Delta_n^2}\right)\\
&\quad + \sum_{a_1 = 1}^m\sum_{b_1 = a_1 + 1}^n\sum_{a_2 = 1}^m\sum_{b_2 = a_2 + 1}^n\expect\left\{\frac{(E_{ia_1}^2 + E_{ib_1}^2)(E_{ia_2}^2 + E_{ib_2}^2)}{n^2\Delta_n^4}\right\}\\
&\leq \frac{m^2}{n^4} + O\left(\frac{m^2\rho_n}{n^2\Delta_n^2}\right)
+ \sum_{a_1 = 1}^m\sum_{b_1 = a_1 + 1}^n\sum_{a_2 = 1}^m\sum_{b_2 = a_2 + 1}^n\expect\left\{\frac{(E_{ia_1}^2E_{ib_2}^2 + E_{ib_1}^2E_{ia_2}^2)}{n^2\Delta_n^4}\right\}
\\
&\quad + \sum_{a_1 = 1}^m\sum_{b_1 = a_1 + 1}^n\sum_{a_2 = 1}^m\sum_{b_2 = a_2 + 1}^n\expect\left\{\frac{(E_{ia_1}^2E_{ia_2}^2 + E_{ib_1}^2E_{ib_2}^2)}{n^2\Delta_n^4}\right\}\\
&\leq \frac{m^2}{n^4} + O\left(\frac{m^2\rho_n}{n^2\Delta_n^2}\right)
+ mn\sum_{a_1 = 1}^m\sum_{b_2 = 1}^n\expect\left(\frac{E_{ia_1}^2E_{ib_2}^2}{n^2\Delta_n^4}\right)
 + mn\sum_{b_1 = 1}^n\sum_{a_2 = 1}^m\expect\left(\frac{E_{ib_1}^2E_{ia_2}^2}{n^2\Delta_n^4}\right)
\\
&\quad + n^2\sum_{a_1 = 1}^m\sum_{a_2 = 1}^m\expect\left(\frac{E_{ia_1}^2E_{ia_2}^2}{n^2\Delta_n^4}\right)
 + m^2\sum_{b_1 = 1}^n\sum_{b_2 = 1}^n\expect\left(\frac{E_{ib_1}^2E_{ib_2}^2}{n^2\Delta_n^4}\right)\\
& = \frac{m^2}{n^4} + O\left(\frac{m^2\rho_n}{n^2\Delta_n^2}\right)
+ 2mn\left\{\sum_{a_1 = 1}^m\sum_{b_2 \neq a_1}^n\expect\left(\frac{E_{ia_1}^2E_{ib_2}^2}{n^2\Delta_n^4}\right) + \sum_{a_1 = 1}^m\expect\left(\frac{E_{ia_1}^4}{n^2\Delta_n^4}\right)\right\}
\\
&\quad + n^2\left\{\sum_{a_1 = 1}^m\sum_{a_2 \neq a_1}^m\expect\left(\frac{E_{ia_1}^2E_{ia_2}^2}{n^2\Delta_n^4}\right) + \sum_{a_1 = 1}^m\expect\left(\frac{E_{ia_1}^4}{n^2\Delta_n^4}\right)\right\}\\
&\quad
 + m^2\left\{\sum_{b_1 = 1}^n\sum_{b_2 \neq b_1}^n\expect\left(\frac{E_{ib_1}^2E_{ib_2}^2}{n^2\Delta_n^4}\right) + \sum_{b_1 = 1}^n\expect\left(\frac{E_{ib_1}^4}{n^2\Delta_n^4}\right)\right\}\\
& = \frac{m^2}{n^4} + O\left(\frac{m^2\rho_n}{n^2\Delta_n^2}\right) + O\left(\frac{m^2\rho_n^2}{\Delta_n^4}\right) + O\left(\frac{m^2\rho_n^2}{\Delta_n^4q_n^2}\right) + O\left(\frac{m^2\rho_n^2}{\Delta_n^4}\right) + O\left(\frac{mn\rho_n^2}{\Delta_n^4q_n^2}\right)
\\
& = O\left\{\frac{m\rho_n^2}{\Delta_n^4}\left(m + \frac{n}{q_n^2}\right)\right\},
\end{align*}
where we use $\beta_{\Delta}\leq 1/2$ to deduce ${m^2}/{n^4} \lesssim {m^2\rho_n}/({n^2\Delta_n^2})$ in the last equality. For $H_3$, we have
\begin{align*}
H_3 & = \sum_{a_1 = 1}^m\sum_{b_1 = a_1 + 1}^n\sum_{a_2 = 1}^m\sum_{b_2 = a_2 + 1}^n
\expect \left\{\frac{(\be_i\transpose\bE\bu_k)^4(E_{ia_1}^4 + E_{ib_1}^4)(E_{ia_2}^4 + E_{ib_2}^4)}{n^4\rho_n^4\Delta_n^4}\right\}\\
&= \sum_{a_1 = 1}^m\sum_{b_1 = a_1 + 1}^n\sum_{a_2 = 1}^m\sum_{b_2 = a_2 + 1}^n
\frac{\expect (\be_i\transpose\bE\bu_k)^4(E_{ia_1}^4E_{ia_2}^4 + E_{ib_1}^4E_{ia_2}^4 + E_{ia_1}^4E_{ib_2}^4 + E_{ib_1}^4E_{ib_2}^4)}{n^4\rho_n^4\Delta_n^4}\\
& \leq n^2\sum_{a_1 = 1}^m\sum_{a_2 = 1}^m
\expect \left\{\frac{(\be_i\transpose\bE\bu_k)^4E_{ia_1}^4E_{ia_2}^4}{n^4\rho_n^4\Delta_n^4}\right\}
 + m^2\sum_{b_1 = 1}^n\sum_{b_2 = 1}^n
\expect \left\{\frac{(\be_i\transpose\bE\bu_k)^4E_{ib_1}^4E_{ib_2}^4}{n^4\rho_n^4\Delta_n^4}\right\}\\
&\quad + mn\sum_{a_1 = 1}^m\sum_{b_2 = 1}^n
\expect \left\{\frac{(\be_i\transpose\bE\bu_k)^4(E_{ia_1}^4E_{ib_2}^4)}{n^4\rho_n^4\Delta_n^4}\right\}
\\&\quad
 + mn\sum_{b_1 = 1}^n\sum_{a_2 = 1}^m
\expect \left\{\frac{(\be_i\transpose\bE\bu_k)^4(E_{ib_1}^4E_{ia_2}^4)}{n^4\rho_n^4\Delta_n^4}\right\}\\
& = n^2\left[\sum_{a_1 = 1}^m\sum_{a_2 \in[m]/\{a_1\}}
\expect \left\{\frac{(\be_i\transpose\bE\bu_k)^4E_{ia_1}^4E_{ia_2}^4}{n^4\rho_n^4\Delta_n^4}\right\} + \sum_{a_1 = 1}^m
\expect \left\{\frac{(\be_i\transpose\bE\bu_k)^4E_{ia_1}^8}{n^4\rho_n^4\Delta_n^4}\right\}\right]
\\&\quad
 + m^2\left[\sum_{b_1 = 1}^n\sum_{b_2 \in [n]/\{b_1\}}
\expect \left\{\frac{(\be_i\transpose\bE\bu_k)^4E_{ib_1}^4E_{ib_2}^4}{n^4\rho_n^4\Delta_n^4}\right\} + \sum_{b_1 = 1}^n
\expect \left\{\frac{(\be_i\transpose\bE\bu_k)^4E_{ib_1}^8}{n^4\rho_n^4\Delta_n^4}\right\}\right]\\
&\quad + 2mn\left[\sum_{a_1 = 1}^m\sum_{b_2\in[n]/\{a_1\}}
\expect \left\{\frac{(\be_i\transpose\bE\bu_k)^4(E_{ia_1}^4E_{ib_2}^4)}{n^4\rho_n^4\Delta_n^4}\right\} + \sum_{a_1 = 1}^m
\expect \left\{\frac{(\be_i\transpose\bE\bu_k)^4(E_{ia_1}^8)}{n^4\rho_n^4\Delta_n^4}\right\}\right]\\
& \leq n^2\left[\sum_{a_1,a_2 = 1}^m
\frac{\{\expect (\be_i\transpose \bE\bu_k)^8\}^{1/2}(\expect E_{ia_1}^8 \expect E_{ia_2}^8)^{1/2}}{n^4\rho_n^4\Delta_n^4} + \sum_{a_1 = 1}^m
\frac{\{\expect(\be_i\transpose\bE\bu_k)^8\}^{1/2}(\expect E_{ia_1}^{16})^{1/2}}{n^4\rho_n^4\Delta_n^4}\right]
\\&\quad
 + m^2\left[\sum_{b_1,b_2 = 1}^n
\frac{\{\expect(\be_i\transpose\bE\bu_k)^8\}^{1/2}(\expect E_{ib_1}^8\expect E_{ib_2}^8)^{1/2}}{n^4\rho_n^4\Delta_n^4}\right]\\
&\quad + m^2\left[\sum_{b_1 = 1}^n
\frac{\{\expect(\be_i\transpose\bE\bu_k)^8\}^{1/2}(\expect E_{ib_1}^{16})^{1/2}}{n^4\rho_n^4\Delta_n^4}\right]\\
&\quad + mn\left[\sum_{a_1 = 1}^m\sum_{b_2\in[n] /\{a_1\}}
\frac{\{\expect(\be_i\transpose\bE\bu_k)^8\}^{1/2}(\expect E_{ia_1}^8\expect E_{ib_2}^8)^{1/2}}{n^4\rho_n^4\Delta_n^4}\right]\\
&\quad + mn\left[\sum_{a_1 = 1}^m
\frac{\{\expect(\be_i\transpose\bE\bu_k)^8\}^{1/2}(\expect E_{ia_1}^{16})^{1/2}}{n^4\rho_n^4\Delta_n^4}\right]\\
& = O\left(\frac{m^2n\rho_n^2}{\Delta_n^4q_n^6} + \frac{mn^2\rho_n^2}{\Delta_n^4\sqrt{n}q_n^7}\right),
 \end{align*}
where in the last equality, we use the fact that 
$
\big\{\expect(\be_i\transpose\bE\bu_k)^8\big\}^{1/2} = {O}(\rho_n^2)
$. Summarizing above results, we obtain
\begin{align}
  \expect\Gamma^2_{ik}(m)&\leq \frac{t^2}{2}\expect\left\{\sum_{a = 1}^m\sum_{b = a + 1}^ng_{\chi(a)\chi(b)}^{(ik)}\right\}^2 + \frac{t^4}{2}\expect\left\{\sum_{a = 1}^m\sum_{b = a + 1}^nh_{\chi(a)\chi(b)}^{(ik)}\right\}^2\nonumber\\
  & = mt^2O\left(\frac{1}{nq_n^2}\right) + t^4O(H_1 + H_2 + H_3)\nonumber\\
  \label{eqn:Gamma_m_bound}
  & = mt^2O\left(\frac{1}{nq_n^2}\right) + mt^4O\left\{\frac{\rho_n^2}{\Delta_n^4}\left(m + \frac{mn}{q_n^6} + \frac{n}{q_n^2} + \frac{n^{3/2}}{q_n^7}\right)\right\}.
\end{align}
Define
\[
f_{ijk}(t) = \expect\exp\{\Lambda_{ijk}(t)\} = \expect\exp\left\{\left(\frac{\mathbbm{i}tE_{ij}u_{jk}}{s_{ik}\lambda_k} - \frac{\mathbbm{i}t}{2}g_{jj}^{(ik)} - \frac{t^2}{2}h_{jj}^{(ik)}\right)\right\}.
\]
Since $\mathbf{Re}(\Lambda_{ijk}(t)) \leq 0$, by Taylor's theorem applied to the complex exponential, 
\begin{align*}
\exp\{\Lambda_{ijk}(t)\}
& = 1 + \frac{\mathbbm{i}tE_{ij}u_{jk}}{s_{ik}\lambda_k} - \frac{\mathbbm{i}t}{2}g_{jj}^{(ik)} - \frac{t^2}{2}h_{jj}^{(ik)} + \frac{1}{2}\left\{\frac{\mathbbm{i}tE_{ij}u_{jk}}{s_{ik}\lambda_k} - \frac{\mathbbm{i}t}{2}g_{jj}^{(ik)} - \frac{t^2}{2}h_{jj}^{(ik)}\right\}^2\\
&\quad + \frac{\theta_{ijk}}{6}\left|\frac{\mathbbm{i}tE_{ij}u_{jk}}{s_{ik}\lambda_k} - \frac{\mathbbm{i}t}{2}g_{jj}^{(ik)} - \frac{t^2}{2}h_{jj}^{(ik)}\right|^3\\
& = 1 + \frac{\mathbbm{i}tE_{ij}u_{jk}}{s_{ik}\lambda_k} - \frac{\mathbbm{i}t}{2}g_{jj}^{(ik)} - \frac{t^2}{2}h_{jj}^{(ik)} - \frac{t^2}{2}\frac{E_{ij}^2u_{jk}^2}{s_{ik}^2\lambda_k^2}\\
&\quad - \frac{\mathbbm{i}tE_{ij}u_{jk}}{s_{ik}\lambda_k}\left(\frac{\mathbbm{i}t}{2}g_{jj}^{(ik)} + \frac{t^2}{2}h_{jj}^{(ik)}\right) + \frac{1}{2}\left(\frac{\mathbbm{i}t}{2}g_{jj}^{(ik)} + \frac{t^2}{2}h_{jj}^{(ik)}\right)^2\\
&\quad
 + \frac{\theta_{ijk}}{6}\left|\frac{\mathbbm{i}tE_{ij}u_{jk}}{s_{ik}\lambda_k} - \frac{\mathbbm{i}t}{2}g_{jj}^{(ik)} - \frac{t^2}{2}h_{jj}^{(ik)}\right|^3,
\end{align*}
where $\theta_{ijk}$ is a complex-valued random variable depending on $\Lambda_{ijk}(t)$ and $|\theta_{ijk}|\leq 1$ with probability one. By Assumption \ref{assumption:Eigenvector_delocalization} and Assumption  \ref{assumption:Noise_matrix_distribution},
\begin{align*}
\max_{j\in [n]}\expect \big|g_{\chi(j)\chi(j)}^{(ik)}\big|&\leq \frac{\{\expect |E_{i\chi(j)}|^3 + \var(E_{i\chi(j)})\expect|E_{i\chi(j)}|\}\|\bu_k\|_\infty^3}{s_{ik}^3|\lambda_k|^3} = O\left(\frac{1}{nq_n}\right),\\
\max_{j\in [n]}\expect \big|g_{\chi(j)\chi(j)}^{(ik)}\big|^2&\leq \frac{2\{\expect E_{i\chi(j)}^6 + \expect E_{i \chi(j)}^2\var(E_{i \chi(j)})^2\}\|\bu_k\|_\infty^6}{s_{ik}^6|\lambda_k|^6} = O\left(\frac{1}{nq_n^4}\right),\\
\max_{j\in [n]}\expect \big|g_{\chi(j)\chi(j)}^{(ik)}\big|^3&\leq \frac{8\{\expect E_{i \chi(j)}^9 + \expect E_{i \chi(j)}^3\var(E_{i \chi(j)})^3\}\|\bu_k\|_\infty^9}{s_{ik}^9|\lambda_k|^9} = O\left(\frac{1}{nq_n^7}\right),\\
\max_{j\in [n]}\expect h_{jj}^{(ik)}& = O\left\{\frac{\rho_n}{\Delta_n^2}\left(\frac{1}{n} + \frac{1}{\sqrt{n}q_n^3}\right)\right\},\\
 \max_{j\in [n]}\expect(h_{jj}^{(ik)})^2 &= O\left\{\frac{\rho_n^2}{\Delta_n^4}\left(\frac{1}{nq_n^2} + \frac{1}{\sqrt{n}q_n^{7}}\right)\right\},\\
\max_{j\in [n]}\expect(h_{jj}^{(ik)})^3 &= O\left\{\frac{\rho_n^3}{\Delta_n^6}\left(\frac{1}{nq_n^4} + \frac{1}{\sqrt{n}q_n^{11}}\right)\right\}.
\end{align*}
Since $|t| = O(\sqrt{n\rho_n}\rho_n/\beta)$ and 
$$
\beta = \max_{i,j\in [n]}\expect|E_{ij}|^3\geq \max_{i,j\in [n]}(\expect E_{ij}^2)^{3/2} = \Theta(\rho_n^{3/2}),
$$
which can be shown by noticing $\min_{i\in[n]}s_{ik}^2 = \Theta(\rho_n/\Delta_n^2)$ with a contradiction argument, it follows that $|t| = O(\sqrt{n})$. Therefore, for any $t$, $|t|\in [n^\gamma, \bar{\eps}\sqrt{n\rho_n}\rho_n/\beta]$, we obtain
\begin{align*}
&\max_{j\in [n]}\left|\expect \exp\{\Lambda_{ijk}(t)\} - 1 + \frac{t^2}{2}\frac{\var(E_{ij})u_{jk}^2}{s_{ik}^2\lambda_k^2}\right|\\
&\quad\leq \frac{|t|}{2}\max_{j\in [n]}\expect |g_{jj}^{(ik)}| + \frac{t^2}{2}\max_{j\in [n]}\expect h_{jj}^{(ik)} + \frac{t^2}{2}\max_{j\in [n]}\expect\left|\frac{E_{ij}u_{jk}g_{jj}^{(ik)}}{s_{ik}\lambda_k}\right|\\
&\quad\quad + \frac{|t|^3}{2}\max_{j\in [n]}\expect\left|\frac{E_{ij}u_{jk}h_{jj}^{(ik)}}{s_{ik}\lambda_k}\right| + \frac{1}{2}\max_{j\in [n]}\{t^2\expect (g_{jj}^{(ik)})^2 + t^4\expect (h_{jj}^{(ik)})^2\}\\
&\quad\quad + \frac{27}{6}\max_{j\in [n]}\left\{\frac{t^3\expect|E_{ij}|^3\|\bu_k\|_\infty^3}{s_{ik}^3|\lambda_k|^3} + \frac{t^3\expect |g_{jj}^{(ik)}|^3}{8} + \frac{t^6\expect (h_{jj}^{(ik)})^3}{8}\right\}\\
&\quad = tO\left(\frac{1}{nq_n}\right)+ t^2O\left\{\frac{\rho_n}{\Delta_n^2}\left(\frac{1}{n} + \frac{1}{\sqrt{n}q_n^3}\right)\right\} + \frac{t^2}{2}\max_{j\in [n]}\left\{\frac{\expect E_{ij}^2u_{jk}^2}{s_{ik}^2\lambda_k^2}\right\}^{1/2}\{\expect(g_{jj}^{(ik)})^2\}^{1/2}\\
&\quad\quad + \frac{t^3}{2}\max_{j\in [n]}\left\{\frac{\expect E_{ij}^2u_{jk}^2}{s_{ik}^2\lambda_k^2}\right\}^{1/2}\{\expect(h_{jj}^{(ik)})^2\}^{1/2} + t^2O\left(\frac{1}{nq_n^4}\right)\\
&\quad\quad + t^4O\left\{\frac{\rho_n^2}{\Delta_n^4}\left(\frac{1}{nq_n^2} + \frac{1}{\sqrt{n}q_n^7}\right)\right\}
+ t^3O\left\{\frac{\beta}{(n\rho_n)^{3/2}}\right\} + t^3O\left(\frac{1}{nq_n^7}\right)\\
&\quad\quad + t^6O\left\{\frac{\rho_n^3}{\Delta_n^6}\left(\frac{1}{nq_n^4} + \frac{1}{\sqrt{n}q_n^{11}} \right)\right\}\\
&\quad = tO\left(\frac{1}{nq_n}\right)+ t^2O\left\{\frac{\rho_n}{\Delta_n^2}\left(\frac{1}{n} + \frac{1}{\sqrt{n}q_n^{3}}\right)\right\} + t^2O\left(\frac{1}{nq_n^2}\right)\\
&\quad\quad + t^3O\left\{\frac{\rho_n}{\Delta_n^2}\left(\frac{1}{nq_n} + \frac{1}{n^{3/4}q_n^{7/2}}\right)\right\}
+ t^2O\left(\frac{1}{nq_n^4}\right)\\
&\quad\quad + t^4O\left\{\frac{\rho_n^2}{\Delta_n^4}\left(\frac{1}{nq_n^2} + \frac{1}{\sqrt{n}q_n^7}\right) \right\} + t^3O\left\{\frac{\beta}{(n\rho_n)^{3/2}}\right\} + t^3O\left(\frac{1}{nq_n^7}\right)\\
&\quad\quad + t^6O\left\{\frac{\rho_n^3}{\Delta_n^6}\left(\frac{1}{nq_n^4} + \frac{1}{\sqrt{n}q_n^{11}}\right)\right\}\\
&\quad = \frac{t^2}{n}O\left(\frac{1}{q_n}\right)+ \frac{t^2}{n}O\left\{\frac{n\rho_n}{\Delta_n^2}\left(\frac{1}{n} + \frac{1}{\sqrt{n}q_n^3}\right)\right\} + \frac{t^2}{n}O\left(\frac{1}{q_n^2}\right)
\\
&\quad\quad + \frac{t^2}{n}O\left(\frac{\sqrt{n}\rho_n}{\Delta_n^2q_n} + \frac{n^{3/4}\rho_n}{\Delta_n^2q_n^{7/2}}\right) + \frac{t^2}{n}O\left(\frac{1}{q_n^4}\right)\\
&\quad\quad + \frac{t^2}{n}O\left(\frac{n\rho_n^2}{\Delta_n^4q_n^2} + \frac{n^{3/2}\rho_n^2}{\Delta_n^4q_n^7}\right) + \frac{t^2}{n}\frac{\bar{\eps}\sqrt{n}\rho_n^{3/2}}{\beta}O\left(\frac{\beta}{\sqrt{n}\rho_n^{3/2}}\right) + \frac{t^2}{n}O\left(\frac{\sqrt{n}}{q_n^7}\right)\\
&\quad\quad + \frac{t^2}{n}O\left(\frac{n^2\rho_n^3}{\Delta_n^6q_n^4} + \frac{n^{5/2}\rho_n^3}{\Delta_n^6q_n^{11}}\right)\\
&\quad \leq \frac{M\bar{\eps}t^2}{n}
\end{align*}
for some constant $M > 0$. 
For $j = 1,\ldots,\lfloor c_1n/(2C^2)\rfloor$, we have
\[
\frac{c}{n}\leq \frac{\var(E_{i\chi(j)})u_{\chi(j)k}}{s_{ik}^2\lambda_k^2}\leq \frac{C^2\rho_n}{ns_{ik}^2\lambda_k^2} = O\left(\frac{1}{n}\right).
\]
This further implies that for $j = 1,\ldots,\lfloor c_1n/(2C^2)\rfloor$,
\[
f_{i\chi(j)k}(t) = \expect \exp\{\Lambda_{i\chi(j)k}(t)\} = 1 - \frac{t^2}{2}\frac{\var(E_{i\chi(j)})u_{\chi(j)k}}{s_{ik}^2\lambda_k^2} + \frac{M\theta_{ijk}'\bar{\eps}t^2}{n},
\]
where $\theta_{ijk}'\in\mathbb{C}$ with $|\theta_{ijk}'|\leq 1$. By setting $\bar{\eps} > 0$ to be sufficiently small, we obtain
\[
\max_{j = 1,\ldots,\lfloor c_1n/(2C^2)\rfloor}|f_{i\chi(j)k}(t)|\leq 1 - \frac{ct^2}{3n}\leq \exp\left(- \frac{ct^2}{3n}\right). 
\]
Now let $m = m(t) = \lceil(6n\log n)/(ct^2)\rceil + 2$. Clearly, $mt^2 = O(n\log n)$ and $mt^2\geq (m - 2)t^2\geq (6n\log n)/c$. Furthermore, since $|t|\geq n^\gamma$ for some small $\gamma > 0$, we see that $m\leq c_1n/(2C^2)$ is satisfied. 
The modulus of the first term in line \eqref{eqn:CHF_intermediate_term_I} can be written as
\begin{align*}
&\left|\expect \exp\left\{\sum_{j = 1}^n\Lambda_{i\chi(j)k}(t) - \sum_{a = m + 1}^{n - 1}\sum_{b = a + 1}^n\left(\frac{\mathbbm{i}t}{2}g_{\chi(a)\chi(b)}^{(ik)} + \frac{t^2}{2}h_{\chi(a)\chi(b)}^{(ik)}\right)\right\}\right|\\
&\quad = \left|\expect\exp\left\{\sum_{j = 1}^m\Lambda_{i\chi(j)k}(t)\right\}\right|
\\&\quad\quad\times 
\left|\expect \exp\left\{\sum_{j = m + 1}^n\Lambda_{i\chi(j)k}(t) - \sum_{a = m + 1}^{n - 1}\sum_{b = a + 1}^n\left(\frac{\mathbbm{i}t}{2}g_{\chi(a)\chi(b)}^{(ik)} + \frac{t^2}{2}h_{\chi(a)\chi(b)}^{(ik)}\right)\right\}\right|\\
&\quad = \left|\prod_{j = 1}^m f_{i\chi(j)k}(t)\right|
\\&\quad\quad\times \left|\expect \exp\left\{\sum_{j = m + 1}^n\Lambda_{i\chi(j)k}(t) - \sum_{a = m + 1}^{n - 1}\sum_{b = a + 1}^n\left(\frac{\mathbbm{i}t}{2}g_{\chi(a)\chi(b)}^{(ik)} + \frac{t^2}{2}h_{\chi(a)\chi(b)}^{(ik)}\right)\right\}\right|\\
&\quad\leq  \left|\prod_{j = 1}^m f_{i\chi(j)k}(t)\right|\leq \exp\left(-\frac{cmt^2}{3n}\right)\leq \exp(-2\log n) = \frac{1}{n^2}.
\end{align*}
Now we analyze the term in line \eqref{eqn:CHF_intermediate_term_II}. Write
\begin{align*}
&\left|\expect \exp\left\{\sum_{j = 1}^n\Lambda_{i\chi(j)k} - \sum_{a = m + 1}^{n - 1}\sum_{b = a + 1}^n\left(\frac{\mathbbm{i}t}{2}g_{\chi(a)\chi(b)}^{(ik)} + \frac{t^2}{2}h_{\chi(a)\chi(b)}^{(ik)}\right)\right\}\Gamma_{ik}(m)\right|\\
&\quad\leq\sum_{a = 1}^m\sum_{b = m + 1}^n\left|\expect \exp\left\{\sum_{j = 1,j\notin\{a,b\}}^m\Lambda_{i\chi(j)k}(t)\right\}\right|\\
&\quad\quad\times\expect\left[\left|\exp\left\{\sum_{j\in([n]\backslash\{m\})\cup\{a,b\}}^m\Lambda_{i\chi(j)k}(t)\right.\right.\right.\\
&\quad\quad\quad\quad 
\left.
\left. - \sum_{a' = m + 1}^{n - 1}\sum_{b' = a + 1}^n\left(\frac{\mathbbm{i}t}{2}g_{\chi(a')\chi(b')}^{(ik)} + \frac{t^2}{2}h_{\chi(a')\chi(b')}^{(ik)}\right)\right\}\right|
\left.\left|\frac{\mathbbm{i}t}{2}g_{\chi(a)\chi(b)}^{(ik)} + \frac{t^2}{2}h_{\chi(a)\chi(b)}^{(ik)}\right|\right]\\
&\quad\leq \sum_{a = 1}^m\sum_{b = m + 1}^n\left|\prod_{j = 1,j\notin\{a,b\}}^mf_{i\chi(j)k}(t)\right|\left(\frac{t}{2}\expect|g_{\chi(a)\chi(b)}^{(ik)}| + \frac{t^2}{2}\expect h_{\chi(a)\chi(b)}^{(ik)}\right)\\
&\quad\leq \sum_{a = 1}^m\sum_{b = m + 1}^n\left|\prod_{j = 1,j\notin\{a,b\}}^mf_{i\chi(j)k}(t)\right|\\
&\quad\quad\times\left(\frac{t}{2}\frac{\{\expect|E_{i\chi(a)}|\expect|E_{i\chi(b)}^2 - \var(E_{i\chi(b)})| + \expect|E_{i\chi(b)}|\expect|E_{i\chi(a)}^2 - \var(E_{i\chi(a)})|\}\|\bu_k\|_\infty^3}{s_{ik}^3|\lambda_k|^3}\right.\\
&\quad\quad\quad\quad\left. + \frac{t^2}{2}\expect h_{\chi(a)\chi(b)}^{(ik)}\right)\\
&\quad = \max_{a,b\in [n]}\left|\prod_{j = 1,j\notin\{a,b\}}^mf_{i\chi(j)k}(t)\right|\left\{mtO\left(\frac{1}{\sqrt{n}}\right) + mt^2O\left(\frac{\rho_n}{\Delta_n^2} + \frac{\sqrt{n}\rho_n}{\Delta_n^2q_n^3}\right)\right\}\\
&\quad\leq \exp\left\{-\frac{(m - 2)ct^2}{3n}\right\}O\left\{\frac{n\log  n}{\sqrt{n}} + n\log n\left(\frac{\rho_n}{\Delta_n^2} + \frac{\sqrt{n}\rho_n}{\Delta_n^2q_n^3}\right)\right\} = O\left(\frac{1}{n} + \frac{n\rho_n}{\Delta_n^2}\right)\\
&\quad = O\left(\frac{n\rho_n}{\Delta_n^2}\times \frac{\Delta_n^2}{n^2\rho_n} + \frac{n\rho_n}{\Delta_n^2}\right) = O\left(\frac{n\rho_n}{\Delta_n^2}\right).
\end{align*}
By \eqref{eqn:Gamma_m_bound}, the modulus of the term in line \eqref{eqn:CHF_intermediate_term_III} is bounded by
\begin{align*}
 &mt^2O\left(\frac{1}{nq_n^2}\right) + mt^4O\left\{\frac{\rho_n^2}{\Delta_n^4}\left(m + \frac{mn}{q_n^6} + \frac{n}{q_n^2} + \frac{n^{3/2}}{q_n^7}\right)\right\}\\
 &\quad = O\left(\frac{n\log n}{nq_n^2}\right) + O\left\{\frac{n^2\rho_n^2(\log n)^2}{\Delta_n^4}\left(1 + \frac{n}{q_n^6}\right)\right\} + t^2O\left\{\frac{n^2\rho_n^2\log n}{\Delta_n^4}\left(\frac{1}{q_n^2} + \frac{\sqrt{n}}{q_n^7}\right)\right\}.
\end{align*}
Combining the upper bounds for lines \eqref{eqn:CHF_intermediate_term_I}-\eqref{eqn:CHF_intermediate_term_III}, we see that for $t\in [n^\gamma, \bar{\eps}\sqrt{n}\rho_n^{3/2}/\beta]$ for sufficiently small $\bar{\eps} > 0$,
\begin{align*}
&\int_{n^\gamma}^{\bar{\eps}\sqrt{n}\rho_n^{3/2}/\beta}\left|\frac{\expect \exp\{\mathbbm{i}t(\widetilde{T}_{ik} + \widetilde{\Delta}_{ik})\}}{t}\right|\mathrm{d}t\\
&\quad = O\left\{\frac{(\log n)}{q_n^2} + \frac{n\rho_n}{\Delta_n^2} + \frac{(n\rho_n)^2}{\Delta_n^4}\frac{(\log n)^2n}{q_n^6}\right\}\int_{n^\gamma}^{\bar{\eps}\sqrt{n}\rho_n^{3/2}/\beta}\frac{\mathrm{d}t}{t}\\
&\quad\quad + O\left\{\frac{n^2\rho_n^2\log n}{\Delta_n^4}\left(\frac{1}{q_n^2} + \frac{\sqrt{n}}{q_n^7}\right)\right\}\int_{n^\gamma}^{\bar{\eps}\sqrt{n}\rho_n^{3/2}/\beta}t\mathrm{d}t\\
&\quad = O\left\{\frac{(\log n)^2}{q_n^2} + \frac{n\rho_n(\log n)}{\Delta_n^2} + \frac{(n\rho_n)^2}{\Delta_n^4}\frac{(\log n)^3n}{q_n^6}
+ \frac{n^2\rho_n^2\log n}{\Delta_n^4}\left(\frac{1}{q_n^2} + \frac{\sqrt{n}}{q_n^7}\right)\frac{n\rho_n^3}{\beta^2}\right\}\\
&\quad = O\left\{\frac{(\log n)^2}{q_n^2} + \frac{n\rho_n(\log n)}{\Delta_n^2} + \frac{(n\rho_n)}{\Delta_n^2}\frac{(\log n)^3}{n^{2\beta_\Delta - 1 + 6\eps}}
\right\}\\
&\quad\quad + O\left\{
\frac{n\rho_n}{\Delta_n^2}\frac{(\log n)}{n^{2\beta_\Delta - 1 - 2\eps}}
+ \frac{\sqrt{n\rho_n}}{\Delta_nq_n}\frac{\log n}{n^{3\beta_\Delta - 1}}\frac{1}{n^{6\eps - 1/2}}\right\}\\
&\quad = O\left\{\frac{(\log n)^2}{q_n^2} + \frac{n\rho_n(\log n)}{\Delta_n^2} + \frac{\sqrt{n\rho_n}}{\Delta_nq_n}
\right\}\\
&\quad = O\left\{\frac{(\log n)^2}{q_n^2} + \frac{n\rho_n(\log n)}{\Delta_n^2}\right\}.
\end{align*}
The same upper bound also holds for the integral over $[-\bar{\eps}\sqrt{n}\rho_n^{3/2}/\beta,-n^\gamma]$. 
The proof is thus completed. 
\end{proof}

\begin{lemma}\label{lemma:CHF_III}
Suppose Assumptions \ref{assumption:Signal_strength}--\ref{assumption:Eigenvector_delocalization} hold and $\min_{i\in[n]}s_{ik}^2 = \Theta(\rho_n/\Delta_n^2)$. If $\beta_\Delta\in(1/2 - \eps, 1/2]$ and $\eps \geq 1/3$, then there exists a sequence $M_n\to\infty$, such that for any $\gamma\in(0, \eps)$ (recall that $\eps$ satisfies Assumption \ref{assumption:Noise_matrix_distribution}),
\begin{align*}
&\int_{\{t:n^\gamma\leq|t|\leq M_n(n^{2\beta_\Delta}\log n)^{1/2}\}}\left|\frac{\expect\exp\{\mathbbm{i}t(\widetilde{T}_{ik} + \widetilde{\Delta}_{ik})\}}{t}\right|\mathrm{d}t
 = O\left\{\frac{(\log n)^4}{q_n^2} + \frac{n\rho_n\log n}{\Delta_n^2}\right\}
.
\end{align*}
\end{lemma}

\begin{proof}
The proof under condition (i) (i.e., $\beta = o(n\rho_n^2/(\Delta_n\sqrt{\log n}))$) follows directly from Lemma \ref{lemma:CHF_intermediate} by setting $M_n = \bar{\eps}\rho_n/(\beta\sqrt{\log n})$. It is sufficient to show that under condition (ii) (i.e., $\beta = \Omega(n\rho_n^2/(\Delta_n\sqrt{\log n}))$ but the Cram\'er's condition holds), we have
\begin{align*}
&\int_{\{t:\bar{\eps}\sqrt{n}\rho_n^{3/2}/\beta\leq|t|\leq M_n(n^{2\beta_\Delta}\log n)^{1/2}\}}\left|\frac{\expect\exp\{\mathbbm{i}t(\widetilde{T}_{ik} + \widetilde{\Delta}_{ik})\}}{t}\right|\mathrm{d}t
= O\left\{\frac{(\log n)^4}{q_n^2}\right\}.
\end{align*}
Let $\calJ_{ik} = \{j\in [n]:|u_{jk}|\geq 1/\sqrt{2n}\}$. By Assumption \ref{assumption:Eigenvector_delocalization}, there exists a constant $C\geq 1$ such that $\|\bu_k\|_\infty\leq C/\sqrt{n}$. It follows that
\[
1 = \sum_{j\in\calJ_{ik}}u_{jk}^2 + \sum_{j\notin \calJ_{ik}}u_{jk}^2\leq \frac{C^2|\calJ_{ik}|}{n} + \frac{(n - |\calJ_{ik}|)}{2n}\leq \frac{C^2|\calJ_{ik}|}{n} + \frac{1}{2}.
\]
This implies that $|\calJ_{ik}|\geq n/(2C^2)$. Therefore, there exists a permutation $\chi:[n]\to[n]$, such that for all $j = 1,\ldots,\lfloor n/(2C^2)\rfloor$, $|u_{\chi(j)k}|\geq 1/\sqrt{2n}$. 
Following the notations in the proof of Lemma \ref{lemma:CHF_intermediate}, we let
\[
\Lambda_{ijk}(t) = \frac{\mathbbm{i}tE_{ij}u_{jk}}{s_{ik}\lambda_k} - \frac{\mathbbm{i}t}{2}g_{jj}^{(ik)} - \frac{t^2}{2}h_{jj}^{(ik)}\quad\mbox{and}\quad f_{ijk}(t) = \expect \exp\{\Lambda_{ijk}(t)\},
\]
where $g_{jj}^{(ik)}$ is defined in \eqref{eqn:gab_formula} and $h_{jj}^{(ik)}$ is defined in \eqref{eqn:hab_formula}. 
Then following the proof of Lemma \ref{lemma:CHF_intermediate}, for any integer $m\in [n - 1]$, we have 
\begin{align*}
|\expect \exp\{\mathbbm{i}t(\widetilde{T}_{ik} + \widetilde{\Delta}_{ik})\}|
& = \prod_{j = 1}^m|f_{i\chi(j)k}(t)| + \max_{a,b\in[n]}\prod_{j = 1,j\notin\{a,b\}}^m\left|
f_{i\chi(j)k}(t)
\right|\\
&\quad\times \left\{m|t|O\left(\frac{1}{\sqrt{n}}\right) + mt^2O\left(\frac{\rho_n}{\Delta_n^2} + \frac{\sqrt{n}\rho_n}{\Delta_n^2q_n^3}\right)\right\}\\
&\quad + mt^2O\left(\frac{1}{nq_n^2}\right) + mt^4O\left\{\frac{\rho_n^2}{\Delta_n^4}\left(m + \frac{mn}{q_n^6} + \frac{n}{q_n^2} + \frac{n^{3/2}}{q_n^7}\right)\right\}.
\end{align*}
Now we analyze $f_{i\chi(j)t}$ for $j\in\{1,\ldots,\lfloor n/(2C^2)\rfloor\}$. 
Let $\psi_{ij}(t) = \expect\exp\{\mathbbm{i}t\rho_nE_{ij}/\beta\}$. 
Because $\psi_{ij}(t)$ is the characteristic function of $\rho_nE_{ij}/\beta$ and $\limsup_{t\to\pm\infty}|\psi_{ij}(t)|\leq 1 - \eta$, by Theorem 1 in Section 2 of Chapter I in \cite{petrov2012sums}, we see that $|\psi_{ij}(t)| \leq 1 - \eta'$ when $|t| \geq c$ for some constant $\eta' \in (0, 1)$. 
Since
\[
\mathbf{Re}\left( - \frac{\mathbbm{i}t}{2}g_{jj}^{(ik)} - \frac{t^2}{2}h_{jj}^{(ik)}\right)\leq 0, 
\]
it follows that
\[
\exp\left\{- \frac{\mathbbm{i}t}{2}g_{jj}^{(ik)} - \frac{t^2}{2}h_{jj}^{(ik)}\right\} = 1 + \theta_{ijk}\left|\frac{\mathbbm{i}t}{2}g_{jj}^{(ik)} + \frac{t^2}{2}h_{jj}^{(ik)}\right|,
\]
where $\theta_{ijk}$ is a complex-valued random variable with $|\theta_{ijk}|\leq 1$ with probability one. This entails that
\begin{align*}
|f_{i\chi(j)k}(t)|& = \left|\expect\exp\left\{\frac{\mathbbm{i}tE_{i\chi(j)}u_{\chi(j)k}}{s_{ik}\lambda_k}\right\}\left\{1 + \theta_{i\chi(j)k}\left|\frac{\mathbbm{i}t}{2}g_{\chi(j)\chi(j)}^{(ik)} + \frac{t^2}{2}h_{\chi(j)\chi(j)}^{(ik)}\right|\right\}\right|\\
&\leq \left|\expect\exp\left\{\frac{\mathbbm{i}tE_{i\chi(j)}u_{\chi(j)k}}{s_{ik}\lambda_k}\right\}\right| + \expect\left|\frac{\mathbbm{i}t}{2}g_{\chi(j)\chi(j)}^{(ik)} + \frac{t^2}{2}h_{\chi(j)\chi(j)}^{(ik)}\right|\\
&\leq \left|\expect\exp\left\{\frac{\mathbbm{i}tE_{i\chi(j)}u_{\chi(j)k}}{s_{ik}\lambda_k}\right\}\right| + tO\left(\frac{1}{nq_n}\right) + t^2O\left\{\frac{\rho_n}{\Delta_n^2}\left(\frac{1}{n} + \frac{1}{\sqrt{n}q_n^4}\right)\right\}.
\end{align*}
By picking a slowing growing sequence $M_n\to\infty$, we can obtain that $tO\{1/(nq_n)\} + t^2O\{1/(n^3q_n)\}\leq \eta'/2$. Then we immediately have
\[
\left|\expect\exp\left\{\frac{\mathbbm{i}tE_{i\chi(j)}u_{\chi(j)k}}{s_{ik}\lambda_k}\right\}\right| = \left|\expect\exp\left\{\mathbbm{i}\left(\frac{\beta u_{\chi(j)k}t}{\rho_ns_{ik}\lambda_k}\right)\frac{\rho_nE_{i\chi(j)}}{\beta}\right\}\right| = \psi_{i\chi(j)}\left(\frac{\beta u_{\chi(j)k}t}{\rho_ns_{ik}\lambda_k}\right).
\]
For $j \leq m\leq n/(2C^2)$ and $t\geq \bar{\eps}\sqrt{n}\rho_n^{3/2}/\beta$, we have
\[
\left|\frac{\beta u_{\chi(j)k}t}{\rho_ns_{ik}\lambda_k}\right|\geq c
\]
for some constant $c > 0$. Therefore, we obtain $|f_{i\chi(j)k}(t)|\leq 1 - \eta'/2$ for $j\leq m\leq n/(2C^2)$ and $|t|\geq\bar{\eps}\sqrt{n}\rho_n^{3/2}/\beta$. Taking $m = \lceil-2\log n/\log(1 - \eta'/2)\rceil + 2$ and a slowing growing sequence $M_n\to\infty$ entails that
\begin{align*}
&|\expect \exp\{\mathbbm{i}t(\widetilde{T}_{ik} + \widetilde{\Delta}_{ik})\}|\\
&\quad\leq\exp\left\{m\log\left(1 - \frac{\eta'}{2}\right)\right\} + \exp\left\{(m - 2)\log\left(1 - \frac{\eta'}{2}\right)\right\}\\
&\quad\quad\times \left\{m|t|O\left(\frac{1}{\sqrt{n}}\right) + mt^2O\left(\frac{\rho_n}{\Delta_n^2} + \frac{\sqrt{n}\rho_n}{\Delta_n^2q_n^3}\right)\right\}\\
&\quad\quad + mt^2O\left(\frac{1}{nq_n^2}\right) + mt^4O\left\{\frac{\rho_n^2}{\Delta_n^4}\left(m + \frac{mn}{q_n^6} + \frac{n}{q_n^2} + \frac{n^{3/2}}{q_n^7}\right)\right\}\\
&\quad = O\left(\frac{1}{n^2}\right) + O\left(\frac{1}{n}\right) + O\left\{\frac{(\log n)^3}{q_n^2}\right\} + 
O\left(\frac{1}{q_n^2}\right)
 = O\left\{\frac{(\log n)^3}{q_n^2}\right\}.
\end{align*}
Hence, we conclude that
\begin{align*}
&\int_{\{t:\bar{\eps}\sqrt{n}\rho_n^{3/2}/\beta\leq|t|\leq M_n(n^{2\beta_\Delta}\log n)^{1/2}\}}\left|\frac{\expect \exp\{\mathbbm{i}t(\widetilde{T}_{ik} + \widetilde{\Delta}_{ik})\}}{t}\right|\mathrm{d}t\\
&\quad = O\left\{\frac{(\log n)^3}{q_n^2}\right\}\int_{\{t:\bar{\eps}\sqrt{n}\rho_n^{3/2}/\beta\leq|t|\leq M_n(n^{2\beta_\Delta}\log n)^{1/2}\}}\frac{\mathrm{d}t}{|t|}
= O\left\{\frac{(\log n)^4}{q_n^2}\right\}.
\end{align*}
The proof is thus completed.
\end{proof}

\begin{lemma}\label{lemma:CHF_IV}
Suppose Assumption \ref{assumption:Signal_strength}--\ref{assumption:Eigenvector_delocalization} hold and $\min_{i\in[n]}s_{ik}^2 = \Theta(\rho_n/\Delta_n^2)$. If $\beta_\Delta\in[1/2 - \eps, 1/2]$ and $\eps \geq 1/3$, then there exists some sufficiently small $\gamma\in(0, \eps)$, such that
\[
\int_{-n^\gamma}^{n^\gamma}\left|\frac{\expect e^{\mathbbm{i}t(\widetilde{T}_{ik} + \widetilde{\Delta}_{ik})} - \mathrm{ch.f.}(t; G_n^{(ik)})}{t}\right|\mathrm{d}t = O\left(\frac{n\rho_n}{\Delta_n^2} + \frac{1}{q_n^2}\right),
\]
where $\mathrm{ch.f.}(t; G_n^{(ik)}) = \int_{-\infty}^{+\infty}e^{\mathbbm{i}tx}\mathrm{d}G_n^{(ik)}(x)$ is the Fourier-Stieltjes transform of $G_n^{(ik)}(x)$.
\end{lemma}

\begin{proof}
Let 
\bea\label{def:gamma}
\gamma = \min(\beta_\Delta/2, \eps/4).
\eae
Recall that
\[
\expect\exp\{\mathbbm{i}t(\widetilde{T}_{ik} + \widetilde{\Delta}_{ik})\} = \expect\exp\left\{\mathbbm{i}t\widetilde{T}_{ik} - \frac{t^2}{2}\sigma_{ik}^2(\be_i\transpose\bE)\right\},
\]
By Taylor's theorem, 
\[
\exp\left\{-\frac{t^2}{2}\sigma_{ik}^2(\be_i\transpose\bE)\right\} = 1 - \frac{t^2}{2}\sigma_{ik}^2(\be_i\transpose\bE) + \frac{t^4}{8}\theta_{ik}\sigma_{ik}^4(\be_i\transpose),
\]
where $\theta_{ik}$ is a random variable depending on $\sigma_{ik}^2(\be_i\transpose\bE)$ and $|\theta_{ik}|\leq 1$ with probability one. This entails that
\begin{equation}
\label{eqn:CHF_IV_decomposition_I}
\begin{aligned}
\expect\exp\{\mathbbm{i}t(\widetilde{T}_{ik} + \widetilde{\Delta}_{ik})\}
& = \expect\exp(\mathbbm{i}t\widetilde{T}_{ik}) - \frac{t^2}{2} \expect\exp(\mathbbm{i}t\widetilde{T}_{ik})\sigma_{ik}^2(\be_i\transpose\bE)\\
&\quad + \frac{t^4}{8}\expect\exp(\mathbbm{i}t\widetilde{T}_{ik})\theta_{ik}\sigma_{ik}^4(\be_i\transpose\bE).
\end{aligned}
\end{equation}
The remaining proof is lengthy and we decompose it into the following three steps:
\begin{itemize}
  \item In \textbf{Step 1}, we establish that
  \[
  \int_{-n^\gamma}^{n^\gamma}\left|\frac{1}{t}\frac{t^4}{8}\expect\exp(\mathbbm{i}t\widetilde{T}_{ik})\theta_{ik}\sigma_{ik}^4(\be_i\transpose\bE)\right|\mathrm{d}t = O\left(\frac{n\rho_n}{\Delta_n^2}\right).
  \]
  \item In \textbf{Step 2}, we establish that
  \[
  \int_{-n^\gamma}^{n^\gamma}\left|\frac{\expect\exp(\mathbbm{i}t\widetilde{T}_{ik}) - \mathrm{ch.f.}(t; G_n^{(ik)})}{t}\right|\mathrm{d}t = O\left(\frac{1}{q_n^2}\right).
  \]
  \item In \textbf{Step 3}, we establish that 
  \[
  \int_{-n^\gamma}^{n^\gamma}\left|\frac{1}{t}\frac{t^2}{2}\expect\exp(\mathbbm{i}t\widetilde{T}_{ik})\sigma_{ik}^2(\be_i\transpose\bE)\right|\mathrm{d}t = O\left(\frac{n\rho_n}{\Delta_n^2} \right).
  \]
\end{itemize}
$\blacksquare$ \textbf{Step 1.} Recall that
\[
\sigma_{ik}^2(\be_i\transpose\bE) = \sum_{a < b}h_{ab}^{(ik)} + \sum_{a = 1}^nh_{aa}^{(ik)},
\]
where $h_{ab}^{(ik)}$ is given in \eqref{eqn:hab_formula}. Since $\rho_n/(s_{ik}^2\lambda_k^2) = \Theta(1)$, by Assumption \ref{assumption:Eigenvector_delocalization} and Assumption \ref{assumption:Noise_matrix_distribution}, there exists a constant $C > 0$ such that $h_{ab}^{(ik)}\leq C\bar{h}_{ab}^{(ik)}$, where
\begin{align*}
\bar{h}_{ab}^{(ik)} =
\left\{\begin{aligned}
&\frac{1}{n^3} + \frac{(E_{ia}^2 + E_{ib}^2)}{n\Delta_n^2} + \frac{(\be_i\transpose\bE\bu_k)^2(E_{ia}^4 + E_{ib}^4)}{n^2\rho_n^2\Delta_n^2}
,&\quad&\mbox{if }a\neq b,\\
& \frac{1}{n^3} + \frac{E_{ia}^2}{n\Delta_n^2} + \frac{(\be_i\transpose\bE\bu_k)^2E_{ia}^4}{n^2\rho_n^2\Delta_n^2},&\quad&\mbox{if }a\neq b.
\end{aligned}
\right.
\end{align*}
By \eqref{firstorder:h}, it is clear that 
\begin{align*}
\max_{a,b\in[n]}\expect \bar{h}_{ab}^{(ik)} = O\left\{\frac{\rho_n}{\Delta_n^2}\left(\frac{1}{n} + \frac{1}{\sqrt{n}q_n^3}\right)\right\},\;
\max_{a,b\in[n]}\expect (\bar{h}_{ab}^{(ik)})^2 = O\left\{\frac{\rho_n^2}{\Delta_n^4}\left(\frac{1}{nq_n^2} + \frac{1}{\sqrt{n}q_n^7}\right)\right\}.
\end{align*}
Furthermore,  the proof of Lemma \ref{lemma:CHF_intermediate} entails that
\begin{align}
\label{eqn:sigma_ik_fourth_moment_upper_bound}
\expect\sigma_{ik}^4(\be_i\transpose\bE)
&\leq C^2\expect\left(\sum_{a < b}\bar{h}_{ab}^{(ik)}+ \sum_{a = 1}^n\bar{h}_{aa}^{(ik)}\right)^2 = O\left\{\frac{(n\rho_n)^2}{\Delta_n^4}\left(1 + \frac{n}{q_n^6} + \frac{\sqrt{n}}{q_n^7}\right)\right\}\\
& = O\left\{\frac{(n\rho_n)}{\Delta_n^2}\left(\frac{1}{n^{2\beta_\Delta}} + \frac{1}{n^{2\beta_\Delta - 1 + 6\eps}} + \frac{1}{n^{2\beta_\Delta - 1/2 + 7\eps}}\right)\right\}\nonumber
\end{align}
The integral of the modulus of the third term on the right-hand side of \eqref{eqn:CHF_IV_decomposition_I} divided by $t$ can be bounded as
\begin{align*}
&\int_{-n^\gamma}^{n^\gamma}\left|\frac{1}{t}\frac{t^4}{8}\expect\exp(\mathbbm{i}t\widetilde{T}_{ik})\theta_{ik}\sigma_{ik}^4(\be_i\transpose\bE)\right|\mathrm{d}t\\
&\quad \leq \int_{-n^\gamma}^{n^\gamma}\frac{|t|^3}{8}\expect\sigma_{ik}^4(\be_i\transpose\bE)\mathrm{d}t\\
&\quad = O\left\{\frac{(n\rho_n)}{\Delta_n^2}\left(\frac{n^{4\gamma}}{n^{2\beta_\Delta}} + \frac{n^{4\gamma}}{n^{2\beta_\Delta - 1 + 6\eps}} + \frac{n^{4\gamma}}{n^{2\beta_\Delta - 1/2 + 7\eps}}\right)\right\}\\
&\quad = O\left\{\frac{(n\rho_n)}{\Delta_n^2}\left(1 + \frac{1}{n^{2\beta_\Delta - 1 + 2\eps}} + \frac{1}{n^{2\beta_\Delta - 1/2 + 3\eps}}\right)\right\}\\
&\quad = O\left( \frac{n\rho_n}{\Delta_n^2}\right).
\end{align*}
$\blacksquare$ \textbf{Step 2.} Recall that 
\bea\nonumber
T_{ik}^\sharp = \frac{\be_i\transpose\bE\bu_k}{s_{ik}\lambda_k}&,\quad \delta_{ik} = \sum_{j = 1}^n\frac{(E_{ij}^2 - \sigma_{ij}^2)u_{jk}^2}{s_{ik}^2\lambda_k^{2}},
\\
\expect \exp(\mathbbm{i}t\widetilde{T}_{ik}) &= \expect\exp\left(\mathbbm{i}tT_{ik}^\sharp - \frac{\mathbbm{i}t}{2}T_{ik}^\sharp\delta_{ik}\right). 
\eae 
By Taylor's theorem, 
\begin{align*}
\exp\left(-\frac{\mathbbm{i}t}{2}T_{ik}^\sharp\delta_{ik}\right)
& = 1 -\frac{\mathbbm{i}t}{2}T_{ik}^\sharp\delta_{ik} - \frac{t^2}{8}(T_{ik}^\sharp\delta_{ik})^2 + \frac{\theta_{ik}t^3}{24}|T_{ik}^\sharp\delta_{ik}|^3,
\end{align*}
where $\theta_{ik}$ is a complex-valued random variable depending on $T_{ik}^\sharp$ and $\delta_{ik}$ and $|\theta_{ik}|\leq 1$ with probability one. It follows that
\begin{equation}
\label{eqn:CHF_IV_decomposition_II}
\begin{aligned}
\expect e^{\mathbbm{i}t\widetilde{T}_{ik}}
& = \expect e^{\mathbbm{i}tT_{ik}^\sharp} - \frac{\mathbbm{i}t}{2}\expect e^{\mathbbm{i}tT_{ik}^\sharp}(T_{ik}^\sharp\delta_{ik})
 - \frac{t^2}{8}\expect e^{\mathbbm{i}tT_{ik}^\sharp}(T_{ik}^\sharp\delta_{ik})^2\\&\quad
 + \frac{t^3}{24}\expect e^{\mathbbm{i}tT_{ik}^\sharp}\theta_{ik}|T_{ik}^\sharp\delta_{ik}|^3
\end{aligned}
\end{equation} 
By Lemma A.1 in \cite{erdos2013}, we know that
\begin{align}
\label{eqn:Tik_sharp_moment_bound}
\expect |T_{ik}^\sharp|^6 &= O\left\{(n\rho_n)^3\left(\frac{\|\bu_k\|_\infty}{s_{ik}\lambda_k}\right)^6\right\} = O(1),\\
\label{eqn:delta_ik_moment_bound}
\expect |\delta_{ik}|^6 &= O\left\{(n\rho_n)^6\left(\frac{\|\bu_k\|_\infty^2}{s_{ik}^2\lambda_k^2q_n}\right)^6\right\} = O\left(\frac{1}{q_n^6}\right).
\end{align}
Therefore, we see that the fourth term in \eqref{eqn:CHF_IV_decomposition_II} satisfies
\begin{align*}
\int_{-n^\gamma}^{n^\gamma}\frac{1}{|t|}\left|\frac{t^3}{24}\expect e^{\mathbbm{i}t\widetilde{T}_{ik}}\theta_{ik}|T_{ik}^\sharp\delta_{ik}|^3\right|\mathrm{d}t&\leq \expect|T_{ik}^\sharp\delta_{ik}|^3\int_{-n^\gamma}^{n^\gamma}\frac{t^2}{24}\mathrm{d}t \leq \frac{n^{3\gamma}}{36}(\expect|T_{ik}^\sharp|^6)^{1/2}(\expect\delta_{ik}^6)^{1/2}\\
& = O\left(\frac{n^{3\eps/4}}{q_n^3}\right) = O\left(\frac{1}{q_n^2}\right).
\end{align*}
The remaining parts of Step 2 breakdowns into the following mini-steps:
\begin{itemize}
  \item \textbf{Step 2 (I):} Show that
  \[
  \int_{-n^\gamma}^{n^\gamma}\frac{1}{|t|}\left|\frac{t^2}{8}\expect e^{\mathbbm{i}tT_{ik}^\sharp}(T_{ik}^\sharp\delta_{ik})^2\right|\mathrm{d}t = O\left(\frac{1}{q_n^2}\right).
  \]
  \item \textbf{Step 2 (II):} Show that
  \[
  \int_{-n^\gamma}^{n^\gamma}\frac{1}{|t|}\left|\expect e^{\mathbbm{i}tT_{ik}^\sharp} - \frac{\mathbbm{i}t}{2}\expect e^{\mathbbm{i}T_{ik}^\sharp}(T_{ik}^\sharp\delta_{ik}) - \mathrm{ch.f.}(t; G_n^{(ik)})\right|\mathrm{d}t = O\left(\frac{1}{q_n^2}\right).
  \]
\end{itemize}
Recall the definitions in \eqref{eqn:gab_formula}. For notation convenience, we simply write $g_{ab} = g_{ab}^{(ik)}$ and $g_a = g_{aa}^{(ik)}$ in the remaining proof of Lemma \ref{lemma:CHF_IV}. Let $\Xi_{ijk} = E_{ij}u_{jk}/(s_{ik}\lambda_k)$. 

\vspace*{1ex}
\noindent $\blacksquare$ \textbf{Step 2 (I).} 
By Lemma \ref{lemma:CHF_local_expansion}, we have
\begin{align}
\label{eqn:CHF_IV_characteristic_function_bound}
\max_{\calS\subset[n]:|\calS|\geq n - 4}\left|\prod_{j \in \calS}\expect e^{it\Xi_{ijk}}\right|\leq 2e^{-t^2/4}.
\end{align}
Next, we write
\begin{equation}
\label{eqn:CHF_IV_Step2_I_decomposition_I}
\begin{aligned}
&\expect e^{\mathbbm{i}T_{ik}^\sharp}(T_{ik}^\sharp\delta_{ik})^2\\
&\quad = \expect e^{\mathbbm{i}T_{ik}^\sharp}\left(\sum_{a < b}g_{ab} + \sum_{a = 1}^ng_a\right)^2\\
&\quad = \sum_{a < b}\sum_{a' < b'}\expect e^{\mathbbm{i}T_{ik}^\sharp}g_{ab}g_{a'b'} + 2\sum_{a < b}\sum_{c = 1}^n\expect e^{\mathbbm{i}T_{ik}^\sharp}g_{ab}g_c + \sum_{a, b = 1}^n\expect e^{\mathbbm{i}T_{ik}^\sharp}g_ag_b.
\end{aligned}
\end{equation}
By Assumption \ref{assumption:Eigenvector_delocalization} and Assumption \ref{assumption:Noise_matrix_distribution}, we have 
\begin{align*}
\max_{a\in [n]}\expect |g_a|& = O\left(\frac{1}{nq_n}\right),\quad\max_{a\in [n]}\expect g_a^2 = O\left(\frac{1}{nq_n^4}\right).
\end{align*}
Clearly, for any $a,b\in [n]$, $a < b$,
  \begin{align*}
  \expect g_{ab} &= \frac{\expect E_{ia}\{E_{ib}^2 - \var(E_{ib})\}u_{ak}u_{bk}^2 + \expect E_{ib}\{E_{ia}^2 - \var(E_{ia})\}u_{ak}^2u_{bk}}{s_{ik}^3\lambda_k^3} = 0,\\
  \expect \Xi_{iak}g_{ab} &= \frac{\expect E_{ia}^2\{E_{ib}^2 - \var(E_{ib})\}u_{ak}^2u_{bk}^2 + \expect E_{ib}\{E_{ia}^3 - E_{ia}\var(E_{ia})\}u_{ak}^3u_{bk}}{s_{ik}^4\lambda_k^4} = 0,\\
  \expect \Xi_{iak}^2g_{ab} &= \frac{\expect E_{ia}^3\{E_{ib}^2 - \var(E_{ib})\}u_{ak}^3u_{bk}^2 + \expect E_{ib}\{E_{ia}^4 - E_{ia}^2\var(E_{ia})\}u_{ak}^4u_{bk}}{s_{ik}^5\lambda_k^5} = 0,\\
  \expect \Xi_{ibk}g_{ab} &= \frac{\expect E_{ia}\{E_{ib}^3 - E_{ib}\var(E_{ib})\}u_{ak}u_{bk}^3 + \expect E_{ib}^2\{E_{ia}^2 - \var(E_{ia})\}u_{ak}^2u_{bk}^2}{s_{ik}^4\lambda_k^4} = 0,\\
  \expect \Xi_{ibk}^2g_{ab} &= \frac{\expect E_{ia}\{E_{ib}^4 - E_{ib}^2\var(E_{ib})\}u_{ak}u_{bk}^4 + \expect E_{ib}^3\{E_{ia}^2 - \var(E_{ia})\}u_{ak}^2u_{bk}^3}{s_{ik}^5\lambda_k^5} = 0,\\
  \expect g_{ab}g_a &=\frac{\expect E_{ia}^2(E_{ia}^2 - \sigma_{ia}^2)(E_{ib}^2 - \sigma^2_{ib})u_{ak}^4u_{bk}^2 + \expect E_{ia}(E_{ia}^{2} - \sigma_{ia}^2)^2E_{ib}u_{ak}^5u_{bk}}{s_{ik}^4\lambda_k^4} = 0,\\
  \expect g_{ab}g_b &=\frac{\expect E_{ia}E_{ib}(E_{ib}^2 - \sigma_{ib}^2)u_{ak}u_{bk}^5 + \expect (E_{ia}^2 - \sigma_{ia}^2)E_{ib}^2(E_{ib}^2 - \sigma_{ib}^2)u_{ak}^2u_{bk}^4}{s_{ik}^4\lambda_k^4} = 0.
  \end{align*}
  Then Taylor's theorem yields that
  \begin{align*}
  \expect e^{\mathbbm{i}t(\Xi_{iak} + \Xi_{ibk})}g_{ab} &= \expect \left\{ - t^2\Xi_{iak}\Xi_{ibk} + \frac{8t^3}{6}\theta_{iabk}(|\Xi_{iak}|^3 + |\Xi_{ibk}|^3)\right\}g_{ab},\\
  \expect e^{\mathbbm{i}t(\Xi_{iak} + \Xi_{ibk})}g_{ab}(g_a + g_b)&
  = \expect\{\theta_{iabk}'\mathbbm{i}t(\Xi_{iak} + \Xi_{ibk})g_{ab}(g_a + g_b)\},
  \end{align*}
  where $\theta_{iabk},\theta_{iabk}'$ are complex-valued random variables depending on $\Xi_{iak}$ and $\Xi_{ibk}$ with $|\theta_{iabk}|\wedge|\theta_{iabk}'|\leq 1$, where $a\wedge b:=\min(a, b)$. Furthermore, by Assumption \ref{assumption:Eigenvector_delocalization} and Assumption \ref{assumption:Noise_matrix_distribution}, we have
  \begin{align*}
  &\max_{a,b\in [n],a\neq b}\expect|\Xi_{iak}\Xi_{ibk}g_{ab}|\\
  &\quad\leq \max_{a,b\in [n],a\neq b}\expect\left(\frac{E_{ia}^2|E_{ib}^3 - E_{ib}\expect E_{ib}^2| + E_{ib}^2|E_{ia}^3 - E_{ia}\expect E_{ia}^2|}{s_{ik}^5|\lambda_k|^5}\right)\|\bu_k\|_\infty^5
  = O\left(\frac{1}{n^2q_n}\right),\\
  &\max_{a,b\in [n],a\neq b}\expect|\Xi_{iak}^3g_{ab}|\\
  &\quad\leq \max_{a,b\in [n],a\neq b}\expect\left(\frac{E_{ia}^4|E_{ib}^2 - \expect E_{ib}^2| + |E_{ib}||E_{ia}^5 - E_{ia}^3\expect E_{ia}^2|}{s_{ik}^6|\lambda_k|^6}\right)\|\bu_k\|_\infty^6
  = O\left(\frac{\sqrt{n}}{n^2q_n^3}\right),\\
  &\max_{a,b\in[n],a\neq b}\expect|\Xi_{iak}g_{ab}g_a|\\
  &\quad\leq \max_{a,b\in[n],a\neq b}\frac{\expect|E_{ia}^3(E_{ia}^2 - \sigma_{ia}^2)|\expect|E_{ib}^2 - \sigma_{ib}^2| + \expect|E_{ia}^2(E_{ia}^{2} - \sigma_{ia}^2)^2|\expect|E_{ib}|}{s_{ik}^7|\lambda_k|^7}\|\bu_k\|_\infty^7\\
  &\quad = O\left(\frac{1}{n^{3/2}q_n^4}\right) = O\left(\frac{1}{n^2q_n^2}\right),\\
  &\max_{a,b\in[n],a\neq b}\expect|\Xi_{ibk}g_{ab}g_a|\\
  &\quad \leq \max_{a,b\in[n],a\neq b}\frac{\expect|E_{ia}^2(E_{ia}^2 - \sigma_{ia}^2)|\expect|E_{ib}^3 - E_{ib}\sigma_{ib}^2| + \expect|E_{ia}(E_{ia}^{2} - \sigma_{ia}^2)^2|\expect|E_{ib}^2|}{s_{ik}^7|\lambda_k|^7}\|\bu_k\|_\infty^7\\
  &\quad = O\left(\frac{1}{n^2q_n^3}\right).
  \end{align*}
  Since $\sqrt{n}\leq n^{2\eps} \leq q_n^2$, we see that 
  \begin{align*}
  \max_{a,b\in [n],a\neq b}\left|\expect e^{\mathbbm{i}t(\Xi_{iak} + \Xi_{ibk})}g_{ab}\right| &= (t^2 + |t|^3)O\left(\frac{1}{n^2q_n}\right),\\
  \max_{a,b\in [n],a\neq b}\left|\expect e^{\mathbbm{i}t(\Xi_{iak} + \Xi_{ibk})}g_{ab}(g_a + g_b)\right|&
   = |t|O\left(\frac{1}{n^2q_n^2}\right).
  \end{align*}
By \eqref{eqn:CHF_IV_characteristic_function_bound}, the modulus of the second term in \eqref{eqn:CHF_IV_Step2_I_decomposition_I} can be bounded as
\begin{equation}
\label{eqn:CHF_IV_Step2_I_decomposition_I_term_I}
\begin{aligned}
\left|2\sum_{a < b}\sum_{c = 1}^n\expect e^{\mathbbm{i}T_{ik}^\sharp}g_{ab}g_c\right|
&\leq 2\sum_{a < b}\sum_{c \notin \{a,b\}}^n\left|\prod_{j = 1, j\notin \{a,b,c\}}^n\expect e^{\mathbbm{i}t\Xi_{ijk}}\right|\expect \left|e^{\mathbbm{i}t(\Xi_{iak} + \Xi_{ibk})}g_{ab}\right|\expect|g_c|\\
&\quad + 2\sum_{a < b}\left|\prod_{j = 1, j\notin \{a,b\}}^n\expect e^{\mathbbm{i}t\Xi_{ijk}}\right|\left|\expect e^{\mathbbm{i}t(\Xi_{iak} + \Xi_{ibk})}g_{ab}(g_a + g_b)\right|\\
&\leq \left\{e^{-t^2/4}(|t| + t^2 + t^3)\right\}O\left(\frac{1}{q_n^2}\right).
\end{aligned}
\end{equation}
By \eqref{eqn:CHF_IV_characteristic_function_bound}, the modulus of the third term in \eqref{eqn:CHF_IV_Step2_I_decomposition_I} can be bounded as
\begin{equation}
\label{eqn:CHF_IV_Step2_I_decomposition_I_term_II}
\begin{aligned}
&\left|\sum_{a, b = 1}^n\expect e^{\mathbbm{i}T_{ik}^\sharp}g_ag_b\right|\\
&\quad\leq \sum_{a\neq b}\left|\prod_{j = 1,j\notin\{a,b\}}^n\expect e^{\mathbbm{i}t\Xi_{ijk}}\right|\expect|g_a|\expect|g_b| + \sum_{a = 1}^n\left|\prod_{j = 1,j\neq a}^n\expect e^{\mathbbm{i}t\Xi_{ijk}}\right|\expect g_a^2\\
&\quad\leq e^{-t^2/4}O\left(\frac{1}{q_n^2}\right) + e^{-t^2/4}O\left(\frac{1}{q_n^4}\right) = e^{-t^2/4}O\left(\frac{1}{q_n^2}\right).
\end{aligned}
\end{equation}
It remains bound bound the modulus of the first term in \eqref{eqn:CHF_IV_Step2_I_decomposition_I}. Note that 
\begin{align}
\label{eqn:CHF_IV_Step2_I_decomposition_II_term_I}
\sum_{a < b}\sum_{a' < b'}\expect e^{\mathbbm{i}T_{ik}^\sharp}g_{ab}g_{a'b'}& = \sum_{a < b}\sum_{a' < b'}\mathbbm{1}(\{a,b\}\cap\{a',b'\} = \varnothing)\expect e^{\mathbbm{i}T_{ik}^\sharp}g_{ab}g_{a'b'}\\
\label{eqn:CHF_IV_Step2_I_decomposition_II_term_II}
&\quad + \sum_{a < b}\sum_{a' < b'}\mathbbm{1}(\{a,b\} = \{a',b'\} )\expect e^{\mathbbm{i}T_{ik}^\sharp}g_{ab}g_{a'b'}\\
\label{eqn:CHF_IV_Step2_I_decomposition_II_term_III}
&\quad + \sum_{a < b}\sum_{a' < b'}\mathbbm{1}(|\{a,b\}\cap\{a',b'\}| = 1)\expect e^{\mathbbm{i}T_{ik}^\sharp}g_{ab}g_{a'b'}.
\end{align}
We consider the terms \eqref{eqn:CHF_IV_Step2_I_decomposition_II_term_I}-\eqref{eqn:CHF_IV_Step2_I_decomposition_II_term_III} separately. 
\begin{itemize}
  \item Term \eqref{eqn:CHF_IV_Step2_I_decomposition_II_term_I}. If $\{a,b\}\cap\{a',b'\} = \varnothing$, then
  \begin{align*}
  \expect e^{\mathbbm{i}tT_{ik}^\sharp}g_{ab}g_{a'b'}& = \prod_{j = 1,j\notin\{a,b,a',b'\}}\expect e^{\mathbbm{i}t\Xi_{ijk}}\left(\expect e^{\mathbbm{i}t(\Xi_{iak} + \Xi_{ibk})}g_{ab}\right)\left(\expect e^{\mathbbm{i}t(\Xi_{ia'k} + \Xi_{ib'k})}g_{a'b'}\right).
  \end{align*}
  Because we have derived that
  \[
  \max_{a,b\in[n]}\left|\expect e^{\mathbbm{i}t(\Xi_{iak} + \Xi_{ibk})}g_{ab}\right|
  = (t^2 + |t|^3)O\left(\frac{1}{n^2q_n}\right)
  \]
  earlier in this proof, then invoking \eqref{eqn:CHF_IV_characteristic_function_bound}, we see that the modulus of the term in line \eqref{eqn:CHF_IV_Step2_I_decomposition_II_term_I} can be bounded as
  \begin{align*}
  &\left|\sum_{a < b}\sum_{a' < b'}\mathbbm{1}(\{a,b\}\cap\{a',b'\} = \varnothing)\expect e^{\mathbbm{i}tT_{ik}^\sharp}g_{ab}g_{a'b'}\right|\\
  &\quad = \sum_{a < b}\sum_{a' < b'}\mathbbm{1}(\{a,b\}\cap\{a',b'\} = \varnothing)\left|\prod_{j = 1,j\notin\{a,b,a',b'\}}\expect e^{\mathbbm{i}t\Xi_{ijk}}\right|\left|\expect e^{\mathbbm{i}t(\Xi_{iak} + \Xi_{ibk})}g_{ab}\right|\\
  &\quad\quad\times \left|\expect e^{\mathbbm{i}t(\Xi_{ia'k} + \Xi_{ib'k})}g_{a'b'}\right|\\
  &\quad\leq e^{-t^2/4}(t^2 + |t|^3)^2O\left(\frac{1}{q_n^2}\right).
  \end{align*}

  \item Term \eqref{eqn:CHF_IV_Step2_I_decomposition_II_term_II}. If $\{a,b\} = \{a',b'\}$, then this occurs only if $a = a'$ and $b = b'$ because of the constraints $a < b$ and $a' < b'$, in which case
  \begin{align*}
  &\sum_{a < b}\sum_{a' < b'}\mathbbm{1}(\{a,b\} = \{a',b'\})\expect e^{\mathbbm{i}tT_{ik}^\sharp}g_{ab}g_{a'b'}\\
  &\quad = \sum_{a < b}\left(\prod_{j = 1,j\notin\{a,b\}}^n\expect e^{\mathbbm{i}t\Xi_{ijk}}\right)\left(\expect e^{\mathbbm{i}t\Xi_{iak} + \mathbbm{i}t\Xi_{ibk}}g_{ab}^2\right).
  \end{align*}
  Because $\max_{a,b\in [n],a<b}\expect g_{ab}^2 = O\{1/(n^2q_n^2)\}$ by Assumption \ref{assumption:Eigenvector_delocalization} and Assumption \ref{assumption:Noise_matrix_distribution}, it follows directly from \eqref{eqn:CHF_IV_characteristic_function_bound} that
  \begin{align*}
  &\left|\sum_{a < b}\sum_{a' < b'}\mathbbm{1}(\{a,b\} = \{a',b'\})\expect e^{\mathbbm{i}tT_{ik}^\sharp}g_{ab}g_{a'b'}\right|\\
  &\quad\leq \sum_{a < b}\max_{a,b\in [n],a < b}\left|\prod_{j = 1,j\notin\{a,b\}}^n\expect e^{\mathbbm{i}t\Xi_{ijk}}\right|\expect g_{ab}^2 = e^{-t^2/4}O\left(\frac{1}{q_n^2}\right)
  \end{align*}

  \item Term \eqref{eqn:CHF_IV_Step2_I_decomposition_II_term_III}. If $|\{a,b\}\cap \{a',b'\}| = 1$, then we may assume that $b = b'$ and $a\neq a'$ without loss of generality because $g_{ab} = g_{ba}$. Then
  \begin{align*}
  \expect e^{\mathbbm{i}tT_{ik}^\sharp}g_{ab}g_{a'b}& = \prod_{j = 1,j\notin\{a,b,a'\}}\expect e^{\mathbbm{i}t\Xi_{ijk}}\left(\expect e^{\mathbbm{i}t(\Xi_{iak} + \Xi_{ibk} + \Xi_{ia'k})}g_{ab}g_{a'b}\right).
  \end{align*}
  Simple computation shows that for $a < b$, $a' < b$, $a\neq a'$, and any $\alpha_1,\alpha_2,\alpha_3\in\mathbb{N}$,
  \begin{align*}
  &\expect \Xi_{iak}^{\alpha_1}\Xi_{ibk}^{\alpha_2}\Xi_{ia'k}^{\alpha_3}g_{ab}g_{a'b}\\
  &\quad = \frac{\expect E_{ia}^{\alpha_1 + 1}E_{ib}^{\alpha_2}\{E_{ib}^2 - \var(E_{ib})\}^2E_{ia'}^{\alpha_3 + 1}u_{ak}^{\alpha_1 + 1}u_{bk}^{\alpha_2 + 4}u_{a'k}^{\alpha_3 + 1}}{(s_{ik}\lambda_k)^{\alpha_1 + \alpha_2 + \alpha_3 + 6}}\\
  &\quad\quad + \frac{\expect E_{ia}^{\alpha_1}\{E_{ia}^2 - \var(E_{ia})\}E_{ib}^{\alpha_2 + 2}E_{ia'}^{\alpha_3}\{E_{ia'}^2 - \var(E_{ia'})\}u_{ak}^{\alpha_1 + 2}u_{bk}^{\alpha_2 + 2}u_{a'k}^{\alpha_3 + 2}}{(s_{ik}\lambda_k)^{\alpha_1 + \alpha_2 + \alpha_3 + 6}}\\
  &\quad\quad + \frac{\expect E_{ia}^{\alpha_1 + 1}E_{ib}^{\alpha_2}\{E_{ib}^3 - E_{ib}\var(E_{ib})\}E_{ia'}^{\alpha_3}\{E_{ia'}^2 - \var(E_{ia'})\}u_{ak}^{\alpha_1 + 1}u_{bk}^{\alpha_2 + 3}u_{a'k}^{\alpha_3 + 2}}{(s_{ik}\lambda_k)^{\alpha_1 + \alpha_2 + \alpha_3 + 6}}\\
  &\quad\quad + \frac{\expect E_{ia}^{\alpha_1}\{E_{ia}^2 - \var(E_{ia})\}E_{ib}^{\alpha_2}\{E_{ib}^3 - E_{ib}\var(E_{ib})\}E_{ia'}^{\alpha_3 + 1}u_{ak}^{\alpha_1 + 2}u_{bk}^{\alpha_2 + 3}u_{a'k}^{\alpha_3 + 1}}{(s_{ik}\lambda_k)^{\alpha_1 + \alpha_2 + \alpha_3 + 6}}.
  \end{align*}
  It is clear that
  \begin{align*}
  \expect \Xi_{iak}g_{ab}g_{a'b}& = \expect \Xi_{ibk}g_{ab}g_{a'b} = \expect \Xi_{ia'k}g_{ab}g_{a'b} = \expect \Xi_{iak}^2g_{ab}g_{a'b} = \expect \Xi_{ibk}^2g_{ab}g_{a'b}\\
  & = \expect \Xi_{ia'k}^2g_{ab}g_{a'b} = \expect \Xi_{iak}\Xi_{ibk}g_{ab}g_{a'b} = \expect \Xi_{ia'k}\Xi_{ibk}g_{ab}g_{a'b} = 0.
  \end{align*}
  Then Taylor's theorem entails that 
  \begin{align*}
  \expect e^{\mathbbm{i}t(\Xi_{iak} + \Xi_{ibk} + \Xi_{ia'k})}g_{ab}g_{a'b}
  & = -t^2\expect \Xi_{iak}\Xi_{ia'k}g_{ab}g_{a'b}\\
  &\quad + \frac{9t^3}{2}\expect\theta_{iaba'k}'\left(|\Xi_{iak}|^3 + |\Xi_{ibk}|^3 + |\Xi_{ia'k}|^3\right)g_{ab}g_{a'b},
  \end{align*}
  where $\theta_{iaba'k}'$ is a complex-valued random variable such that $|\theta_{iaba'k}'|\leq 1$ with probability one.
  Furthermore, by Assumption \ref{assumption:Eigenvector_delocalization} and Assumption \ref{assumption:Noise_matrix_distribution}, uniformly over $a,b,a'\in[n]$,
  \begin{align*}
  \expect |\Xi_{iak}\Xi_{ia'k}g_{ab}g_{a'b}|
  &\leq \frac{\expect E_{ia}^2\expect \{E_{ib}^2 - \var(E_{ib})\}^2 \expect E_{ia'}^2 \|\bu_k\|_\infty^8}{s_{ik}^8\lambda_k^8}\\
  &\quad + \frac{\expect |E_{ia}^3 - E_{ia}\var(E_{ia})|\expect E_{ib}^2 \expect |E_{ia'}^3 - E_{ia'}\var(E_{ia'})| \|\bu_k\|_\infty^8}{s_{ik}^8\lambda_k^8}\\
  &\quad + \frac{\expect E_{ia}^2\expect |E_{ib}^3 - E_{ib}\var(E_{ib})| \expect |E_{ia'}^3 - E_{ia'}\var(E_{ia'})| \|\bu_k\|_\infty^8}{s_{ik}^8\lambda_k^8}\\
  &\quad + \frac{\expect |E_{ia}^3 - E_{ia}\var(E_{ia})|\expect |E_{ib}^3 - E_{ib}\var(E_{ib})| \expect E_{ia'}^2 \|\bu_k\|_\infty^8}{s_{ik}^8\lambda_k^8}\\
  & = O\left[\frac{1}{n^4\rho_n^4}\left\{\rho_n^2\frac{(n\rho_n)^2}{nq_n^2} + \rho_n\left(\frac{(n\rho_n)^{3/2}}{nq_n} + \rho_n^{3/2}\right)^2\right\}\right]\\
  & = O\left(\frac{1}{n^3q_n^2}\right),\\
  \expect |\Xi_{iak}^3g_{ab}g_{a'b}|
  &\leq \frac{\expect E_{ia}^4\expect \{E_{ib}^2 - \var(E_{ib})\}^2 \expect |E_{ia'}| \|\bu_k\|_\infty^9}{s_{ik}^9\lambda_k^9}\\
  &\quad + \frac{\expect |E_{ia}^5 - E_{ia}^3\var(E_{ia})|\expect E_{ib}^2 \expect |E_{ia'}^2 - \var(E_{ia'})| \|\bu_k\|_\infty^9}{s_{ik}^9\lambda_k^9}\\
  &\quad + \frac{\expect E_{ia}^4\expect |E_{ib}^3 - E_{ib}\var(E_{ib})| \expect |E_{ia'}^2 -\var(E_{ia'})| \|\bu_k\|_\infty^9}{s_{ik}^9\lambda_k^9}\\
  &\quad + \frac{\expect |E_{ia}^5 - E_{ia}^3\var(E_{ia})|\expect |E_{ib}^3 - E_{ib}\var(E_{ib})| \expect |E_{ia'}| \|\bu_k\|_\infty^9}{s_{ik}^9\lambda_k^9}\\
  & = O\left(\frac{1}{n^{5/2}q_n^4}\right) = O\left(\frac{1}{n^3q_n^2}\right),\\
  \expect |\Xi_{ibk}^3g_{ab}g_{a'b}|
  &\leq \frac{\expect |E_{ia}|\expect |E_{ib}^3\{E_{ib}^2 - \var(E_{ib})\}^2| \expect |E_{ia'}| \|\bu_k\|_\infty^9}{s_{ik}^9\lambda_k^9}\\
  &\quad + \frac{\expect |E_{ia}^2 - \var(E_{ia})|\expect |E_{ib}^5| \expect |E_{ia'}^2 - \var(E_{ia'})| \|\bu_k\|_\infty^9}{s_{ik}^9\lambda_k^9}\\
  &\quad + \frac{\expect |E_{ia}|\expect |E_{ib}^6 - E_{ib}^4\var(E_{ib})| \expect |E_{ia'}^2 -\var(E_{ia'})| \|\bu_k\|_\infty^9}{s_{ik}^9\lambda_k^9}\\
  &\quad + \frac{\expect |E_{ia}^2 - \var(E_{ia})|\expect |E_{ib}^6 - E_{ib}^4\var(E_{ib})| \expect |E_{ia'}| \|\bu_k\|_\infty^9}{s_{ik}^9\lambda_k^9}\\
  & = O\left(\frac{1}{n^2q_n^5}\right) = O\left(\frac{1}{n^3q_n^2}\right).
  \end{align*}
  These expected value upper bounds imply
  \[
  \max_{a,b,a'\in [n],a,a'<b,a\neq a'}\left|\expect e^{\mathbbm{i}t(\Xi_{iak} + \Xi_{ibk} + \Xi_{ia'k})}g_{ab}g_{a'b}\right| = (t^2 + |t|^3)O\left(\frac{1}{n^3q_n^2}\right).
  \]
  The exact same reasoning shows that
  \begin{align*}
  &\max_{a,b,b'\in [n],a < b,b\neq b'}\left|\expect e^{\mathbbm{i}t(\Xi_{iak} + \Xi_{ibk} + \Xi_{ib'k})}g_{ab}g_{ab'}\right| = (t^2 + |t|^3)O\left(\frac{1}{n^3q_n^2}\right),
    \end{align*}
  and hence, by \eqref{eqn:CHF_IV_characteristic_function_bound},
  \begin{align*}
  &\left|\sum_{a < b}\sum_{a' < b'}\mathbbm{1}(|\{a,b\}\cap\{a',b'\}| = 1)\expect e^{\mathbbm{i}tT_{ik}^\sharp}g_{ab}g_{a'}\right|\\
  &\quad\leq \sum_{a < b}\sum_{a' < b'}\mathbbm{1}(\{a,b,a',b'\} = \{a,b,b'\})\\
  &\quad\quad\quad\times\left|\prod_{j = 1,j\notin\{a,b,b'\}}\expect e^{\mathbbm{i}t\Xi_{ijk}}\left\{\expect e^{\mathbbm{i}t(\Xi_{iak} + \Xi_{ibk} + \Xi_{ib'k})}g_{ab}g_{ab'}\right\}\right|\\
  &\quad\quad + \sum_{a < b}\sum_{a' < b'}\mathbbm{1}(\{a,b,a',b'\} = \{a,b,a'\})\\
  &\quad\quad\quad\quad\times\left|\prod_{j = 1,j\notin\{a,b,a'\}}\expect e^{\mathbbm{i}t\Xi_{ijk}}\left\{\expect e^{\mathbbm{i}t(\Xi_{iak} + \Xi_{ibk} + \Xi_{ia'k})}g_{ab}g_{a'b}\right\}\right|\\
  &\quad = e^{-t^2/4}\left(t^2 + |t|^3\right)O\left(\frac{1}{q_n^2}\right).
  \end{align*}
\end{itemize}
Combining the above upper bounds entails that the modulus of the first term in \eqref{eqn:CHF_IV_Step2_I_decomposition_I} can be bounded as
\begin{align}
\label{eqn:CHF_IV_Step2_I_decomposition_I_term_III}
\left|\sum_{a < b}\sum_{a' < b'}\expect e^{\mathbbm{i}T_{ik}^\sharp}g_{ab}g_{a'b'}\right|\leq e^{-t^2/4}\mathrm{Poly}(|t|)O\left(\frac{1}{q_n^2}\right),
\end{align}
where $\mathrm{Poly}(t)$ is a generic polynomial of $t$. Now combining \eqref{eqn:CHF_IV_Step2_I_decomposition_I}, \eqref{eqn:CHF_IV_Step2_I_decomposition_I_term_I}, \eqref{eqn:CHF_IV_Step2_I_decomposition_I_term_II}, and \eqref{eqn:CHF_IV_Step2_I_decomposition_I_term_III} leads to
\begin{align*}
\left|\expect e^{\mathbbm{i}T_{ik}^\sharp \delta_{ik}}(T_{ik}^\sharp\delta_{ik})^2\right|\leq e^{-t^2/4}\mathrm{Poly}(|t|)O\left(\frac{1}{q_n^2}\right).
\end{align*} 
Hence, we conclude
\begin{align*}
\int_{-n^\gamma}^{n^\gamma}\left|\frac{1}{t}\frac{t^2}{8}\expect e^{\mathbbm{i}T_{ik}^\sharp \delta_{ik}}(T_{ik}^\sharp\delta_{ik})^2\right|\mathrm{d}t\leq O\left(\frac{1}{q_n^2}\right)\int_{-n^\gamma}^{n^\gamma}e^{-t^2/4}\mathrm{Poly}(|t|)\mathrm{d}t = O\left(\frac{1}{q_n^2}\right),
\end{align*}
thereby finishing Step 2 (I). 

\vspace*{1ex}\noindent
$\blacksquare$ \textbf{Step 2 (II).}
First write
\begin{equation}
\label{eqn:CHF_IV_Step2_II_decomposition}
\begin{aligned}
\expect e^{\mathbbm{i}tT_{ik}^\sharp}(T_{ik}^\sharp\delta_{ik})
& = \expect e^{\mathbbm{i}tT_{ik}^\sharp}\sum_{j = 1}^n\frac{E_{ij}\{E_{ij}^2 - \var(E_{ij})\}u_{jk}^3}{s_{ik}^3\lambda_k^3}\\
&\quad + \expect e^{\mathbbm{i}tT_{ik}^\sharp}\sum_{a\neq b}\frac{E_{ia}\{E_{ib}^2 - \var(E_{ib})\}u_{ak}u_{bk}^2}{s_{ik}^3\lambda_k^3}.
\end{aligned}
\end{equation}
For the second term in \eqref{eqn:CHF_IV_Step2_II_decomposition}, we have
\begin{align*}
&\expect e^{\mathbbm{i}tT_{ik}^\sharp}\sum_{a\neq b}\frac{E_{ia}\{E_{ib}^2 - \var(E_{ib})\}u_{ak}u_{bk}^2}{s_{ik}^3\lambda_k^3}\\
&\quad = \sum_{a\neq b}\prod_{j = 1,j\notin\{a,b\}}^n\expect e^{\mathbbm{i}t\Xi_{ijk}}\left(\expect e^{\mathbbm{i}t\Xi_{iak}}\frac{E_{ia}u_{ak}}{s_{ik}\lambda_k}\right)\left(\expect e^{\mathbbm{i}t\Xi_{ibk}}\frac{\{E_{ib}^2 - \var(E_{ib})\}u_{bk}^2}{s_{ik}^2\lambda_k^2}\right).
\end{align*}
By Taylor's theorem, Assumption \ref{assumption:Eigenvector_delocalization}, and Assumption \ref{assumption:Noise_matrix_distribution},
\begin{align*}
\expect e^{\mathbbm{i}t\Xi_{iak}}\frac{E_{ia}u_{ak}}{s_{ik}\lambda_k}
& = \expect\left(1 + \mathbbm{i}t\frac{E_{ia}u_{ak}}{s_{ik}\lambda_k} + \frac{t^2}{2}\theta_{iak}\left|\frac{E_{ia}u_{ak}}{s_{ik}\lambda_k}\right|^2\right)\frac{E_{ia}u_{ak}}{s_{ik}\lambda_k} \\
& = \mathbbm{i}t\frac{\var(E_{ia})u_{ak}^2}{s_{ik}^2\lambda_k^2} + \frac{t^2}{2}\expect\theta_{iak}\left|\frac{E_{ia}u_{ak}}{s_{ik}\lambda_k}\right|^2\frac{E_{ia}u_{ak}}{s_{ik}\lambda_k}\\
\expect e^{\mathbbm{i}t\Xi_{ibk}}\frac{(E_{ib}^2 - \sigma_{ib}^2)u_{bk}^2}{s_{ik}^2\lambda_k^2}
& = \expect\left(1 + \mathbbm{i}t\frac{E_{ib}u_{bk}}{s_{ik}\lambda_k} + \frac{t^2}{2}\theta_{ibk}\left|\frac{E_{ib}u_{bk}}{s_{ik}\lambda_k}\right|^2\right)\frac{(E_{ib}^2 - \sigma_{ib}^2)u_{bk}^2}{s_{ik}^2\lambda_k^2}\\
& = \mathbbm{i}t\frac{\expect E_{ib}^3u_{bk}^3}{s_{ik}^3\lambda_k^3} + \frac{t^2}{2}\expect \theta_{ibk}\frac{E_{ib}^2(E_{ib}^2 - \sigma_{ib}^2)u_{bk}^4}{s_{ik}^4\lambda_k^4},
\end{align*}
where $\theta_{iak},\theta_{ibk}$ are a complex-valued random variables and $|\theta_{iak}|,|\theta_{ibk}|\leq 1$ with probability one. It follows that
\begin{align*}
&\max_{a\in [n]}\left|\expect e^{\mathbbm{i}t\Xi_{iak}}\frac{E_{ia}u_{ak}}{s_{ik}\lambda_k}\right|\leq t\max_{a\in [n]}\expect\frac{E_{ia}^2u_{ak}^2}{s_{ik}^2\lambda_k^2} = |t|O\left(\frac{1}{n}\right),\\
&\max_{b\in [n]}\left|\expect e^{\mathbbm{i}t\Xi_{ibk}}\frac{(E_{ib}^2 - \sigma_{ib}^2)u_{bk}^2}{s_{ik}^2\lambda_k^2}\right|\leq t\max_{b\in [n]}\frac{\expect |E_{ib}^3 - E_{ib}\sigma_{ib}^2|\|\bu_k\|_\infty^3}{s_{ik}^3|\lambda_k|^3} = |t|O\left(\frac{1}{nq_n}\right),\\
&\max_{a\in [n]}\left|\frac{t^2}{2}\expect\theta_{iak}\left|\frac{E_{ia}u_{ak}}{s_{ik}\lambda_k}\right|^2\frac{E_{ia}u_{ak}}{s_{ik}\lambda_k}\right|\leq \frac{t^2}{2}\expect\left|\frac{E_{ia}u_{ak}}{s_{ik}\lambda_k}\right|^3 = t^2O\left(\frac{1}{nq_n}\right),\\
&\max_{b\in [n]}\left|\frac{t^2}{2}\expect \theta_{ibk}\frac{E_{ib}^2(E_{ib}^2 - \sigma^2_{ib})u_{bk}^4}{s_{ik}^4\lambda_k^4}\right|\leq \max_{b\in [n]}\frac{t^2}{2}\frac{\expect E_{ib}^4 + (\expect E_{ib}^2)^2\|\bu_k\|_\infty^4}{s_{ik}^4\lambda_k^4} = t^2O\left(\frac{1}{nq_n^2}\right).
\end{align*}
It follows that
\begin{align*}
&\max_{a,b\in [n],a\neq b}\left|\left(\expect e^{\mathbbm{i}t\Xi_{iak}}\frac{E_{ia}u_{ak}}{s_{ik}\lambda_k}\right)\left(\expect e^{\mathbbm{i}t\Xi_{ibk}}\frac{\{E_{ib}^2 - \var(E_{ib})\}u_{bk}^2}{s_{ik}^2\lambda_k^2}\right) + t^2\frac{\var(E_{ia})u_{ak}^2}{s_{ik}^2\lambda_k^2}\frac{\expect E_{ib}^3u_{bk}^3}{s_{ik}^3\lambda_k^3}\right|\\
&\quad = |t|^3O\left(\frac{1}{n^2q_n^2}\right).
\end{align*}
Therefore, by Lemma \ref{lemma:CHF_local_expansion}, we further obtain
\begin{align*}
&\left|\expect e^{\mathbbm{i}tT_{ik}^\sharp}\sum_{a\neq b}\frac{E_{ia}(E_{ib}^2 - \sigma_{ib}^2)u_{ak}u_{bk}^2}{s_{ik}^3\lambda_k^3} + e^{-t^2/2}t^2\kappa_n^{(ik)}\right|\\
&\quad \leq \Bigg|\sum_{a\neq b}\prod_{j = 1,j\notin\{a,b\}}^n\expect e^{\mathbbm{i}t\Xi_{ijk}}\left(\expect e^{\mathbbm{i}t\Xi_{iak}}\frac{E_{ia}u_{ak}}{s_{ik}\lambda_k}\right)\left(\expect e^{\mathbbm{i}t\Xi_{ibk}}\frac{\{E_{ib}^2 - \var(E_{ib})\}u_{bk}^2}{s_{ik}^2\lambda_k^2}\right)\\
&\quad\quad + \sum_{a\neq b}\prod_{j = 1,j\notin\{a,b\}}^n\expect e^{\mathbbm{i}t\Xi_{ijk}}t^2\frac{\var(E_{ia})u_{ak}^2}{s_{ik}^2\lambda_k^2}\frac{\expect E_{ib}^3u_{bk}^3}{s_{ik}^3\lambda_k^3}\Bigg|\\
&\quad\quad + \Bigg|-\sum_{a\neq b}\prod_{j = 1,j\notin\{a,b\}}^n\expect e^{\mathbbm{i}t\Xi_{ijk}}t^2\frac{\var(E_{ia})u_{ak}^2}{s_{ik}^2\lambda_k^2}\frac{\expect E_{ib}^3u_{bk}^3}{s_{ik}^3\lambda_k^3} + e^{-t^2/2}t^2\kappa_n^{(ik)}\Bigg|\\
&\quad\leq n^2\max_{\calS\subset[n]:|\calS| = n - 2}\left|\prod_{j\in\calS}\expect e^{\mathbbm{i}t\Xi_{ijk}}\right|\\
&\quad\quad\times \max_{a,b\in[n],a\neq b}\left|\left(\expect e^{\mathbbm{i}t\Xi_{iak}}\frac{E_{ia}u_{ak}}{s_{ik}\lambda_k}\right)\left(\expect e^{\mathbbm{i}t\Xi_{ibk}}\frac{(E_{ib}^2 - \sigma_{ib}^2)u_{bk}^2}{s_{ik}^2\lambda_k^2}\right) + t^2\frac{\sigma_{ia}^2u_{ak}^2}{s_{ik}^2\lambda_k^2}\frac{\expect E_{ib}^3u_{bk}^3}{s_{ik}^3\lambda_k^3}\right|\\
&\quad\quad + \Bigg|\sum_{a,b\in[n]}\prod_{j = 1,j\notin\{a,b\}}^n\expect e^{\mathbbm{i}t\Xi_{ijk}}t^2\frac{\var(E_{ia})u_{ak}^2}{s_{ik}^2\lambda_k^2}\frac{\expect E_{ib}^3u_{bk}^3}{s_{ik}^3\lambda_k^3} - e^{-t^2/2}t^2\kappa_n^{(ik)}\Bigg|\\
&\quad\quad + \Bigg|\sum_{a = 1}^n\prod_{j\in[n]\backslash\{a\}}\expect e^{\mathbbm{i}t\Xi_{ijk}}t^2\frac{\var(E_{ia})u_{ak}^2}{s_{ik}^2\lambda_k^2}\frac{\expect E_{ia}^3u_{ak}^3}{s_{ik}^3\lambda_k^3}\Bigg|\\
&\quad\leq 2n^2e^{-t^2/4}|t|^3O\left(\frac{1}{n^2q_n^2}\right)\\
&\quad\quad + \Bigg|\sum_{a,b\in[n]}\Bigg(\prod_{j = 1,j\notin\{a,b\}}^n\expect e^{\mathbbm{i}t\Xi_{ijk}} - e^{-t^2/2}\Bigg)t^2\frac{\var(E_{ia})u_{ak}^2}{s_{ik}^2\lambda_k^2}\frac{\expect E_{ib}^3u_{bk}^3}{s_{ik}^3\lambda_k^3}\Bigg|\\
&\quad\quad + n\max_{\calS\subset[n]:|\calS| = n - 1}\Bigg|\prod_{j\in[n]\backslash\{a\}}\expect e^{\mathbbm{i}t\Xi_{ijk}}\Bigg|t^2
\max_{a\in[n]}\frac{\var(E_{ia})u_{ak}^2}{s_{ik}^2\lambda_k^2}\left|\frac{\expect E_{ia}^3u_{ak}^3}{s_{ik}^3\lambda_k^3}\right|\\
&\quad\leq 2n^2e^{-t^2/4}|t|^3O\left(\frac{1}{n^2q_n^2}\right)\\
&\quad\quad + n^2\max_{\calS\subset[n]:|\calS|\geq n - 2}\Bigg|\prod_{j = 1,j\notin\{a,b\}}^n\expect e^{\mathbbm{i}t\Xi_{ijk}} - e^{-t^2/2}\Bigg|t^2\max_{a,b\in[n]}\frac{\var(E_{ia})u_{ak}^2}{s_{ik}^2\lambda_k^2}\left|\frac{\expect E_{ib}^3u_{bk}^3}{s_{ik}^3\lambda_k^3}\right|\\
&\quad\quad + n\max_{\calS\subset[n]:|\calS| = n - 1}\Bigg|\prod_{j\in[n]\backslash\{a\}}\expect e^{\mathbbm{i}t\Xi_{ijk}}\Bigg|t^2
\max_{a\in[n]}\frac{\var(E_{ia})u_{ak}^2}{s_{ik}^2\lambda_k^2}\left|\frac{\expect E_{ia}^3u_{ak}^3}{s_{ik}^3\lambda_k^3}\right|\\
&\quad = 2n^2e^{-t^2/4}|t|^3O\left(\frac{1}{n^2q_n^2}\right) + n^2e^{-t^2/2}(t^2 + t^4 + t^6)O\left(\frac{1}{q_n^2}\right)t^2O\left(\frac{1}{n^2q_n}\right)\\
&\quad\quad + 2ne^{-t^2/4}t^2O\left(\frac{1}{n^2q_n}\right)\\
&\quad = e^{-t^2/4}\left(t^2 + |t|^3 + t^4 + t^6 + t^8\right)O\left(\frac{1}{q_n^2}\right).
\end{align*}
This implies that
\begin{align}
&\Bigg|\sum_{a\neq b}\prod_{j = 1,j\notin\{a,b\}}^n\expect e^{\mathbbm{i}t\Xi_{ijk}}\left(\expect e^{\mathbbm{i}t\Xi_{iak}}\frac{E_{ia}u_{ak}}{s_{ik}\lambda_k}\right)\left(\expect e^{\mathbbm{i}t\Xi_{ibk}}\frac{\{E_{ib}^2 - \var(E_{ib})\}u_{bk}^2}{s_{ik}^2\lambda_k^2}\right)\nonumber\\
\label{eqn:CHF_IV_Step2_II_decomposition_term_I}
&\quad + e^{-t^2/2}t^2\kappa_n^{(ik)}\Bigg|\leq e^{-t^2/4}t^2\mathrm{Poly}(|t|)O\left(\frac{1}{q_n^2}\right).
\end{align}
For the first term in \eqref{eqn:CHF_IV_Step2_II_decomposition}, by Taylor's theorem and Lemma \ref{lemma:CHF_local_expansion}, for each $a\in [n]$, there exists some complex-valued random variable $\theta_{iak}\in\mathbb{C}$, $|\theta_{iak}|\leq 1$ with probability one, such that
\begin{align*}
&\left|\expect e^{\mathbbm{i}tT_{ik}^\sharp}\sum_{a = 1}^n\frac{E_{ia}\{E_{ia}^2 - \var(E_{ia})\}u_{ak}^3}{s_{ik}^3\lambda_k^3} - e^{-t^2/2}\kappa_n^{(ik)}\right|\\
&\quad = \left|\sum_{a = 1}^n\prod_{j = 1,j\neq a}^n\expect e^{\mathbbm{i}t\Xi_{ijk}}\expect\left(1 + t\theta_{iak}|\Xi_{iak}|\right)\frac{E_{ia}\{E_{ia}^2 - \var(E_{ia})\}u_{ak}^3}{s_{ik}^3\lambda_k^3} - e^{-t^2/2}\kappa_n^{(ik)}\right|\\
&\quad \leq \left|\sum_{a = 1}^n\prod_{j = 1,j\neq a}^n\expect e^{\mathbbm{i}t\Xi_{ijk}}\frac{\expect E_{ia}^3u_{ak}^3}{s_{ik}^3\lambda_k^3} - e^{-t^2/4}\kappa_{n}^{(ik)}\right|\\
&\quad\quad + \left|\sum_{a = 1}^n\prod_{j = 1,j\neq a}^n\expect e^{\mathbbm{i}t\Xi_{ijk}}t\expect\left\{\theta_{iak}\frac{|\Xi_{iak}|E_{ia}(E_{ia}^2 - \sigma_{ia}^2)u_{ak}^3}{s_{ik}^3\lambda_k^3}\right\}\right|\\
&\quad \leq \sum_{a = 1}^n\left|\prod_{j = 1,j\neq a}^n\expect e^{\mathbbm{i}t\Xi_{ijk}} - e^{-t^2/4}\right|\frac{\expect |E_{ia}|^3|u_{ak}|^3}{s_{ik}^3\lambda_k^3}\\
&\quad\quad + n\max_{\calS\subset[n]:|\calS| = n - 1}\left|\prod_{j = 1,j\neq a}^n\expect e^{\mathbbm{i}t\Xi_{ijk}}\right||t|\max_{a\in[n]}\expect\left|\frac{E_{ia}^2(E_{ia}^2 - \sigma_{ia}^2)u_{ak}^4}{s_{ik}^4\lambda_k^4}\right|\\
&\quad \leq n\max_{\calS\subset[n]:|\calS| = n - 1}\left|\prod_{j\in\calS}\expect e^{\mathbbm{i}t\Xi_{ijk}} - e^{-t^2/4}\right|\max_{a\in[n]}\frac{\expect |E_{ia}|^3|u_{ak}|^3}{s_{ik}^3\lambda_k^3}\\
&\quad\quad + n\max_{\calS\subset[n]:|\calS| = n - 1}\left|\prod_{j = 1,j\neq a}^n\expect e^{\mathbbm{i}t\Xi_{ijk}}\right||t|\max_{a\in[n]}\expect\left|\frac{E_{ia}^2(E_{ia}^2 - \sigma_{ia}^2)u_{ak}^4}{s_{ik}^4\lambda_k^4}\right|\\
&\quad = ne^{-t^2/2}(t^2 + t^4 + t^6)O\left(\frac{1}{q_n^2}\right)O\left(\frac{1}{nq_n}\right) + 2ne^{-t^2/4}|t|O\left(\frac{1}{nq_n^2}\right)\\
&\quad = e^{-t^2/4}(|t| + t^2 + t^4 + t^6)O\left(\frac{1}{q_n^2}\right).
\end{align*}
This implies that
\begin{equation}
\label{eqn:CHF_IV_Step2_II_decomposition_term_II}
\begin{aligned}
&\left|\expect e^{\mathbbm{i}tT_{ik}^\sharp}\sum_{a = 1}^n\frac{E_{ia}\{E_{ia}^2 - \var(E_{ia})\}u_{ak}^3}{s_{ik}^3\lambda_k^3} - e^{-t^2/2}\kappa_n^{(ik)}\right|\\
&\quad\leq e^{-t^2/4}|t|\mathrm{Poly}(|t|)O\left(\frac{1}{q_n^2}\right).
\end{aligned}
\end{equation}
Using the properties of Hermite polynomials and Fourier-Stieltjes transform, we see that
\bea\label{chf:expansion}
\mathrm{ch.f.}(t; G_n^{(ik)}) = e^{-t^2/2} + \frac{\mathbbm{i}t^3}{3}\kappa_n^{(ik)}e^{-t^2/2} - \frac{\mathbbm{i}t}{2}\kappa_n^{(ik)}e^{-t^2/2}.
\eae
By Lemma \ref{lemma:CHF_local_expansion}, we have
\[
\left|\expect e^{\mathbbm{i}tT_{ik}^\sharp} - e^{-t^2/2}\left(1 - \frac{\mathbbm{i}t^3}{6}\kappa_n^{(ik)}\right)\right|\leq e^{-t^2/4}|t|\mathrm{Poly}(|t|)O\left(\frac{1}{q_n^2}\right)
\]
Combining \eqref{eqn:CHF_IV_Step2_II_decomposition}, \eqref{eqn:CHF_IV_Step2_II_decomposition_term_I}, and \eqref{eqn:CHF_IV_Step2_II_decomposition_term_II} yields that
\begin{align*}
&\left|\expect e^{\mathbbm{i}tT_{ik}^\sharp} - \frac{\mathbbm{i}t}{2}\expect e^{\mathbbm{i}tT_{ik}^\sharp}(T_{ik}^\sharp\delta_{ik}) - \mathrm{ch.f.}(t; G_n^{(ik)})\right|\\
&\quad = \Bigg|\expect e^{\mathbbm{i}tT_{ik}^\sharp} - \frac{\mathbbm{i}t}{2}\expect e^{\mathbbm{i}tT_{ik}^\sharp}(T_{ik}^\sharp\delta_{ik}) - e^{-t^2/2}\left(1 - \frac{\mathbbm{i}t^3}{6}\kappa_n^{(ik)}\right) + \frac{\mathbbm{i}t}{2}e^{-t^2/2}\kappa_n^{(ik)}\\
&\quad\quad - \frac{\mathbbm{i}t^3}{2}e^{-t^2/2}\kappa_n^{(ik)}\Bigg|\\
&\quad = e^{-t^2/4}|t|\mathrm{Poly}(|t|)O\left(\frac{1}{q_n^2}\right).
\end{align*}
Hence, we conclude that
\begin{align*}
\int_{-n^\gamma}^{n^\gamma}\frac{1}{|t|}\left|\expect e^{\mathbbm{i}tT_{ik}^\sharp} - \frac{\mathbbm{i}t}{2}\expect e^{\mathbbm{i}T_{ik}^\sharp}(T_{ik}^\sharp\delta_{ik}) - \mathrm{ch.f.}(t; G_n^{(ik)})\right|\mathrm{d}t
& = O\left(\frac{1}{q_n^2}\right)\int_{-n^\gamma}^{n^\gamma}e^{-t^2/4}\mathrm{Poly}(|t|)\mathrm{d}t\\
& =  O\left(\frac{1}{q_n^2}\right).
\end{align*}

\vspace*{1ex}\noindent
$\blacksquare$ \textbf{Step 3.} By Taylor's theorem, there exists a complex-valued random variable $\theta_{ik}$, $|\theta_{ik}|\leq 1$ with probability one, such that
\begin{align*}
&\frac{t^2}{2}|\expect e^{\mathbbm{i}t\widetilde{T}_{ik}}\sigma^2_{ik}(\be_i\transpose\bE)|\\
&\quad = \frac{t^2}{2}\left|\expect e^{\mathbbm{i}tT_{ik}^\sharp}\exp\left(-\frac{\mathbbm{i}t}{2}T_{ik}^\sharp\delta_{ik}\right)\sigma^2_{ik}(\be_i\transpose\bE)\right|\\
&\quad = \frac{t^2}{2}\left|\expect e^{\mathbbm{i}tT_{ik}^\sharp}\left(1 + \frac{t}{2}\theta_{ik}|T_{ik}^\sharp\delta_{ik}|\right)\sigma^2_{ik}(\be_i\transpose\bE)\right|\\
&\quad\leq \frac{t^2}{2}\left|\expect e^{\mathbbm{i}tT_{ik}^\sharp}\sigma_{ik}^2(\be_i\transpose\bE)\right| + \frac{|t|^3}{4}\expect\left|T_{ik}^\sharp\delta_{ik}\right|\sigma_{ik}^2(\be_i\transpose\bE)\\
&\quad\leq \frac{t^2}{2}\left|\expect e^{\mathbbm{i}tT_{ik}^\sharp}\sigma_{ik}^2(\be_i\transpose\bE)\right| + \frac{|t|^3}{4}\left\{\expect(T_{ik}^\sharp\delta_{ik})^2\right\}^{1/2}\left\{\expect\sigma_{ik}^4(\be_i\transpose\bE)\right\}^{1/2}\\
&\quad\leq \frac{t^2}{2}\left|\expect e^{\mathbbm{i}tT_{ik}^\sharp}\sigma_{ik}^2(\be_i\transpose\bE)\right| + \frac{|t|^3}{4}\left\{\expect(T_{ik}^\sharp)^4\right\}^{1/4}\left\{\expect\delta_{ik}^4\right\}^{1/4}\left\{\expect\sigma_{ik}^4(\be_i\transpose\bE)\right\}^{1/2}\\
&\quad\leq \frac{t^2}{2}\left|\expect e^{\mathbbm{i}tT_{ik}^\sharp}\sigma_{ik}^2(\be_i\transpose\bE)\right| + \frac{|t|^3}{4}\left\{\expect(T_{ik}^\sharp)^6\right\}^{1/6}\left\{\expect\delta_{ik}^6\right\}^{1/6}\left\{\expect\sigma_{ik}^4(\be_i\transpose\bE)\right\}^{1/2}\\
&\quad = \frac{t^2}{2}\left|\expect e^{\mathbbm{i}tT_{ik}^\sharp}\sigma_{ik}^2(\be_i\transpose\bE)\right| + |t|^3O\left(\frac{n\rho_n}{\Delta_n^2q_n}\right).
\end{align*}
because $\expect \sigma_{ik}^4(\be_i\transpose\bE) = O\{(n\rho_n)^2/\Delta_n^4\}$ and $\{\expect(T_{ik}^\sharp)^6\}^{1/6}\{\expect\delta_{ik}^6\}^{1/6} = O(1/q_n)$ by \eqref{eqn:sigma_ik_fourth_moment_upper_bound}, \eqref{eqn:Tik_sharp_moment_bound}, and \eqref{eqn:delta_ik_moment_bound}. Now, it is sufficient to focus on the first term on the right-hand side of the preceding display since 
\[
\int_{-n^\gamma}^{n^\gamma}\frac{1}{|t|}\left|t^3O\left(\frac{n\rho_n}{\Delta_n^2q_n}\right)\right|\mathrm{d}t = O\left\{\frac{n^{3\gamma}(n\rho_n)}{\Delta_n^2q_n}\right\} = O\left(\frac{n\rho_n}{\Delta_n^2} \right).
\]
By Lemma \ref{lemma:CHF_local_expansion} and the fact that $\max_{a,b\in [n]}\expect h_{ab}^{(ik)} = O\{\rho_n/(n\Delta_n^2)\}$, we have
\begin{align*}
&\frac{t^2}{2}\left|\expect e^{\mathbbm{i}tT_{ik}^\sharp}\sigma_{ik}^2(\be_i\transpose\bE)\right|\\
&\quad \leq \frac{t^2}{2}\sum_{a < b}\left|\prod_{j = 1,j\notin\{a,b\}}^n\expect e^{\mathbbm{i}t\Xi_{ijk}}\right|\expect h_{ab}^{(ik)} + \frac{t^2}{2}\sum_{a = 1}^n\left|\prod_{j = 1,j\notin\{a\}}^n\expect e^{\mathbbm{i}t\Xi_{ijk}}\right|\expect h_{aa}^{(ik)}\\
&\quad\leq t^2e^{-t^2/4}O\left(\frac{n\rho_n}{\Delta_n^2}\right).
\end{align*}
Hence, we conclude
\[
\int_{-n^\gamma}^{n^\gamma}\frac{1}{|t|}\left|\frac{t^2}{2}\expect e^{\mathbbm{i}tT_{ik}^\sharp}\sigma_{ik}^2(\be_i\transpose\bE)\right|\mathrm{d}t = \int_{-n^\gamma}^{n^\gamma}\frac{1}{|t|}\left|t^2e^{-t^2/4}O\left(\frac{n\rho_n}{\Delta_n^2}\right)\right|\mathrm{d}t = O\left(\frac{n\rho_n}{\Delta_n^2}\right),
\]
thereby completing the proof.
\end{proof}

\subsection{Finishing the Proof of Theorem \ref{thm:edgeworth_expansion}}
\label{sub:finishing_proof_of_edgeworth_expansion}
We are now in a position to finish the proof of Theorem \ref{thm:edgeworth_expansion}. Recall \eqref{def:tildeT}--\eqref{decom:T}. It is sufficient to show that
\begin{equation}
\begin{aligned}
\label{eqn:edgeworth_remainder_I}
&\left\|\prob\Big(T_{ik}\leq x\Big) - \prob\left(\widetilde{T}_{ik} + \Delta_{ik}\leq x\right)\right\|_\infty\\
&\quad = O\left\{
\frac{(\log n)^{2\xi}}{n\rho_n^{1/2}} + \frac{(\log n)^{3\xi}}{\Delta_n} + \frac{(n\rho_n)(\log n)^{3\xi}}{\Delta_n^2} + \frac{(\log n)^{3\xi\vee4}}{q_n^2} + \frac{(n\rho_n)^2(\log n)^{3\xi}}{\Delta_n^3}
\right\},
\end{aligned}
\end{equation}
and
\begin{align}
\label{eqn:edgeworth_remainder_II}
&\left\|\prob\left(\widetilde{T}_{ik} + \Delta_{ik}\leq x\right) - \prob\left(\widetilde{T}_{ik} + \widetilde{\Delta}_{ik}\leq x\right)\right\|_\infty = O\left(\frac{1}{q_n^2}\right).
\end{align}
The above results, combined with \eqref{eqn:edgeworth_expansion_self_smoothed_version}, entail our main result immediately. We first show \eqref{eqn:edgeworth_remainder_II} directly. By Lemma \ref{lemma:conditional_Berry_Esseen}, we have
\[
\sup_{x\in\mathbb{R}}\left|\prob(\Delta_{ik}\leq x\mid \be_i\transpose\bE) - \prob(\widetilde{\Delta}_{ik}\leq x\mid \be_i\transpose\bE)\right| = \widetilde{O}_{\prob}\left(\frac{1}{q_n^2}\right).
\]
Recall that $\widetilde{T}_{ik} = T_{ik}^\sharp - (1/2)T_{ik}^\sharp\delta_{ik}$ is a function of $\be_i\transpose\bE$, and its value is determined given $\be_i\transpose\bE$. It follows that for any $x \in \mathbb{R}$,
\begin{align*}
&\left|\prob(\widetilde{T}_{ik} + \Delta_{ik}\leq x) - \prob(\widetilde{T}_{ik} + \widetilde{\Delta}_{ik}\leq x)\right|\\
&\quad = \left|\int\left\{\prob(\widetilde{T}_{ik} + \Delta_{ik}\leq x\mid \be_i\transpose\bE) - \prob(\widetilde{T}_{ik} + \widetilde{\Delta}_{ik}\leq x\mid \be_i\transpose\bE)\right\}\prob(\mathrm{d}\be_i\transpose\bE)\right|\\
&\quad\leq \int\left|\prob(\widetilde{T}_{ik} + \Delta_{ik}\leq x\mid \be_i\transpose\bE) - \prob(\widetilde{T}_{ik} + \widetilde{\Delta}_{ik}\leq x\mid \be_i\transpose\bE)\right|\prob(\mathrm{d}\be_i\transpose\bE)\\
&\quad\leq \int\left|\prob(\Delta_{ik}\leq x - \widetilde{T}_{ik}\mid \be_i\transpose\bE) - \prob(\widetilde{\Delta}_{ik}\leq x - \widetilde{T}_{ik}\mid \be_i\transpose\bE)\right|\prob(\mathrm{d}\be_i\transpose\bE)\\
&\quad\leq \int\sup_{u\in\mathbb{R}}\left|\prob(\Delta_{ik}\leq u\mid \be_i\transpose\bE) - \prob(\widetilde{\Delta}_{ik}\leq u\mid \be_i\transpose\bE)\right|\prob(\mathrm{d}\be_i\transpose\bE)\\
&\quad = \int\widetilde{O}_{\prob}\left(\frac{1}{q_n^2}\right)\prob(\mathrm{d}\be_i\transpose\bE) = O\left(\frac{1}{q_n^2}\right),
\end{align*}
where we have used the fact that the total variation distance is always upper bounded by $1$. Then \eqref{eqn:edgeworth_remainder_II} follows directly by taking supremum on both sides of the above inequality.

To show \eqref{eqn:edgeworth_remainder_I}, we need the following lemma due to \cite{10.1214/21-AOS2125}. 
\begin{lemma}[Lemma 8.2 in \cite{10.1214/21-AOS2125}]
\label{lemma:edgeworth_remainder}
Suppose we have random variables $X, Y, Z$ satisfying $X = Y + Z$, such that the cumulative distribution function $F_Y(\cdot)$ of $Y$ is smooth, and there exists constants $M\in (0,+\infty)$ such that $F_Y(u + a) - F_Y(u)\leq Ma + O(\zeta_n)$ for any $u\in\mathbb{R}$ and $a > 0$. Also assume that $\prob(|Z|\geq \widetilde{\zeta}_n)\leq n^{-1}$. Then 
\[
\|F_X(u) - F_Y(u)\|_\infty = O(\zeta_n + \widetilde{\zeta}_n + n^{-1}),
\] 
where $F_X(u) = \prob(X\leq u)$. 
\end{lemma}
To apply Lemma \ref{lemma:edgeworth_remainder}, we set $Y = \widetilde{T}_{ik} + \Delta_{ik}$ and $X = T_{ik}$. By \eqref{decom:T} in Section \ref{sub:proof_sketch_edgeworth_expansion}, we see that $Z = X - Y$ satisfies 
\[
Z = \widetilde{O}_{\prob}\left\{\frac{(\log n)^{3\xi}}{n\rho_n^{1/2}} + \frac{(\log n)^{3\xi}}{\Delta_n} + \frac{(n\rho_n)(\log n)^{3\xi}}{\Delta_n^2} + \frac{(\log n)^{3\xi}}{q_n^2} + \frac{(n\rho_n)^2(\log n)^{3\xi}}{\Delta_n^3}\right\}
\]
for any $\xi > 1$.
This shows that there exists a constant $C > 0$ such that
\begin{align*}
&\prob\left[|Z|\geq C\left\{\frac{(\log n)^{3\xi}}{n\rho_n^{1/2}} + \frac{(\log n)^{3\xi}}{\Delta_n} + \frac{(n\rho_n)(\log n)^{3\xi}}{\Delta_n^2} + \frac{(\log n)^{3\xi}}{q_n^2} + \frac{(n\rho_n)^2(\log n)^{3\xi}}{\Delta_n^3}\right\}\right]\\
&\quad\leq \frac{1}{n}.
\end{align*}
To show that $F_Y(u):=\prob(Y\leq u)$ has the smoothness property required by Lemma \ref{lemma:edgeworth_remainder}, for any $u\in\mathbb{R}$ and $a > 0$, we apply \eqref{eqn:edgeworth_expansion_self_smoothed_version}, \eqref{eqn:edgeworth_remainder_II}, and the fact that $\sup_{x,i,n}|\mathrm{d}G_n^{(ik)}(x)/\mathrm{d}x|\leq M$ for some constant $M > 0$ to obtain
\begin{align*}
&\prob\left(\widetilde{T}_{ik} + \Delta_{ik}\leq u + a\right) - \prob\left(\widetilde{T}_{ik} + \Delta_{ik}\leq u\right)\\
&\quad\leq \left|\prob\left(\widetilde{T}_{ik} + \Delta_{ik}\leq u + a\right) - \prob\left(\widetilde{T}_{ik} + \widetilde{\Delta}_{ik}\leq u + a\right)\right|\\
&\quad\quad + \left|\prob\left(\widetilde{T}_{ik} + \widetilde{\Delta}_{ik}\leq u + a\right) - G_n^{(ik)}(u + a)\right|\\
&\quad\quad + \left|G_n^{(ik)}(u + a) - G_n^{(ik)}(u)\right| + \left|G_n^{(ik)}(u) - \prob\left(\widetilde{T}_{ik} + \widetilde{\Delta}_{ik}\leq u\right)\right|\\
&\quad\quad + \left|\prob\left(\widetilde{T}_{ik} + \widetilde{\Delta}_{ik}\leq u\right) - \prob\left(\widetilde{T}_{ik} + \Delta_{ik}\leq u\right)\right|\\
&\quad\leq 2\left\|\prob\left(\widetilde{T}_{ik} + \Delta_{ik}\leq x\right) - \prob\left(\widetilde{T}_{ik} + \widetilde{\Delta}_{ik}\leq x\right)\right\|_\infty + \max_{i\in[n]}\sup_{x\in\mathbb{R}}\left|\frac{\mathrm{d}G_n^{(ik)}}{\mathrm{d}x}(x)\right|a\\
&\quad\quad + 2\left\|\prob\left(\widetilde{T}_{ik} + \widetilde{\Delta}_{ik}\leq x\right) - G_n^{(ik)}(x)\right\|_\infty\\
&\quad\leq Ma + O\left\{\frac{1}{n} + \frac{n\rho_n\log n}{\Delta_n^2} + \frac{(\log n)^4}{q_n^2}\right\}.
\end{align*}
Then \eqref{eqn:edgeworth_remainder_I} follows from Lemma \ref{lemma:edgeworth_remainder} with
\begin{align*}
\widetilde{\zeta}_n &= C\left\{\frac{(\log n)^{3\xi}}{n\rho_n^{1/2}} + \frac{(\log n)^{3\xi}}{\Delta_n} + \frac{(n\rho_n)(\log n)^{3\xi}}{\Delta_n^2} + \frac{(\log n)^{3\xi}}{q_n^2} + \frac{(n\rho_n)^2(\log n)^{3\xi}}{\Delta_n^3}\right\}\\
\zeta_n &= \frac{1}{n} + \frac{n\rho_n\log n}{\Delta_n^2} + \frac{(\log n)^4}{q_n^2}.
\end{align*}

\section{Proof of Theorem \ref{thm:matrix_denoising}}
\label{sec:proof_of_matrix_denoising}

The proof of Theorem \ref{thm:matrix_denoising} is simply applying Theorem \ref{thm:edgeworth_expansion} with the ``symmetric dilation'' trick in Remark \ref{remark:symmetric_dilation} of the manuscript. Using our definitions in the main paper, we set
\bea\label{def:ape}
\bA = \begin{bmatrix}\zero_{p_1\times p_1} & \bX \\ \bX\transpose & \zero_{p_2\times p_2}\end{bmatrix},\quad
\bP = \begin{bmatrix}\zero_{p_1\times p_1} & \bM \\ \bM\transpose & \zero_{p_2\times p_2}\end{bmatrix},\quad
\bE = \begin{bmatrix}\zero_{p_1\times p_1} & \bY \\ \bY\transpose & \zero_{p_2\times p_2}\end{bmatrix}.
\eae
The spectral decomposition of $\bP$ is given by $\bP = \bU_\bP\bS_\bP\bU_\bP\transpose$, where
\bea\label{defupsp}
\bU_\bP = \frac{1}{\sqrt{2}}\begin{bmatrix}\bU & \bU\bPi \\ \bV & -\bV\bPi\end{bmatrix},\quad
\bS_\bP = \begin{bmatrix} \bSigma & \zero_{p_1\times p_2} \\ \zero_{p_2\times p_1} & -\bPi\transpose\bSigma\bPi\end{bmatrix},
\eae
$\bU = [\bu_1,\ldots,\bu_r]$, $\bV = [\bv_1,\ldots,\bv_r]$, $\bSigma = \mathrm{diag}(\sigma_1,\ldots,\sigma_r)$, $\bPi$ is a permutation matrix such that $\bPi\transpose\bSigma\bPi = \mathrm{diag}(\sigma_r, \ldots, \sigma_1)$, $\lambda_k = \sigma_k$, and $\lambda_{k + r} = -\sigma_k$ for all $k\in [r]$. Here $d = \mathrm{rank}(\bP) = 2r$. The sample eigenvector of $\bA$ corresponding to its $r$ largest eigenvalues are
\[
\bU_\bA = \frac{1}{\sqrt{2}}\begin{bmatrix}\widehat{\bU} \\ \widehat{\bV} \end{bmatrix},
\]
where $\widehat{\bU} = [\widehat{\bu}_1,\ldots,\widehat{\bu}_r]$ and $\widehat{\bV} = [\widehat{\bv}_1,\ldots,\widehat{\bv}_r]$. Next we verify Assumptions \ref{assumption:Signal_strength}--\ref{assumption:Eigenvector_delocalization}. For Assumption \ref{assumption:Signal_strength}, we have
\begin{align*}
\frac{\min_{k\in [2d]}|\lambda_k|}{\sqrt{n\rho_n}} &= \frac{\sigma_r}{\sqrt{n\rho_n}}\geq C_1\sqrt{n\rho_n} = \Omega(\sqrt{n\rho_n}) = \Omega(n^{\beta_\Delta}),\\
\frac{\max_{k\in [2d]}|\lambda_k|}{\sqrt{n\rho_n}} &= \frac{\sigma_1}{\sqrt{n\rho_n}}\leq C_2\sqrt{n\rho_n} = O(\sqrt{n\rho_n}) = O(n^{\beta_\Delta}).
\end{align*}
Also, $n\rho_n = \omega(n^a)$ for some $a > 0$ and $\beta_\Delta \leq 1$ are satisfied. For Assumption \ref{assumption:Eigenvalue_separation}, by condition (c), we have
\[
\min_{k\in [d - 1]}(\lambda_k - \lambda_{k + 1}) = \min\left\{\min_{k\in [r - 1]}(\sigma_k - \sigma_{k + 1}), 2\sigma_r\right\}\geq \min(2\sigma_r, \delta_0\sigma_r) = \min(2,\delta_0)\sigma_r.
\]
Assumption \ref{assumption:Noise_matrix_distribution} holds with $q_n = \sqrt{n}(2\log n)^{-(3A_0/2)\log\log n}$ and any $\eps \in (0, 1/2)$ by Example \ref{example:sub_exponential_noise} and condition (a). Here we take $\eps$ to be sufficiently close to $1/2$ such that $\beta_\Delta \geq 1/2 - \eps$ and $\eps\geq 1/3$, \emph{e.g.}, $\eps = \max(\beta_\Delta - 1/2, 1/3)$. Assumption \ref{assumption:Eigenvector_delocalization} holds by condition (d). 

Next, we compute $s_{ik}^2$. Recalling the definition \eqref{eqn:plug_in_estimate_variance} and \eqref{def:ape}--\eqref{defupsp}, under the current i.i.d. noise setup, we have
\begin{align*}
\color{black}s_{ik}^2 = \rho_n\sum_{j = 1}^{p_2}\frac{v_{jk}^2}{2\sigma_k^2} {=} \frac{\rho_n}{2\sigma_k^2},\quad
\widehat{s}_{ik}^2 = \sum_{j = 1}^{p_2}(X_{ij} - \widehat{M}_{ij})^2\frac{\widehat{v}_{jk}^2}{2\widehat{\sigma}_k^2}.
\end{align*}
Clearly, $\min_{i\in [n]}s_{ik}^2 = \Theta(\rho_n/\Delta_n^2)$, where $\Delta_n \asymp \sigma_r$. For $\beta = \expect |Y_{ij}|^3$, it is immediate that $\beta = O(\rho_n^{3/2}) = o(n\rho_n^2/(\Delta_n\sqrt{\log n}))$ because $\sigma_r\sqrt{\log n} = o(n\rho_n^{1/2})$. Therefore, the Cram\'er's condition is not required. Then the Edgeworth expansion formula \eqref{eqn:edgeworth_expansion_formula} has the form
\begin{align*}
G_n^{(ik)}(x) & = \Phi(x) + \frac{(2x^2 + 1)}{6}\phi(x)\sum_{j = 1}^{p_2}\frac{\expect Y_{ij}^3u_{jk}^3}{\sqrt{8}s_{ik}^3\sigma_k^3}\\
& = \Phi(x) + \frac{(2x^2 + 1)}{6}\phi(x)\sum_{j = 1}^{p_2}\frac{\expect Y_{ij}^3u_{jk}^3}{\rho_n^{3/2}}.
\end{align*}
Now we compute the second-order bias $\bb_{k} = \bb(\hat{\bu}_k)$ (c.f.~\eqref{def:bias}), which is specified by formulae \eqref{def:bias} and \eqref{eqn:eigenvector_second_order_term}. 
First, we observe that,
\begin{align*}
\bD &= \expect \bE^2 = \expect\begin{bmatrix}\zero_{p_1\times p_1} & \bY\\\bY\transpose & \zero_{p_2\times p_2}\end{bmatrix}^2 = \expect\begin{bmatrix}\bY\bY\transpose & \zero_{p_1\times p_2}\\\zero_{p_2\times p_1} & \bY\transpose\bY\end{bmatrix} = \begin{bmatrix}p_2\rho_n\eye_{p_1} &\\ & p_1\rho_n\eye_{p_2}\end{bmatrix}.
\end{align*}
Then by Theorem \ref{thm:Eigenvector_Expansion}, we have
\begin{align*}
\frac{1}{\sqrt{2}}\bb_{k}& = \left\{\eye_n - \frac{3}{4}\begin{bmatrix}\bu_k\\\bv_k\end{bmatrix}\begin{bmatrix}\bu_k\transpose&\bv_k\transpose\end{bmatrix} + 
\sum_{m\in[r]/\{k\}}\frac{\sigma_m}{2(\sigma_k - \sigma_m)}
\begin{bmatrix}\bu_m\\\bv_m\end{bmatrix}\begin{bmatrix}\bu_m\transpose&\bv_m\transpose\end{bmatrix}
 \right\}\frac{\expect\bE^2}{\sigma_k^2}\frac{1}{\sqrt{2}}\begin{bmatrix}\bu_k\\\bv_k\end{bmatrix}\\
 &\quad - 
 \sum_{m = 1}^r\frac{\sigma_m}{2(\sigma_k + \sigma_m)}
 \begin{bmatrix}\bu_m\\-\bv_m\end{bmatrix}\begin{bmatrix}\bu_m\transpose&-\bv_m\transpose\end{bmatrix}\frac{\expect\bE^2}{\sigma_k^2}\frac{1}{\sqrt{2}}\begin{bmatrix}\bu_k\\\bv_k\end{bmatrix}\\
& = \frac{\rho_n}{\sqrt{2}\sigma_k^2}\begin{bmatrix}p_2\bu_k\\p_1\bv_k\end{bmatrix} - \frac{3(p_2 + p_1)\rho_n}{4\sqrt{2}\sigma_k^2}\begin{bmatrix}\bu_k\\\bv_k\end{bmatrix} - \frac{(p_2 - p_1)\rho_n}{4\sqrt{2}\sigma_k^2}\begin{bmatrix}\bu_k\\-\bv_k\end{bmatrix}
\\
& = -\frac{\rho_n}{2\sqrt{2}\sigma_k^2}\begin{bmatrix}p_1\bu_k\\p_2\bv_k\end{bmatrix}.
\end{align*}
In particular, we have $b_{ik}$ being the $i$th entry of $\bb_k$:
\[
b_{ik} = -\frac{\rho_n}{2\sigma_k^2}p_1 u_{ik}.
\]
Note that
\begin{align*}
\frac{(\widehat{u}_{ik} - u_{ik} - b_{ik})/\sqrt{2}}{\widehat{s}_{ik}} = 
\frac{\widehat{u}_{ik} - u_{ik} - b_{ik}}{\sqrt{2}\widehat{s}_{ik}} = T_{ik}.
\end{align*}
Therefore, Theorem \ref{thm:edgeworth_expansion} entails that
\begin{align*}
\|\prob(T_{ik}\leq x) - G_n^{(ik)}(x)\|_\infty = O\left\{\frac{(\log n)^{3\xi}}{n\rho_n} + \frac{(2\log n)^{3\xi\vee4 + 33\log\log n}}{n}\right\}.
\end{align*}

\section{Proof of Theorem \ref{thm:bootstrap_subgaussian}}
\label{sec:proof_of_bootstrap_subgaussian}
Under our bootstrap setting, we consider the $\bP$ and $\bE$ in the main paper to be $\widehat{\bP}$ and $\bE^*$, respectively, and $\widehat{\beta}=\max_{i,j\in [n]}\expect_*|E^*_{ij}|^3$ in the position of $\beta = \max_{i,j\in[n]}\expect|E_{ij}|^3$, where $\mathbb{E}_*$ is the expected value with regard to the conditional distribution of $\bE^*$ given $\bA$. Then the proof strategy of Theorem \ref{thm:bootstrap_subgaussian} is to show that Assumptions \ref{assumption:Signal_strength}--\ref{assumption:Eigenvector_delocalization} still hold with high probability with respect to the randomness of $\bA$, possibly for different constants $C_1, C_2, a, \delta_0, C_\mu, C$ but for the same $\beta_\Delta$, with $q_n := \sqrt{n}(2\log n)^{-33\log\log n}$. Let 
\begin{align*}
\calE_{1n} &= \left\{\bA:\frac{1}{2}\lambda_{\min}(|\bS_\bP|)\leq \lambda_{\min}(|\bS_\bA|),\lambda_{\max}(|\bS_\bA|)\leq 2\lambda_{\max}(|\bS_\bP|)\right\},\\
\calE_{2n} &= \left\{\bA:\lambda_k(\bS_\bA) - \lambda_{k + 1}(\bS_\bA)\geq \frac{1}{2}\Delta_n,k\in [d - 1]\right\},\\
\calE_{3n} &= \left\{\bA:\max_{k\in [d]}\|\widehat{\bu}_k\|_\infty \leq \frac{2C_\mu}{\sqrt{n}}\right\},
\end{align*}
By Weyl's inequality and Result \ref{result:Noise_matrix_concentration}, $\calE_{1n}$ w.h.p.. For any $k\in [d]$, by Weyl's inequality and Result \ref{result:Noise_matrix_concentration} again, 
\begin{align*}
&\lambda_k(\bS_\bA)\geq \lambda_k(\bS_\bP) - \|\bE\|_2\geq \lambda_k(\bS_\bP) - \frac{1}{4}\Delta_n\quad\mbox{w.h.p.},\\
&\lambda_{k + 1}(\bS_\bA)\leq \lambda_{k + 1}(\bS_\bP) + \|\bE\|_2\leq \lambda_{k + 1}(\bS_\bP) + \frac{1}{4}\Delta_n\quad\mbox{w.h.p.};
\end{align*}
Note that $\Delta_n\gg\|\bE\|$ under the setting of Assumptions \ref{assumption:Signal_strength} and \ref{assumption:Eigenvalue_separation}. Hence, by Assumption \ref{assumption:Eigenvalue_separation}, we obtain
\begin{align*}
\lambda_k(\bS_\bA) - \lambda_{k + 1}(\bS_\bA) \geq \lambda_k(\bS_\bP) - \lambda_{k + 1}(\bS_\bP) - \frac{1}{2}\Delta_n\geq \frac{1}{2}\Delta_n\quad\mbox{w.h.p.}.
\end{align*}
Namely, $\calE_{2n}$ occurs w.h.p.. For $\calE_{3n}$, by Lemma \ref{lemma:Eigenvector_zero_order_deviation}, we have
\[
\|\widehat{\bu}_k\|_\infty\leq \|\widehat{\bu}_k - \bu_k\mathrm{sgn}(\bu_k\transpose\widehat{\bu}_k)\|_\infty + \|\bu_k\|_\infty = \widetilde{O}_{\prob}\left\{\|\bU_\bP\|_{2\to\infty}\frac{(n\rho_n)^{1/2}(\log n)^\xi}{\Delta_n}\right\} + \frac{C_\mu}{\sqrt{n}}.
\]
Namely, $\max_{k\in [d]}\|\widehat{\bu}_k\|_\infty\leq 2C_\mu/\sqrt{n}$ w.h.p.. Hence, we see that $\calE_{1n}\cap\calE_{2n}\cap\calE_{3n}$ occurs w.h.p., over which Assumptions \ref{assumption:Signal_strength}, \ref{assumption:Eigenvalue_separation}, and \ref{assumption:Eigenvector_delocalization} hold with $\bP$, $\bE$ replaced with $\widehat{\bP}$, $\bE^*$. Now consider Assumption \ref{assumption:Noise_matrix_distribution}. For any $p$, $2\leq p\leq (2\log n)^{11\log\log n}$, we have
\begin{align*}
\int_{\mathbb{R}}|x|^p\mathrm{d}\widehat{F}_n(x)
& = \frac{2}{n(n + 1)}\sum_{1\leq i\leq j\leq n}\left|E_{ij} + (p_{ij} - \widehat{p}_{ij}) - \widehat{\mu}_\bE\right|^p\\
&\leq \frac{2}{n(n + 1)}\sum_{1\leq i\leq j\leq n}3^p|E_{ij}|^p + \frac{2}{n(n + 1)}\sum_{1\leq i\leq j\leq n}3^p|p_{ij} - \widehat{p}_{ij}|^p +3^p|\widehat{\mu}_\bE|^p\\
&\leq \frac{2}{n(n + 1)}\sum_{1\leq i\leq j\leq n}3^p|E_{ij}|^p + 3^p\max_{i,j\in[n]}|p_{ij} - \widehat{p}_{ij}|^p\\
&\quad + 3^p\left|\frac{2}{n(n + 1)}\sum_{1\leq i\leq j\leq n}(E_{ij} + p_{ij} - \widehat{p}_{ij})\right|^p\\
&\leq \frac{2}{n(n + 1)}\sum_{1\leq i\leq j\leq n}3^p|E_{ij}|^p + (3^p + 6^p)\max_{i,j\in[n]}|p_{ij} - \widehat{p}_{ij}|^p
\\&\quad
 + 6^p\left|\frac{2}{n(n + 1)}\sum_{1\leq i\leq j\leq n}E_{ij}\right|^p.
\end{align*}
By \eqref{eqn:pij_zeroth_order_deviation}, the second term is
$O\{C^p(\log n)^{p\xi}{\rho_n^{p/2}}/{n^{p/2}}\}$
with probability at least $1 - n^{-2}$, where $C > 0$ is a generic constant. The last term is 
$O\{C^p(\log n)^{p\xi}\rho_n^{p/2}/{n^p}\}$
with probability at least $1 - n^{-2}$ by Result \ref{result:concentration}.
The sum of these two terms is bounded by
\bea\label{generalbound:qn:1}
\frac{C^p\{\rho_n(\log n)^\xi\}^{p/2}}{n^{p/2}}\leq\left\{
\begin{aligned}
&\frac{C^p(n\rho_n)^{p/2}}{nq_n^{p - 2}},&\quad&\mbox{if }p \geq 3,\\
&C\rho_n,&\quad&\mbox{if }p = 2.
\end{aligned}
\right.
\eae
Now, we focus on the first term. Write
\begin{align*}
&\expect\left\{\frac{2}{n(n + 1)}\sum_{1\leq i\leq j\leq n}3^p|E_{ij}|^p\right\}
 = 3^p\expect|E_{ij}|^p = 3^p\left[p^{p}\left\{\frac{1}{p}(\expect|E_{ij}|^p)^{1/p}\right\}^p\right]\\
&\quad\leq (3p)^p\|E_{ij}\|_{\psi_2}^p\leq C^p(\rho_np^2)^{p/2}.
\end{align*}
When $p \geq 3$ and $p\leq (2\log n)^{11\log\log n}$, with $q_n := \sqrt{n}(2\log n)^{-33\log\log n}$, we have
\bea\label{generalbound:qn:2}
\expect\left\{\frac{2}{n(n + 1)}\sum_{1\leq i\leq j\leq n}3^p|E_{ij}|^p\right\}\leq \frac{C^p(n\rho_n)^{p/2}}{nq_n^{p - 2}},
\eae
and when $p = 2$, similarly, we have
\begin{align*}
\expect\left\{\frac{2}{n(n + 1)}\sum_{1\leq i\leq j\leq n}3^2|E_{ij}|^2\right\} = O(\rho_n).
\end{align*}
Now we compute the variance of this term:
\begin{align*}
&\var\left\{\frac{2}{n(n + 1)}\sum_{1\leq i\leq j\leq n}3^p|E_{ij}|^p\right\}
= \frac{2\times 9^p\var(|E_{ij}|^p)}{n(n + 1)}\leq \frac{C^p\expect|E_{ij}|^{2p}}{n^2}\\
&\quad = \frac{C^p(2p)^{2p}}{n^2}\left\{\frac{1}{2p}\left(\expect|E_{ij}|^{2p}\right)^{1/(2p)}\right\}^{2p}\leq \frac{C^p(4\rho_np^2)^p}{n^2},
\end{align*}
recalling that $\sup_{p\geq 1}p^{-1}(\expect|E^p_{ij}|^{1/p}) = \|E_{ij}\|_{\psi_1}\leq C\rho_n^{1/2}$. Then, by Markov's inequality,
\begin{align*}
&\prob\left[\left|\frac{2}{n(n + 1)}\sum_{1\leq i\leq j\leq n}3^p|E_{ij}|^p - \expect\left\{\frac{2}{n(n + 1)}\sum_{1\leq i\leq j\leq n}3^p|E_{ij}|^p\right\}\right| > \frac{C^{p/2}(4\rho_np^2)^{p/2}}{n^{1/4}}\right]\\
&\quad\leq \frac{n^{1/2}C^p(4\rho_np^2)^p}{n^2C^p(4\rho_np^2)^p} = \frac{1}{n^{3/2}}.
\end{align*}
 This shows that, with probability at least $1 - n^{-3/2}$, 
\begin{align*}
\frac{2}{n(n + 1)}\sum_{1\leq i\leq j\leq n}3^p|E_{ij}|^p
&\leq \expect\left\{\frac{2}{n(n + 1)}\sum_{1\leq i\leq j\leq n}3^p|E_{ij}|^p\right\} + \frac{C^{p/2}(4\rho_np^2)^{p/2}}{n^{1/4}}\\
&\leq C^p(\rho_np^2)^{p/2} + C^p(\rho_np^2)^{p/2} \leq C^p(\rho_n p^2)^{p/2}.
\end{align*}
\par
Summarizing all results above, when $p \geq 3$, we conclude that
$
\int_{\mathbb{R}}|x|^p\mathrm{d}\widehat{F}_n(x)
$
is bounded by $C^p(n\rho_n)^{p/2}/(nq_n^{p - 2})$, and when $p = 2$, it is bounded by $O(\rho_n)$ with probability at least $1 - O(n^{-1})$. We also have the corresponding $\widehat{\beta}$ satisfies a slightly sharper bound: with probability at least $1 - O(n^{-1})$,
$$
\widehat{\beta} = \int_{\mathbb{R}}|x|^3\mathrm{d}\widehat{F}_n(x) = O(\rho_n^{3/2}) = o\left(\frac{n\rho_n^2}{\Delta_n\sqrt{\log n}}\right),
$$
recalling $\Delta_n = o(n\rho_n^{1/2}/\sqrt{\log n})$. We note such a sharper bound can be derived by following the same argument as above, yet focusing on the typical case that $p = 3$ and not introducing $q_n$ in \eqref{generalbound:qn:1} and \eqref{generalbound:qn:2}. Then, we conclude that the event
\[
\calE_{4n} = \bigcap_{p = 2}^{\lceil(2\log n)^{11\log\log n}\rceil}
\left\{\bA:\int_{\mathbb{R}}|x|^p\mathrm{d}\widehat{F}_n(x)\leq \frac{C^p(n\rho_n)^{p/2}}{nq_n^{p - 2}},\,\, \widehat{\beta} = o\left(\frac{n\rho_n^2}{\Delta_n\sqrt{\log n}}\right)
\right\}
\]
occurs with probability at least $1 - O(n^{-1})$, over which Assumption \ref{assumption:Noise_matrix_distribution} holds. Let 
\[
\calE_{5n} = \left\{\min_{i\in[n]}\widehat{s}_{ik}^2 \geq \frac{1}{2}\min_{i\in[n]}{s}_{ik}^2\right\}.
\]
By Lemma \ref{lemma:variance_expansion} and Result~\ref{result:concentration}, and recalling $q_n := \sqrt{n}(2\log n)^{-33\log\log n}$, we know that
\begin{align*}
\min_{i\in[n]}\widehat{s}_{ik}^2&\geq \min_{i\in[n]}s_{ik}^2 - \max_{i\in[n]}\left|\sum_{j = 1}^n\frac{(E_{ij}^2 - \expect E_{ij}^2)u_{jk}^2}{\lambda_k^2}\right| - \max_{i\in[n]}\left|\sum_{a,b = 1}^n\frac{2E_{ia}^2E_{ab}u_{ak}u_{bk}}{\lambda_k^3}\right|\\
&\quad - \left|\widetilde{O}_{\prob}\left\{\frac{\rho_n^{1/2}(\log n)^{2\xi}}{\Delta_n^2n} + \frac{\rho_n(\log n)^{2\xi}}{\Delta_n^3} + \frac{n\rho_n^2(\log n)^{2\xi}}{\Delta_n^4}\right\}\right|\\
& = \min_{i\in[n]}s_{ik}^2 - \left|\widetilde{O}_{\prob}\left\{\frac{\rho_n(\log n)^{\xi}}{q_n\Delta_n^2}\right\}\right| - \left|\widetilde{O}_{\prob}\left\{\frac{\rho_n}{\Delta_n^2}\times{\frac{\sqrt{n\rho_n}}{\Delta_n}(\log n)^{\xi}}\right\}\right|\\
&\quad - \left|\widetilde{O}_{\prob}\left\{\frac{\rho_n^{1/2}(\log n)^{2\xi}}{\Delta_n^2n} + \frac{\rho_n(\log n)^{2\xi}}{\Delta_n^3} + \frac{n\rho_n^2(\log n)^{2\xi}}{\Delta_n^4}\right\}\right|\\
&\geq \frac{1}{2}\min_{i\in[n]}s_{ik}^2\quad\mbox{w.h.p.},
\end{align*}
and hence, $\calE_{5n}$ occurs w.h.p..
Applying Theorem \ref{thm:edgeworth_expansion} with regard to the $\prob^*$ probability over the event $\calE_{1n}\cap\calE_{2n}\cap\calE_{3n}\cap\calE_{4n}\cap\calE_{5n}$, we conclude w.h.p.,
\begin{align*}
&\left\|\prob^*\left(T_{ik}^*\leq x\right) - G_n^{(ik)*}(x)\right\|_\infty\\
&\quad = O\left\{
\frac{(\log n)^{2\xi}}{n\rho_n} + \frac{(\log n)^{3\xi}}{\Delta_n} + \frac{(n\rho_n)(\log n)^{3\xi}}{\Delta_n^2} + \frac{(\log n)^{3\xi\vee4}}{q_n^2} + \frac{(n\rho_n)^2(\log n)^{3\xi}}{\Delta_n^3}
\right\},
\end{align*}
where
\begin{align*}
G_n^{(ik)*}(x) = \Phi(x) + \frac{(2x^2 + 1)}{6}\phi(x)\int_{\mathbb{R}}u^3\mathrm{d}\widehat{F}_n(u)\sum_{j = 1}^n\frac{\widehat{u}_{jk}^3}{\widehat{s}_{ik}^3\widehat{\lambda}_k^3}.
\end{align*}
By Theorem \ref{thm:edgeworth_expansion} and the fact that $\beta = O(\rho_n^{3/2}) = o(n\rho_n^2/(\Delta_n\sqrt{\log n}))$, we also have
\begin{align*}
&\left\|\prob\left(T_{ik}\leq x\right) - G_n^{(ik)}(x)\right\|_\infty\\
&\quad = O\left\{
\frac{(\log n)^{2\xi}}{n\rho_n^{1/2}} + \frac{(\log n)^{3\xi}}{\Delta_n} + \frac{(n\rho_n)(\log n)^{3\xi}}{\Delta_n^2} + \frac{(\log n)^{3\xi\vee4}}{q_n^2} + \frac{(n\rho_n)^2(\log n)^{3\xi}}{\Delta_n^3}
\right\},
\end{align*}
where
\begin{align*}
G_n^{(ik)}(x) = \Phi(x) + \frac{(2x^2 + 1)}{6}\phi(x)\expect E_{ij}^3\sum_{j = 1}^n\frac{{u}_{jk}^3}{{s}_{ik}^3{\lambda}_k^3}.
\end{align*}
So to finally show the high-probability bound of  $\left\|\prob^*\left(T_{ik}^*\leq x\right) -\prob\left(T_{ik}\leq x\right)\right\|_\infty$, it is left to bound $\|G_n^{(ik)*}(x) - G_n^{(ik)}(x)\|_\infty$ w.h.p..
\par
By \eqref{eqn:pij_zeroth_order_deviation}, concentration of $\mu_\bE = 2n^{-1}(n + 1)^{-1}\sum_{i\leq j}E_{ij}$, and Result~\ref{result:concentration},
\begin{align*}
&\left|\int_{\mathbb{R}}x^3\mathrm{d}\widehat{F}_n(x) - \expect E_{ij}^3\right|\\
&\quad
= \left|\frac{2}{n(n + 1)}\sum_{i\leq j}\left[\{E_{ij} + (p_{ij} - \widehat{p}_{ij}) - \mu_{\bE}\}^3 - \expect E_{ij}^3\right]\right|\\
&\quad \leq
\left|\frac{2}{n(n + 1)}\sum_{i\leq j}(E_{ij}^3 - \expect E_{ij}^3)\right|
+ \left|\frac{6}{n(n + 1)}\sum_{i\leq j}E_{ij}^2(p_{ij} - \widehat{p}_{ij} - \mu_{\bE})\right|\\
&\quad\quad 
+ \left|\frac{6}{n(n + 1)}\sum_{i\leq j}E_{ij}(p_{ij} - \widehat{p}_{ij} - \mu_{\bE})^2\right|
+ \max_{i,j\in[n]}|p_{ij} - \widehat{p}_{ij} - \mu_{\bE}|^3\\
&\quad\leq \left|\frac{2}{n(n + 1)}\sum_{i\leq j}(E_{ij}^3 - \expect E_{ij}^3)\right| + \left|\frac{6}{n(n + 1)}\sum_{i\leq j}E_{ij}^2\right|\left(\max_{i,j\in[n]}|p_{ij} - \widehat{p}_{ij}| + |\mu_{\bE}|\right)\\
&\quad\quad + \frac{6}{n(n + 1)}\sum_{i\leq j}|E_{ij}|\left(\max_{i,j\in[n]}|p_{ij} - \widehat{p}_{ij}| + |\mu_{\bE}|\right)^2
+ \left(\max_{i,j\in[n]}|p_{ij} - \widehat{p}_{ij}| + |\mu_{\bE}|\right)^3\\
&\quad = \left|\frac{2}{n(n + 1)}\sum_{i\leq j}(E_{ij}^3 - \expect E_{ij}^3)\right| + \widetilde{O}_{\prob}\left\{\frac{\rho_n^{3/2}(\log n)^{3\xi}}{n^{3/2}}\right\}\\
&\quad\quad + \left\{\left|\frac{6}{n(n + 1)}\sum_{i\leq j}(E_{ij}^2- \sigma_{ij}^2)\right| + O(\rho_n)\right\}\widetilde{O}_{\prob}\left\{(\log n)^\xi\sqrt{\frac{\rho_n}{n}}\right\}\\
&\quad\quad + \left\{\left|\frac{6}{n(n + 1)}\sum_{i\leq j}(|E_{ij}|- \expect|E_{ij}|)\right| + O(\rho_n^{1/2})\right\}\widetilde{O}_{\prob}\left\{\frac{\rho_n(\log n)^{2\xi}}{n}\right\}\\
&\quad = \left|\frac{2}{n(n + 1)}\sum_{i\leq j}(E_{ij}^3 - \expect E_{ij}^3)\right| + \widetilde{O}_{\prob}\left\{\frac{\rho_n^{3/2}(\log n)^\xi}{\sqrt{n}}\right\}.
\end{align*}
By Markov's inequality, with probability at least $1 - n^{-1}$, we have
\begin{align*}
\left|\frac{2}{n(n + 1)}\sum_{i\leq j}(E_{ij}^3 - \expect E_{ij}^3)\right| = O\left(\frac{\rho_n^{3/2}}{\sqrt{n}}\right).
\end{align*}
This shows that, with probability at least $1 - n^{-1}$, 
\[
\left|\int_{\mathbb{R}}x^3\mathrm{d}\widehat{F}_n(x) - \expect E_{ij}^3\right| = O\left\{\frac{\rho_n^{3/2}(\log n)^\xi}{\sqrt{n}}\right\}.
\]
By Lemma \ref{lemma:variance_expansion} and Result~\ref{result:concentration}, we know that 
\begin{align*}
|\widehat{s}_{ik}^2 - s_{ik}^2|
 = O\left[\frac{\rho_n}{\Delta_n^2}
 \left\{ \frac{(\log n)^\xi}{q_n} + \frac{\sqrt{n\rho_n}(\log n)^\xi}{\Delta_n} + \frac{(\log n)^{2\xi}}{n\rho_n^{1/2}}\right\}
\right]
\end{align*} 
and $\widehat{s}_{ik}^2 = \Theta(\rho_n/\Delta_n^2)$ with probability at least $1 - n^{-2}$. 
Then we proceed to bound
\begin{align*}
&\|G_n^{(ik)*}(x) - G_n^{(ik)}(x)\|_\infty\\
&\quad \lesssim \left|\int_{\mathbb{R}}u^3\mathrm{d}\widehat{F}_n(u)\sum_{j = 1}^n\frac{\widehat{u}_{jk}^3}{\widehat{s}_{ik}^3\widehat{\lambda}_k^3} - \expect E_{ij}^3\sum_{j = 1}^n\frac{{u}_{jk}^3}{{s}_{ik}^3{\lambda}_k^3}\right|\\
&\quad \leq \left|\int_{\mathbb{R}}u^3\mathrm{d}\widehat{F}_n(u) - \expect E_{ij}^3\right|\sum_{j = 1}^n\frac{|\widehat{u}_{jk}|^3}{\widehat{s}_{ik}^3|\widehat{\lambda}_k|^3} + \expect E_{ij}^3\sum_{j = 1}^n\left|\frac{\widehat{u}_{jk}^3}{\widehat{s}_{ik}^3\widehat{\lambda}_k^3} - \frac{{u}_{jk}^3}{{s}_{ik}^3{\lambda}_k^3}\right|\\
&\quad = O\left\{\frac{(\log n)^\xi}{n}\right\} + O(\rho_n^{3/2})\sum_{j = 1}^n\left|\frac{\widehat{u}_{jk}}{\widehat{s}_{ik}\widehat{\lambda}_k} - \frac{{u}_{jk}}{{s}_{ik}{\lambda}_k}\right|\left(\frac{\widehat{u}_{jk}^2}{\widehat{s}_{ik}^2\widehat{\lambda}_k^2} + \frac{u_{jk}^2}{s_{ik}^2\lambda_k^2}\right)\\
&\quad = O\left\{\frac{(\log n)^\xi}{n}\right\} + O\left(\frac{\rho_n^{1/2}}{n}\right)\sum_{j = 1}^n\left|\frac{\widehat{u}_{jk}}{\widehat{s}_{ik}\widehat{\lambda}_k} - \frac{{u}_{jk}}{{s}_{ik}{\lambda}_k}\right|\\
&\quad = O\left\{\frac{(\log n)^\xi}{n}\right\}\\
&\quad\quad + O\left(\rho_n^{1/2}\right)\left\{\frac{\|\widehat{\bu}_k - \bu_k\|_\infty}{\widehat{s}_{ik}\widehat{\lambda}_k}
+ \frac{\|\bu_k\|_\infty|s_{ik} - \widehat{s}_{ik}||\lambda_k| + \|\bu_k\|_\infty\widehat{s}_{ik}|\lambda_k - \widehat{\lambda}_k|}{|\widehat{s}_{ik}\widehat{\lambda}_ks_{ik}\lambda_k|}
\right\}\\
&\quad = O\left\{\frac{\rho_n^{1/2}(\log n)^\xi}{\Delta_n} + \frac{(\log n)^{2\xi}}{q_n^2}\right\}
\end{align*}
with probability at least $1 - n^{-3/2}$ by Lemma \ref{lemma:Eigenvector_zero_order_deviation}, Weyl's inequality, and Result \ref{result:Noise_matrix_concentration}. The proof is thus completed.

\section{Proofs of Theorems \ref{thm:edgeworth_expansion_smoothed} and \ref{thm:bootstrap_GRDPG}}
\label{sec:proof_of_smoothed_edgeworth_expansion}

\begin{proof}[Proof of Theorem \ref{thm:edgeworth_expansion_smoothed}]
The proof sketch of Theorem \ref{thm:edgeworth_expansion_smoothed} is similar to that of Theorem \ref{thm:edgeworth_expansion}. In light of Esseen's smoothing lemma (Theorem \ref{thm:esseen_smoothing_lemma}), we first show that
\begin{align}
\label{eqn:CHF_smoothed_comparison}
\int_{-n^{2\beta_\Delta}}^{n^{2\beta_\Delta}}\left|\frac{\expect e^{\mathbbm{i}(\widetilde{T}_{ik} + \widetilde{\Delta}_{ik} + w_n)} - \mathrm{ch.f.}(t; G_n^{(ik)})}{t}\right|\mathrm{d}t = O\left\{\frac{n\rho_n\log n}{\Delta_n^2} + \frac{(\log n)^2}{q_n^2}\right\},
\end{align}
where $w_n\sim\mathrm{N}(0, (\tau^2\beta^2\log n)/(n\rho_n^3))$ is independent of $\bA$ and $\tau > 0$ is a sufficiently large constant. This entails that
\begin{align}
\label{eqn:edgeworth_expansion_smoothed_version}
\|\prob(\widetilde{T}_{ik} + \widetilde{\Delta}_{ik} + w_n\leq x) - G_n^{(ik)}(x)\|_\infty = O\left\{\frac{n\rho_n\log n}{\Delta_n^2} + \frac{(\log n)^2}{q_n^2}\right\}.
\end{align}
By Lemma \ref{lemma:CHF_II} and Lemma \ref{lemma:CHF_intermediate}, we have
\begin{align*}
&\int_{\{t:M_n(n^{2\beta_\Delta}\log n)^{1/2}\leq|t|\leq n^{2\beta_\Delta}\}}\left|\frac{\expect \exp\{\mathbbm{i}t(\widetilde{T}_{ik} + \widetilde{\Delta}_{ik} + w_n)\}}{t}\right|\mathrm{d}t\\
&\quad = \int_{\{t:M_n(n^{2\beta_\Delta}\log n)^{1/2}\leq|t|\leq n^{2\beta_\Delta}\}}\left|\frac{\expect \exp\{\mathbbm{i}t(\widetilde{T}_{ik} + \widetilde{\Delta}_{ik})\}}{t}\right|e^{-\frac{\tau^2t^2\beta^2\log n}{2n\rho_n^3}}\mathrm{d}t = O\left(\frac{1}{n}\right),\\
&\int_{\{t:{n^\gamma}\leq|t|\leq \bar{\eps}\sqrt{n\rho_n}(\rho_n/\beta)\}}\left|\frac{\expect\exp\{\mathbbm{i}t(\widetilde{T}_{ik} + \widetilde{\Delta}_{ik} + w_n)\}}{t}\right|\mathrm{d}t\\
&\quad = \int_{\{t:{n^\gamma}\leq|t|\leq \bar{\eps}\sqrt{n\rho_n}(\rho_n/\beta)\}}\left|\frac{\expect\exp\{\mathbbm{i}t(\widetilde{T}_{ik} + \widetilde{\Delta}_{ik})\}}{t}\right|e^{-\frac{\tau^2t^2\beta^2\log n}{2n\rho_n^3}}\mathrm{d}t\\
&\quad =  O\left\{\frac{(\log n)^2}{q_n^2} + \frac{n\rho_n\log n}{\Delta_n^2}\right\},
\end{align*}
where $\gamma = \min(\beta_\Delta/2, \eps/4)$, $\bar{\eps} > 0$ is a sufficiently small constant, and $M_n\to\infty$ sufficiently slow. For the integral over the gap $\{t:\bar{\eps}\sqrt{n\rho_n}(\rho_n/\beta)\leq|t|\leq M_n(n^{2\beta_\Delta}\log n)^{1/2}\}$, we can directly obtain
\begin{align*}
&\int_{\{t:\bar{\eps}\sqrt{n\rho_n}(\rho_n/\beta)\leq|t|\leq M_n(n^{2\beta_\Delta}\log n)^{1/2}\}}\left|\frac{\expect \exp\{\mathbbm{i}t(\widetilde{T}_{ik} + \widetilde{\Delta}_{ik} + w_n)\}}{t}\right|\mathrm{d}t\\
&\quad = \int_{\{t:\bar{\eps}\sqrt{n\rho_n}(\rho_n/\beta)\leq|t|\leq M_n(n^{2\beta_\Delta}\log n)^{1/2}\}}\left|\frac{1}{t}\expect \exp\left\{\mathbbm{i}t\widetilde{T}_{ik} - \frac{t^2}{2}\sigma^2_{ik}(\be_i\transpose\bE)\right\}\right|e^{-\frac{\tau^2t^2\beta^2\log n}{n\rho_n^3}}\mathrm{d}t\\
&\quad \leq \int_{\{t:\bar{\eps}\sqrt{n\rho_n}(\rho_n/\beta)\leq|t|\leq M_n(n^{2\beta_\Delta}\log n)^{1/2}\}}\frac{e^{-\frac{\tau^2t^2\beta^2\log n}{2n\rho_n^3}}}{|t|}\mathrm{d}t\\
&\quad\leq \exp\left(-\frac{\bar{\eps}^2\tau^2\log n}{2}\right)\int_{\{t:\bar{\eps}\sqrt{n\rho_n}(\rho_n/\beta)\leq|t|\leq M_n(n^{2\beta_\Delta}\log n)^{1/2}\}}\frac{\mathrm{d}t}{|t|}\\
&\quad\lesssim \log n\exp\left(-\frac{\bar{\eps}^2\tau^2\log n}{2}\right) = O\left(\frac{1}{n}\right)
\end{align*}
by taking a sufficiently large $\tau > 0$. For the integral over $t\in[-n^\gamma, n^\gamma]$, we have
\begin{align*}
&\int_{-n^\gamma}^{n^\gamma}\left|\frac{\expect e^{\mathbbm{i}t(\widetilde{T}_{ik} + \widetilde{\Delta}_{ik} + w_n)} - \mathrm{ch.f.}(t; G_n^{(ik)})}{t}\right|\mathrm{d}t\\
&\quad = \int_{-n^\gamma}^{n^\gamma}\left|\frac{\expect e^{\mathbbm{i}t(\widetilde{T}_{ik} + \widetilde{\Delta}_{ik})}\exp\{-(\tau^2t^2\beta^2\log n)/(2n\rho_n^3)\} - \mathrm{ch.f.}(t; G_n^{(ik)})}{t}\right|\mathrm{d}t.
\end{align*}
By Taylor's theorem, there exists some $\theta_n(t)\in\mathbb{R}$, $|\theta_n(t)|\leq 1$, such that 
\[
\exp\left(-\frac{\tau^2t^2\beta^2\log n}{2n\rho_n^3}\right) = 1 + \theta_n(t)\frac{\tau^2t^2\beta^2\log n}{2n\rho_n^3}.
\]
By Assumption \ref{assumption:Noise_matrix_distribution}, for any $t\in[-n^\gamma,n^\gamma]$ and $\gamma = \min(\beta_\Delta/2, \eps/4)$,
\[
\frac{|t|\beta^2\log n}{n\rho_n^3}\leq \frac{C^2n^{\eps/4}(n\rho_n)^3\log n}{n^3\rho_n^3q_n^2}\leq \frac{C^2n^{\eps/4}\log n}{q_n^2} = o(1).
\]
It follows that
\begin{align*}
&\int_{-n^\gamma}^{n^\gamma}\left|\frac{\expect e^{\mathbbm{i}t(\widetilde{T}_{ik} + \widetilde{\Delta}_{ik} + w_n)} - \mathrm{ch.f.}(t; G_n^{(ik)})}{t}\right|\mathrm{d}t\\
&\quad = \int_{-n^\gamma}^{n^\gamma}\left|\frac{1}{t}\left[\expect e^{\mathbbm{i}t(\widetilde{T}_{ik} + \widetilde{\Delta}_{ik})}\left\{1 + \theta_n(t)\frac{\tau^2t^2\beta^2\log n}{2n\rho_n^3}\right\}
 - \mathrm{ch.f.}(t; G_n^{(ik)})\right]\right|\mathrm{d}t\\
&\quad\leq \int_{-n^\gamma}^{n^\gamma}\left|\frac{1}{t}\left\{\expect e^{\mathbbm{i}t(\widetilde{T}_{ik} + \widetilde{\Delta}_{ik})}
 - \mathrm{ch.f.}(t; G_n^{(ik)})\right\}\right|\mathrm{d}t
 \\&\quad\quad
 + \int_{-n^\gamma}^{n^\gamma}\left|\left\{\expect e^{\mathbbm{i}t(\widetilde{T}_{ik} + \widetilde{\Delta}_{ik})}{\theta_n(t)}\frac{\tau^2t\beta^2\log n}{2n\rho_n^3}
 \right\}\right|\mathrm{d}t\\
&\quad\leq \int_{-n^\gamma}^{n^\gamma}\left|\frac{1}{t}\left\{\expect e^{\mathbbm{i}t(\widetilde{T}_{ik} + \widetilde{\Delta}_{ik})}
 - \mathrm{ch.f.}(t; G_n^{(ik)})\right\}\right|\mathrm{d}t
 \\&\quad\quad
 + \int_{-n^\gamma}^{n^\gamma}\left|\left[\left\{\expect e^{\mathbbm{i}t(\widetilde{T}_{ik} + \widetilde{\Delta}_{ik})} - \mathrm{ch.f.}(t; G_n^{(ik)})\right\}\theta_n(t)\frac{\tau^2t\beta^2\log n}{2n\rho_n^3}
 \right]\right|\mathrm{d}t
 \\&\quad\quad
 + \int_{-n^\gamma}^{n^\gamma}\left|\mathrm{ch.f.}(t; G_n^{(ik)}){\theta_n(t)}\frac{\tau^2t\beta^2\log n}{2n\rho_n^3}\right|\mathrm{d}t\\
&\quad\leq 2\int_{-n^\gamma}^{n^\gamma}\left|\frac{1}{t}\left\{\expect e^{\mathbbm{i}t(\widetilde{T}_{ik} + \widetilde{\Delta}_{ik})}
 - \mathrm{ch.f.}(t; G_n^{(ik)})\right\}\right|\mathrm{d}t
 + \int_{-n^\gamma}^{n^\gamma}\left|\mathrm{ch.f.}(t; G_n^{(ik)})\frac{\tau^2t\beta^2\log n}{2n\rho_n^3}\right|\mathrm{d}t.
\end{align*}
The first term above is $O\{(n\rho_n)/\Delta_n^2 + 1/q_n^2\}$ by Lemma \ref{lemma:CHF_IV}. The second term above is
\begin{align*}
&\int_{-n^\gamma}^{n^\gamma}\left|\mathrm{ch.f.}(t; G_n^{(ik)})\frac{\tau^2t\beta^2\log n}{2n\rho_n^3}\right|\mathrm{d}t\\
&\quad = \frac{\tau^2\beta^2\log n}{2n\rho_n^3}\int_{-n^\gamma}^{n^\gamma}\left|t\left\{e^{-t^2/2} + \frac{\mathbbm{i}t^3}{3}\kappa_n^{(ik)}e^{-t^2/2} - \frac{\mathbbm{i}t}{2}\kappa_n^{(ik)}e^{-t^2/2}\right\}\right|\mathrm{d}t\\
&\quad\leq \frac{C\tau^2\beta^2\log n}{n\rho_n^3}\int_{-n^\gamma}^{n^\gamma}|t|e^{-t^2/4}\mathrm{d}t = O\left(\frac{\log n}{q_n^2}\right),
\end{align*}
recalling that $\kappa_n^{(ik)} = \sum_{j = 1}^n{\expect E_{ij}^3u_{jk}^3}/{(s_{ik}\lambda_k)^3}$ and \eqref{chf:expansion}. Therefore, we conclude that
\[
\int_{-n^\gamma}^{n^\gamma}\left|\frac{\expect e^{\mathbbm{i}t(\widetilde{T}_{ik} + \widetilde{\Delta}_{ik} + w_n)} - \mathrm{ch.f.}(t; G_n^{(ik)})}{t}\right|\mathrm{d}t = O\left(\frac{n\rho_n}{\Delta_n^2} + \frac{\log n}{q_n^2}\right),
\]
thereby establishing \eqref{eqn:CHF_smoothed_comparison}. To connect $\widetilde{T}_{ik} + \widetilde{\Delta}_{ik} + w_n$ with $T_{ik} + w_n$, it is now sufficient to establish the analogous versions of \eqref{eqn:edgeworth_remainder_I} and \eqref{eqn:edgeworth_remainder_II}. Observe that for any $x\in\mathbb{R}$,
\begin{align*}
&\left|\prob(\widetilde{T}_{ik} + \Delta_{ik} + w_n\leq x) - \prob(\widetilde{T}_{ik} + \widetilde{\Delta}_{ik} + w_n\leq x)\right|\\
&\quad = \left|\expect_{\be_i\transpose\bE,w_n}\left\{\prob(\widetilde{T}_{ik} + \Delta_{ik} + w_n\leq x\mid \be_i\transpose\bE, w_n) - \prob(\widetilde{T}_{ik} + \widetilde{\Delta}_{ik} + w_n\leq x\mid \be_i\transpose\bE, w_n)\right\}\right|\\
&\quad\leq \expect_{\be_i\transpose\bE,w_n}\left|\prob(\widetilde{T}_{ik} + \Delta_{ik} + w_n\leq x\mid \be_i\transpose\bE, w_n) - \prob(\widetilde{T}_{ik} + \widetilde{\Delta}_{ik}\leq x\mid \be_i\transpose\bE, w_n)\right|\\
&\quad\leq \expect_{\be_i\transpose\bE, w_n}\left|\prob(\Delta_{ik}\leq x - \widetilde{T}_{ik} - w_n\mid \be_i\transpose\bE, w_n) - \prob(\widetilde{\Delta}_{ik}\leq x - \widetilde{T}_{ik} - w_n\mid \be_i\transpose\bE, w_n)\right|\\
&\quad\leq \expect_{\be_i\transpose\bE, w_n}\left\{\sup_{u\in\mathbb{R}}\left|\prob(\Delta_{ik}\leq u\mid \be_i\transpose\bE, w_n) - \prob(\widetilde{\Delta}_{ik}\leq u\mid \be_i\transpose\bE, w_n)\right|\right\}\\
&\quad = \expect_{\be_i\transpose\bE, w_n}\left\{\sup_{u\in\mathbb{R}}\left|\prob(\Delta_{ik}\leq u\mid \be_i\transpose\bE) - \prob(\widetilde{\Delta}_{ik}\leq u\mid \be_i\transpose\bE)\right|\right\}\\
&\quad = \int\widetilde{O}_{\prob}\left(\frac{1}{q_n^2}\right)\prob(\mathrm{d}\be_i\transpose\bE) = O\left(\frac{1}{q_n^2}\right),
\end{align*}
by Lemma~\ref{lemma:conditional_Berry_Esseen}, and thereby establishing the analogous version of \eqref{eqn:edgeworth_remainder_II}, i.e.,
\begin{align}
\label{eqn:edgeworth_remainder_smoothed_II}
\left\|\prob(\widetilde{T}_{ik} + \Delta_{ik} + w_n\leq x) - \prob(\widetilde{T}_{ik} + \widetilde{\Delta}_{ik} + w_n\leq x)\right\|_\infty = O\left(\frac{1}{q_n^2}\right).
\end{align}
To establish the analogous version of \eqref{eqn:edgeworth_remainder_I}, namely, 
\begin{equation}
\label{eqn:edgeworth_remainder_smoothed_I}
\begin{aligned}
&\left\|\prob\left(T_{ik} + w_n\leq x\right) - \prob\left(\widetilde{T}_{ik} + \Delta_{ik} + w_n\leq x\right)\right\|_\infty\\
&\quad = O\left\{\frac{(\log n)^{3\xi}}{n\rho_n^{1/2}} + \frac{(\log n)^{3\xi}}{\Delta_n} + \frac{(n\rho_n)(\log n)^{3\xi}}{\Delta_n^2} + \frac{(\log n)^{3\xi}}{q_n^2} + \frac{(n\rho_n)^2(\log n)^{3\xi}}{\Delta_n^3}\right\},
\end{aligned}
\end{equation}
we apply Lemma \ref{lemma:edgeworth_remainder} with $Y = \widetilde{T}_{ik} + w_n + \Delta_{ik}$ and $X = T_{ik} + w_n$. By \eqref{decom:T}, $Z = X - Y$ satisfies
\begin{align*}
&\prob\left[|Z|\geq C\left\{\frac{(\log n)^{3\xi}}{n\rho_n^{1/2}} + \frac{(\log n)^{3\xi}}{\Delta_n} + \frac{(n\rho_n)(\log n)^{3\xi}}{\Delta_n^2} + \frac{(\log n)^{3\xi}}{q_n^2} + \frac{(n\rho_n)^2(\log n)^{3\xi}}{\Delta_n^3}\right\}\right]\\
&\quad\leq \frac{1}{n}
\end{align*}
for some constant $C > 0$. 
To show that $F_Y(u):=\prob(Y\leq u)$ has the smoothness property required by Lemma \ref{lemma:edgeworth_remainder}, for any $u\in\mathbb{R}$ and $a > 0$, we apply \eqref{eqn:edgeworth_expansion_smoothed_version}, \eqref{eqn:edgeworth_remainder_smoothed_II}, and the fact that $\sup_{x,i,n}|\mathrm{d}G_n^{(ik)}(x)/\mathrm{d}x|\leq M$ for some constant $M > 0$ to obtain
\begin{align*}
&\prob\left(\widetilde{T}_{ik} + \Delta_{ik} + w_n\leq u + a\right) - \prob\left(\widetilde{T}_{ik} + \Delta_{ik} + w_n\leq u\right)\\
&\quad\leq \left|\prob\left(\widetilde{T}_{ik} + \Delta_{ik} + w_n\leq u + a\right) - \prob\left(\widetilde{T}_{ik} + \widetilde{\Delta}_{ik} + w_n\leq u + a\right)\right|\\
&\quad\quad + \left|\prob\left(\widetilde{T}_{ik} + \widetilde{\Delta}_{ik} + w_n\leq u + a\right) - G_n^{(ik)}(u + a)\right|\\
&\quad\quad + \left|G_n^{(ik)}(u + a) - G_n^{(ik)}(u)\right| + \left|G_n^{(ik)}(u) - \prob\left(\widetilde{T}_{ik} + \widetilde{\Delta}_{ik} + w_n\leq u\right)\right|\\
&\quad\quad + \left|\prob\left(\widetilde{T}_{ik} + \widetilde{\Delta}_{ik} + w_n\leq u\right) - \prob\left(\widetilde{T}_{ik} + \Delta_{ik} + w_n\leq u\right)\right|\\
&\quad\leq 2\left\|\prob\left(\widetilde{T}_{ik} + \Delta_{ik} + w_n\leq x\right) - \prob\left(\widetilde{T}_{ik} + \widetilde{\Delta}_{ik} + w_n\leq x\right)\right\|_\infty + \max_{i\in[n]}\sup_{x\in\mathbb{R}}\left|\frac{\mathrm{d}G_n^{(ik)}}{\mathrm{d}x}(x)\right|a\\
&\quad\quad + 2\left\|\prob\left(\widetilde{T}_{ik} + \widetilde{\Delta}_{ik} + w_n\leq x\right) - G_n^{(ik)}(x)\right\|_\infty\\
&\quad\leq Ma + O\left\{\frac{n\rho_n\log n}{\Delta_n^2} + \frac{(\log n)^2}{q_n^2}\right\}.
\end{align*}
Then \eqref{eqn:edgeworth_remainder_smoothed_I} follows from Lemma \ref{lemma:edgeworth_remainder} with
\begin{align*}
\widetilde{\zeta}_n& = C\left\{\frac{(\log n)^{3\xi}}{n\rho_n^{1/2}} + \frac{(\log n)^{3\xi}}{\Delta_n} + \frac{(n\rho_n)(\log n)^{3\xi}}{\Delta_n^2} + \frac{(\log n)^{3\xi}}{q_n^2} + \frac{(n\rho_n)^2(\log n)^{3\xi}}{\Delta_n^3}\right\}\\ 
\zeta_n& = \frac{n\rho_n\log n}{\Delta_n^2} + \frac{(\log n)^2}{q_n^2},
\end{align*}
thereby completing the proof.
\end{proof}

\begin{proof}[Proof of Theorem \ref{thm:bootstrap_GRDPG}]
The proof strategy is the same as that of Theorem \ref{thm:bootstrap_subgaussian}.
Suppose $\delta_0 > 0$ is a constant such that $\min_{k\in [d - 1]}(\lambda_k - \lambda_{k + 1}) \geq \delta_0n\rho_n$. 
 Let
\begin{align*}
\calE_{1n} &= \left\{\bA:\frac{1}{2}\lambda_{\min}(|\bS_\bP|)\leq \lambda_{\min}(|\bS_\bA|),\lambda_{\max}(|\bS_\bA|)\leq 2\lambda_{\max}(|\bS_\bP|)\right\},\\
\calE_{2n} &= \left\{\bA:\lambda_k(\bS_\bA) - \lambda_{k + 1}(\bS_\bA)\geq \frac{1}{2}\delta_0n\rho_n,k\in [d - 1]\right\},\\
\calE_{3n} &= \left\{\bA:\max_{k\in [d]}\|\widehat{\bu}_k\|_\infty \leq \frac{2C_\mu}{\sqrt{n}}\right\}.
\end{align*}
From the proof there, we know that $\calE_{1n}\cap\calE_{2n}\cap\calE_{3n}$ occurs w.h.p., over which Assumptions \ref{assumption:Signal_strength}, \ref{assumption:Eigenvalue_separation}, and \ref{assumption:Eigenvector_delocalization} hold with $(\bS_\bP,\lambda_k,\bu_k)$ replaced by $(\bS_\bA,\widehat{\lambda}_k,\widehat{\bu}_k)$. 

Now let
\[
\calE_{4n} = \left\{\bA:\frac{\delta_1}{2}\rho_n\leq \widehat{p}_{ij}\leq 2\delta_2\rho_n\right\}.
\]
Clearly, when $A_{ij}\sim\mathrm{Bernoulli}(p_{ij})$ independently for all $i\leq j$ with $p_{ij} = \Theta(\rho_n)$ uniformly over $i,j\in[n]$, Assumption \ref{assumption:Noise_matrix_distribution} holds with $q_n = \sqrt{n\rho_n}$ and any $\eps > 0$. Therefore, by \eqref{eqn:pij_zeroth_order_deviation}, event $\calE_{4n}$ occurs w.h.p.. It is also clear that $\beta = \max_{i,j\in[n]}\expect |E_{ij}|^3 = \Theta(\rho_n) = \Omega(\rho_n/\sqrt{\log n})$ and
\[
\min_{i\in[n]}s_{ik}^2 = \min_{i\in[n]}\sum_{j = 1}^n\frac{p_{ij}(1 - p_{ij})u_{jk}^2}{\lambda_k^2} = \Theta\left(\frac{1}{n^2\rho_n}\right).
\] 
Applying Theorem \ref{thm:edgeworth_expansion_smoothed} entails
\begin{align*}
\left\|\prob\left(T_{ik} + \tau\sqrt{\frac{\beta^2\log n}{n\rho_n^3}}z\leq x\right) - G_n^{(ik)}(x)\right\| = O\left\{\frac{(\log n)^{3\xi}}{n\rho_n}\right\}
\end{align*}
for a sufficiently large $\tau > 0$, where
\[
G_n^{(ik)}(x) = \Phi(x) + \frac{(2x^2 + 1)}{6}\phi(x)\sum_{j = 1}^n\frac{p_{ij}(1 - p_{ij})(1 - 2p_{ij})u_{jk}^3}{s_{ik}^3\lambda_k^3}.
\]
Furthermore, over the event $\calE_{4n}$, we know that
\[
\expect^*|E_{ij}^*|^p = \widehat{p}_{ij}(1 - \widehat{p}_{ij})^p + (1 - \widehat{p}_{ij})|\widehat{p}_{ij}|^p\leq 2\widehat{p}_{ij}\leq 4\max_{i,j\in[n]}p_{ij}\leq \frac{2^p(n\rho_n)^{p/2}}{n(n\rho_n)^{(p - 2)/2}}
\]
for any $p\geq 2$, so that Assumption \ref{assumption:Noise_matrix_distribution} holds with $(\expect, \bE)$ replaced by $(\expect^*, \bE^*)$ and $q_n:=\sqrt{n\rho_n}$ over the event $\calE_{4n}$. Also, $\widehat{\beta} = \max_{i,j\in[n]}\expect^*|E_{ij}^*|^3 = O(\rho_n) = \Omega(\rho_n/\sqrt{\log n})$ over $\calE_{4n}$. 
Applying Theorem \ref{thm:edgeworth_expansion_smoothed} with regard to the probability $\prob^*$ over the event $\calE_{1n}\cap\calE_{2n}\cap\calE_{3n}\cap\calE_{4n}$, we further obtain
\begin{align*}
\left\|\prob^*\left(T_{ik}^* + \tau\sqrt{\frac{\beta^2\log n}{n\rho_n^3}}z^*\leq x\right) - G_n^{(ik)*}(x)\right\| = O\left\{\frac{(\log n)^{3\xi}}{n\rho_n}\right\},
\end{align*}
where
\[
G_n^{(ik)*}(x) = \Phi(x) +\frac{(2x^2 + 1)}{6}\phi(x)\sum_{j = 1}^n\frac{\widehat{p}_{ij}(1 - \widehat{p}_{ij})(1 - 2\widehat{p}_{ij})\widehat{u}_{jk}^3}{\widehat{s}_{ik}^3\widehat{\lambda}_k^3}.
\]
Let $\psi(x) = x(1 - x)(1 - 2x)$. By mean-value theorem, for each $(i, j)\in[n]\times[n]$, there exists some $\bar{p}_{ij}$ between $p_{ij}$ and $\widehat{p}_{ij}$, such that
\begin{align*}
\max_{i,j\in[n]}|\psi(\widehat{p}_{ij}) - \psi(p_{ij})|
& \leq \max_{i,j\in[n]}|1 - 6\bar{p}_{ij} + 6\bar{p}_{ij}^2||\widehat{p}_{ij} - p_{ij}|
 = \widetilde{O}_{\prob}\left\{\frac{(\log n)^\xi\rho_n^{1/2}}{\sqrt{n}}\right\}
\end{align*}
by \eqref{eqn:pij_zeroth_order_deviation}. From the proof of Theorem \ref{thm:bootstrap_subgaussian}, we know that
\begin{align*}
&\max_{j\in[n]}\left|\frac{\widehat{u}_{jk}^3}{\widehat{s}_{ik}^3\widehat{\lambda}_k^3} - \frac{{u}_{jk}^3}{{s}_{ik}^3{\lambda}_k^3}\right|\\
&\quad\leq 2\max_{j\in[n]}\left|\frac{\widehat{u}_{jk}}{\widehat{s}_{ik}\widehat{\lambda}_k} - \frac{{u}_{jk}}{{s}_{ik}{\lambda}_k}\right|\left(\frac{\widehat{u}_{jk}^2}{\widehat{s}_{ik}^2\widehat{\lambda}_k^2} + \frac{u_{jk}^2}{s_{ik}^2\lambda_k^2}\right)\\
&\quad = \widetilde{O}_{\prob}\left(\frac{1}{n\rho_n}\right)\left\{\frac{\|\widehat{\bu}_k - \bu_k\|_\infty}{\widehat{s}_{ik}\widehat{\lambda}_k}
+ \frac{\|\bu_k\|_\infty|s_{ik} - \widehat{s}_{ik}||\lambda_k| + \|\bu_k\|_\infty\widehat{s}_{ik}|\lambda_k - \widehat{\lambda}_k|}{|\widehat{s}_{ik}\widehat{\lambda}_ks_{ik}\lambda_k|}
\right\}\\
&\quad = \widetilde{O}_{\prob}\left\{\frac{(\log n)^\xi}{(n\rho_n)^2}\right\}.
\end{align*}
Hence, we conclude that
\begin{align*}
&\|G_n^{(ik)}(x) - G_n^{(ik)*}(x)\|_\infty\\
&\quad \lesssim n\max_{j \in[n]}\left|\psi(\widehat{p}_{ij})\frac{\widehat{u}_{jk}^3}{\widehat{s}_{ik}^3\widehat{\lambda}_k^3} - \psi(p_{ij})\frac{{u}_{jk}^3}{{s}_{ik}^3{\lambda}_k^3}\right|\\
&\quad \leq n\max_{j \in[n]}\left|\psi(\widehat{p}_{ij})\left(\frac{\widehat{u}_{jk}^3}{\widehat{s}_{ik}^3\widehat{\lambda}_k^3} - \frac{{u}_{jk}^3}{{s}_{ik}^3{\lambda}_k^3}\right)\right|
 + n\max_{j \in[n]}\left|\{\psi(\widehat{p}_{ij}) - \psi(p_{ij})\}\frac{{u}_{jk}^3}{{s}_{ik}^3{\lambda}_k^3}\right|\\
&\quad = O\left\{\frac{(\log n)^\xi}{n\rho_n}\right\}
\end{align*}
with probability at least $1 - n^{-1}$, thereby completing the proof.
\end{proof}

\bibliographystyle{acm} 
\bibliography{reference1,reference2,reference_RMT}       

\begin{thebibliography}{100}

\bibitem{abbe2017community}
{\sc Abbe, E.}
\newblock Community detection and stochastic block models: recent developments.
\newblock {\em J. Mach. Learn. Res. 18}, 1 (2017), 6446--6531.

\bibitem{10.1214/19-AOS1854}
{\sc Abbe, E., Fan, J., Wang, K., and Zhong, Y.}
\newblock {Entrywise eigenvector analysis of random matrices with low expected
  rank}.
\newblock {\em Ann. Statist. 48}, 3 (2020), 1452--1474.

\bibitem{agterberg2021entrywise}
{\sc Agterberg, J., Lubberts, Z., and Priebe, C.~E.}
\newblock Entrywise estimation of singular vectors of low-rank matrices with
  heteroskedasticity and dependence.
\newblock {\em IEEE Trans. Inform. Theory 68}, 7 (2022), 4618--4650.

\bibitem{airoldi2008mixed}
{\sc Airoldi, E.~M., Blei, D.~M., Fienberg, S.~E., and Xing, E.~P.}
\newblock Mixed membership stochastic blockmodels.
\newblock {\em J. Mach. Learn. Res. 9\/} (2008), 1981--2014.

\bibitem{arnold1967asymptotic}
{\sc ARNOLD, L.}
\newblock On the asymptotic distribution of the eigenvalues of random matrices.
\newblock {\em J. Math. Anal. Appl. 20\/} (1967), 262--268.

\bibitem{arnold1971wigner}
{\sc Arnold, L.}
\newblock On wigner's semicircle law for the eigenvalues of random matrices.
\newblock {\em Probab. Theory Related Fields 19}, 3 (1971), 191--198.

\bibitem{arroyo2021inference}
{\sc Arroyo, J., Athreya, A., Cape, J., Chen, G., Priebe, C.~E., and
  Vogelstein, J.~T.}
\newblock Inference for multiple heterogeneous networks with a common invariant
  subspace.
\newblock {\em J. Mach. Learn. Res. 22}, 142 (2021), 1--49.

\bibitem{athreya2016limit}
{\sc Athreya, A., Priebe, C.~E., Tang, M., Lyzinski, V., Marchette, D.~J., and
  Sussman, D.~L.}
\newblock A limit theorem for scaled eigenvectors of random dot product graphs.
\newblock {\em Sankhya A 78}, 1 (2016), 1--18.

\bibitem{10.2307/25050481}
{\sc Babu, G.~J., and Singh, K.}
\newblock On one term edgeworth correction by efron's bootstrap.
\newblock {\em Sankhya A 46}, 2 (1984), 219--232.

\bibitem{bai2010spectral}
{\sc Bai, Z., and Silverstein, J.~W.}
\newblock {\em Spectral analysis of large dimensional random matrices},
  vol.~20.
\newblock Springer, 2010.

\bibitem{AIHPB_2008__44_3_447_0}
{\sc Bai, Z., and Yao, J.-F.}
\newblock Central limit theorems for eigenvalues in a spiked population model.
\newblock {\em Ann. Inst. H. Poincar\'e Probab. Statist. 44}, 3 (2008),
  447--474.

\bibitem{10.1214/009117905000000233}
{\sc Baik, J., Arous, G.~B., and P{\'e}ch{\'e}, S.}
\newblock {Phase transition of the largest eigenvalue for nonnull complex
  sample covariance matrices}.
\newblock {\em Ann. Probab. 33}, 5 (2005), 1643--1697.

\bibitem{10.1214/20-AOS1960}
{\sc Bao, Z., Ding, X., , and Wang, K.}
\newblock {Singular vector and singular subspace distribution for the matrix
  denoising model}.
\newblock {\em The Annals of Statistics 49}, 1 (2021), 370--392.

\bibitem{10.1214/21-AOS2143}
{\sc Bao, Z., Ding, X., Wang, J., and Wang, K.}
\newblock {Statistical inference for principal components of spiked covariance
  matrices}.
\newblock {\em The Annals of Statistics 50}, 2 (2022), 1144--1169.

\bibitem{BENAYCHGEORGES2011494}
{\sc Benaych-Georges, F., and Nadakuditi, R.~R.}
\newblock The eigenvalues and eigenvectors of finite, low rank perturbations of
  large random matrices.
\newblock {\em Adv. Math. 227}, 1 (2011), 494--521.

\bibitem{BENAYCHGEORGES2012120}
{\sc Benaych-Georges, F., and Nadakuditi, R.~R.}
\newblock The singular values and vectors of low rank perturbations of large
  rectangular random matrices.
\newblock {\em J. Multivariate Anal. 111\/} (2012), 120--135.

\bibitem{bennett2007netflix}
{\sc Bennett, J., Lanning, S., et~al.}
\newblock The {N}etflix {P}rize.
\newblock In {\em Proceedings of KDD cup and workshop\/} (2007), vol.~2007, New
  York, NY, USA., p.~35.

\bibitem{10.1214/aos/1031833676}
{\sc Bentkus, V., G{\"o}tze, F., and van Zwet, W.~R.}
\newblock {An Edgeworth expansion for symmetric statistics}.
\newblock {\em Ann. Statist. 25}, 2 (1997), 851--896.

\bibitem{bhatia2013matrix}
{\sc Bhatia, R.}
\newblock {\em Matrix analysis}, vol.~169.
\newblock Springer Science \& Business Media, 2013.

\bibitem{10.1214/aos/1176344134}
{\sc Bhattacharya, R.~N., and Ghosh, J.~K.}
\newblock {On the Validity of the Formal Edgeworth Expansion}.
\newblock {\em Ann. Statist. 6}, 2 (1978), 434--451.

\bibitem{10.1214/aos/1176342609}
{\sc Bickel, P.~J.}
\newblock {Edgeworth Expansions in Nonparametric Statistics}.
\newblock {\em Ann. Statist. 2}, 1 (1974), 1--20.

\bibitem{10.1214/aos/1176350170}
{\sc Bickel, P.~J., Gotze, F., and van Zwet, W.~R.}
\newblock {The Edgeworth Expansion for $U$-Statistics of Degree Two}.
\newblock {\em Ann. Statist. 14}, 4 (1986), 1463--1484.

\bibitem{blei2003latent}
{\sc Blei, D.~M., Ng, A.~Y., and Jordan, M.~I.}
\newblock Latent dirichlet allocation.
\newblock {\em J. Mach. Learn. Res. 3}, Jan (2003), 993--1022.

\bibitem{Bloznelis2022}
{\sc Bloznelis, M., and G{\"o}tze, F.}
\newblock Edgeworth approximations for distributions of symmetric statistics.
\newblock {\em Probab. Theory Related Fields 183}, 3 (2022), 1153--1235.

\bibitem{bourgade2017eigenvector}
{\sc Bourgade, P., and Yau, H.-T.}
\newblock The eigenvector moment flow and local quantum unique ergodicity.
\newblock {\em Commun. Math. Phys. 350}, 1 (2017), 231--278.

\bibitem{https://doi.org/10.1002/cpa.21895}
{\sc Bourgade, P., Yau, H.-T., and Yin, J.}
\newblock Random band matrices in the delocalized phase i: Quantum unique
  ergodicity and universality.
\newblock {\em Comm. Pure Appl. Math. 73}, 7 (2020), 1526--1596.

\bibitem{BURA20082275}
{\sc Bura, E., and Pfeiffer, R.}
\newblock On the distribution of the left singular vectors of a random matrix
  and its applications.
\newblock {\em Stat. Probab. Lett. 78}, 15 (2008), 2275--2280.

\bibitem{10.1214/20-AOS1986}
{\sc Cai, C., Li, G., Chi, Y., Poor, H.~V., and Chen, Y.}
\newblock {Subspace estimation from unbalanced and incomplete data matrices:
  ${\ell _{2,\infty }}$ statistical guarantees}.
\newblock {\em Ann. Statist. 49}, 2 (2021), 944--967.

\bibitem{10.1214/17-AOS1541}
{\sc Cai, T.~T., and Zhang, A.}
\newblock {Rate-optimal perturbation bounds for singular subspaces with
  applications to high-dimensional statistics}.
\newblock {\em Ann. Statist. 46}, 1 (2018), 60--89.

\bibitem{10.1214/aos/1176344955}
{\sc Callaert, H., Janssen, P., and Veraverbeke, N.}
\newblock {An Edgeworth Expansion for $U$-Statistics}.
\newblock {\em Ann. Statist. 8}, 2 (1980), 299--312.

\bibitem{tight_oracle_inequalities}
{\sc Candes, E.~J., and Plan, Y.}
\newblock Tight oracle inequalities for low-rank matrix recovery from a minimal
  number of noisy random measurements.
\newblock {\em IEEE Trans. Inform. Theory 57}, 4 (April 2011), 2342--2359.

\bibitem{candes2009exact}
{\sc Cand{\`e}s, E.~J., and Recht, B.}
\newblock Exact matrix completion via convex optimization.
\newblock {\em Found. Comput. Math. 9}, 6 (2009), 717--772.

\bibitem{6545395}
{\sc Cand\'es, E.~J., Sing-Long, C.~A., and Trzasko, J.~D.}
\newblock Unbiased risk estimates for singular value thresholding and spectral
  estimators.
\newblock {\em IEEE Trans. Signal Process. 61}, 19 (2013), 4643--4657.

\bibitem{candes2010power}
{\sc Cand{\`e}s, E.~J., and Tao, T.}
\newblock The power of convex relaxation: Near-optimal matrix completion.
\newblock {\em IEEE Trans. Inform. Theory 56}, 5 (2010), 2053--2080.

\bibitem{cape2019signal}
{\sc Cape, J., Tang, M., and Priebe, C.~E.}
\newblock Signal-plus-noise matrix models: eigenvector deviations and
  fluctuations.
\newblock {\em Biometrika 106}, 1 (2019), 243--250.

\bibitem{cape2017two}
{\sc Cape, J., Tang, M., and Priebe, C.~E.}
\newblock {The two-to-infinity norm and singular subspace geometry with
  applications to high-dimensional statistics}.
\newblock {\em Ann. Statist. 47}, 5 (2019), 2405--2439.

\bibitem{Capitaine2018}
{\sc Capitaine, M.}
\newblock {\em Limiting Eigenvectors of Outliers for Spiked
  Information-Plus-Noise Type Matrices}.
\newblock Springer International Publishing, Cham, 2018, pp.~119--164.

\bibitem{10.1214/08-AOP394}
{\sc Capitaine, M., Donati-Martin, C., and F{\'e}ral, D.}
\newblock {The largest eigenvalues of finite rank deformation of large Wigner
  matrices: Convergence and nonuniversality of the fluctuations}.
\newblock {\em Ann. Probab. 37}, 1 (2009), 1--47.

\bibitem{chatterjee2015}
{\sc Chatterjee, S.}
\newblock Matrix estimation by universal singular value thresholding.
\newblock {\em Ann. Statist. 43}, 1 (02 2015), 177--214.

\bibitem{MAL-079}
{\sc Chen, Y., Chi, Y., Fan, J., and Ma, C.}
\newblock Spectral methods for data science: A statistical perspective.
\newblock {\em Found. Trends Mach. Learn. 14}, 5 (2021), 566--806.

\bibitem{doi:10.1073/pnas.1910053116}
{\sc Chen, Y., Fan, J., Ma, C., and Yan, Y.}
\newblock Inference and uncertainty quantification for noisy matrix completion.
\newblock {\em Proc. Natl. Acad. Sci. U.S.A. 116}, 46 (2019), 22931--22937.

\bibitem{doi:10.1080/01621459.2022.2089574}
{\sc Chen, Y., He, S., Yang, Y., and Liang, F.}
\newblock Learning topic models: Identifiability and finite-sample analysis.
\newblock {\em J. Amer. Statist. Assoc. 0}, 0 (2022), 1--16.

\bibitem{8811622}
{\sc Chi, Y., Lu, Y.~M., and Chen, Y.}
\newblock Nonconvex optimization meets low-rank matrix factorization: An
  overview.
\newblock {\em IEEE Trans. Signal Process. 67}, 20 (2019), 5239--5269.

\bibitem{chung2001course}
{\sc Chung, K.~L.}
\newblock {\em A {C}ourse in {P}robability {T}heory}.
\newblock Academic press, New York-London, 2001.

\bibitem{doi:10.1137/0707001}
{\sc Davis, C., and Kahan, W.~M.}
\newblock The rotation of eigenvectors by a perturbation. iii.
\newblock {\em SIAM J. Numer. Anal. 7}, 1 (1970), 1--46.

\bibitem{davison_hinkley_1997}
{\sc Davison, A.~C., and Hinkley, D.~V.}
\newblock {\em Bootstrap Methods and their Application}.
\newblock Cambridge Series in Statistical and Probabilistic Mathematics.
  Cambridge University Press, 1997.

\bibitem{dekel2007eigenvectors}
{\sc Dekel, Y., Lee, J.~R., and Linial, N.}
\newblock Eigenvectors of random graphs: Nodal domains.
\newblock In {\em Approximation, Randomization, and Combinatorial Optimization.
  Algorithms and Techniques}. Springer, 2007, pp.~436--448.

\bibitem{devroye2013probabilistic}
{\sc Devroye, L., Gy{\"o}rfi, L., and Lugosi, G.}
\newblock {\em A probabilistic theory of pattern recognition}, vol.~31.
\newblock Springer Science \& Business Media, 2013.

\bibitem{10.3150/19-BEJ1129}
{\sc Ding, X.}
\newblock {High dimensional deformed rectangular matrices with applications in
  matrix denoising}.
\newblock {\em Bernoulli 26}, 1 (2020), 387--417.

\bibitem{1614066}
{\sc Donoho, D.}
\newblock Compressed sensing.
\newblock {\em IEEE Trans. Inform. Theory 52}, 4 (2006), 1289--1306.

\bibitem{10.1214/14-AOS1257}
{\sc Donoho, D., and Gavish, M.}
\newblock {Minimax risk of matrix denoising by singular value thresholding}.
\newblock {\em Ann. Statist. 42}, 6 (2014), 2413--2440.

\bibitem{Edgeworth1905}
{\sc Edgeworth, F.~Y.}
\newblock The law of error.
\newblock {\em Nature 36\/} (1887), 482--483.

\bibitem{efron1994introduction}
{\sc Efron, B., and Tibshirani, R.~J.}
\newblock {\em An introduction to the bootstrap}.
\newblock Chapman and Hall/CRC, New York, NY, 1994.

\bibitem{eichler2017complete}
{\sc Eichler, K., Li, F., Litwin-Kumar, A., Park, Y., Andrade, I.,
  Schneider-Mizell, C.~M., Saumweber, T., Huser, A., Eschbach, C., Gerber, B.,
  et~al.}
\newblock The complete connectome of a learning and memory centre in an insect
  brain.
\newblock {\em Nature 548}, 7666 (2017), 175--182.

\bibitem{10.1214/009117906000000917}
{\sc {El Karoui}, N.}
\newblock {Tracy–Widom limit for the largest eigenvalue of a large class of
  complex sample covariance matrices}.
\newblock {\em Ann. Probab. 35}, 2 (2007), 663--714.

\bibitem{elad2010sparse}
{\sc Elad, M.}
\newblock {\em Sparse and redundant representations: from theory to
  applications in signal and image processing}, vol.~2.
\newblock Springer, 2010.

\bibitem{eldar2012compressed}
{\sc Eldar, Y.~C., and Kutyniok, G.}
\newblock {\em Compressed sensing: theory and applications}.
\newblock {C}ambridge {U}niversity {P}ress, Cambridge, 2012.

\bibitem{pmlr-v83-eldridge18a}
{\sc Eldridge, J., Belkin, M., and Wang, Y.}
\newblock Unperturbed: spectral analysis beyond davis-kahan.
\newblock In {\em Proceedings of Algorithmic Learning Theory\/} (07--09 Apr
  2018), F.~Janoos, M.~Mohri, and K.~Sridharan, Eds., vol.~83 of {\em Proc.
  Mach. Learn. Res. (PMLR)}, PMLR, pp.~321--358.

\bibitem{erdHos2013delocalization}
{\sc Erd{\H{o}}s, L., Knowles, A., Yau, H.-T., and Yin, J.}
\newblock Delocalization and diffusion profile for random band matrices.
\newblock {\em Communications in Mathematical Physics 323}, 1 (2013), 367--416.

\bibitem{erdos2013}
{\sc Erd{\"o}s, L., Knowles, A., Yau, H.-T., and Yin, J.}
\newblock Spectral statistics of {E}rd{\"o}s-{R}{\'e}nyi graphs i: Local
  semicircle law.
\newblock {\em Ann. Probab. 41}, 3B (05 2013), 2279--2375.

\bibitem{ERDOS20121435}
{\sc Erd{\H{o}}s, L., Yau, H.-T., and Yin, J.}
\newblock Rigidity of eigenvalues of generalized wigner matrices.
\newblock {\em Adv. Math. 229}, 3 (2012), 1435--1515.

\bibitem{10.1007/BF02392223}
{\sc Esseen, C.-G.}
\newblock {Fourier analysis of distribution functions. A mathematical study of
  the Laplace-Gaussian law}.
\newblock {\em Acta Math. 77}, none (1945), 1--125.

\bibitem{doi:10.1080/01621459.2020.1840990}
{\sc Fan, J., Fan, Y., Han, X., and Lv, J.}
\newblock Asymptotic theory of eigenvectors for random matrices with diverging
  spikes.
\newblock {\em J. Amer. Statist. Assoc. 0}, 0 (2020), 1--14.

\bibitem{10.1214/aos/1176345638}
{\sc Freedman, D.~A.}
\newblock {Bootstrapping Regression Models}.
\newblock {\em Ann. Statist. 9}, 6 (1981), 1218--1228.

\bibitem{furedi1981eigenvalues}
{\sc F{\"u}redi, Z., and Koml{\'o}s, J.}
\newblock The eigenvalues of random symmetric matrices.
\newblock {\em Combinatorica 1}, 3 (1981), 233--241.

\bibitem{10.5555/3122009.3153016}
{\sc Gao, C., Ma, Z., Zhang, A.~Y., and Zhou, H.~H.}
\newblock Achieving optimal misclassification proportion in stochastic block
  models.
\newblock {\em J. Mach. Learn. Res. 18}, 1 (jan 2017), 1980–2024.

\bibitem{6846297}
{\sc Gavish, M., and Donoho, D.~L.}
\newblock The optimal hard threshold for singular values is $4/\sqrt {3}$.
\newblock {\em IEEE Trans. Inform. Theory 60}, 8 (2014), 5040--5053.

\bibitem{https://doi.org/10.1002/jcc.540140115}
{\sc Glunt, W., Hayden, T., and Raydan, M.}
\newblock Molecular conformations from distance matrices.
\newblock {\em J. Comput. Chem. 14}, 1 (1993), 114--120.

\bibitem{goldberg1992using}
{\sc Goldberg, D., Nichols, D., Oki, B.~M., and Terry, D.}
\newblock Using collaborative filtering to weave an information tapestry.
\newblock {\em Commun. ACM 35}, 12 (1992), 61--70.

\bibitem{10.1214/aop/1176992073}
{\sc Hall, P.}
\newblock {Edgeworth Expansion for Student's $t$ Statistic Under Minimal Moment
  Conditions}.
\newblock {\em Ann. Probab. 15}, 3 (1987), 920--931.

\bibitem{hall2013bootstrap}
{\sc Hall, P.}
\newblock {\em The bootstrap and {E}dgeworth expansion}.
\newblock Springer Science \& Business Media, New York NY, 2013.

\bibitem{10.1214/aos/1176347994}
{\sc Helmers, R.}
\newblock {On the Edgeworth Expansion and the Bootstrap Approximation for a
  Studentized $U$-Statistic}.
\newblock {\em Ann. Statist. 19}, 1 (1991), 470--484.

\bibitem{10.1145/312624.312649}
{\sc Hofmann, T.}
\newblock Probabilistic latent semantic indexing.
\newblock In {\em Proceedings of the 22nd Annual International ACM SIGIR
  Conference on Research and Development in Information Retrieval\/} (New York,
  NY, USA, 1999), SIGIR '99, Association for Computing Machinery, pp.~50--57.

\bibitem{HOLLAND1983109}
{\sc Holland, P.~W., Laskey, K.~B., and Leinhardt, S.}
\newblock Stochastic blockmodels: First steps.
\newblock {\em Soc. Netw. 5}, 2 (1983), 109--137.

\bibitem{PhysRevE.83.016107}
{\sc Karrer, B., and Newman, M. E.~J.}
\newblock Stochastic blockmodels and community structure in networks.
\newblock {\em Phys. Rev. E 83\/} (Jan 2011), 016107.

\bibitem{5466511}
{\sc Keshavan, R.~H., Montanari, A., and Oh, S.}
\newblock Matrix completion from a few entries.
\newblock {\em IEEE Trans. Inform. Theory 56}, 6 (2010), 2980--2998.

\bibitem{https://doi.org/10.1002/cpa.21450}
{\sc Knowles, A., and Yin, J.}
\newblock The isotropic semicircle law and deformation of wigner matrices.
\newblock {\em Commun. Pur. Appl. Math. 66}, 11 (2013), 1663--1749.

\bibitem{10.1214/13-AOP855}
{\sc Knowles, A., and Yin, J.}
\newblock {The outliers of a deformed Wigner matrix}.
\newblock {\em Ann. Probab. 42}, 5 (2014), 1980--2031.

\bibitem{knowles2017anisotropic}
{\sc Knowles, A., and Yin, J.}
\newblock Anisotropic local laws for random matrices.
\newblock {\em Probab. Theory Related Fields 169}, 1 (2017), 257--352.

\bibitem{LAHIRI1993247}
{\sc Lahiri, S.}
\newblock Bootstrapping the studentized sample mean of lattice variables.
\newblock {\em J. Multivariate Anal. 45}, 2 (1993), 247--256.

\bibitem{10.2307/24304972}
{\sc Lai, T.~L., and Wang, J.~Q.}
\newblock {E}dgeworth {E}xpansions for {S}ymmetric {S}tatistics with
  {A}pplications to {B}ootstrap {M}ethods.
\newblock {\em Stat. Sin. 3}, 2 (1993), 517--542.

\bibitem{1407873}
{\sc Lee, K.-C., Ho, J., and Kriegman, D.~J.}
\newblock Acquiring linear subspaces for face recognition under variable
  lighting.
\newblock {\em IEEE Trans. Pattern Anal. Mach. Intell. 27}, 5 (May 2005),
  684--698.

\bibitem{10.1093/bioinformatics/btt091}
{\sc Lee, S., Chugh, P.~E., Shen, H., Eberle, R., and Dittmer, D.~P.}
\newblock Poisson factor models with applications to non-normalized microrna
  profiling.
\newblock {\em Bioinformatics 29}, 9 (02 2013), 1105--1111.

\bibitem{10.1214/15-AOS1370}
{\sc Lei, J.}
\newblock {A goodness-of-fit test for stochastic block models}.
\newblock {\em Ann. Statist. 44}, 1 (2016), 401--424.

\bibitem{levin2019bootstrapping}
{\sc Levin, K., and Levina, E.}
\newblock Bootstrapping networks with latent space structure.
\newblock {\em arXiv preprint:1907.10821\/} (2019).

\bibitem{8007083}
{\sc Lu, Y.~M., and Li, G.}
\newblock Spectral initialization for nonconvex estimation: High-dimensional
  limit and phase transitions.
\newblock In {\em 2017 IEEE International Symposium on Information Theory
  (ISIT)\/} (2017), pp.~3015--3019.

\bibitem{lyzinski2014}
{\sc Lyzinski, V., Sussman, D.~L., Tang, M., Athreya, A., and Priebe, C.~E.}
\newblock Perfect clustering for stochastic blockmodel graphs via adjacency
  spectral embedding.
\newblock {\em Electron. J. Stat. 8}, 2 (2014), 2905--2922.

\bibitem{ImplicitRegularizationMa2020}
{\sc Ma, C., Wang, K., Chi, Y., and Chen, Y.}
\newblock Implicit regularization in nonconvex statistical estimation: Gradient
  descent converges linearly for phase retrieval, matrix completion, and blind
  deconvolution.
\newblock {\em Found. Comput. Math. 20}, 3 (2020), 451--632.

\bibitem{doi:10.1080/01621459.2020.1751645}
{\sc Mao, X., Sarkar, P., and Chakrabarti, D.}
\newblock Estimating mixed memberships with sharp eigenvector deviations.
\newblock {\em J. Amer. Statist. Assoc. 116}, 536 (2021), 1928--1940.

\bibitem{marchenko1967distribution}
{\sc Marchenko, V.~A., and Pastur, L.~A.}
\newblock Distribution of eigenvalues for some sets of random matrices.
\newblock {\em Math. USSR Sb. 114}, 4 (1967), 507--536.

\bibitem{mykland1989edgeworth}
{\sc Mykland, P.}
\newblock Edgeworth and bootstrap methods for dependent variables (ph. d.
  thesis, university of california, berkeley).

\bibitem{10.1214/aos/1176348649}
{\sc Mykland, P.~A.}
\newblock {Asymptotic Expansions and Bootstrapping Distributions for Dependent
  Variables: A Martingale Approach}.
\newblock {\em Ann. Statist. 20}, 2 (1992), 623--654.

\bibitem{10.2307/2244676}
{\sc Mykland, P.~A.}
\newblock Asymptotic expansions for martingales.
\newblock {\em Ann. Probab. 21}, 2 (1993), 800--818.

\bibitem{10.1214/aos/1176324617}
{\sc Mykland, P.~A.}
\newblock {Martingale Expansions and Second Order Inference}.
\newblock {\em Ann. Statist. 23}, 3 (1995), 707--731.

\bibitem{https://doi.org/10.1111/j.1468-0262.2004.00482.x}
{\sc Newey, W.~K., and Smith, R.~J.}
\newblock Higher order properties of {GMM} and generalized empirical likelihood
  estimators.
\newblock {\em Econometrica 72}, 1 (2004), 219--255.

\bibitem{nickel2008random}
{\sc Nickel, C. L.~M.}
\newblock {\em Random dot product graphs a model for social networks}.
\newblock PhD thesis, Johns Hopkins University, 2008.

\bibitem{OROURKE201826}
{\sc O'Rourke, S., Vu, V., and Wang, K.}
\newblock Random perturbation of low rank matrices: Improving classical bounds.
\newblock {\em Linear Algebra Appl. 540\/} (2018), 26--59.

\bibitem{paulsen2002completely}
{\sc Paulsen, V.}
\newblock {\em Completely bounded maps and operator algebras}.
\newblock No.~78. Cambridge University Press, Cambridge, 2002.

\bibitem{PhysRevX.4.011047}
{\sc Peixoto, T.~P.}
\newblock Hierarchical block structures and high-resolution model selection in
  large networks.
\newblock {\em Phys. Rev. X 4\/} (Mar 2014), 011047.

\bibitem{petrov2012sums}
{\sc Petrov, V.~V.}
\newblock {\em Sums of independent random variables}, vol.~82.
\newblock Springer Science \& Business Media, 2012.

\bibitem{AIHPB_2013__49_1_64_0}
{\sc Pizzo, A., Renfrew, D., and Soshnikov, A.}
\newblock On finite rank deformations of {Wigner} matrices.
\newblock {\em Ann. Inst. H. Poincar\'e Probab. Statist. 49}, 1 (2013), 64--94.

\bibitem{10.1214/aoms/1177729989}
{\sc Quenouille, M.~H.}
\newblock {Problems in Plane Sampling}.
\newblock {\em Ann. Math. Statist. 20}, 3 (1949), 355--375.

\bibitem{10.1093/biomet/43.3-4.353}
{\sc QUENOUILLE, M.~H.}
\newblock Notes on {B}ias in {E}stimation.
\newblock {\em Biometrika 43}, 3-4 (12 1956), 353--360.

\bibitem{doi:10.1142/S2010326312500153}
{\sc Renfrew, D., and Soshnikov, A.}
\newblock On finite rank deformations of {W}igner matrices {II}: Delocalized
  perturbations.
\newblock {\em Random Matrices: Theory Appl. 02}, 01 (2013), 1250015.

\bibitem{rohe2011}
{\sc Rohe, K., Chatterjee, S., and Yu, B.}
\newblock Spectral clustering and the high-dimensional stochastic blockmodel.
\newblock {\em Ann. Statist. 39}, 4 (08 2011), 1878--1915.

\bibitem{rubin2017statistical}
{\sc Rubin-Delanchy, P., Cape, J., Tang, M., and Priebe, C.~E.}
\newblock {A Statistical Interpretation of Spectral Embedding: The Generalised
  Random Dot Product Graph}.
\newblock {\em J. R. Stat. Soc. Ser. B Methodol. 84}, 4 (06 2022), 1446--1473.

\bibitem{10.1215/00127094-3129809}
{\sc Rudelson, M., and Vershynin, R.}
\newblock {Delocalization of eigenvectors of random matrices with independent
  entries}.
\newblock {\em Duke Math. J. 164}, 13 (2015), 2507--2538.

\bibitem{rudelson2016no}
{\sc Rudelson, M., and Vershynin, R.}
\newblock No-gaps delocalization for general random matrices.
\newblock {\em Geom. Funct. Anal. 26}, 6 (2016), 1716--1776.

\bibitem{1143830}
{\sc Schmidt, R.}
\newblock Multiple emitter location and signal parameter estimation.
\newblock {\em IEEE Trans. Antennas Propag. 34}, 3 (1986), 276--280.

\bibitem{https://doi.org/10.1111/rssb.12245}
{\sc Sengupta, S., and Chen, Y.}
\newblock A block model for node popularity in networks with community
  structure.
\newblock {\em J. R. Stat. Soc. Ser. B Methodol. 80}, 2 (2018), 365--386.

\bibitem{SHABALIN201367}
{\sc Shabalin, A.~A., and Nobel, A.~B.}
\newblock Reconstruction of a low-rank matrix in the presence of gaussian
  noise.
\newblock {\em J. Multivariate Anal. 118\/} (2013), 67--76.

\bibitem{868688}
{\sc Shi, J., and {Malik}, J.}
\newblock Normalized cuts and image segmentation.
\newblock {\em IEEE Trans. Pattern Anal. Mach. Intell. 22}, 8 (Aug 2000),
  888--905.

\bibitem{10.1214/aos/1176345636}
{\sc Singh, K.}
\newblock {On the Asymptotic Accuracy of Efron's Bootstrap}.
\newblock {\em Ann. Statist. 9}, 6 (1981), 1187--1195.

\bibitem{10.2307/4615841}
{\sc Skovgaard, I.~M.}
\newblock Edgeworth expansions of the distributions of maximum likelihood
  estimators in the general (non i.i.d.) case.
\newblock {\em Scand. J. Stat. 8}, 4 (1981), 227--236.

\bibitem{10.2307/4615839}
{\sc Skovgaard, I.~M.}
\newblock Transformation of an edgeworth expansion by a sequence of smooth
  functions.
\newblock {\em Scand. J. Stat. 8}, 4 (1981), 207--217.

\bibitem{stein2010complex}
{\sc Stein, E.~M., and Shakarchi, R.}
\newblock {\em Complex analysis}, vol.~2.
\newblock Princeton University Press, 2010.

\bibitem{sussman2012consistent}
{\sc Sussman, D.~L., Tang, M., Fishkind, D.~E., and Priebe, C.~E.}
\newblock A consistent adjacency spectral embedding for stochastic blockmodel
  graphs.
\newblock {\em J. Amer. Statist. Assoc. 107}, 499 (2012), 1119--1128.

\bibitem{tang2017}
{\sc Tang, M., Athreya, A., Sussman, D.~L., Lyzinski, V., and Priebe, C.~E.}
\newblock A nonparametric two-sample hypothesis testing problem for random
  graphs.
\newblock {\em Bernoulli 23}, 3 (08 2017), 1599--1630.

\bibitem{tang2018}
{\sc Tang, M., and Priebe, C.~E.}
\newblock Limit theorems for eigenvectors of the normalized {L}aplacian for
  random graphs.
\newblock {\em Ann. Statist. 46}, 5 (10 2018), 2360--2415.

\bibitem{8570772}
{\sc {Tang}, R., {Ketcha}, M., {Badea}, A., {Calabrese}, E.~D., {Margulies},
  D.~S., {Vogelstein}, J.~T., {Priebe}, C.~E., and {Sussman}, D.~L.}
\newblock Connectome smoothing via low-rank approximations.
\newblock {\em IEEE Trans. Med. Imaging 38}, 6 (June 2019), 1446--1456.

\bibitem{tao2012topics}
{\sc Tao, T.}
\newblock {\em Topics in random matrix theory}.
\newblock American Mathematical Society, Providence, RI, 2012.

\bibitem{212753}
{\sc Tufts, D., and Shah, A.}
\newblock Estimation of a signal waveform from noisy data using low-rank
  approximation to a data matrix.
\newblock {\em IEEE Trans. Signal Process. 41}, 4 (1993), 1716--1721.

\bibitem{10.1093/bioinformatics/17.suppl_1.S279}
{\sc Venet, D., Pecasse, F., Maenhaut, C., and Bersini, H.}
\newblock Separation of samples into their constituents using gene expression
  data.
\newblock {\em Bioinformatics 17}, suppl\_1 (06 2001), S279--S287.

\bibitem{vershynin2018high}
{\sc Vershynin, R.}
\newblock {\em High-dimensional probability, volume 47 of Cambridge Series in
  Statistical and Probabilistic Mathematics}.
\newblock Cambridge University Press, Cambridge, Cambridge, 2018.

\bibitem{10.2307/2237255}
{\sc Wallace, D.~L.}
\newblock Asymptotic approximations to distributions.
\newblock {\em Ann. Math. Statist. 29}, 3 (1958), 635--654.

\bibitem{wedin1972perturbation}
{\sc Wedin, P.-{\AA}.}
\newblock Perturbation bounds in connection with singular value decomposition.
\newblock {\em BIT 12}, 1 (1972), 99--111.

\bibitem{PhysRev.98.145}
{\sc Wigner, E.~P.}
\newblock Lower limit for the energy derivative of the scattering phase shift.
\newblock {\em Phys. Rev. 98\/} (Apr 1955), 145--147.

\bibitem{wu2022statistical}
{\sc Wu, D., and Xie, F.}
\newblock Statistical inference of random graphs with a surrogate likelihood
  function.
\newblock {\em arXiv preprint:2207.01702\/} (2022).

\bibitem{xie2021entrywise}
{\sc Xie, F.}
\newblock {Entrywise limit theorems for eigenvectors of signal-plus-noise
  matrix models with weak signals}.
\newblock {\em Bernoulli 30}, 1 (2024), 388--418.

\bibitem{xie2019efficient}
{\sc Xie, F., and Xu, Y.}
\newblock Efficient estimation for random dot product graphs via a one-step
  procedure.
\newblock {\em J. Amer. Statist. Assoc. 118}, 541 (2023), 651--664.

\bibitem{10.1214/17-AOAS1123}
{\sc Xie, F., Zhou, M., and Xu, Y.}
\newblock {BayCount: A Bayesian decomposition method for inferring tumor
  heterogeneity using RNA-Seq counts}.
\newblock {\em Ann. Appl. Stat. 12}, 3 (2018), 1605--1627.

\bibitem{yau2012universality}
{\sc Yau, H.-T.}
\newblock {675Universality of generalized Wigner matrices}.
\newblock In {\em {Quantum Theory from Small to Large Scales: Lecture Notes of
  the Les Houches Summer School: Volume 95, August 2010}}. Oxford University
  Press, 05 2012.

\bibitem{young2007random}
{\sc Young, S.~J., and Scheinerman, E.~R.}
\newblock Random dot product graph models for social networks.
\newblock In {\em International Workshop on Algorithms and Models for the
  Web-Graph\/} (2007), Springer, pp.~138--149.

\bibitem{doi:10.1093/biomet/asv008}
{\sc Yu, Y., Wang, T., and Samworth, R.~J.}
\newblock A useful variant of the davis–kahan theorem for statisticians.
\newblock {\em Biometrika 102}, 2 (2015), 315--323.

\bibitem{zhang2022perturbation}
{\sc Zhang, Y., and Tang, M.}
\newblock Perturbation analysis of randomized svd and its applications to
  high-dimensional statistics.
\newblock {\em arXiv preprint arXiv:2203.10262\/} (2022).

\bibitem{10.1214/21-AOS2125}
{\sc Zhang, Y., and Xia, D.}
\newblock Edgeworth expansions for network moments.
\newblock {\em Ann. Statist. 50}, 2 (2022), 726--753.

\end{thebibliography}

\end{document}